\documentclass[letterpaper,twoside,onecolumn,openright,10pt,reqno,extrafontsizes]{memoir}

\makeatletter{}%

\usepackage{amsmath}
\usepackage{amssymb}
\usepackage{amsthm}

\allowdisplaybreaks

\usepackage{fourier}
\usepackage{bm}
\usepackage{mathrsfs}
\usepackage{wasysym}
\usepackage{latexsym}
\usepackage{mathdots}
\usepackage{mathtools}

\usepackage[dvips]{graphicx}
\usepackage{xcolor}
\usepackage{tikz}

\usepackage{url}

\usepackage{makeidx}
\makeindex

\usepackage{hyperref}
\usepackage{memhfixc}

\definecolor{dkgray}{rgb}{.4,.4,.4}
\definecolor{dkblue}{rgb}{0,0,.4}
\definecolor{medblue}{rgb}{0,0,.75}
\definecolor{rust}{rgb}{0.6,0.15,0.15}

\hypersetup{
    pdftitle={Matrix Concentration Inequalities},    %
    pdfauthor={Joel A.~Tropp},     %
    pdfkeywords={random matrix theory} {large deviations} {probability inequality}, %
    colorlinks=true,       %
    linkcolor=medblue,          %
    citecolor=rust,        %
    filecolor=magenta,      %
    urlcolor=rust           %
}

\usepackage{etex}

\usepackage{fixltx2e}

\usepackage{alltt}

\usepackage{fontenc}

\usepackage{alg}

\makeatletter
\def\algspacing{\alg@unmargin}
\makeatother

\newtheorem{thm}{Theorem}
\newtheorem{cor}[thm]{Corollary}
\newtheorem{lemma}[thm]{Lemma}
\newtheorem{prop}[thm]{Proposition}

\newtheorem{defn}[thm]{Definition}

\newtheorem{example}[thm]{Example}

\newtheorem{fact}[thm]{Fact}

\theoremstyle{remark}

\DeclareMathAlphabet{\mathsfsl}{OT1}{cmss}{m}{sl}

\newcommand{\lang}{\textit}

\newcommand{\term}[1]{\emph{#1}\index{#1}}

	{\makebox{\phantom{}} \\ %
		\textsc{Input:}\begin{itemize}}%
	{\end{itemize}}

	{\textsc{Output:}\begin{itemize}}%
	{\end{itemize}}

	{\textsc{Procedure:}\begin{enumerate}}%
	{\end{enumerate}}

\renewcommand{\phi}{\varphi}

\newcommand{\eps}{\varepsilon}

\newcommand{\half}{\tfrac{1}{2}}

\newcommand{\econst}{\mathrm{e}}
\newcommand{\iunit}{\mathrm{i}}

\newcommand{\Id}{\mathbf{I}}

\newcommand{\coll}[1]{\mathscr{#1}}

\newcommand{\R}{\mathbb{R}}
\newcommand{\C}{\mathbb{C}}

\newcommand{\abs}[1]{\left\vert {#1} \right\vert}
\newcommand{\abssq}[1]{{\abs{#1}}^2}

\newcommand{\sgn}[1]{\operatorname{sgn}{#1}}
\newcommand{\real}{\operatorname{Re}}

\newcommand{\diff}[1]{\mathrm{d}{#1}}
\newcommand{\idiff}[1]{\, \diff{#1}}

\newcommand{\grad}{\nabla}

\newcommand{\Prob}[1]{\mathbb{P}\left\{ {#1} \right\}}
\newcommand{\Expect}{\operatorname{\mathbb{E}}}

\newcommand{\uniform}{\textsc{uniform}}
\newcommand{\normal}{\textsc{normal}}

\newcommand{\Var}{\operatorname{Var}}
\newcommand{\mVar}{\operatorname{\mathbf{Var}}}

\newcommand{\vct}[1]{\bm{#1}}
\newcommand{\mtx}[1]{\bm{#1}}

\newcommand{\adj}{*}

\newcommand{\rank}{\operatorname{rank}}

\newcommand{\diag}{\operatorname{diag}}
\newcommand{\trace}{\operatorname{tr}}

\newcommand{\psdle}{\preccurlyeq}
\newcommand{\psdge}{\succcurlyeq}

\newcommand{\psdlt}{\prec}
\newcommand{\psdgt}{\succ}

\newcommand{\ip}[2]{\left\langle {#1},\ {#2} \right\rangle}

\newcommand{\abssqip}[2]{\abssq{\ip{#1}{#2}}}

\newcommand{\norm}[1]{\left\Vert {#1} \right\Vert}
\newcommand{\normsq}[1]{\norm{#1}^2}

\newcommand{\enorm}[1]{\norm{#1}_2}
\newcommand{\enormsq}[1]{\enorm{#1}^2}

\newcommand{\fnorm}[1]{\norm{#1}_{\mathrm{F}}}
\newcommand{\fnormsq}[1]{\fnorm{#1}^2}

\newcommand{\pnorm}[2]{\norm{#2}_{#1}}

\newcommand{\conv}{\operatorname{conv}}

\newcommand{\bigO}{\operatorname{O}}

\newcommand{\strank}{\operatorname{srank}}

\newcommand{\intdim}{\operatorname{intdim}}

\settrimmedsize{11in}{210mm}{*}
\settypeblocksize{7.75in}{33pc}{*}
\setulmargins{4cm}{*}{*}
\setlrmargins{1.25in}{*}{*}
\setmarginnotes{17pt}{51pt}{\onelineskip}
\setheadfoot{\onelineskip}{2\onelineskip}
\setheaderspaces{*}{2\onelineskip}{*}
\checkandfixthelayout

\newcommand{\asudedication}[1]{%
{\clearpage\mbox{}\vfill\centering #1 \par\vfill\clearpage}}

\begin{document}

\pagestyle{headings}

\headstyles{default}

\chapterstyle{veelo}

\numberwithin{equation}{section}
\numberwithin{thm}{section}
\numberwithin{exer}{section}

\frontmatter

\title{An Introduction to \\ Matrix Concentration Inequalities}
\author{Joel A.~Tropp}
\date{24 December 2014 \\ FnTML Draft, Revised}

\maketitle
I%

\asudedication{For Margot and Benjamin}

\setsecnumdepth{subsection}
\maxsecnumdepth{subsection}

\newpage

\tableofcontents

\makeatletter{}%

\chapter{Preface}

In recent years, random matrices have come to play a major role in computational mathematics, but most of the classical areas of random matrix theory remain the province of experts.  Over the last decade, with the advent of matrix concentration inequalities, research has advanced to the point where we can conquer many (formerly) challenging problems with a page or two of arithmetic.

My aim is to describe the most successful methods from this area along with some interesting examples that these techniques can illuminate.  I hope that the results in these pages will inspire future work on applications of random matrices as well as refinements of the matrix concentration inequalities discussed herein.

I have chosen to present a coherent body of results based on a generalization of the Laplace transform method for establishing scalar concentration inequalities.  In the last two years, Lester Mackey and I, together with our coauthors, have developed an alternative approach to matrix concentration using exchangeable pairs and Markov chain couplings.  With some regret, I have chosen to omit this theory because the ideas seem less accessible to a broad audience of researchers.  The interested reader will find pointers to these articles in the annotated bibliography.

The work described in these notes reflects the influence of many researchers.  These include Rudolf Ahlswede, Rajendra Bhatia, Eric Carlen, Sourav Chatterjee, Edward Effros, Elliott Lieb, Roberto Imbuzeiro Oliveira, D{\' e}nes Petz, Gilles Pisier, Mark Rudelson, Roman Vershynin, and Andreas Winter.  I have also learned a great deal from other colleagues and friends along the way.

I would like to thank some people who have helped me improve this work.
Several readers informed me about errors in the initial version of this manuscript;
these include Serg Bogdanov, Peter Forrester, Nikos Karampatziakis, and Guido Lagos.
The anonymous reviewers tendered many useful suggestions,
and they pointed out a number of errors.
Sid Barman gave me feedback on the final revisions to the monograph.
Last, I want to thank L{\'e}on Nijensohn for his continuing encouragement.

I gratefully acknowledge financial support from the Office of Naval Research under awards N00014-08-1-0883 and N00014-11-1002, the Air Force Office of Strategic Research under award FA9550-09-1-0643, and an Alfred P.~Sloan Fellowship.  Some of this research was completed at the Institute of Pure and Applied Mathematics at UCLA.  I would also like to thank the California Institute of Technology and the Moore Foundation.

\vspace{1pc}

{\flushright

Joel A.~Tropp \\
Pasadena, CA \\
December 2012 \\
Revised, March 2014 and December 2014 \\}

\mainmatter

\makeatletter{}%

\chapter{Introduction}

Random matrix theory has grown into a vital area of probability, and it has found applications in many other fields.  To motivate the results in this monograph, we begin with an overview of the connections between random matrix theory and computational mathematics.  We introduce the basic ideas underlying our approach, and we state one of our main results on the behavior of random matrices.  As an application, we examine the properties of the sample covariance estimator, a random matrix that arises in statistics.  Afterward, we summarize the other types of results that appear in these notes, and we assess the novelties in this presentation.

\section{Historical Origins}

Random matrix theory sprang from several different sources in the first half of the 20th century.

\begin{description}
\item	[Geometry of Numbers.]  Peter Forrester~\cite[p.~v]{For10:Log-Gases} traces the field of random matrix theory to work of Hurwitz, who defined the invariant integral over a Lie group.  Specializing this analysis to the orthogonal group, we can reinterpret this integral as the expectation of a function of a uniformly random orthogonal matrix.

\item	[Multivariate Statistics.]  Another early example of a random matrix appeared in the work of John Wishart~\cite{Wis28:Generalised-Product}.  Wishart was studying the behavior of the sample covariance estimator for the covariance matrix of a multivariate normal random vector.  He showed that the estimator, which is a random matrix, has the distribution that now bears his name.  Statisticians have often used random matrices as models for multivariate data~\cite{MKB79:Multivariate-Analysis,Mui82:Aspects-Multivariate}.

\item	[Numerical Linear Algebra.]  In their remarkable work~\cite{NG47:Numerical-Inverting,NG51:Numerical-Inverting-II} on computational methods for solving systems of linear equations, von Neumann and Goldstine considered a random matrix model for the floating-point errors that arise from an \textsf{LU} decomposition.%
\footnote{von Neumann and Goldstine invented and analyzed this algorithm before they had any digital computer on which to implement it!  See~\cite{Grc11:John-von-Neumanns} for a historical account.}
They obtained a high-probability bound for the norm of the random matrix, which they took as an estimate for the error the procedure might typically incur.  Curiously, in subsequent years, numerical linear algebraists became very suspicious of probabilistic techniques, and only in recent years have randomized algorithms reappeared in this field.
See the surveys~\cite{Mah11:Randomized-Algorithms,HMT11:Finding-Structure,Woo14:Sketching-Tool} for more details and references.

\item	[Nuclear Physics.]  In the early 1950s, physicists had reached the limits of deterministic analytical techniques for studying the energy spectra of heavy atoms undergoing slow nuclear reactions.  Eugene Wigner was the first researcher to surmise that a random matrix with appropriate symmetries might serve as a suitable model for the Hamiltonian of the quantum mechanical system that describes the reaction.  The eigenvalues of this random matrix model the possible energy levels of the system.  See Mehta's book~\cite[\S1.1]{Meh04:Random-Matrices} for an account of all this.
\end{description}

\noindent
In each area, the motivation was quite different and led to distinct sets of questions.
Later, random matrices began to percolate into other fields such as graph theory (the Erd{\H os}--R{\'e}nyi model~\cite{ER60:Evolution-Random} for a random graph) and number theory (as a model for the spacing of zeros of the Riemann zeta function~\cite{Mon73:Pair-Correlation}).

\section{The Modern Random Matrix}

By now, random matrices are ubiquitous.  They arise throughout modern mathematics and statistics, as well as in many branches of science and engineering.  Random matrices have several different purposes that we may wish to distinguish.  They can be used within randomized computer algorithms; they serve as models for data and for physical phenomena; and they are subjects of mathematical inquiry.  This section offers a taste of these applications.  Note that the ideas and references here reflect the author's interests, and they are far from comprehensive!

\subsection{Algorithmic Applications}

The striking mathematical properties of random matrices can be harnessed to develop algorithms for solving many different problems.

\begin{description}
\item	[Computing Matrix Approximations.]  Random matrices can be used to develop fast algorithms for computing a truncated singular-value decomposition.  In this application, we multiply a large input matrix by a smaller random matrix to extract information about the dominant singular vectors of the input matrix.  The seed of this idea appears in~\cite{FKV98:Fast-Monte-Carlo,DFK+99:Clustering-Large}.  The survey~\cite{HMT11:Finding-Structure} explains how to implement this method in practice, while the two monographs~\cite{Mah11:Randomized-Algorithms, Woo14:Sketching-Tool} cover more theoretical aspects.

\item	[Sparsification.]  One way to accelerate spectral computations on large matrices is to replace the original matrix by a sparse proxy that has similar spectral properties.  An elegant way to produce the sparse proxy is to zero out entries of the original matrix at random while rescaling the entries that remain.  This approach was proposed in~\cite{AM01:Fast-Computation,AM07:Fast-Computation}, and the papers~\cite{AKL13:Near-Optimal-Entrywise,KD14:Note-Randomized} contain recent innovations.  Related ideas play an important role in Spielman and Teng's work~\cite{ST04:Nearly-Linear-Time} on fast algorithms for solving linear systems.

\item	[Subsampling of Data.]  In large-scale machine learning, one may need to subsample data randomly to reduce the computational costs of fitting a model.  For instance, we can combine random sampling with the Nystr{\"o}m decomposition to obtain a randomized approximation of a kernel matrix.  This method was introduced by Williams \& Seeger~\cite{WS01:Nystrom-Method}.  The paper~\cite{DM05:Nystrom-Method} provides the first theoretical analysis, and the survey~\cite{GM13:Revisiting-Nystrom} contains more complete results.

\item	[Dimension Reduction.]  A basic template in the theory of algorithms invokes randomized projection to reduce the dimension of a computational problem.  Many types of dimension reduction are based on properties of random matrices.  The two papers~\cite{JL84:Extensions-Lipschitz,Bou85:Lipschitz-Embedding} established the mathematical foundations of this approach.  The earliest applications in computer science appear in the work~\cite{LLR95:Geometry-Graphs}.  Many contemporary variants depend on ideas from~\cite{AC09:Fast-Johnson-Lindenstrauss} and~\cite{CW13:Low-Rank-Approximation}.

\item	[Combinatorial Optimization.]  One approach to solving a computationally difficult optimization problem is to relax (i.e., enlarge) the constraint set so the problem becomes tractable, to solve the relaxed problem, and then to use a randomized procedure to map the solution back to the original constraint set~\cite[\S4.3]{BN01:Lectures-Modern}.  This technique is called \term{relaxation and rounding}.  For hard optimization problems involving a matrix variable, the analysis of the rounding procedure often involves ideas from random matrix theory~\cite{So09:Moment-Inequalities, NRV13:Efficient-Rounding}.

\item	[Compressed Sensing.]  When acquiring data about an object with relatively few degrees of freedom as compared with the ambient dimension, we may be able to sieve out the important information from the object by taking a small number of random measurements, where the number of measurements is comparable to the number of degrees of freedom~\cite{GGIMS02:Near-Optimal-Sparse,CRT06:Robust-Uncertainty,Don06:Compressed-Sensing}.  This observation is now referred to as \term{compressed sensing}.   Random matrices play a central role in the design and analysis of measurement procedures.  For example, see \cite{RF13:Mathematical-Introduction,CRPW12:Convex-Geometry,ALMT14:Living-Edge,Tro14:Convex-Recovery}.

\end{description}

\subsection{Modeling}

Random matrices also appear as models for multivariate data or multivariate phenomena.  By studying the properties of these models, we may hope to understand the typical behavior of a data-analysis algorithm or a physical system.

\begin{description}
\item	[Sparse Approximation for Random Signals.]  Sparse approximation has become an important problem in statistics, signal processing, machine learning and other areas.  One model for a ``typical'' sparse signal poses the assumption that the nonzero coefficients that generate the signal are chosen at random.  When analyzing methods for identifying the sparse set of coefficients, we must study the behavior of a random column submatrix drawn from the model matrix~\cite{Tro08:Conditioning-Random,Tro08:Norms-Random}.

\item	[Demixing of Structured Signals.]	In data analysis, it is common to encounter a mixture of two structured signals, and the goal is to extract the two signals using prior information about the structures.  A common model for this problem assumes that the signals are randomly oriented with respect to each other, which means that it is usually possible to discriminate the underlying structures.  Random orthogonal matrices arise in the analysis of estimation techniques for this problem~\cite{MT12:Sharp-Recovery,ALMT14:Living-Edge,MT13:Achievable-Performance}.

\item	[Stochastic Block Model.]  One probabilistic framework for describing community structure in a network assumes that each pair of individuals in the same community has a relationship with high probability, while each pair of individuals drawn from different communities has a relationship with lower probability.  This is referred to as the \term{stochastic block model}~\cite{HLL83:Stochastic-Block}.  It is quite common to analyze algorithms for extracting community structure from data by positing that this model holds.  See~\cite{ABH14:Exact-Recovery} for a recent contribution, as well as a summary of the extensive literature.

\item	[High-Dimensional Data Analysis.]  More generally, random models are pervasive in the analysis of statistical estimation procedures for high-dimensional data.  Random matrix theory plays a key role in this field~\cite{MKB79:Multivariate-Analysis,Mui82:Aspects-Multivariate,Kol11:Oracle-Inequalities,BG11:Statistics-High-Dimensional}.

\item	[Wireless Communication.]  Random matrices are commonly used as models for wireless channels.  See the book of Tulino and Verd{\'u} for more information~\cite{TV04:Random-Matrix}.
\end{description}

\noindent
In these examples, it is important to recognize that random models may not coincide very well with reality, but they allow us to get a sense of what might be possible in some generic cases.

\subsection{Theoretical Aspects}

Random matrices are frequently studied for their intrinsic mathematical interest.
In some fields, they provide examples of striking phenomena.  In other areas,
they furnish counterexamples to ``intuitive'' conjectures.  
Here are a few disparate problems where random matrices play a role.

\begin{description}
\item	[Combinatorics.]  An expander graph has the property that every small set of vertices has edges linking it to a large proportion of the vertices.  The expansion property is closely related to the spectral behavior of the adjacency matrix of the graph.  The easiest construction of an expander involves a random matrix argument~\cite[\S9.2]{AS00:Probabilistic-Method}.

\item	[Numerical Analysis.]  For worst-case examples, the Gaussian elimination method for solving a linear system is not numerically stable.  In practice, however, stability problems rarely arise.  One explanation for this phenomenon is that, with high probability, a small random perturbation of any fixed matrix is well conditioned.  As a consequence, it can be shown that Gaussian elimination is stable for most matrices~\cite{SST06:Smoothed-Analysis}.

\item	[High-Dimensional Geometry.]  Dvoretzky's Theorem states that, when $N$ is large, the unit ball of each $N$-dimensional Banach space has a slice of dimension $n \approx \log N$ that is close to a Euclidean ball with dimension $n$.  It turns out that a \emph{random} slice of dimension $n$ realizes this property~\cite{Mil71:New-Proof}.  This result can be framed as a statement about spectral properties of a random matrix~\cite{Gor85:Some-Inequalities}.

\item	[Quantum Information Theory.]  Random matrices appear as counterexamples for a number of conjectures in quantum information theory.  Here is one instance.  In classical information theory, the total amount of information that we can transmit through a pair of channels equals the sum of the information we can send through each channel separately.  It was conjectured that the same property holds for quantum channels.  In fact, a pair of quantum channels can have strictly larger capacity than a single channel.  This result depends on a random matrix construction~\cite{Has09:Superadditivity-Communication}.
See~\cite{HW08:Counterexamples-Maximal} for related work.
\end{description}

\section{Random Matrices for the People}

Historically, random matrix theory has been regarded as a very challenging field.  Even now, many well-established methods are only comprehensible to researchers with significant experience, and it may take months of intensive effort to prove new results.
There are a small number of classes of random matrices that have been studied so completely that we know almost everything about them.  Yet, moving beyond this \lang{terra firma}, one quickly encounters examples where classical methods are brittle.

We hope to democratize random matrix theory.  These notes describe tools that deliver useful information about a wide range of random matrices.  In many cases, a modest amount of straightforward arithmetic leads to strong results.  The methods here should be accessible to computational scientists working in a variety of fields.
Indeed, the techniques in this work have already found an extensive number of applications.

\section{Basic Questions in Random Matrix Theory}

Random matrices merit special attention because they have spectral properties
that are quite different from familiar deterministic matrices.  Here are
some of the questions we might want to investigate.

\begin{itemize}
\item	What is the expectation of the maximum eigenvalue of a random Hermitian matrix?  What about the minimum eigenvalue?

\item	How is the maximum eigenvalue of a random Hermitian matrix distributed?  What is the probability that it takes values substantially different from its mean? What about the minimum eigenvalue?

\item	What is the expected spectral norm of a random matrix?  What is the probability that the norm takes a value substantially different from its mean?

\item	What about the other eigenvalues or singular values?  Can we say something about the ``typical'' spectrum of a random matrix?

\item	Can we say anything about the eigenvectors or singular vectors?  For instance, is each one distributed almost uniformly on the sphere? %

\item	We can also ask questions about the operator norm of a random matrix acting as a map between two normed linear spaces.  In this case, the geometry of the domain and codomain play a role.
\end{itemize}

\noindent
In this work, we focus on the first three questions above.  We study the expectation of the extreme eigenvalues of a random Hermitian matrix, and we attempt to provide bounds on the probability that they take an unusual value.  As an application of these results, we can control the expected spectral norm of a general matrix and bound the probability of a large deviation.  These are the most relevant problems in many (but not all!) applications.
The remaining questions are also important, but we will not touch on them here.  We recommend the book~\cite{Tao12:Topics-Random} for a friendly introduction to other branches of random matrix theory.

\section{Random Matrices as Independent Sums}

Our approach to random matrices depends on a fundamental principle:

\begin{quotation}
\noindent
\textbf{In applications, it is common that a random matrix can be expressed as a sum of independent random matrices.}
\end{quotation}

\noindent
The examples that appear in these notes should provide ample evidence for this claim.  For now, let us describe a specific problem that will serve as an illustration throughout the Introduction.  We hope this example is complicated enough to be interesting but simple enough to elucidate the main points.

\subsection{Example: The Sample Covariance Estimator}

Let $\vct{x} = (X_1, \dots, X_p)$ be a complex random vector with zero mean: $\Expect \vct{x} = \vct{0}$.
The \term{covariance matrix} $\mtx{A}$ of the random vector $\vct{x}$ is the positive-semidefinite matrix
\begin{equation} \label{eqn:covariance-mtx}
\mtx{A} = \Expect( \vct{xx}^\adj ) = \sum_{j,k=1}^p \Expect\big( X_j X^\adj_k \big) \, \mathbf{E}_{jk}
\end{equation}
The star ${}^\adj$ refers to the conjugate transpose operation, and
the standard basis matrix $\mathbf{E}_{jk}$ has a one in the $(j, k)$ position and zeros elsewhere.
In other words, the $(j,k)$ entry of the sample covariance matrix $\mtx{A}$ records the covariance between
the $j$th and $k$th entry of the vector $\vct{x}$.

One basic problem in statistical practice is to estimate the covariance matrix from data.  Imagine that we have access to $n$ independent samples $\vct{x}_1, \dots, \vct{x}_n$, each distributed the same way as $\vct{x}$.  The \term{sample covariance estimator} $\mtx{Y}$ is the random matrix
\begin{equation} \label{eqn:sample-covariance-mtx}
\mtx{Y} = \frac{1}{n} \sum_{k=1}^n \vct{x}_k \vct{x}_k^\adj.
\end{equation}
The random matrix $\mtx{Y}$ is an unbiased estimator%
\footnote{The formula~\eqref{eqn:sample-covariance-mtx} supposes that the random vector $\vct{x}$ is known to have zero mean.  Otherwise, we have to make an adjustment to incorporate an estimate for the sample mean.}
for the sample covariance matrix: $\Expect \mtx{Y} = \mtx{A}$.
Observe that the sample covariance estimator $\mtx{Y}$ fits neatly into our paradigm:

\begin{quotation}
\noindent
\textbf{The sample covariance estimator can be expressed as a sum of independent random matrices.}
\end{quotation}

\noindent
This is precisely the type of decomposition that allows us to apply the tools in these notes.

\section{Exponential Concentration Inequalities for Matrices}

An important challenge in probability theory is to study the
probability that a real random variable $Z$ takes a value substantially
different from its mean.  That is, we seek 
a bound of the form
\begin{equation} \label{eqn:scalar-concentration}
\Prob{ \abs{ Z - \Expect Z } \geq t } \leq \underline{\hspace{1pc} ??? \hspace{1pc}}
\end{equation}
for a positive parameter $t$.
When $Z$ is expressed as a sum of independent random variables,
the literature contains many tools for addressing this problem.
See~\cite{BLM13:Concentration-Inequalities} for an overview.

For a random matrix $\mtx{Z}$, a variant of~\eqref{eqn:scalar-concentration}
is the question of whether $\mtx{Z}$ deviates substantially from
its mean value.  We might frame this question as
\begin{equation} \label{eqn:matrix-concentration}
\Prob{ \norm{ \mtx{Z} - \Expect \mtx{Z} } \geq t }
\leq \underline{\hspace{1pc} ??? \hspace{1pc}}.
\end{equation}
Here and elsewhere, $\norm{\cdot}$ denotes the spectral norm %
of a matrix.
As noted, it is frequently possible to decompose $\mtx{Z}$
as a sum of independent random matrices.  We might even dream
that established methods for studying the scalar concentration
problem~\eqref{eqn:scalar-concentration} extend
to~\eqref{eqn:matrix-concentration}.  %

\subsection{The Bernstein Inequality}

To explain what kind of results we have in mind, let us return to the scalar problem~\eqref{eqn:scalar-concentration}.
First, to simplify formulas, we assume that the real random variable $Z$ has zero mean: $\Expect Z = 0$.  If not, we can simply center the random variable by subtracting its mean.  Second, and more restrictively, we suppose that $Z$ can be expressed as a sum of independent, real random variables.

To control $Z$, we rely on two types of information: global properties of the sum (such as its mean and variance) and local properties of the summands (such as their maximum fluctuation).  These pieces of data are usually easy to obtain.  Together, they determine how well $Z$ concentrates around zero, its mean value.

\begin{thm}[Bernstein Inequality] \label{thm:scalar-bernstein}
Let $S_1, \dots, S_n$ be independent, centered, real random variables,
and assume that each one is uniformly bounded:
$$
\Expect S_k = 0
\quad\text{and}\quad
\abs{ S_k } \leq L
\quad\text{for each $k = 1, \dots, n$.}
$$
Introduce the sum $Z = \sum_{k=1}^n S_k$,
and let $v(Z)$ denote the variance of the sum:
$$
v(Z) = \Expect{} Z^2
	= \sum_{k=1}^n \Expect{} S_k^2.
$$
Then
$$
\Prob{ \abs{ Z } \geq t }
	\leq 2\, \exp\left( \frac{-t^2/2}{v(Z) + Lt/3} \right)
	\quad\text{for all $t \geq 0$.}
$$
\end{thm}

\noindent
See~\cite[\S2.8]{BLM13:Concentration-Inequalities} for a proof of this result.
We refer to Theorem~\ref{thm:scalar-bernstein} as an \term{exponential concentration inequality}
because it yields exponentially decaying bounds on the probability that $Z$ deviates substantially
from its mean.

\subsection{The Matrix Bernstein Inequality}

What is truly astonishing is that the scalar Bernstein inequality, Theorem~\ref{thm:scalar-bernstein}, lifts directly to matrices.  Let us emphasize this remarkable fact:

\begin{quotation}
\noindent
\textbf{There are exponential concentration inequalities for the spectral norm of a sum of independent random matrices.}
\end{quotation}

\noindent
As a consequence, once we decompose a random matrix as an independent sum, we can harness global properties (such as the mean and the variance) and local properties (such as a uniform bound on the summands) to obtain detailed information about the norm of the sum.  As in the scalar case, it is usually easy to acquire the input data for the inequality.  But the output of the inequality is highly nontrivial.

To illustrate these claims, we will state one of the major results from this monograph.
This theorem is a matrix extension of Bernstein's inequality that was developed independently
in the two papers~\cite{Oli10:Concentration-Adjacency,Tro11:User-Friendly-FOCM}.
After presenting the result, we give some more details about its interpretation.
In the next section, we apply this result to study the covariance estimation problem.

\begin{thm}[Matrix Bernstein] \label{thm:matrix-bernstein-intro}
Let $\mtx{S}_1, \dots, \mtx{S}_n$ be independent, centered random matrices with common dimension $d_1 \times d_2$,
and assume that each one is uniformly bounded
$$
\Expect \mtx{S}_k = \mtx{0}
\quad\text{and}\quad
\norm{ \mtx{S}_k } \leq L
\quad\text{for each $k = 1, \dots, n$.}
$$
Introduce the sum
\begin{equation} \label{eqn:matrix-bernstein-Z-intro}
\mtx{Z} = \sum_{k=1}^n \mtx{S}_k,
\end{equation}
and let $v(\mtx{Z})$ denote the matrix variance statistic of the sum:
\begin{equation} \label{eqn:matrix-variance-intro}
\begin{aligned}
v(\mtx{Z}) &= \max\big\{ \norm{ \smash{\Expect ( \mtx{ZZ}^\adj )} }, \ 
	\norm{ \smash{\Expect ( \mtx{Z}^\adj \mtx{Z}) } } \big\} \\
	&= \max\left\{ \norm{ \sum_{k=1}^n \Expect \big( \mtx{S}_k \mtx{S}_k^\adj \big) }, \
	\norm{ \sum_{k=1}^n \Expect \big( \mtx{S}_k^\adj \mtx{S}_k \big) } \right\}.
\end{aligned}
\end{equation}
Then
\begin{equation} \label{eqn:bernstein-tail-intro}
\Prob{ \norm{ \mtx{Z} } \geq t }
	\leq (d_1 + d_2) \cdot \exp\left( \frac{-t^2/2}{v(\mtx{Z}) + Lt/3} \right)
	\quad\text{for all $t \geq 0$.}
\end{equation}
Furthermore,
\begin{equation} \label{eqn:bernstein-expect-intro}
\Expect \norm{ \mtx{Z} }
	\leq \sqrt{ 2 v(\mtx{Z}) \log(d_1 + d_2) } + \frac{1}{3} L \,\log(d_1 + d_2).
\end{equation}
\end{thm}

\noindent
The proof of this result appears in Chapter~\ref{chap:matrix-bernstein}.

To appreciate what Theorem~\ref{thm:matrix-bernstein-intro} means, it is valuable
to make a direct comparison with the scalar version, Theorem~\ref{thm:scalar-bernstein}.
In both cases, we express the object of interest as an independent sum, and we instate
a uniform bound on the summands.  There are three salient changes:

\begin{itemize}
\item	The variance $v(\mtx{Z})$ in the result for matrices can be interpreted
as the magnitude of the expected squared deviation of $\mtx{Z}$ from its mean.
The formula reflects the fact that a general matrix $\mtx{B}$ has \emph{two}
different squares $\mtx{BB}^\adj$ and $\mtx{B}^\adj \mtx{B}$.  For an Hermitian
matrix, the two squares coincide.

\item	The tail bound has a dimensional factor $d_1 + d_2$ that depends on the size of
the matrix.  This factor reduces to two in the scalar setting.  In the matrix case, it
limits the range of $t$ where the tail bound is informative.

\item	We have included a bound for $\Expect{} \norm{ \mtx{Z} }$.  This estimate is not particularly interesting in the scalar setting, but it is usually quite challenging to prove results of this type for matrices.  In fact, the expectation bound is often more useful than the tail bound.
\end{itemize}

\noindent
The latter point deserves amplification:

\begin{quotation} \noindent
\textbf{The expectation bound~\eqref{eqn:bernstein-expect-intro} is the most important aspect of the
matrix Bernstein inequality.}
\end{quotation}

\noindent
For further discussion of this result, turn to Chapter~\ref{chap:matrix-bernstein}.  Chapters~\ref{chap:matrix-series} and~\ref{chap:intrinsic} contain related results and interpretations.

\subsection{Example: The Sample Covariance Estimator} \label{sec:sample-covar}

We will apply the matrix Bernstein inequality, Theorem~\ref{thm:matrix-bernstein-intro},
to measure how well the sample covariance estimator approximates the true covariance matrix.
As before, let $\vct{x}$ be a zero-mean random vector with dimension $p$.
Introduce the $p \times p$ covariance matrix $\mtx{A} = \Expect( \vct{xx}^\adj )$.  
Suppose we have $n$ independent samples
$\vct{x}_1, \dots, \vct{x}_n$ with the same distribution as $\vct{x}$.
Form the $p \times p$ sample covariance estimator
$$
\mtx{Y} = \frac{1}{n} \sum_{k=1}^n \vct{x}_k \vct{x}_k^\adj.
$$
Our goal is to study how the spectral-norm distance $\norm{ \mtx{Y} - \mtx{A} }$ between the
sample covariance and the true covariance depends on the number $n$ of samples.

For simplicity, we will perform the analysis under the extra assumption that the
$\ell_2$ norm of the random vector is bounded: $\normsq{\vct{x}} \leq B$.
This hypothesis can be relaxed if we apply a variant of the matrix Bernstein
inequality that reflects the typical magnitude of a summand $\mtx{S}_k$.
One such variant appears in the formula~\eqref{eqn:matrix-moment-ineq}.

We are in a situation where it is quite easy to see how the matrix Bernstein inequality applies.
Define the random deviation $\mtx{Z}$ of the estimator $\mtx{Y}$ from the true covariance matrix $\mtx{A}$:
$$
\mtx{Z} = \mtx{Y} - \mtx{A} = \sum_{k=1}^n \mtx{S}_k
\quad\text{where}\quad
\mtx{S}_k = \frac{1}{n} \big(\vct{x}_k \vct{x}_k^\adj - \mtx{A} \big)
\quad\text{for each index $k$.}
$$
The random matrices $\mtx{S}_k$ are independent, identically distributed, and centered.
To apply Theorem~\ref{thm:matrix-bernstein-intro},
we need to find a uniform bound $L$ for the summands,
and we need to control the matrix variance statistic $v(\mtx{Z})$.

First, let us develop a uniform bound on the spectral norm of each summand.
We may calculate that
$$
\norm{ \mtx{S}_k }
	= \frac{1}{n} \norm{ \smash{\vct{x}_k \vct{x}_k^\adj - \mtx{A}} }
	\leq \frac{1}{n} \big( \norm{ \smash{\vct{x}_k\vct{x}_k^\adj} } + \norm{\mtx{A}} \big)
	\leq \frac{2B}{n}.	
$$
The first relation is the triangle inequality.  The second follows from the assumption that
$\vct{x}$ is bounded and the observation that
\begin{equation*} \label{eqn:intro-scov-normA}
\norm{\mtx{A}} = \norm{ \smash{\Expect( \vct{xx}^\adj )} } \leq \Expect \norm{ \smash{\vct{xx}^\adj} }
	= \Expect{} \normsq{\vct{x}} \leq B.
\end{equation*}
This expression depends on Jensen's inequality and the hypothesis that $\vct{x}$ is bounded.

Second, we need to bound the matrix variance statistic $v(\mtx{Z})$
defined in~\eqref{eqn:matrix-variance-intro}.
The matrix $\mtx{Z}$ is Hermitian, so the two squares in this formula coincide with each other:
\begin{equation*} %
v(\mtx{Z}) = \norm{ \Expect{} \mtx{Z}^2 }
	= \norm{ \sum_{k=1}^n \Expect{} \mtx{S}_k^2 }.
\end{equation*}
We need to determine the variance of each summand.  By direct calculation,
\begin{align*} %
\Expect{} \mtx{S}_k^2
	= \frac{1}{n^2} \Expect{} \big(\vct{x}_k \vct{x}_k^\adj - \mtx{A} \big)^2
	&= \frac{1}{n^2} \Expect \big[ \normsq{\vct{x}_k} \cdot \vct{x}_k \vct{x}_k^\adj
	- \big(\vct{x}_k \vct{x}_k^\adj \big) \, \mtx{A} - \mtx{A} \, \big(\vct{x}_k \vct{x}_k^\adj \big)
	+ \mtx{A}^2 \big] \notag \\
	&\psdle \frac{1}{n^2} \big[ B \cdot \Expect \big( \vct{x}_k \vct{x}_k^\adj \big) - \mtx{A}^2 - \mtx{A}^2 + \mtx{A}^2 \big] \notag \\
	&\psdle \frac{B}{n^2} \cdot \mtx{A}
\end{align*}
The expression $\mtx{H} \psdle \mtx{T}$ means that $\mtx{T} - \mtx{H}$ is positive semidefinite.
We used the norm bound for the random vector $\vct{x}$ and the fact that expectation preserves
the semidefinite order.  In the last step, we dropped the negative-semidefinite term $-\mtx{A}^2$.
Summing this relation over $k$, we reach
$$
\mtx{0} \psdle \sum_{k=1}^n \Expect{} \mtx{S}_k^2
	\psdle \frac{B}{n} \cdot \mtx{A}.
$$ 
The matrix is positive-semidefinite because it is a sum of squares of Hermitian matrices.
Extract the spectral norm to arrive at
$$
v(\mtx{Z}) = \norm{\sum_{k=1}^n \Expect{} \mtx{S}_k^2 }
	\leq \frac{B \norm{\mtx{A}}}{n}.
$$
We have now collected the information we need to analyze the sample covariance estimator.

We can invoke the estimate~\eqref{eqn:bernstein-expect-intro} %
from the matrix Bernstein inequality, Theorem~\ref{thm:matrix-bernstein-intro},
with the uniform bound $L = 2B / n$ and the variance bound $v(\mtx{Z}) \leq B \norm{\mtx{A}} / n$.
We attain
$$
\Expect \norm{ \mtx{Y} - \mtx{A} }
	= \Expect \norm{ \mtx{Z} }
	\leq \sqrt{\frac{2B \norm{\mtx{A}} \log(2p)}{n}} + \frac{2B \log(2p)}{3n}.
$$
In other words, the error in approximating the sample covariance matrix is not too large when we have a sufficient number of samples.  If we wish to obtain a relative error on the order of $\eps$, we may take
$$
n \geq \frac{2B \log(2p)}{\eps^2 \norm{\mtx{A}}}.
$$
This selection yields
$$
\Expect \norm{ \mtx{Y} - \mtx{A} }
	\leq \big(\eps + \eps^2 \big) \cdot \norm{\mtx{A}}.
$$
It is often the case that $B = \textrm{Const} \cdot p$, so we discover that $n = \mathrm{Const} \cdot \eps^{-2} p \log p$ samples are sufficient for the sample covariance estimator to provide a relatively accurate estimate of the true covariance matrix $\mtx{A}$.  This bound is qualitatively sharp for worst-case distributions.

The analysis in this section applies to many other examples.  We encapsulate the argument
in Corollary~\ref{cor:matrix-approx-sampling}, which we use to study several more problems.

\subsection{History of this Example}

Covariance estimation may be the earliest application of matrix concentration tools in random matrix theory.
Rudelson~\cite{Rud99:Random-Vectors}, building on a suggestion of Pisier, showed how to use the noncommutative Khintchine inequality~\cite{LP86:Inegalites-Khintchine,LPP91:Noncommutative-Khintchine,Buc01:Operator-Khintchine,Buc05:Optimal-Constants} to obtain essentially optimal bounds on the sample covariance estimator of a bounded random vector.  The tutorial~\cite{Ver12:Introduction-Nonasymptotic} of Roman Vershynin offers an overview of this problem as well as many results and references.
The analysis of the sample covariance matrix here is adapted from the technical~\cite{GT11:Tail-Bounds}.  It leads to a result similar with the one Rudelson obtained in~\cite{Rud99:Random-Vectors}.

\subsection{Optimality of the Matrix Bernstein Inequality}

Theorem~\ref{thm:matrix-bernstein-intro} can be sharpened very little because it
applies to every random matrix $\mtx{Z}$ of the form~\eqref{eqn:matrix-bernstein-Z-intro}.
Let us say a few words about optimality now, postponing the details to~\S\ref{sec:bernstein-optimality}.

Suppose that $\mtx{Z}$ is a random matrix of the form~\eqref{eqn:matrix-bernstein-Z-intro}.
To make the comparison simpler,
we also insist that each summand $\mtx{S}_k$ is a \term{symmetric} random variable;
that is, $\mtx{S}_k$ and $- \mtx{S}_k$ have the same distribution for each index $k$.
Introduce the quantity
$$
L_{\star}^2 = \Expect{} \max\nolimits_k \normsq{\mtx{S}_k}.
$$
In \S\ref{sec:bernstein-optimality}, we will argue that these assumptions imply
\begin{equation} \label{eqn:norm-lower-upper-intro}
\begin{aligned}
\mathrm{const} \cdot \big[ v(\mtx{Z})\ + \ L_{\star}^2 \big]
	\quad&\leq\quad \Expect \normsq{ \mtx{Z} } \\
	&\leq\quad \mathrm{Const} \cdot \big[ v(\mtx{Z}) \log(d_1 + d_2)\ +\  L_{\star}^2 \log^2(d_1 + d_2) \big].
\end{aligned}
\end{equation}
In other words, the scale of $\Expect \normsq{\mtx{Z}}$ must depend on the matrix variance
statistic $v(\mtx{Z})$ and the average upper bound $L_{\star}^2$ for the summands.  The quantity
$L = \sup \norm{\mtx{S}_k}$ that appears in the matrix Bernstein inequality always exceeds
$L_{\star}$, sometimes by a large margin, but they capture the same type of information.

The significant difference between the lower and upper bound
in~\eqref{eqn:norm-lower-upper-intro} comes from the dimensional factor
$\log(d_1 + d_2)$.
There are random matrices $\mtx{Z}$ for which the lower bound gives a more
accurate reflection of $\Expect \normsq{\mtx{Z}}$, but there are also many
random matrices where the upper bound describes the behavior correctly.
At present, there is no method known for distinguishing between
these two extremes under the model~\eqref{eqn:matrix-bernstein-Z-intro}
for the random matrix. %

The tail bound~\eqref{eqn:bernstein-tail-intro} provides a useful tool
in practice, but it is not necessarily the best way to collect information about
large deviation probabilities.  To obtain more precise results, we recommend
using the expectation bound~\eqref{eqn:bernstein-expect-intro}
to control $\Expect \norm{\mtx{Z}}$ and then applying scalar concentration
inequalities to estimate $\Prob{ \norm{\mtx{Z}} \geq \Expect \norm{\mtx{Z}} + t }$.
The book~\cite{BLM13:Concentration-Inequalities} offers a good treatment of
the methods that are available for establishing scalar concentration.

\section{The Arsenal of Results}

The Bernstein inequality is probably the most familiar exponential tail bound
for a sum of independent random variables, but there are many more.
It turns out that essentially all of these scalar results admit extensions
that hold for random matrices.  In fact, many of the established techniques
for scalar concentration have analogs in the matrix setting.

\subsection{What's Here...}

This monograph focuses on a few key exponential concentration inequalities
for a sum of independent random matrices, and it describes some specific
applications of these results.

\begin{description}
\item	[Matrix Gaussian Series.]  A matrix Gaussian series is a random matrix that can be expressed as a sum of fixed matrices, each weighted by an independent standard normal random variable.  This formulation includes a surprising number of examples.  The most important are undoubtedly Wigner matrices and rectangular Gaussian matrices.  Another interesting case is a Toeplitz matrix with Gaussian entries.  The analysis of matrix Gaussian series appears in Chapter~\ref{chap:matrix-series}.

\item	[Matrix Rademacher Series.]  A matrix Rademacher series is a random matrix that can be written as a sum of fixed matrices, each weighted by an independent Rademacher random variable.%
\footnote{A \term{Rademacher random variable} takes the two values $\pm 1$ with equal probability.}
This construction includes things like random sign matrices, as well as a fixed matrix whose entries are modulated by random signs.  There are also interesting examples that arise in combinatorial optimization.  We treat these problems in Chapter~\ref{chap:matrix-series}.

\item	[Matrix Chernoff Bounds.]  The matrix Chernoff bounds apply to a random matrix that can be decomposed as a sum of independent, random positive-semidefinite matrices whose maximum eigenvalues are subject to a uniform bound.
These results allow us to obtain information about the norm of a random submatrix drawn from a fixed matrix.
They are also appropriate for studying the Laplacian matrix of a random graph.  See Chapter~\ref{chap:matrix-chernoff}.

\item	[Matrix Bernstein Bounds.]  The matrix Bernstein inequality concerns a random matrix that can be expressed as a sum of independent, centered random matrices that admit a uniform spectral-norm bound.  This result has many applications, including the analysis of randomized algorithms for matrix sparsification and matrix multiplication.  It can also be used to study the random features paradigm for approximating a kernel matrix.  Chapter~\ref{chap:matrix-bernstein} contains this material.

\item	[Intrinsic Dimension Bounds.]  Some matrix concentration inequalities can be improved when the random matrix has limited spectral content in most dimensions.  In this situation, we may be able to obtain bounds that do not depend on the ambient dimension.  See Chapter~\ref{chap:intrinsic} for details.
\end{description}

\noindent
We have chosen to present these results because they are illustrative,
and they have already found concrete applications.

\subsection{What's Not Here...}

The program of extending scalar concentration results to the matrix setting
has been quite fruitful, and there are many useful results beyond the ones
that we detail.  Let us mention some of the other tools that are available.
For further information, see the annotated bibliography.

First, there are additional exponential concentration inequalities for
a sum of independent random matrices.  All of the following results can
be established within the framework of this monograph.
\begin{itemize}
\item	\textbf{Matrix Hoeffding.}  This result concerns a sum of independent random matrices
whose squares are subject to semidefinite upper bounds~\cite[\S7]{Tro11:User-Friendly-FOCM}.

\item	\textbf{Matrix Bennett.}  This estimate sharpens the tail bound from the matrix Bernstein
inequality~\cite[\S6]{Tro11:User-Friendly-FOCM}.

\item	\textbf{Matrix Bernstein, Unbounded Case.}  The matrix Bernstein
inequality extends to the case where the moments of the summands grow at a controlled rate.
See~\cite[\S6]{Tro11:User-Friendly-FOCM} or~\cite{Kol11:Oracle-Inequalities}.

\item	\textbf{Matrix Bernstein, Nonnegative Summands.}  The lower tail of the
Bernstein inequality can be improved when the summands are positive semidefinite~\cite{Mau03:Bound-Deviation};
this result extends to the matrix setting.  By a different argument, the dimensional factor can be removed
from this bound for a class of interesting examples~\cite[Thm.~3.1]{Oli13:Lower-Tail}.
\end{itemize}

The approach in this monograph can be adapted to obtain exponential concentration
for matrix-valued martingales.  Here are a few results from this category:

\begin{itemize}
\item	\textbf{Matrix Azuma.}  This is the martingale version of the matrix Hoeffding bound~\cite[\S7]{Tro11:User-Friendly-FOCM}.

\item	\textbf{Matrix Bounded Differences}.  The matrix Azuma inequality gives bounds for the spectral
norm of a matrix-valued function of independent random variables~\cite[\S7]{Tro11:User-Friendly-FOCM}.

\item	\textbf{Matrix Freedman.}  This result can be viewed as the martingale extension of the matrix
Bernstein inequality~\cite{Oli10:Concentration-Adjacency,Tro11:Freedmans-Inequality}.

\end{itemize}

\noindent
The technical report~\cite{Tro11:User-Friendly-Martingale-TR} explains how to extend other bounds for
a sum of independent random matrices to the martingale setting.

\term{Polynomial moment inequalities} provide bounds for the expected trace
of a power of a random matrix.
Moment inequalities for a sum of independent
random matrices can provide useful information when the summands have heavy
tails or else a uniform bound does not reflect the typical size of the summands.

\begin{itemize}
\item	\textbf{Matrix Khintchine.}  The matrix Khintchine inequality is
the polynomial version of the exponential bounds for matrix Gaussian series and
matrix Rademacher series.  This result is presented in~\eqref{eqn:nc-khintchine}.
See the papers~\cite{LP86:Inegalites-Khintchine,Buc01:Operator-Khintchine,Buc05:Optimal-Constants}
or~\cite[Cor.~7.3]{MJCFT12:Matrix-Concentration} for proofs.

\item	\textbf{Matrix Moment Inequalities.}  The matrix Chernoff inequality
admits a polynomial variant; the simplest form appears in~\eqref{eqn:matrix-rosenthal}.
The matrix Bernstein inequality also has a polynomial variant,
stated in~\eqref{eqn:matrix-moment-ineq}.  These
bounds are drawn from~\cite[App.]{CGT12:Masked-Sample}.
\end{itemize}

\noindent
The methods that lead to polynomial moment inequalities differ
substantially from the techniques in this monograph, so we cannot
include the proofs here.
The annotated bibliography includes references to the large
literature on moment inequalities for random matrices.

Recently, Lester Mackey and the author, in collaboration with Daniel Paulin
and several other researchers~\cite{MJCFT12:Matrix-Concentration, PMT14:Efron-Stein},
have developed another framework for establishing matrix concentration.
This approach extends a scalar argument, introduced by
Chatterjee~\cite{Cha08:Concentration-Inequalities,Cha07:Steins-Method},
that depends on exchangeable pairs and Markov chain couplings.
The \term{method of exchangeable pairs} delivers both exponential concentration
inequalities and polynomial moment inequalities for random matrices,
and it can reproduce many of the bounds mentioned above.
It also leads to new results:

\begin{itemize}
\item	\textbf{Polynomial Efron--Stein Inequality for Matrices.}  This bound is a matrix version
of the polynomial Efron--Stein inequality~\cite[Thm.~1]{BBLM05:Moment-Inequalities}.
It controls the polynomial moments of a centered random matrix that is a
function of independent random variables~\cite[Thm.~4.2]{PMT14:Efron-Stein}.

\item	\textbf{Exponential Efron--Stein Inequality for Matrices.}  This bound is the matrix extension
of the exponential Efron--Stein inequality~\cite[Thm.~1]{BLM03:Concentration-Inequalities}.
It leads to exponential concentration inequalities for a centered random matrix constructed
from independent random variables~\cite[Thm.~4.3]{PMT14:Efron-Stein}.
\end{itemize}

\noindent
Another significant advantage is that the method of exchangeable pairs can sometimes handle
random matrices built from dependent random variables.  Although the simplest version of
the exchangeable pairs argument is more elementary than the approach in this monograph,
it takes a lot of effort to establish the more useful inequalities.  With some
regret, we have chosen not to include this material because the method and results
are accessible to a narrower audience.

Finally, we remark that the modified logarithmic Sobolev inequalities
of~\cite{BLM03:Concentration-Inequalities,BBLM05:Moment-Inequalities}
also extend to the matrix setting~\cite{CT14:Subadditivity-Matrix}.
Unfortunately, the matrix variants do not seem to be as useful as the scalar results.

\section{About This Monograph}

This monograph is intended for graduate students and researchers in computational mathematics who want to learn some modern techniques for analyzing random matrices.  The preparation required is minimal.   We assume familiarity with calculus, applied linear algebra, the basic theory of normed spaces, and classical probability theory up through the elementary concentration inequalities (such as Markov and Bernstein).  Beyond the basics, which can be gleaned from any good textbook,
we include all the required background in Chapter~\ref{chap:matrix-functions}.

The material here is based primarily on the paper ``User-Friendly Tail Bounds for Sums of Random Matrices'' by the present author~\cite{Tro11:User-Friendly-FOCM}.  There are several significant revisions to this earlier work:

\begin{description}
\item	[Examples and Applications.]  Many of the papers on matrix concentration give limited information about how the results can be used to solve problems of interest.  A major part of these notes consists of worked examples and applications that indicate how matrix concentration inequalities apply to practical questions.

\item	[Expectation Bounds.]  This work collects bounds for the expected value of the spectral norm of a random matrix and bounds for the expectation of the smallest and largest eigenvalues of a random symmetric matrix.  Some of these useful results have appeared piecemeal in the literature~\cite{CGT12:Masked-Sample,MJCFT12:Matrix-Concentration}, but they have not been included in a unified presentation.

\item	[Optimality.]  We explain why each matrix concentration inequality is (nearly) optimal.  This presentation includes examples to show that each term in each bound is necessary to describe some particular phenomenon.

\item	[Intrinsic Dimension Bounds.]  Over the last few years, there have been some refinements to the basic matrix concentration bounds that improve the dependence on dimension~\cite{HKZ12:Tail-Inequalities,Min11:Some-Extensions}.  We describe a new framework that allows us to prove these results with ease.

\item	[Lieb's Theorem.]  The matrix concentration inequalities in this monograph depend on a deep theorem~\cite[Thm.~6]{Lie73:Convex-Trace} from matrix analysis due to Elliott Lieb.  We provide a complete proof of this result, along with all the background required to understand the argument.  %

\item	[Annotated Bibliography.]  We have included a list of the major works on matrix concentration, including a short summary of the main contributions of these papers.  We hope this catalog will be a valuable guide for further reading.
\end{description}

The organization of the notes is straightforward.  Chapter~\ref{chap:matrix-functions} contains background material that is needed for the analysis.  Chapter~\ref{chap:matrix-lt} describes the framework for developing exponential concentration inequalities for matrices.  Chapter~\ref{chap:matrix-series} presents the first set of results and examples, concerning matrix Gaussian and Rademacher series.  Chapter~\ref{chap:matrix-chernoff} introduces the matrix Chernoff bounds and their applications, and Chapter~\ref{chap:matrix-bernstein} expands on our discussion of the matrix Bernstein inequality.  Chapter~\ref{chap:intrinsic} shows how to sharpen some of the results so that they depend on an intrinsic dimension parameter.
Chapter~\ref{chap:lieb} contains the proof of Lieb's theorem.  We conclude with resources on matrix concentration and a bibliography.

To make the presentation smoother, we have not followed all of the conventions for scholarly articles in journals.
In particular, almost all the citations appear in the notes at the end of each chapter.
Our aim has been to explain the ideas as clearly as possible,
rather than to interrupt the narrative with an elaborate genealogy of results.

\makeatletter{}%

\chapter[Matrix Functions \& Probability with Matrices]{Matrix Functions \& \\ Probability with Matrices} \label{chap:matrix-functions}

We begin the main development with a short overview of the background material that is required to understand the proofs and, to a lesser extent, the statements of matrix concentration inequalities.  We have been careful to provide cross-references to these foundational results, so most readers will be able to proceed directly to the main theoretical development in Chapter~\ref{chap:matrix-lt} or the discussion of specific random matrix inequalities in Chapters~\ref{chap:matrix-series},~\ref{chap:matrix-chernoff}, and~\ref{chap:matrix-bernstein}.

\subsubsection{Overview}

Section~\ref{sec:matrix-background} covers material from matrix theory concerning the behavior of matrix functions.  Section~\ref{sec:probability-background} reviews relevant results from probability, especially the parts involving matrices.

\section{Matrix Theory Background} \label{sec:matrix-background}

Let us begin with the results we require from the field of matrix analysis.

\subsection{Conventions}

We write $\R$ and $\C$ for the real and complex fields.
A \term{matrix} is a finite, two-dimensional array of complex numbers.
Many parts of the discussion do not depend on the size of a matrix, so we specify dimensions only when it really matters.
Readers who wish to think about real-valued matrices will find that none of the results require any essential modification in this setting.

\subsection{Spaces of Vectors}

The symbol $\C^d$ denotes the complex linear space consisting of $d$-dimensional column vectors
with complex entries, equipped with the usual componentwise addition and multiplication by a complex scalar.
We endow this space with the standard $\ell_2$ inner product
$$
\ip{ \vct{x} }{ \smash{ \vct{y} }} = \vct{x}^\adj \vct{y} = \sum_{i=1}^d x_i^\adj y_i
\quad\text{for all $\vct{x}, \vct{y} \in \C^d$.}
$$
The symbol ${}^\adj$ denotes the complex conjugate of a number, as well as the conjugate transpose of a vector or matrix.  The inner product induces the $\ell_2$ norm:
\begin{equation} \label{eqn:l2-norm}
\normsq{ \vct{x} } = \ip{\vct{x}}{\vct{x}} = \sum_{i=1}^d \abssq{x_i}
\quad\text{for all $\vct{x} \in \C^d$.}
\end{equation}

Similarly, the real linear space $\R^d$ consists of $d$-dimensional column vectors with real entries,
equipped with the usual componentwise addition and multiplication by a real scalar.  The inner
product and $\ell_2$ norm on $\R^d$ are defined by the same relations as for $\C^d$.

\subsection{Spaces of Matrices}

We write $\mathbb{M}^{d_1 \times d_2}$ for the complex linear space consisting of $d_1 \times d_2$ matrices with complex entries, equipped with the usual componentwise addition and multiplication by a complex scalar.  It is convenient to identify $\C^d$ with the space $\mathbb{M}^{d \times 1}$.  We write $\mathbb{M}_d$ for the algebra of $d \times d$ square, complex matrices.  The term ``algebra'' just means that we can multiply two matrices in $\mathbb{M}_d$ to obtain another matrix in $\mathbb{M}_d$.

\subsection{Topology \& Convergence}

We can endow the space of matrices with the Frobenius norm:
\begin{equation} \label{eqn:frobenius-norm}
\fnormsq{ \mtx{B} } = \sum_{j=1}^{d_1} \sum_{k=1}^{d_2} \abssq{ \smash{b_{jk}} }
\quad\text{for $\mtx{B} \in \mathbb{M}^{d_1 \times d_2}$.}
\end{equation}
Observe that the Frobenius norm on $\mathbb{M}^{d \times 1}$ coincides with the $\ell_2$
norm~\eqref{eqn:l2-norm} on $\C^d$.

The Frobenius norm induces a norm topology on the space of matrices.  In particular,
given a sequence $\{ \mtx{B}_n : n = 1,2,3,\dots \} \subset \mathbb{M}^{d_1\times d_2}$,
the symbol
$$
\mtx{B}_n \to \mtx{B}
\quad\text{means that}\quad
\fnorm{ \mtx{B}_n - \mtx{B} } \to 0
\quad\text{as $n \to \infty$.}
$$
Open and closed sets are also defined with respect to the Frobenius-norm topology.
Every other norm topology on $\mathbb{M}^{d_1 \times d_2}$ induces the same notions
of convergence and open sets.  We use the same topology for the normed linear spaces
$\mathbb{C}^d$ and $\mathbb{M}_d$.

\subsection{Basic Vectors and Matrices}

We write $\mtx{0}$ for the zero vector or the zero matrix, while $\Id$ denotes the identity matrix.  Occasionally, we add a subscript to specify the dimension.  For instance, $\Id_d$ is the $d \times d$ identity.

The standard basis for the linear space $\C^d$ consists of \term{standard basis vectors}.  
The standard basis vector $\mathbf{e}_k$ is a column vector with a one in position $k$ and zeros elsewhere.
We also write $\mathbf{e}$ for the column vector whose entries all equal one.
There is a related notation for the standard basis of $\mathbb{M}^{d_1 \times d_2}$.  
We write $\mathbf{E}_{jk}$ for the standard basis matrix with a one in position $(j, k)$ and zeros elsewhere.  
The dimension of a standard basis vector and a standard basis matrix is typically determined by the context.

A square matrix $\mtx{Q}$ that satisfies $\mtx{QQ}^\adj = \Id =  \mtx{Q}^\adj \mtx{Q}$ is called a \term{unitary matrix}.
We reserve the letter $\mtx{Q}$ for a unitary matrix.  %
Readers who prefer the real setting may prefer to regard $\mtx{Q}$ as an orthogonal matrix.

\subsection{Hermitian Matrices and Eigenvalues}

An \term{Hermitian matrix} $\mtx{A}$ is a square matrix that satisfies $\mtx{A} = \mtx{A}^\adj$.  A useful intuition from operator theory is that Hermitian matrices are analogous with real numbers, while general square matrices are analogous with complex numbers.

We write $\mathbb{H}_d$ for the collection of $d \times d$ Hermitian matrices.  The set $\mathbb{H}_d$ is a linear space over the real field.  That is, we can add Hermitian matrices and multiply them by real numbers.
The space $\mathbb{H}_d$ inherits the Frobenius-norm topology from $\mathbb{M}_d$.
We adopt Parlett's convention~\cite{Par98:Symmetric-Eigenvalue} that bold Latin and Greek letters that are symmetric around the vertical axis ($\mtx{A}$, $\mtx{H}$, \dots, $\mtx{Y}$; $\mtx{\Delta}$, $\mtx{\Theta}$, \dots, $\mtx{\Omega}$) always represent Hermitian~matrices.

Each Hermitian matrix $\mtx{A} \in \mathbb{H}_d$ has an \term{eigenvalue decomposition}
\begin{equation} \label{eqn:eigenvalue-decomposition}
\mtx{A} = \mtx{Q \Lambda Q}^\adj
\quad\text{where $\mtx{Q} \in \mathbb{M}_d$ is unitary and $\mtx{\Lambda} \in \mathbb{H}_d$ is diagonal}.
\end{equation}
The diagonal entries of $\mtx{\Lambda}$ are real numbers, which are referred to as the \term{eigenvalues} of $\mtx{A}$.  The unitary matrix $\mtx{Q}$ in the eigenvalue decomposition is not determined completely, but the list of eigenvalues is unique modulo permutations.  The eigenvalues of an Hermitian matrix are often referred to as its \term{spectrum}.

We denote the algebraic minimum and maximum eigenvalues of an Hermitian matrix $\mtx{A}$ by $\lambda_{\min}(\mtx{A})$ and $\lambda_{\max}(\mtx{A})$.  The extreme eigenvalue maps are positive homogeneous:
\begin{equation} \label{eqn:eig-pos-homo}
\lambda_{\min}(\alpha \mtx{A}) = \alpha \lambda_{\min}(\mtx{A})
\quad\text{and}\quad
\lambda_{\max}(\alpha \mtx{A}) = \alpha \lambda_{\max}(\mtx{A})
\quad\text{for $\alpha \geq 0$.}
\end{equation}
There is an important relationship between minimum and maximum eigenvalues:
\begin{equation} \label{eqn:min-max-sign-eig}
\lambda_{\min}(-\mtx{A}) = - \lambda_{\max}(\mtx{A}).
\end{equation}
The fact~\eqref{eqn:min-max-sign-eig} warns us that we must be careful
passing scalars through an eigenvalue map.

This work rarely requires any eigenvalues of an Hermitian matrix
aside from the minimum and maximum.
When they do arise, we usually order the other eigenvalues in the weakly decreasing sense:
$$
\lambda_1(\mtx{A}) \geq \lambda_2(\mtx{A}) \geq \dots \geq \lambda_d(\mtx{A})
\quad\text{for $\mtx{A} \in \mathbb{H}_d$.}
$$
On occasion, it is more natural to arrange eigenvalues in the weakly increasing sense:
$$
\lambda_1^\uparrow(\mtx{A}) \leq \lambda_2^\uparrow(\mtx{A}) \leq \dots \leq \lambda_d^\uparrow(\mtx{A})
\quad\text{for $\mtx{A} \in \mathbb{H}_d$.}
$$
To prevent confusion, we will accompany this notation with a reminder.

Readers who prefer the real setting may read ``symmetric'' in place of ``Hermitian.''  In this case, the eigenvalue decomposition involves an orthogonal matrix $\mtx{Q}$.  Note, however, that the term ``symmetric'' has a different meaning when applied to random variables!

\subsection{The Trace of a Square Matrix}

The \term{trace} of a square matrix, denoted by $\trace$, is the sum of its diagonal entries.
\begin{equation} \label{eqn:trace}
\trace \mtx{B} = \sum_{j=1}^d b_{jj}
\quad\text{for $\mtx{B} \in \mathbb{M}_d$.}
\end{equation}
The trace is unitarily invariant:
\begin{equation} \label{eqn:trace-unitary-invar}
\trace \mtx{B} = \trace( \mtx{Q} \mtx{B} \mtx{Q}^\adj )
\quad\text{for each $\mtx{B} \in \mathbb{M}_d$ and each unitary $\mtx{Q} \in \mathbb{M}_d$.}
\end{equation}
In particular, the existence of an eigenvalue decomposition~\eqref{eqn:eigenvalue-decomposition} shows that the trace of an Hermitian~matrix equals the sum of its eigenvalues.%
\footnote{This fact also holds true for a general square matrix.}

Another valuable relation connects the trace with the Frobenius norm:
\begin{equation} \label{eqn:trace-Frobenius}
\fnormsq{\mtx{C}} = \trace( \mtx{CC}^\adj) = \trace( \mtx{C}^\adj \mtx{C} ).
\quad\text{for all $\mtx{C} \in \mathbb{M}^{d_1 \times d_2}$.}
\end{equation}
This expression follows from the definitions~\eqref{eqn:frobenius-norm}
and~\eqref{eqn:trace} and a short calculation.

\subsection{The Semidefinite Partial Order} \label{sec:psd-order}

A matrix $\mtx{A} \in \mathbb{H}_d$ is \term{positive semidefinite} when it satisfies
\begin{equation} \label{eqn:psd-rayleigh}
\vct{u}^\adj \mtx{A} \vct{u} \geq 0
\quad\text{for each vector $\vct{u} \in \C^d$.}
\end{equation}
Equivalently, a matrix $\mtx{A}$ is positive semidefinite when it is Hermitian
and its eigenvalues are all nonnegative.
Similarly, we say that $\mtx{A} \in \mathbb{H}_d$ is \term{positive definite} when
\begin{equation} \label{eqn:pd-rayleigh}
\vct{u}^\adj \mtx{A} \vct{u} > 0
\quad\text{for each nonzero vector $\vct{u} \in \C^d$.}
\end{equation}
Equivalently, $\mtx{A}$ is positive definite when it is Hermitian
and its eigenvalues are all positive.

Positive-semidefinite and positive-definite matrices play a special role in matrix
theory, analogous with the role of nonnegative and positive numbers in real analysis.
In particular, observe that the square of an Hermitian matrix is always positive
semidefinite.  The square of a nonsingular Hermitian matrix is always positive definite.

The family of positive-semidefinite matrices in $\mathbb{H}_d$ forms a closed convex cone.%
\footnote{A \term{convex cone} is a subset $C$ of a linear space that is closed under conic combinations.
That is, $\tau_1 \vct{x}_1 + \tau_2 \vct{x}_2 \in C$ for all $\vct{x}_1, \vct{x}_2 \in C$ and all
$\tau_1, \tau_2 > 0$.  Equivalently, $C$ is a set that is both convex and positively homogeneous.}
This geometric fact follows easily from the definition~\eqref{eqn:psd-rayleigh}.  Indeed,
for each vector $\vct{u} \in \mathbb{C}^d$, the condition
$$
\big\{ \mtx{A} \in \mathbb{H}_d : \vct{u}^\adj \mtx{A} \vct{u} \geq 0 \big\}
$$
describes a closed halfspace in $\mathbb{H}_d$.  As a consequence, the family of positive-semidefinite
matrices in $\mathbb{H}_d$ is an intersection of closed halfspaces.  Therefore, it is a closed convex set.
To see why this convex set is a cone, just note that
$$
\text{$\mtx{A}$ positive semidefinite}
\quad\text{implies}\quad
\text{$\alpha \mtx{A}$ is positive semidefinite for $\alpha \geq 0$.}
$$
Beginning from~\eqref{eqn:pd-rayleigh}, similar considerations
show that the family of positive-definite matrices
in $\mathbb{H}_d$ forms an (open) convex cone.

We may now define the \term{semidefinite partial order} $\psdle$
on the real-linear space $\mathbb{H}_d$ using the rule
\begin{equation} \label{eqn:semidefinite-order}
\mtx{A} \psdle \mtx{H}
\quad\text{if and only if}\quad
\text{$\mtx{H} - \mtx{A}$ is positive semidefinite.}
\end{equation}
In particular, we write $\mtx{A} \psdge \mtx{0}$ to indicate that $\mtx{A}$ is positive semidefinite and $\mtx{A} \psdgt \mtx{0}$ to indicate that $\mtx{A}$ is positive definite.  For a diagonal matrix $\mtx{\Lambda}$, the expression $\mtx{\Lambda} \psdge \mtx{0}$ means that each entry of $\mtx{\Lambda}$ is nonnegative.

The semidefinite order is preserved by conjugation, a simple fact whose importance cannot be overstated.

\begin{prop}[Conjugation Rule]
\label{prop:conjugation-rule}
Let $\mtx{A}$ and $\mtx{H}$ be Hermitian matrices of the same dimension, and let $\mtx{B}$
be a general matrix with compatible dimensions.  Then
\begin{equation} \label{eqn:conjugation-rule}
\mtx{A} \psdle \mtx{H}
\quad\text{implies}\quad
\mtx{B}\mtx{A} \mtx{B}^\adj  \psdle \mtx{B} \mtx{H} \mtx{B}^\adj.
\end{equation}
\end{prop}

Finally, we remark that the trace of a positive-semidefinite matrix is at least as large as its maximum eigenvalue:
\begin{equation}\label{eqn:maxeig-trace}
\lambda_{\max}(\mtx{A}) \leq \trace \mtx{A}
\quad\text{when $\mtx{A}$ is positive semidefinite.}
\end{equation}
This property follows from the definition of a positive-semidefinite matrix and the
fact that the trace of $\mtx{A}$ equals the sum of the eigenvalues.

\subsection{Standard Matrix Functions}

Let us describe the most direct method for extending a function on the real numbers to
a function on Hermitian matrices.  The basic idea is to apply the function to
each eigenvalue of the matrix to construct a new matrix.

\begin{defn}[Standard Matrix Function] \label{def:standard-matrix-fn}
Let $f : I \to \R$ where $I$ is an interval of the real line.  Consider a matrix $\mtx{A} \in \mathbb{H}_d$ whose eigenvalues are contained in $I$.  Define the matrix $f(\mtx{A}) \in \mathbb{H}_d$ using an eigenvalue decomposition
of $\mtx{A}$:
$$
f(\mtx{A}) = \mtx{Q} \begin{bmatrix} f(\lambda_1) \\ & \ddots \\ && f(\lambda_d) \end{bmatrix} \mtx{Q}^\adj
\quad\text{where}\quad
\mtx{A} = \mtx{Q} \begin{bmatrix} \lambda_1 \\ & \ddots \\ && \lambda_d \end{bmatrix} \mtx{Q}^\adj.
$$
In particular, we can apply $f$ to a real diagonal matrix by applying the function to each diagonal entry.

It can be verified that the definition of $f(\mtx{A})$ does not depend on which eigenvalue decomposition $\mtx{A} = \mtx{Q\Lambda Q}^\adj$ that we choose.  Any matrix function that arises in this fashion is called a \term{standard matrix function}.
\end{defn}

To confirm that this definition is sensible, consider the power function $f(t) = t^q$ for a natural number $q$.  When $\mtx{A}$ is Hermitian, the power function $f(\mtx{A}) = \mtx{A}^q$, where $\mtx{A}^q$ is the $q$-fold product of $\mtx{A}$.

For an Hermitian matrix $\mtx{A}$, whenever we write the power function $\mtx{A}^q$ or the exponential $\econst^{\mtx{A}}$ or the logarithm $\log \mtx{A}$, we are always referring to a standard matrix function.  Note that we only define the matrix logarithm for positive-definite matrices, and non-integer powers are only valid for positive-semidefinite matrices.

The following result is an immediate, but important, consequence of the definition of a standard matrix function.

\begin{prop}[Spectral Mapping Theorem]
\label{prop:spectral-mapping}
Let $f : I \to \R$ be a function on an interval $I$ of the real line, and let $\mtx{A}$ be an Hermitian
matrix whose eigenvalues are contained in $I$.  If $\lambda$ is an eigenvalue of $\mtx{A}$, then
$f(\lambda)$ is an eigenvalue of $f(\mtx{A})$.
\end{prop}

When a real function has a power series expansion, we can also represent the standard matrix function with the same power series expansion.  Indeed, suppose that $f : I \to \R$ is defined on an interval $I$ of the real line, and assume that the eigenvalues of $\mtx{A}$ are contained in $I$.  Then
$$
f(a) = c_0 + \sum\limits_{q=1}^\infty c_q a^q
\quad\text{for $a \in I$}
\quad\text{implies}\quad
f(\mtx{A}) = c_0 \Id + \sum\limits_{q=1}^\infty c_q \mtx{A}^q.
$$
This formula can be verified using an eigenvalue decomposition of $\mtx{A}$ and the definition of a standard matrix function.

\subsection{The Transfer Rule}

In most cases, the ``obvious'' generalization of an inequality for real-valued functions fails to hold in the semidefinite order.  Nevertheless, there is one class of inequalities for real functions that extends to give semidefinite relationships for standard matrix functions.

\begin{prop}[Transfer Rule]
\label{prop:transfer-rule}
Let $f$ and $g$ be real-valued functions defined on an interval $I$ of the real line,
and let $\mtx{A}$ be an Hermitian matrix whose eigenvalues are contained in $I$.
Then
\begin{equation} \label{eqn:transfer-rule}
f(a) \leq g(a)
\quad\text{for each $a \in I$}\quad
\text{implies}\quad
f(\mtx{A}) \psdle g(\mtx{A}).
\end{equation}
\end{prop}

\begin{proof}
Decompose $\mtx{A} = \mtx{Q\Lambda Q}^\adj$.  It is immediate that $f(\mtx{\Lambda}) \psdle g(\mtx{\Lambda})$.  The Conjugation Rule~\eqref{eqn:conjugation-rule} allows us to conjugate this relation by $\mtx{Q}$.  Finally, we invoke Definition~\ref{def:standard-matrix-fn}, of a standard matrix function, to complete the argument.
\end{proof}

\subsection{The Matrix Exponential}

For any Hermitian matrix $\mtx{A}$, we can introduce the matrix exponential $\econst^{\mtx{A}}$
using Definition~\ref{def:standard-matrix-fn}.  Equivalently,
we can use a power series expansion:
\begin{equation} \label{eqn:exp-series}
\econst^{\mtx{A}} = \exp(\mtx{A}) = \Id + \sum\limits_{q=1}^\infty \frac{\mtx{A}^q}{q!}.
\end{equation}
The Spectral Mapping Theorem, Proposition~\ref{prop:spectral-mapping}, implies that the
exponential of an Hermitian matrix is always positive definite.

We often work with the trace of the matrix exponential:
$$
\trace \exp : \mtx{A} \longmapsto \trace \econst^{\mtx{A}}.
$$
This function has a monotonicity property that we use extensively.
For Hermitian matrices $\mtx{A}$ and $\mtx{H}$ with the same dimension,
\begin{equation} \label{eqn:exp-trace-monotone}
\mtx{A} \psdle \mtx{H}
\quad\text{implies}\quad
\trace \econst^{\mtx{A}} \leq \trace \econst^{\mtx{H}}.
\end{equation}
We establish this result in \S\ref{sec:monotone-trace}.

\subsection{The Matrix Logarithm}

We can define the matrix logarithm as a standard matrix function.  The matrix logarithm is also the functional inverse of the matrix exponential:
\begin{equation} \label{eqn:log-defn}
\log\bigl( \econst^{\mtx{A}} \bigr) = \mtx{A}
\quad\text{for each Hermitian matrix $\mtx{A}$}.
\end{equation}
A valuable fact about the matrix logarithm is that it preserves the semidefinite order.
For positive-definite matrices $\mtx{A}$ and $\mtx{H}$ with the same dimension,
\begin{equation} \label{eqn:log-monotone}
\mtx{A} \psdle \mtx{H}
\quad\text{implies}\quad
\log \mtx{A}  \psdle \log \mtx{H}.
\end{equation}
We establish this result in \S\ref{sec:log-monotone}.
Let us stress that the matrix exponential \emph{does not} have any operator monotonicity property analogous with~\eqref{eqn:log-monotone}!

\subsection{Singular Values of Rectangular Matrices}

A general matrix does not have an eigenvalue decomposition, but it admits a different representation that is just as useful.  Every $d_1 \times d_2$ matrix $\mtx{B}$ has a \term{singular value decomposition}
\begin{equation} \label{eqn:singular-value-decomposition}
\mtx{B} = \mtx{Q}_1 \mtx{\Sigma} \mtx{Q}_2^\adj
\quad\text{where $\mtx{Q}_1$ and $\mtx{Q}_2$ are unitary and $\mtx{\Sigma}$ is nonnegative diagonal.}
\end{equation}
The unitary matrices $\mtx{Q}_1$ and $\mtx{Q}_2$ have dimensions $d_1 \times d_1$ and $d_2 \times d_2$, respectively.
The inner matrix $\mtx{\Sigma}$ has dimension $d_1 \times d_2$, and we use the term diagonal in the sense that only the diagonal entries $(\mtx{\Sigma})_{jj}$ may be nonzero.

The diagonal entries of $\mtx{\Sigma}$ are called the \term{singular values} of $\mtx{B}$, and they are denoted as $\sigma_j(\mtx{B})$.  The singular values are determined completely modulo permutations, and it is conventional to arrange them in weakly decreasing order: $$
\sigma_1(\mtx{B}) \geq \sigma_2(\mtx{B}) \geq \dots \geq \sigma_{\min\{d_1,\ d_2\}}(\mtx{B}).
$$

There is an important relationship between singular values and eigenvalues.  A general matrix has two squares associated with it, $\mtx{BB}^\adj$ and $\mtx{B}^\adj \mtx{B}$, both of which are positive semidefinite.  We can use a singular value decomposition of $\mtx{B}$ to construct eigenvalue decompositions of the two squares:
\begin{equation} \label{eqn:svd-square}
\mtx{BB}^\adj = \mtx{Q}_1  (\mtx{\Sigma \Sigma}^\adj) \mtx{Q}_1^\adj
\quad\text{and}\quad
\mtx{B}^\adj \mtx{B} = \mtx{Q}_2 (\mtx{\Sigma}^\adj \mtx{\Sigma}) \mtx{Q}_2^\adj
\end{equation}
The two squares of $\mtx{\Sigma}$ are square, diagonal matrices with nonnegative entries.
Conversely, we can always extract a singular value decomposition
from the eigenvalue decompositions of the two squares.

We can write the Frobenius norm of a matrix in terms of the singular values:
\begin{equation} \label{eqn:frobenius-svals}
\fnormsq{\mtx{B}} = \sum_{j=1}^{\min\{d_1, d_2\}} \sigma_j(\mtx{B})^2
\quad\text{for $\mtx{B} \in \mathbb{M}^{d_1 \times d_2}$.}
\end{equation}
This expression follows from the expression~\eqref{eqn:trace-Frobenius}
for the Frobenius norm, the property~\eqref{eqn:svd-square} of the
singular value decomposition, and the unitary invariance~\eqref{eqn:trace-unitary-invar}
of the trace.

\subsection{The Spectral Norm}

The \term{spectral norm} of an Hermitian matrix $\mtx{A}$ is defined by the relation
\begin{equation} \label{eqn:spectral-norm-herm}
\norm{ \mtx{A} } = \max\big\{ \lambda_{\max}(\mtx{A}), \ - \lambda_{\min}(\mtx{A}) \big\}.
\end{equation}
For a general matrix $\mtx{B}$, the spectral norm is defined to be the largest singular value:
\begin{equation} \label{eqn:spectral-norm}
\norm{ \mtx{B} } = \sigma_1(\mtx{B}).
\end{equation}
These two definitions are consistent for Hermitian matrices because of~\eqref{eqn:svd-square}.
When applied to a row vector or a column vector,
the spectral norm coincides with the $\ell_2$ norm~\eqref{eqn:l2-norm}.

We will often need the fact that
\begin{equation} \label{eqn:spectral-norm-square}
\normsq{\mtx{B}} = \norm{\smash{\mtx{BB}^\adj}} = \norm{\smash{\mtx{B}^\adj \mtx{B}}}.
\end{equation}
This identity also follows from~\eqref{eqn:svd-square}.

\subsection{The Stable Rank}
\label{sec:stable-rank}

In several of the applications, we need an analytic measure
of the collinearity of the rows and columns of a matrix
called the \term{stable rank}.  For a general matrix $\mtx{B}$,
the stable rank is defined as
\begin{equation} \label{eqn:stable-rank}
\strank(\mtx{B}) = \frac{\fnormsq{\mtx{B}}}{\normsq{\mtx{B}}}.
\end{equation}
The stable rank is a lower bound for the algebraic rank:
$$
1 \leq \strank(\mtx{B}) \leq \rank(\mtx{B}).
$$
This point follows when we use~\eqref{eqn:frobenius-svals}
and~\eqref{eqn:spectral-norm} to express the two norms
in terms of the singular values of $\mtx{B}$.
In contrast to the algebraic rank,
the stable rank is a continuous function of the matrix,
so it is more suitable for numerical applications.

\subsection{Dilations} \label{sec:dilation}

An extraordinarily fruitful idea from operator theory is to embed matrices within larger block matrices, called \term{dilations}.  Dilations have an almost magical power.
In this work, we will use dilations to extend matrix concentration
inequalities from Hermitian matrices to general matrices.

\begin{defn}[Hermitian Dilation] \label{def:herm-dilation}
The \term{Hermitian dilation}
$$
\coll{H} : \mathbb{M}^{d_1 \times d_2} \longrightarrow \mathbb{H}_{d_1 + d_2}
$$
is the map from a general matrix to an Hermitian matrix defined by
\begin{equation} \label{eqn:herm-dilation}
\coll{H}(\mtx{B}) = \begin{bmatrix} \mtx{0} & \mtx{B} \\ \mtx{B}^\adj & \mtx{0} \end{bmatrix}.
\end{equation}
\end{defn}

It is clear that the Hermitian dilation is a real-linear map.
Furthermore, the dilation retains important spectral information.
To see why, note that the square of the dilation satisfies
\begin{equation} \label{eqn:herm-dilation-square}
\coll{H}( \mtx{B} )^2
	= \begin{bmatrix} \mtx{BB}^\adj & \mtx{0} \\
	\mtx{0} & \mtx{B}^\adj\mtx{B} \end{bmatrix}.
\end{equation}
We discover that the squared eigenvalues of $\coll{H}(\mtx{B})$ coincide with the squared singular values of $\mtx{B}$, along with an appropriate number of zeros.  As a consequence, $\norm{\coll{H}(\mtx{B})} = \norm{\mtx{B}}$.  Moreover,
\begin{equation} \label{eqn:herm-dilation-norm}
\lambda_{\max}(\coll{H}(\mtx{B})) = \norm{ \coll{H}(\mtx{B}) } = \norm{ \mtx{B} }.
\end{equation}
We will invoke the identity~\eqref{eqn:herm-dilation-norm} repeatedly.

One way to justify the first relation in~\eqref{eqn:herm-dilation-norm} is to introduce the first columns
$\vct{u}_1$ and $\vct{u}_2$ of the unitary matrices $\mtx{Q}_1$ and $\mtx{Q}_2$
that appear in the singular value decomposition $\mtx{B} = \mtx{Q}_1 \mtx{\Sigma} \mtx{Q}_2^\adj$.
Then we may calculate that
$$
\norm{\mtx{B}}
	= \real(\vct{u}_1^\adj \mtx{B} \vct{u}_2)
	= \frac{1}{2} \begin{bmatrix} \vct{u}_1 \\ \vct{u}_2 \end{bmatrix}^\adj
	\begin{bmatrix} \mtx{0} & \mtx{B} \\ \mtx{B}^\adj & \mtx{0} \end{bmatrix}
	\begin{bmatrix} \vct{u}_1 \\ \vct{u}_2 \end{bmatrix}
	\leq \lambda_{\max}(\coll{H}(\mtx{B}))
	\leq \norm{\coll{H}(\mtx{B})}
	= \norm{\mtx{B}}.
$$
Indeed, the spectral norm of $\mtx{B}$ equals its largest singular value $\sigma_1(\mtx{B})$,
which coincides with $\vct{u}_1^\adj \mtx{B} \vct{u}_2$
by construction of $\vct{u}_1$ and $\vct{u}_2$.  The second
identity relies on a direct calculation. %
The first inequality follows from the variational representation of the maximum
eigenvalue as a Rayleigh quotient; this fact can also be derived as a
consequence of~\eqref{eqn:eigenvalue-decomposition}.
The second inequality depends on the definition~\eqref{eqn:spectral-norm-herm} of the spectral
norm of an Hermitian matrix.

\subsection{Other Matrix Norms}
\label{sec:other-norms}

There are a number of other matrix norms that arise sporadically in this work.
The \term{Schatten 1-norm} of a matrix can be defined as the sum of
its singular values:
\begin{equation} \label{eqn:schatten-1-norm}
\pnorm{S_1}{ \mtx{B} } = \sum_{j=1}^{\min\{d_1, d_2\}} \sigma_j(\mtx{B})
\quad\text{for $\mtx{B} \in \mathbb{M}^{d_1 \times d_2}$.}
\end{equation}
The entrywise $\ell_1$ norm of a matrix is defined as
\begin{equation} \label{eqn:matrix-l1-norm}
\pnorm{\ell_1}{\mtx{B}} = \sum_{j=1}^{d_1} \sum_{k=1}^{d_2} \abs{ \smash{b_{jk}} }
\quad\text{for $\mtx{B} \in \mathbb{M}^{d_1 \times d_2}$.}
\end{equation}
We always have the relation
\begin{equation} \label{eqn:matrix-l1-l2}
\pnorm{\ell_1}{\mtx{B}} \leq \sqrt{ d_1 d_2 } \fnorm{\mtx{B}}
\quad\text{for $\mtx{B} \in \mathbb{M}^{d_1 \times d_2}$}
\end{equation}
because of the Cauchy--Schwarz inequality.

\section{Probability with Matrices} \label{sec:probability-background}

We continue with some material from probability, focusing on connections with matrices.

\subsection{Conventions}

We prefer to avoid abstraction and unnecessary technical detail, so we frame the standing assumption that all random variables are sufficiently regular that we are justified in computing expectations, interchanging limits, and so forth.
The manipulations we perform are valid if we assume that all random variables are bounded,
but the results hold in broader circumstances if we instate appropriate regularity conditions.
Since the expectation operator is linear, we typically do not use parentheses with it.
We instate the convention that powers and products take precedence over the expectation operator.
In particular,
$$
\Expect X^q = \Expect(X^q).
$$
This position helps us reduce the clutter of parentheses.
We sometimes include extra delimiters when it is helpful for clarity.

\subsection{Some Scalar Random Variables}

We use consistent notation for some of the basic scalar random variables.

\begin{description}
\item	[Standard normal variables.]  We reserve the letter $\gamma$ for a $\normal(0, 1)$ random variable.  That is,
$\gamma$ is a real Gaussian with mean zero and variance one.

\item	[Rademacher random variables.]  We reserve the letter $\varrho$ for a random variable that takes the two values $\pm 1$
with equal probability.

\item	[Bernoulli random variables.]  A $\textsc{bernoulli}(p)$ random variable takes the value one with probability $p$ and the value zero with probability $1 - p$, where $p \in [0, 1]$.  We use the letters $\delta$ and $\xi$ for Bernoulli random variables.
\end{description}

\subsection{Random Matrices}

Let $(\Omega, \coll{F}, \mathbb{P})$ be a probability space.
A \term{random matrix} $\mtx{Z}$ is a measurable map
$$
\mtx{Z} : \Omega \longrightarrow \mathbb{M}^{d_1 \times d_2}.
$$
It is more natural to think of the entries of $\mtx{Z}$ as complex random variables that may or may not be correlated with each other.  We reserve the letters $\mtx{X}$ and $\mtx{Y}$ for random Hermitian matrices, while the letter $\mtx{Z}$ denotes a general random matrix.

A finite sequence $\{ \mtx{Z}_k \}$ of random matrices is \term{independent} when
$$
\Prob{ \mtx{Z}_k \in F_k \ \text{for each $k$} }
	= \prod\nolimits_k \Prob{ \mtx{Z}_k \in F_k }
$$
for every collection $\{F_k\}$ of Borel subsets of $\mathbb{M}^{d_1\times d_2}$.

\subsection{Expectation}

The \term{expectation} of a random matrix $\mtx{Z} = [ Z_{jk} ]$ is simply the matrix formed by taking the componentwise expectation.  That is,
$$
(\Expect{} \mtx{Z})_{jk} = \Expect{}  Z_{jk} .
$$
Under mild assumptions, expectation commutes with linear and real-linear maps.
Indeed, expectation commutes with multiplication by a fixed matrix:
$$
\Expect(\mtx{BZ}) = \mtx{B}\, (\Expect \mtx{Z})
\quad\text{and}\quad
\Expect(\mtx{ZB}) =  (\Expect \mtx{Z}) \,\mtx{B}.
$$
In particular, the product rule for the expectation of independent random variables extends to matrices:
$$
\Expect(\mtx{SZ}) = (\Expect \mtx{S} )(\Expect \mtx{Z})
\quad\text{when $\mtx{S}$ and $\mtx{Z}$ are independent.}
$$
We use these identities liberally, without any further comment.

\subsection{Inequalities for Expectation}

Markov's inequality states that a nonnegative (real) random variable $X$ obeys the probability bound
\begin{equation} \label{eqn:markov}
\Prob{ X \geq t } \leq \frac{\Expect X}{t}
\quad\text{for $t > 0$.}
\end{equation}
The Markov inequality is a central tool for establishing concentration inequalities.

Jensen's inequality describes how averaging interacts with convexity.
Let $\mtx{Z}$ be a random matrix, and let $h$ be a real-valued function on matrices.  Then
\begin{equation} \label{eqn:jensen}
\begin{array}{l}
\Expect h(\mtx{Z}) \leq h(\Expect \mtx{Z}) \quad\text{when $h$ is concave, and} \\
\Expect h(\mtx{Z}) \geq h(\Expect \mtx{Z}) \quad\text{when $h$ is convex.}
\end{array}
\end{equation}

The family of positive-semidefinite matrices in $\mathbb{H}_d$ forms a convex cone,
and the expectation of a random matrix can be viewed as a convex combination.
Therefore, expectation preserves the semidefinite order:
$$
\mtx{X} \psdle \mtx{Y}
\quad\text{implies}\quad
\Expect \mtx{X} \psdle \Expect \mtx{Y}.
$$
We use this result many times without direct reference.

\subsection{The Variance of a Random Hermitian Matrix}

The variance of a real random variable $Y$ is defined as the expected squared deviation
from the mean:
$$
\Var(Y) = \Expect{} (Y - \Expect Y)^2 
$$
There are a number of natural extensions of this concept in the matrix setting that
play a role in our theory.  %

Suppose that $\mtx{Y}$ is a random Hermitian matrix.
We can define a matrix-valued variance:
\begin{equation} \label{eqn:matrix-val-var-herm}
\mVar(\mtx{Y}) = \Expect{} (\mtx{Y} - \Expect \mtx{Y})^2
	= \Expect{} \mtx{Y}^2 - (\Expect \mtx{Y})^2.
\end{equation}
The matrix $\mVar(\mtx{Y})$ is always positive semidefinite.
We can interpret the $(j, k)$ entry of this matrix as the covariance
between the $j$th and $k$th columns of $\mtx{Y}$:
$$
(\mVar(\mtx{Y}))_{jk} = \Expect\big[ (\vct{y}_{:j} - \Expect \vct{y}_{:j} )^{\adj}( \vct{y}_{:k} - \Expect \vct{y}_{:k} ) \big],
$$
where we have written $\vct{y}_{:j}$ for the $j$th column of $\mtx{Y}$.

The matrix-valued variance contains a lot of information about the fluctuations
of the random matrix.  We can summarize $\mVar(\mtx{Y})$ using
a single number $v(\mtx{Y})$, which we call the \term{matrix variance statistic}:
\begin{equation} \label{eqn:matrix-variance-herm}
v(\mtx{Y}) = \norm{\mVar(\mtx{Y})} =
	\norm{ \smash{\Expect{} (\mtx{Y} - \Expect \mtx{Y})^2} }.
\end{equation}
To understand what this quantity means, one may wish to rewrite it as
$$
v(\mtx{Y}) = \sup_{\norm{\vct{u}} = 1} \ \Expect \normsq{ (\mtx{Y}\vct{u}) - \Expect( \mtx{Y} \vct{u}) }.
$$
Roughly speaking, the matrix variance statistic describes the maximum
variance of $\mtx{Y}\vct{u}$ for any unit vector $\vct{u}$.

\subsection{The Variance of a Sum of Independent, Random Hermitian Matrices}

The matrix-valued variance interacts beautifully with a sum of independent random matrices.
Consider a finite sequence $\{\mtx{X}_k \}$ of independent, random Hermitian matrices
with common dimension $d$.  Introduce the sum $\mtx{Y} = \sum_k \mtx{X}_k$.  Then
\begin{align} \label{eqn:matrix-variance-add}
\mVar(\mtx{Y})
	= \mVar \left(\sum\nolimits_k \mtx{X}_k \right)
	&= \Expect{} \left( \sum\nolimits_k (\mtx{X}_k - \Expect \mtx{X}_k) \right)^2 \notag \\
	&= \sum\nolimits_{j,k} \Expect \big[ (\mtx{X}_j - \Expect \mtx{X}_j)(\mtx{X}_k - \Expect \mtx{X}_k) \big] \notag \\
	&= \sum\nolimits_{k} \Expect{} (\mtx{X}_k - \Expect \mtx{X}_k)^2 \notag \\
	&= \sum\nolimits_k \mVar(\mtx{X}_k).
\end{align}
This identity matches the familiar result for the variance
of a sum of independent scalar random variables.
It follows that the matrix variance statistic satisfies
\begin{equation} \label{eqn:indep-sum-herm}
v(\mtx{Y}) = \norm{ \sum\nolimits_k \mVar(\mtx{X}_k) }.
\end{equation}
The fact that the sum remains inside the norm is very important.
Indeed, the best general inequalities between $v(\mtx{Y})$ and the
matrix variance statistics $v(\mtx{X}_k)$ of the summands are 
$$
v(\mtx{Y}) \leq \sum\nolimits_k v(\mtx{X}_k)
	\leq d \cdot v(\mtx{Y}).
$$
These relations can be improved in some special cases.
For example, when the matrices $\mtx{X}_k$ are identically distributed,
the left-hand inequality becomes an identity.

\subsection{The Variance of a Rectangular Random Matrix}

We will often work with non-Hermitian random matrices.  In this case, we need to account
for the fact that a general matrix has \emph{two} different squares.
Suppose that $\mtx{Z}$ is a random matrix with dimension $d_1 \times d_2$.  Define
\begin{equation} \label{eqn:matrix-val-var-rect}
\begin{aligned}
\mVar_1(\mtx{Z}) &= \Expect\big[ (\mtx{Z} - \Expect \mtx{Z})(\mtx{Z} - \Expect \mtx{Z})^\adj \big],
\quad\text{and}\quad \\
\mVar_2(\mtx{Z}) &= \Expect\big[ (\mtx{Z} - \Expect \mtx{Z})^\adj(\mtx{Z} - \Expect \mtx{Z}) \big].
\end{aligned}
\end{equation}
The matrix $\mVar_1(\mtx{Z})$ is a positive-semidefinite matrix with dimension $d_1 \times d_1$,
and it describes the fluctuation of the rows of $\mtx{Z}$.  The matrix $\mVar_2(\mtx{Z})$
is a positive-semidefinite matrix with dimension $d_2 \times d_2$, and it reflects the
fluctuation of the columns of $\mtx{Z}$.  For an Hermitian random matrix $\mtx{Y}$,
$$
\mVar(\mtx{Y}) = \mVar_1(\mtx{Y}) = \mVar_2(\mtx{Y}).
$$
In other words, the two variances coincide in the Hermitian setting.

As before, it is valuable to reduce these matrix-valued variances
to a single scalar parameter.  %
We define the matrix variance statistic of a general random matrix $\mtx{Z}$ as
\begin{equation} \label{eqn:matrix-variance-rect}
v(\mtx{Z}) = \max\big\{ \norm{\mVar_1(\mtx{Z})}, \ \norm{\mVar_2(\mtx{Z})} \big\}.
\end{equation}
When $\mtx{Z}$ is Hermitian, the definition~\eqref{eqn:matrix-variance-rect} coincides with the original
definition~\eqref{eqn:matrix-variance-herm}.

To promote a deeper appreciation for the formula~\eqref{eqn:matrix-variance-rect},
let us explain how it arises from the Hermitian dilation~\eqref{eqn:herm-dilation}.
By direct calculation,
\begin{align} \label{eqn:var-dilation}
\mVar( \coll{H}(\mtx{Z}) )
	&= \Expect \begin{bmatrix} \mtx{0} & (\mtx{Z} - \Expect \mtx{Z}) \\ (\mtx{Z} - \Expect \mtx{Z})^\adj & \mtx{0} \end{bmatrix}^2 \notag \\
	&= \Expect \begin{bmatrix} (\mtx{Z} - \Expect\mtx{Z})(\mtx{Z} - \Expect \mtx{Z})^\adj & \mtx{0}  \notag \\
	\mtx{0} & (\mtx{Z} - \Expect \mtx{Z})^\adj (\mtx{Z} - \Expect \mtx{Z}) \end{bmatrix} \notag \\
	&= \begin{bmatrix} \mVar_1(\mtx{Z}) & \mtx{0} \\ \mtx{0} & \mVar_2(\mtx{Z}) \end{bmatrix}.
\end{align}
The first identity is the definition~\eqref{eqn:matrix-val-var-herm} of the matrix-valued variance.
The second line follows from the formula~\eqref{eqn:herm-dilation-square} for the square of the dilation.
The last identity depends on the definition~\eqref{eqn:matrix-val-var-rect} of the two matrix-valued variances.
Therefore, using the definitions~\eqref{eqn:matrix-variance-herm} and~\eqref{eqn:matrix-variance-rect}
of the matrix variance statistics,
\begin{equation} \label{eqn:var-stat-dilation}
v(\coll{H}(\mtx{Z})) = \norm{ \mVar(\coll{H}(\mtx{Z}) }
	= \max\big\{ \norm{ \mVar_1(\mtx{Z}) }, \ \norm{\mVar_2(\mtx{Z})} \big\}
	= v(\mtx{Z}).
\end{equation}
The second identity holds because the spectral norm of a block-diagonal matrix
is the maximum norm achieved by one of the diagonal blocks.

\subsection{The Variance of a Sum of Independent Random Matrices}

As in the Hermitian case, the matrix-valued variances interact nicely with an independent sum.
Consider a finite sequence $\{ \mtx{S}_k \}$ of independent random matrices with the same dimension.
Form the sum $\mtx{Z} = \sum_k \mtx{S}_k$.  Repeating the calculation leading up to~\eqref{eqn:indep-sum-herm},
we find that
$$
\mVar_1(\mtx{Z})
	= \sum\nolimits_k \mVar_1(\mtx{S}_k)
\quad\text{and}\quad
\mVar_2(\mtx{Z}) %
	= \sum\nolimits_k \mVar_2(\mtx{S}_k).
$$
In summary, the matrix variance statistic of an independent sum satisfies 
\begin{equation} \label{eqn:indep-sum-rect}
v(\mtx{Z}) = \max\left\{ \norm{\sum\nolimits_k \mVar_1(\mtx{S}_k)}, \
	\norm{ \sum\nolimits_k \mVar_2(\mtx{S}_k) } \right\}.
\end{equation}
This formula arises time after time.

\section{Notes}

Everything in this chapter is firmly established.  We have culled the results that are relevant to our discussion.  Let us give some additional references for readers who would like more information.

\subsection{Matrix Analysis}

Our treatment of matrix analysis is drawn from Bhatia's excellent books on matrix analysis~\cite{Bha97:Matrix-Analysis,Bha07:Positive-Definite}.  The two books~\cite{HJ13:Matrix-Analysis,HJ94:Topics-Matrix} of Horn \& Johnson also serve as good general references.  Higham's work~\cite{Hig08:Functions-Matrices} is a generous source of information about matrix functions.  Other valuable resources include Carlen's lecture notes~\cite{Car10:Trace-Inequalities}, the book of Petz~\cite{Pet11:Matrix-Analysis}, and the book of Hiai \& Petz~\cite{HP14:Introduction-Matrix}.

\subsection{Probability with Matrices}

The classic introduction to probability is the two-volume treatise~\cite{Fel68:Introduction-Probability-I,Fel71:Introduction-Probability-II} of Feller.  The book~\cite{GS01:Probability-Random} of Grimmett \& Stirzaker offers a good treatment of probability theory and random processes at an intermediate level.  For a more theoretical presentation, consider the book~\cite{Shi96:Probability} of Shiryaev.

There are too many books on random matrix theory for us to include a comprehensive list; here is a selection that the author finds useful.  Tao's book~\cite{Tao12:Topics-Random} gives a friendly introduction to some of the major aspects of classical and modern random matrix theory.  The lecture notes~\cite{Kem13:Introduction-Random} of Kemp are also extremely readable.  The survey of Vershynin~\cite{Ver12:Introduction-Nonasymptotic} provides a good summary of techniques from asymptotic convex geometry that are relevant to random matrix theory.  The works of Mardia, Kent, \& Bibby~\cite{MKB79:Multivariate-Analysis} and Muirhead~\cite{Mui82:Aspects-Multivariate} present classical results on random matrices that are particularly useful in statistics, while Bai \& Silverstein~\cite{BS10:Spectral-Analysis} contains a comprehensive modern treatment.  Nica and Speicher~\cite{NS06:Lectures-Combinatorics} offer an entr{\'e}e to the beautiful field of free probability.  Mehta's treatise~\cite{Meh04:Random-Matrices} was the first book on random matrix theory available, and it remains solid.  


%
\makeatletter{}%

\chapter[The Matrix Laplace Transform Method]{The Matrix \\ Laplace Transform Method} \label{chap:matrix-lt}

This chapter contains the core part of the analysis that ultimately delivers matrix concentration inequalities.  Readers who are only interested in the concentration inequalities themselves or the example applications may wish to move on to Chapters~\ref{chap:matrix-series},~\ref{chap:matrix-chernoff}, and~\ref{chap:matrix-bernstein}.

In the scalar setting, the Laplace transform method provides a simple but powerful way to develop concentration inequalities for a sum of independent random variables.  This technique is sometimes referred to as the ``Bernstein trick'' or ``Chernoff bounding.''  For a primer, we recommend~\cite[Chap.~2]{BLM13:Concentration-Inequalities}.

In the matrix setting, there is a very satisfactory extension of this argument that allows us to prove concentration inequalities for a sum of independent random matrices.  As in the scalar case, the matrix Laplace transform method
is both easy to use and incredibly useful.  In contrast to the scalar case, the arguments that lead to matrix concentration are no longer elementary.  The purpose of this chapter is to install the framework we need to support these results.  Fortunately, in practical applications, all of the technical difficulty remains invisible.

\subsubsection{Overview}

We first define matrix analogs of the moment generating function and the cumulant generating function, which pack up information about the fluctuations of a random Hermitian matrix.  Section~\ref{sec:matrix-lt} explains how we can use the matrix mgf to obtain probability inequalities for the maximum eigenvalue of a random Hermitian matrix.  The next task is to develop a bound for the mgf of a sum of independent random matrices using information about the summands.  In~\S\ref{sec:mom-fail}, we discuss the challenges that arise;~\S\ref{sec:lieb} presents the ideas we need to overcome these obstacles.  Section~\ref{sec:cum-subadd} establishes that the classical result on additivity of cumulants has a companion in the matrix setting.  This result allows us to develop a collection of abstract probability inequalities in~\S\ref{sec:master-tail} that we can specialize to obtain matrix Chernoff bounds, matrix Bernstein bounds, and so forth. %

\section{Matrix Moments and Cumulants} \label{sec:mom-cum}

At the heart of the Laplace transform method are
the moment generating function (mgf) and the
cumulant generating function (cgf) of a random variable.
We begin by presenting matrix versions of the mgf and cgf.

\begin{defn}[Matrix Mgf and Cgf] \label{def:matrix-mgf-cgf}
Let $\mtx{X}$ be a random Hermitian matrix.
The \term{matrix moment generating function} $\mtx{M}_{\mtx{X}}$
and the \term{matrix cumulant generating function} $\mtx{\Xi}_{\mtx{X}}$
are given by
\begin{equation} \label{eqn:matrix-mgf-cgf}
\mtx{M}_{\mtx{X}}(\theta) = \Expect \econst^{\theta \mtx{X}}
\quad\text{and}\quad
\mtx{\Xi}_{\mtx{X}}(\theta) = \log{} \Expect \econst^{\theta \mtx{X}}
\quad\text{for $\theta \in \mathbb{R}$.}
\end{equation}
Note that the expectations may not exist for all values of $\theta$.
\end{defn}

\noindent
The matrix mgf $\mtx{M}_{\mtx{X}}$ and matrix cgf $\mtx{\Xi}_{\mtx{X}}$ contain information about how much the random matrix $\mtx{X}$ varies.  We aim to exploit the data encoded in these functions to control the eigenvalues.

Let us take a moment to expand on Definition~\ref{def:matrix-mgf-cgf}; this discussion is not important for subsequent developments.  Observe that the matrix mgf and cgf have formal power series expansions:
$$
\mtx{M}_{\mtx{X}}(\theta) = \Id + \sum\limits_{q=1}^\infty
	\frac{\theta^q }{ q! } (\Expect{} \mtx{X}^q)
\quad\text{and}\quad
\mtx{\Xi}_{\mtx{X}}(\theta) = \sum\limits_{q=1}^\infty \frac{\theta^q}{q!} \mtx{\Psi}_q.
$$
We call the coefficients $\Expect{} \mtx{X}^q $ \term{matrix moments}, and we refer to $\mtx{\Psi}_q$ as a \term{matrix cumulant}.
The matrix cumulant $\mtx{\Psi}_q$ has a formal expression as a (noncommutative) polynomial in the matrix moments up to order $q$.  In particular, the first cumulant is the mean and the second cumulant is the variance:
$$
\mtx{\Psi}_1 = \Expect \mtx{X}
\quad\text{and}\quad
\mtx{\Psi}_2 = \Expect{} \mtx{X}^2 - (\Expect \mtx{X})^2 = \mVar(\mtx{X})
$$
The matrix variance was introduced in~\eqref{eqn:matrix-val-var-herm}.  Higher-order cumulants are harder to write down and interpret.

\section{The Matrix Laplace Transform Method} \label{sec:matrix-lt}

In the scalar setting, the Laplace transform method allows us to obtain tail bounds for a random variable in terms of its mgf.  The starting point for our theory is the observation that a similar result holds in the matrix setting.

\begin{prop}[Tail Bounds for Eigenvalues] \label{prop:matrix-lt}
Let $\mtx{Y}$ be a random Hermitian~matrix.  For all $t \in \mathbb{R}$,
\begin{align}
\Prob{ \lambda_{\max}(\mtx{Y}) \geq t }
	&\leq \inf_{\theta > 0} \ \econst^{-\theta t}
	\, \Expect \trace \econst^{\theta \mtx{Y}}, \quad\text{and}\quad
	 \label{eqn:matrix-lt-upper-bound} \\
\Prob{ \lambda_{\min}(\mtx{Y}) \leq t }
	&\leq \inf_{\theta < 0} \ \econst^{-\theta t}
	\, \Expect \trace \econst^{\theta \mtx{Y}}.
	 \label{eqn:matrix-lt-lower-bound}
\end{align}
\end{prop}

In words, we can control the tail probabilities of the extreme eigenvalues of a random matrix by producing a bound for the \emph{trace} of the matrix mgf.  The proof of this fact parallels the
classical argument, but there is a twist.

\begin{proof}
We begin with~\eqref{eqn:matrix-lt-upper-bound}.
Fix a positive number $\theta$, and observe that
$$
\Prob{ \lambda_{\max}(\mtx{Y}) \geq t }
	= \Prob{ \econst^{\theta \lambda_{\max}(\mtx{Y})} \geq \econst^{\theta t} } 
	\leq \econst^{- \theta t} \, \Expect \econst^{\theta \lambda_{\max}(\mtx{Y})}
	= \econst^{- \theta t} \, \Expect \econst^{\lambda_{\max}(\theta \mtx{Y})}
$$
The first identity holds because $a \mapsto \econst^{\theta a}$ is a monotone increasing function,
so the event does not change under the mapping.  The second relation is Markov's inequality~\eqref{eqn:markov}.
The last holds because the maximum eigenvalue is a positive-homogeneous map, as stated in~\eqref{eqn:eig-pos-homo}.   
To control the exponential, note that
\begin{equation} \label{eqn:lt-apply-spectral-map}
\econst^{\lambda_{\max}(\theta \mtx{Y})}
	= \lambda_{\max}\bigl(\econst^{\theta\mtx{Y}}\bigr)
	\leq \trace \econst^{\theta\mtx{Y}}.
\end{equation}
The first identity depends on the Spectral Mapping Theorem, Proposition~\ref{prop:spectral-mapping}, and the fact that the exponential function is increasing.  The inequality follows because the exponential of an Hermitian matrix is positive definite, and~\eqref{eqn:maxeig-trace} shows that the maximum eigenvalue of a positive-definite matrix is dominated by the trace.  Combine the latter two displays to reach
$$
\Prob{ \lambda_{\max}(\mtx{Y}) \geq t }
	\leq \econst^{- \theta t} \, \Expect \trace \econst^{\theta\mtx{Y}}.
$$
This inequality is valid for any positive $\theta$, so we may take an infimum to achieve the tightest possible bound.

To prove~\eqref{eqn:matrix-lt-lower-bound}, we use a similar approach.  Fix a negative number $\theta$,
and calculate that
$$
\Prob{ \lambda_{\min}(\mtx{Y}) \leq t }
	= \Prob{ \econst^{\theta \lambda_{\min}(\mtx{Y})} \geq \econst^{\theta t} }
	\leq \econst^{- \theta t} \, \Expect \econst^{\theta \lambda_{\min}(\mtx{Y})}
	= \econst^{- \theta t} \, \Expect \econst^{\lambda_{\max}(\theta \mtx{Y})}.
$$
The function $a \mapsto \econst^{\theta a}$ reverses the inequality in the event because it is monotone decreasing.
The last identity depends on the relationship~\eqref{eqn:min-max-sign-eig} between minimum and maximum eigenvalues.  Finally, we introduce the inequality~\eqref{eqn:lt-apply-spectral-map} for the trace exponential and minimize over negative values of $\theta$.
\end{proof}

In the proof of Proposition~\ref{prop:matrix-lt}, it may seem crude to bound the maximum eigenvalue by the trace.  In fact, our overall approach leads to matrix concentration inequalities that are sharp for specific examples (see the discussion in~\S\S\ref{sec:matrix-gauss-sharp}, \ref{sec:matrix-chernoff-sharp}, and~\ref{sec:bernstein-optimality}), so we must conclude that the loss in this bound is sometimes inevitable.  At the same time, this maneuver allows us to exploit some amazing convexity properties of the trace exponential.  
  
We can adapt the proof of Proposition~\ref{prop:matrix-lt} to obtain bounds for the expectation of the maximum eigenvalue of a random Hermitian matrix.  This argument does not have a perfect analog in the scalar setting.

\begin{prop}[Expectation Bounds for Eigenvalues] \label{prop:matrix-expect}
Let $\mtx{Y}$ be a random Hermitian matrix.  Then
\begin{align}
\Expect \lambda_{\max}(\mtx{Y}) &\leq \inf_{\theta > 0} \ 
\frac{1}{\theta}\log{} \Expect \trace \econst^{\theta \mtx{Y}}, \quad\text{and}\quad
\label{eqn:matrix-lt-upper-mean} \\
\Expect \lambda_{\min}(\mtx{Y}) &\geq \sup_{\theta < 0} \ 
\frac{1}{\theta}\log{} \Expect \trace \econst^{\theta \mtx{Y}}.
\label{eqn:matrix-lt-lower-mean}
\end{align}
\end{prop}

\begin{proof}
We establish the bound~\eqref{eqn:matrix-lt-upper-mean}; the proof of~\eqref{eqn:matrix-lt-lower-mean} is quite similar.  Fix a positive number $\theta$, and calculate that
$$
\Expect \lambda_{\max}(\mtx{Y}) = \frac{1}{\theta} \Expect{} \log{} \econst^{ \lambda_{\max}(\theta \mtx{Y}) }
	\leq \frac{1}{\theta} \log{} \Expect \econst^{ \lambda_{\max}(\theta \mtx{Y}) }
	= \frac{1}{\theta} \log{} \Expect \lambda_{\max}\big(\econst^{\theta \mtx{Y}}\big)
	\leq \frac{1}{\theta} \log{} \Expect \trace \econst^{\theta \mtx{Y}}.
$$
The first identity holds because the maximum eigenvalue is a positive-homogeneous map, as stated in~\eqref{eqn:eig-pos-homo}.
The second relation is Jensen's inequality.  The third follows when we use the Spectral Mapping Theorem, Proposition~\ref{prop:spectral-mapping}, to draw the eigenvalue map through the exponential.  The final inequality depends on the fact~\eqref{eqn:maxeig-trace} that the trace of a positive-definite matrix dominates the maximum eigenvalue.
\end{proof}

\section{The Failure of the Matrix Mgf} \label{sec:mom-fail}

We would like the use the Laplace transform bounds from Section~\ref{sec:matrix-lt} to study a sum of independent random matrices.  In the scalar setting, the Laplace transform method is effective for studying an independent sum because the mgf and the cgf decompose.  In the matrix case, the situation is more subtle, and the goal of this section is to indicate where things go awry.

Consider an independent sequence $\{ X_k \}$ of real random variables.  The mgf of the sum satisfies a multiplication rule:
\begin{equation} \label{eqn:mgf-mult}
M_{(\sum_k X_k)}(\theta)
	= \Expect \exp\left( \sum\nolimits_k \theta X_k \right )
	= \Expect \prod\nolimits_k \econst^{\theta X_k}
	= \prod\nolimits_k \Expect \econst^{\theta X_k}
	= \prod\nolimits_k M_{X_k}(\theta).
\end{equation}
The first identity is the definition of an mgf.  The second relation holds because the exponential map converts a sum of real scalars to a product, and the third relation requires the independence of the random variables.  The last identity, again, is the definition.

At first, we might imagine that a similar relationship holds for the matrix mgf.  Consider an independent sequence $\{ \mtx{X}_k \}$ of random Hermitian matrices.  Perhaps,
\begin{equation} \label{eqn:matrix-mgf-mult-fail}
\mtx{M}_{(\sum_k \mtx{X}_k)}(\theta) \quad \overset{?}{=} \quad
\prod\nolimits_k \mtx{M}_{\mtx{X}_k}(\theta).
\end{equation}
Unfortunately, this hope shatters when we subject it to interrogation.

It is not hard to find the reason that~\eqref{eqn:matrix-mgf-mult-fail} fails.  The identity~\eqref{eqn:mgf-mult} depends on the fact that the scalar exponential converts a sum into a product.  In contrast, for Hermitian matrices,
$$
\econst^{\mtx{A} + \mtx{H}} \neq \econst^{\mtx{A}} \econst^{\mtx{H}}
\quad\text{unless $\mtx{A}$ and $\mtx{H}$ commute.}
$$
If we introduce the trace, the situation improves somewhat:
\begin{equation} \label{eqn:golden-thompson}
\trace \econst^{\mtx{A} + \mtx{H}} \leq \trace \econst^{\mtx{A}} \econst^{\mtx{H}}
\quad\text{for all Hermitian $\mtx{A}$ and $\mtx{H}$.}
\end{equation}
The result~\eqref{eqn:golden-thompson} is known as the Golden--Thompson inequality, a famous theorem from statistical physics.  Unfortunately, the analogous bound may fail for three matrices:
$$
\trace \econst^{\mtx{A} + \mtx{H} + \mtx{T}} \not\leq \trace \econst^{\mtx{A}} \econst^{\mtx{H}} \econst^{\mtx{T}}
\quad\text{for certain Hermitian $\mtx{A},\mtx{H}$, and $\mtx{T}$.}
$$
It seems that we have reached an impasse.

What if we consider the cgf instead?  The cgf of a sum of independent real random variables satisfies an addition rule:
\begin{equation} \label{eqn:cgf-add}
\Xi_{(\sum_k X_k)}(\theta)
	= \log{} \Expect \exp\left( \sum\nolimits_k \theta X_k \right)
	= \log{} \prod\nolimits_k \Expect \econst^{\theta X_k}
	= \sum\nolimits_k \Xi_{X_k}(\theta).
\end{equation}
The relation~\eqref{eqn:cgf-add} follows when we extract the logarithm of the multiplication rule~\eqref{eqn:mgf-mult}.  This result looks like a more promising candidate for generalization because a sum of Hermitian~matrices remains Hermitian.  We might hope that
$$
\mtx{\Xi}_{(\sum_k \mtx{X}_k)}(\theta)
\quad \overset{?}{=} \quad
\sum\nolimits_k \mtx{\Xi}_{\mtx{X}_k}(\theta).
$$
As stated, this putative identity also fails.  Nevertheless, the addition rule~\eqref{eqn:cgf-add} admits a very satisfactory extension to matrices.  In contrast with the scalar case, the proof involves much deeper considerations.

\section{A Theorem of Lieb} \label{sec:lieb}

To find the appropriate generalization of the addition rule for cgfs, we turn to the literature on matrix analysis.  Here, we discover a famous result %
of Elliott Lieb on the convexity properties of the trace exponential function.

\begin{thm}[Lieb] \label{thm:lieb}
Fix an Hermitian~matrix $\mtx{H}$ with dimension $d$.  The function
$$
\mtx{A} \longmapsto \trace \exp( \mtx{H} + \log \mtx{A} )
$$
is a concave map on the convex cone of $d \times d$ positive-definite matrices.
\end{thm}

\noindent
In the scalar case, the analogous function $a \mapsto \exp(h + \log a)$ is linear, so this result describes a new type of phenomenon that emerges when we move to the matrix setting.  We present a complete proof of Theorem~\ref{thm:lieb} in
Chapter~\ref{chap:lieb}.

For now, let us focus on the consequences of this remarkable result.  Lieb's Theorem is valuable to us because the Laplace transform bounds from Section~\ref{sec:matrix-lt} involve the trace exponential function.  To highlight the connection, we rephrase Theorem~\ref{thm:lieb} in probabilistic terms.

\begin{cor} %
\label{cor:cum-ineq}
Let $\mtx{H}$ be a fixed Hermitian matrix, and let $\mtx{X}$ be a random Hermitian matrix of the same dimension.  Then
$$
\Expect \trace \exp( \mtx{H} + \mtx{X} )
	\leq \trace \exp\left( \mtx{H} + \log{} \Expect \econst^{\mtx{X}}  \right).
$$
\end{cor}

\begin{proof}
Introduce the random matrix $\mtx{Y} = \econst^{\mtx{X}}$.  Then
$$
\begin{aligned}
\Expect \trace \exp(\mtx{H} + \mtx{X})
	&= \Expect \trace \exp( \mtx{H} + \log{} \mtx{Y}  ) \\
	&\leq \trace \exp( \mtx{H} + \log{}\Expect \mtx{Y} )
	= \trace \exp\left( \mtx{H} + \log{}\Expect \econst^{\mtx{X}} \right).
\end{aligned}
$$
The first identity follows from the interpretation~\eqref{eqn:log-defn} of the matrix logarithm as the functional inverse of the matrix exponential.  Theorem~\ref{thm:lieb} shows that the trace function is concave in $\mtx{Y}$, so Jensen's inequality~\eqref{eqn:jensen} allows us to draw the expectation inside the function.
\end{proof}

\section{Subadditivity of the Matrix Cgf} \label{sec:cum-subadd}

We are now prepared to generalize the addition rule~\eqref{eqn:cgf-add} for scalar cgfs to the matrix setting.  The following result is fundamental to our approach to random matrices.

\begin{lemma}[Subadditivity of Matrix Cgfs] \label{lem:cgf-indep}
Consider a finite sequence $\{ \mtx{X}_k \}$ of independent, random, Hermitian matrices of the same dimension.  Then
\begin{equation} \label{eqn:cgf-subadd}
\Expect \trace \exp\left( \sum\nolimits_k \theta \mtx{X}_k \right)
	\leq \trace \exp\left( \sum\nolimits_k \log{} \Expect \econst^{\theta \mtx{X}_k} \right)
	\quad\text{for $\theta \in \mathbb{R}$.}
\end{equation}
Equivalently,
\begin{equation} \label{eqn:cgf-subadd-rule}
\trace \exp\left( \mtx{\Xi}_{(\sum_k \mtx{X}_k)}(\theta) \right)
	\leq \trace \exp\left( \sum\nolimits_k \mtx{\Xi}_{\mtx{X}_k}(\theta) \right)
	\quad\text{for $\theta \in \mathbb{R}$.}
\end{equation}
\end{lemma}

The parallel between the additivity rule~\eqref{eqn:cgf-add} and the subadditivity rule~\eqref{eqn:cgf-subadd-rule} is striking.
With our level of preparation, it is easy to prove this result.
We just apply the bound from Corollary~\ref{cor:cum-ineq} repeatedly.

\begin{proof}
Without loss of generality, we assume that $\theta = 1$ by absorbing the parameter into the random matrices.
Let $\Expect_k$ denote the expectation with respect to $\mtx{X}_k$, the remaining random matrices held fixed.  Abbreviate
\begin{equation*}\label{eqn:xi-k}
\mtx{\Xi}_k = \log{} \Expect_{k} \econst^{\mtx{X}_k}
	= \log{} \Expect \econst^{\mtx{X}_k}.
\end{equation*}
We may calculate that
\begin{align*}
\Expect \trace \exp\left( \sum\nolimits_{k=1}^n \mtx{X}_k \right)
	&= \Expect \Expect_n
		\trace \exp\left( \sum\nolimits_{k=1}^{n-1} \mtx{X}_k
		+ \mtx{X}_n \right)  \\
	&\leq \Expect
		\trace \exp\left( \sum\nolimits_{k=1}^{n-1} \mtx{X}_k
		+ \log{} \Expect_{n} \econst^{ \mtx{X}_n} \right) \\
	&= \Expect \Expect_{n-1}
		\trace \exp\left( \sum\nolimits_{k=1}^{n-2} \mtx{X}_k
		+ \mtx{X}_{n-1} + \mtx{\Xi}_n \right) \\
	&\leq \Expect \Expect_{n-2}
		\trace \exp\left( \sum\nolimits_{k=1}^{n-2} \mtx{X}_k
		+ \mtx{\Xi}_{n-1} + \mtx{\Xi}_{n} \right) \\
\cdots\quad
	&\leq \trace \exp\biggl( \sum\nolimits_{k=1}^{n} \mtx{\Xi}_k \biggr).
\end{align*}
We can introduce iterated expectations because of the tower property of conditional expectation.
To bound the expectation $\Expect_m$ for an index $m = 1, 2, 3, \dots, n$,
we invoke Corollary~\ref{cor:cum-ineq} with the fixed matrix $\mtx{H}$ equal to
$$
\mtx{H}_m = \sum\limits_{k=1}^{m-1} \mtx{X}_k +
	\sum\limits_{k=m+1}^n \mtx{\Xi}_k.
$$
This argument is legitimate because $\mtx{H}_m$ is independent from $\mtx{X}_m$.

The formulation~\eqref{eqn:cgf-subadd-rule} follows from~\eqref{eqn:cgf-subadd} when we substitute the expression~\eqref{eqn:matrix-mgf-cgf} for the matrix cgf and make some algebraic simplifications.
\end{proof}

\section{Master Bounds for Sums of Independent Random Matrices} \label{sec:master-tail}

Finally, we can present some general results on the behavior of a sum of independent random matrices.  At this stage, we simply combine the Laplace transform bounds with the subadditivity of the matrix cgf to obtain abstract inequalities.  Later, we will harness properties of the summands to develop more concrete estimates that apply to specific examples of interest.

\begin{thm}[Master Bounds for a Sum of Independent Random Matrices] \label{thm:master-ineq}
Consider a finite sequence $\{ \mtx{X}_k \}$ of independent, random, Hermitian matrices
of the same size.  Then
\begin{align}
\Expect \lambda_{\max}\left( \sum\nolimits_k \mtx{X}_k \right)
	&\leq \inf_{\theta > 0} \ \frac{1}{\theta} \log{}
	\trace \exp\left( \sum\nolimits_k
	\log{} \Expect \econst^{\theta \mtx{X}_k} \right),
	\quad\text{and}
	\label{eqn:master-upper-expect} \\
\Expect \lambda_{\min}\left( \sum\nolimits_k \mtx{X}_k \right)
	&\geq \sup_{\theta < 0} \ \frac{1}{\theta} \log{}
	\trace \exp\left( \sum\nolimits_k
	\log{} \Expect \econst^{\theta \mtx{X}_k} \right).
	\label{eqn:master-lower-expect}
\end{align}
Furthermore, for all $t \in \mathbb{R}$,
\begin{align}
\Prob{ \lambda_{\max}\left( \sum\nolimits_k \mtx{X}_k \right) \geq t }
	&\leq \inf_{\theta > 0} \ \econst^{-\theta t}
	\,\trace \exp\left( \sum\nolimits_k
	\log{} \Expect \econst^{\theta \mtx{X}_k} \right),
	\quad\text{and}
	\label{eqn:master-upper-tail}\\
\Prob{ \lambda_{\min}\left( \sum\nolimits_k \mtx{X}_k \right) \leq t }
	&\leq \inf_{\theta < 0} \ \econst^{-\theta t}
	\, \trace \exp\left( \sum\nolimits_k
	 \log{} \Expect \econst^{\theta \mtx{X}_k} \right).
	\label{eqn:master-lower-tail}
\end{align}
\end{thm}

\begin{proof}
Substitute the subadditivity rule for matrix cgfs, Lemma~\ref{lem:cgf-indep}, into the two matrix Laplace transform results, Proposition~\ref{prop:matrix-lt} and Proposition~\ref{prop:matrix-expect}.
\end{proof}

In this chapter, we have focused on probability inequalities for the extreme eigenvalues of a sum of independent random matrices.  Nevertheless, these results also give information about the spectral norm of a sum of independent, random, rectangular matrices because we can apply them to the Hermitian dilation~\eqref{eqn:herm-dilation} of the sum.  Instead of presenting a general theorem, we find it more natural to extend individual results to the non-Hermitian case.

\section{Notes}

This section includes some historical discussion about the results we have described in this chapter, along with citations for the results that we have established.

\subsection{The Matrix Laplace Transform Method}

The idea of lifting the ``Bernstein trick'' to the matrix setting is due to two researchers in quantum information theory, Rudolf Ahlswede and Andreas Winter, who were working on a problem concerning transmission of information through a quantum channel~\cite{AW02:Strong-Converse}.  Their paper contains a version of the matrix Laplace transform result, Proposition~\ref{prop:matrix-lt}, along with a substantial number of related foundational ideas.  Their work is one of the major inspirations for the tools that are described in these notes.

The statement of Proposition~\ref{prop:matrix-lt} and the proof that we present appear in the paper~\cite{Oli10:Sums-Random} of Roberto Oliveira.  The subsequent result on expectations, Proposition~\ref{prop:matrix-expect}, first appeared in the paper~\cite{CGT12:Masked-Sample}.

\subsection{Subadditivity of Cumulants}

The major impediment to applying the matrix Laplace transform method is the need to produce a bound for the trace of the matrix moment generating function (the trace mgf).  This is where all the technical difficulty in the argument resides.

Ahlswede \& Winter~\cite[App.]{AW02:Strong-Converse} proposed an approach for bounding the trace mgf of an independent sum, based on a repeated application of the Golden--Thompson inequality~\eqref{eqn:golden-thompson}.  Their argument leads to a cumulant bound of the form
\begin{equation} \label{eqn:aw-cumulant}
\Expect \trace \exp\left( \sum\nolimits_k \mtx{X}_k \right)
	\leq d \cdot \exp\left( \sum\nolimits_k \lambda_{\max}\bigl(\log{} \Expect \econst^{\mtx{X}_k} \bigr) \right)
\end{equation}
when the random Hermitian matrices $\mtx{X}_k$ have dimension $d$.
In other words, Ahlswede \& Winter bound the cumulant of a sum in terms of the sum of the \emph{maximum eigenvalues} of the cumulants.  There are cases where the bound~\eqref{eqn:aw-cumulant} is equivalent with Lemma~\ref{lem:cgf-indep}.
For example, the estimates coincide when each matrix $\mtx{X}_k$ is identically distributed.
In general, however, the estimate~\eqref{eqn:aw-cumulant} leads to fundamentally weaker results
than our bound from Lemma~\ref{lem:cgf-indep}.  In the worst case, the approach of Ahlswede \& Winter
may produce an unnecessary factor of the dimension $d$ in the \emph{exponent}.
See~\cite[\S\S3.7, 4.8]{Tro11:User-Friendly-FOCM} for details.

The first major technical advance beyond the original argument of Ahlswede \& Winter appeared in a paper~\cite{Oli10:Concentration-Adjacency} of Oliveira.  He developed a more effective way to deploy the Golden--Thompson inequality, and he used this technique to establish a matrix version of Freedman's inequality~\cite{Fre75:Tail-Probabilities}.  In the scalar setting, Freedman's inequality extends the Bernstein concentration inequality to martingales; Oliveira obtained the analogous extension of Bernstein's inequality for matrix-valued martingales.  When specialized to independent sums, his result is quite similar with the matrix Bernstein inequality, Theorem~\ref{thm:matrix-bernstein-intro}, apart from the precise values of the constants.  Oliveira's method, however, does not seem to deliver the full spectrum of matrix concentration inequalities that we discuss in these notes.

The approach here, based on Lieb's Theorem, was introduced in the article~\cite{Tro11:User-Friendly-FOCM} by the author of these notes.  This paper was apparently the first to recognize that Lieb's Theorem has probabilistic content, as stated in Corollary~\ref{cor:cum-ineq}.  This idea leads to Lemma~\ref{lem:cgf-indep}, on the subadditivity of cumulants, along with the master tail bounds from Theorem~\ref{thm:master-ineq}.  Note that the two articles~\cite{Oli10:Concentration-Adjacency,Tro11:User-Friendly-FOCM} are independent works.

For a detailed discussion of the benefits of Lieb's Theorem over the Golden--Thompson inequality, see~\cite[\S4]{Tro11:User-Friendly-FOCM}.  In summary, to get the sharpest concentration results for random matrices, Lieb's Theorem appears
to be indispensible.  The approach of Ahlswede \& Winter seems intrinsically weaker.  Oliveira's argument has certain advantages, however, in that it extends from matrices to the fully noncommutative setting~\cite{JZ12:Noncommutative-Martingale}.

Subsequent research on the underpinnings of the matrix Laplace transform method has led to a martingale version of the subadditivity of cumulants~\cite{Tro11:Freedmans-Inequality, Tro11:User-Friendly-Martingale-TR};
these works also depend on Lieb's Theorem.
The technical report~\cite{GT11:Tail-Bounds} shows how to use a related result,
called the Lieb--Seiringer Theorem~\cite{LS05:Stronger-Subadditivity},
to obtain upper and lower tail bounds for all eigenvalues of a sum of
independent random Hermitian matrices.

\subsection{Noncommutative Moment Inequalities}

There is a closely related, and much older, line of research on noncommutative moment inequalities.  These results provide information about the expected trace of a power of a sum of independent random matrices.  The matrix Laplace transform method, as encapsulated in Theorem~\ref{thm:master-ineq}, gives analogous bounds for the exponential moments.

Research on noncommutative moment inequalities dates to an
important paper~\cite{LP86:Inegalites-Khintchine} of Fran{\c c}oise Lust-Piquard,
which contains an operator extension of the Khintchine inequality.  Her result, now called
the \term{noncommutative Khintchine inequality}, controls the trace moments of a sum
of fixed matrices, each modulated by an independent Rademacher random variable;
see Section~\ref{sec:nc-khintchine} for more details.

In recent years, researchers have generalized many other moment inequalities for a sum of scalar random variables to matrices (and beyond).  For instance, the Rosenthal--Pinelis inequality for a sum of independent zero-mean random variables admits a matrix version~\cite{JZ13:Noncommutative-Bennett,MJCFT12:Matrix-Concentration,CGT12:Masked-Sample}.  We present a variant of the latter result below in~\eqref{eqn:matrix-moment-ineq}.  See the paper~\cite{JX05:Best-Constants} for a good overview of some other noncommutative moment inequalities.

Finally, and tangentially, we mention that a different notion of matrix moments and cumulants plays a central role in the theory of free probability~\cite{NS06:Lectures-Combinatorics}.

\subsection{Quantum Statistical Mechanics}

A curious feature of the theory of matrix concentration inequalities is that the most powerful tools come from the mathematical theory of quantum statistical mechanics.  This field studies the bulk statistical properties of interacting quantum systems, and it would seem quite distant from the field of random matrix theory.  The connection between these two areas has emerged because of research on quantum information theory, which studies how information can be encoded, operated upon, and transmitted via quantum mechanical systems.

The Golden--Thompson inequality is a major result from quantum statistical mechanics.
Bhatia's book~\cite[Sec.~IX.3]{Bha97:Matrix-Analysis} contains a detailed treatment of this result from the perspective of matrix theory.  For an account with more physical content, see the book of Thirring~\cite{Thi02:Quantum-Mathematical}.
The fact that the Golden--Thompson inequality fails for three matrices can be obtained from simple examples, such as combinations of Pauli spin matrices~\cite[Exer.~IX.8.4]{Bha97:Matrix-Analysis}.

Lieb's Theorem~\cite[Thm.~6]{Lie73:Convex-Trace} was first established in an important paper of Elliott Lieb on the convexity of trace functions.  His main goal was to establish concavity properties for a function that measures the amount of information in a quantum system.  See the notes in Chapter~\ref{chap:lieb} for a more detailed discussion.

\makeatletter{}%

\chapter[Matrix Gaussian \& Rademacher Series]{Matrix Gaussian Series \& \\ Matrix Rademacher Series} \label{chap:matrix-series}

In this chapter, we present our first set of matrix concentration inequalities.  These results provide spectral information about a sum of fixed matrices, each modulated by an independent scalar random variable.  This type of formulation is surprisingly versatile, and it captures a range of interesting examples.  Our main goal, however, is to introduce
matrix concentration in the simplest setting possible.

To be more precise about our scope, let us introduce the concept of a matrix Gaussian series.
Consider a finite sequence $\{ \mtx{B}_k \}$ of fixed matrices with the same dimension,
along with a finite sequence $\{ \gamma_k \}$ of independent standard normal random variables.
We will study the spectral norm of the random matrix
$$
\mtx{Z} = \sum\nolimits_k \gamma_k \mtx{B}_k.
$$
This expression looks abstract, but it has concrete modeling power.
For example, we can express a Gaussian Wigner matrix, one of the
classical random matrices, in this fashion.  But the real value of this
approach is that we can use matrix Gaussian series to represent
many kinds of random matrices built from Gaussian random variables.
This technique allows us to attack problems that classical methods
do not handle gracefully.  For instance, we can easily study a
Toeplitz matrix with Gaussian entries.

Similar ideas allow us to treat a \term{matrix Rademacher series}, a sum of fixed matrices modulated by random signs.
(Recall that a \term{Rademacher random variable} takes the values $\pm 1$ with equal probability.)
The results in this case are almost identical with the results for matrix Gaussian series,
but they allow us to consider new problems.
As an example, we can study the expected spectral norm of a fixed real matrix after flipping
the signs of the entries at random.

\subsubsection{Overview}

In \S\ref{sec:matrix-gauss-rect}, we begin
with an overview of our results for matrix Gaussian
series; very similar results also hold for matrix Rademacher series.
Afterward, we discuss the accuracy of the theoretical bounds.
The subsequent sections, \S\S\ref{sec:gauss-matrices}--\ref{sec:toeplitz},
describe what the matrix concentration inequalities
tell us about some classical and not-so-classical examples of random matrices.
Section~\ref{sec:maxqp} includes an overview of a more substantial
application in combinatorial optimization.  The final part \S\ref{sec:matrix-gauss-proof}
contains detailed proofs of the bounds.
We conclude with bibliographical notes.

\section{A Norm Bound for Random Series with Matrix Coefficients} \label{sec:matrix-gauss-rect}

Consider a finite sequence $\{ b_k \}$ of real numbers and a finite sequence $\{ \gamma_k \}$ of independent standard normal random variables.  Form the random series $Z = \sum_k \gamma_k b_k$.
A routine invocation of the scalar Laplace transform method demonstrates that
\begin{equation} \label{eqn:real-gauss}
\Prob{ \abs{Z} \geq t }
	\leq 2\, \exp\left( \frac{-t^2}{2v} \right)
\quad\text{where $v = \Var(Z) = \sum\nolimits_k b_k^2$.}
\end{equation}
It turns out that the inequality~\eqref{eqn:real-gauss} extends directly to the matrix setting.

\begin{thm}[Matrix Gaussian \& Rademacher Series] \label{thm:matrix-gauss-rect}
Consider a finite sequence $\{ \mtx{B}_k \}$ of fixed complex matrices with dimension $d_1 \times d_2$,
and let $\{\gamma_k\}$ be a finite sequence of independent standard normal variables.
Introduce the matrix Gaussian series
\begin{equation} \label{eqn:matrix-gauss-series}
\mtx{Z} = \sum\nolimits_k \gamma_k \mtx{B}_k.
\end{equation}
Let $v(\mtx{Z})$ be the matrix variance statistic of the sum:
\begin{align}
v(\mtx{Z}) &= \max\left\{ \norm{ \smash{\Expect ( \mtx{ZZ}^\adj )} }, \
	\norm{ \smash{\Expect ( \mtx{Z}^\adj \mtx{Z} )} } \right\}  \label{eqn:matrix-gauss-sigma2-rect} \\
	&= \max\left\{ \norm{ \sum\nolimits_k \mtx{B}_k \mtx{B}_k^\adj }, \
	\norm{ \sum\nolimits_k \mtx{B}_k^\adj \mtx{B}_k } \right\}. \label{eqn:matrix-gauss-rect-var-calc}
\end{align}
Then  %
\begin{equation} \label{eqn:matrix-gauss-expect-rect}
\Expect \norm{ \mtx{Z} }
	\leq \sqrt{2 v(\mtx{Z}) \log (d_1 + d_2)}.
\end{equation}
Furthermore, for all $t \geq 0$,
\begin{equation} \label{eqn:matrix-gauss-tail-rect}
\Prob{ \norm{ \mtx{Z} } \geq t}
	\leq (d_1 + d_2) \, \exp\left( \frac{- t^2}{2v(\mtx{Z})} \right).
\end{equation}
The same bounds hold when we replace $\{\gamma_k\}$ by a finite sequence $\{\varrho_k\}$ of independent Rademacher random variables.
\end{thm}

\noindent
The proof of Theorem~\ref{thm:matrix-gauss-rect} appears below
in~\S\ref{sec:matrix-gauss-proof}.

\subsection{Discussion}

Let us take a moment to discuss the content of Theorem~\ref{thm:matrix-gauss-rect}.
The main message is that the expectation of $\norm{\mtx{Z}}$
is controlled by the matrix variance statistic $v(\mtx{Z})$.  Furthermore,
$\norm{\mtx{Z}}$ has a subgaussian tail whose decay rate depends on $v(\mtx{Z})$.

The matrix variance statistic $v(\mtx{Z})$ defined in~\eqref{eqn:matrix-gauss-sigma2-rect}
specializes the general formulation~\eqref{eqn:matrix-variance-rect}.  The second
expression~\eqref{eqn:matrix-gauss-rect-var-calc} follows from the additivity
property~\eqref{eqn:indep-sum-rect} for the variance of an independent sum.
When the summands are Hermitian, observe that the two terms in the maximum coincide.
The formulas~\eqref{eqn:matrix-gauss-sigma2-rect} and~\eqref{eqn:matrix-gauss-rect-var-calc}
are a direct extension of the variance that arises in the scalar bound~\eqref{eqn:real-gauss}.  

As compared with~\eqref{eqn:real-gauss},
a new feature of the bound~\eqref{eqn:matrix-gauss-tail-rect} is the
dimensional factor $d_1 + d_2$.  When $d_1 = d_2 = 1$,
the matrix bound reduces to the scalar result~\eqref{eqn:real-gauss}.
In this case, at least, we have lost nothing by lifting the Laplace transform
method to matrices.
The behavior of the matrix tail bound~\eqref{eqn:matrix-gauss-tail-rect}
is more subtle than the behavior of the scalar tail bound~\eqref{eqn:real-gauss}.
See Figure~\ref{fig:gauss-tail-schema} for an illustration.

\begin{figure}
\begin{center}
\includegraphics[width=0.9\textwidth]{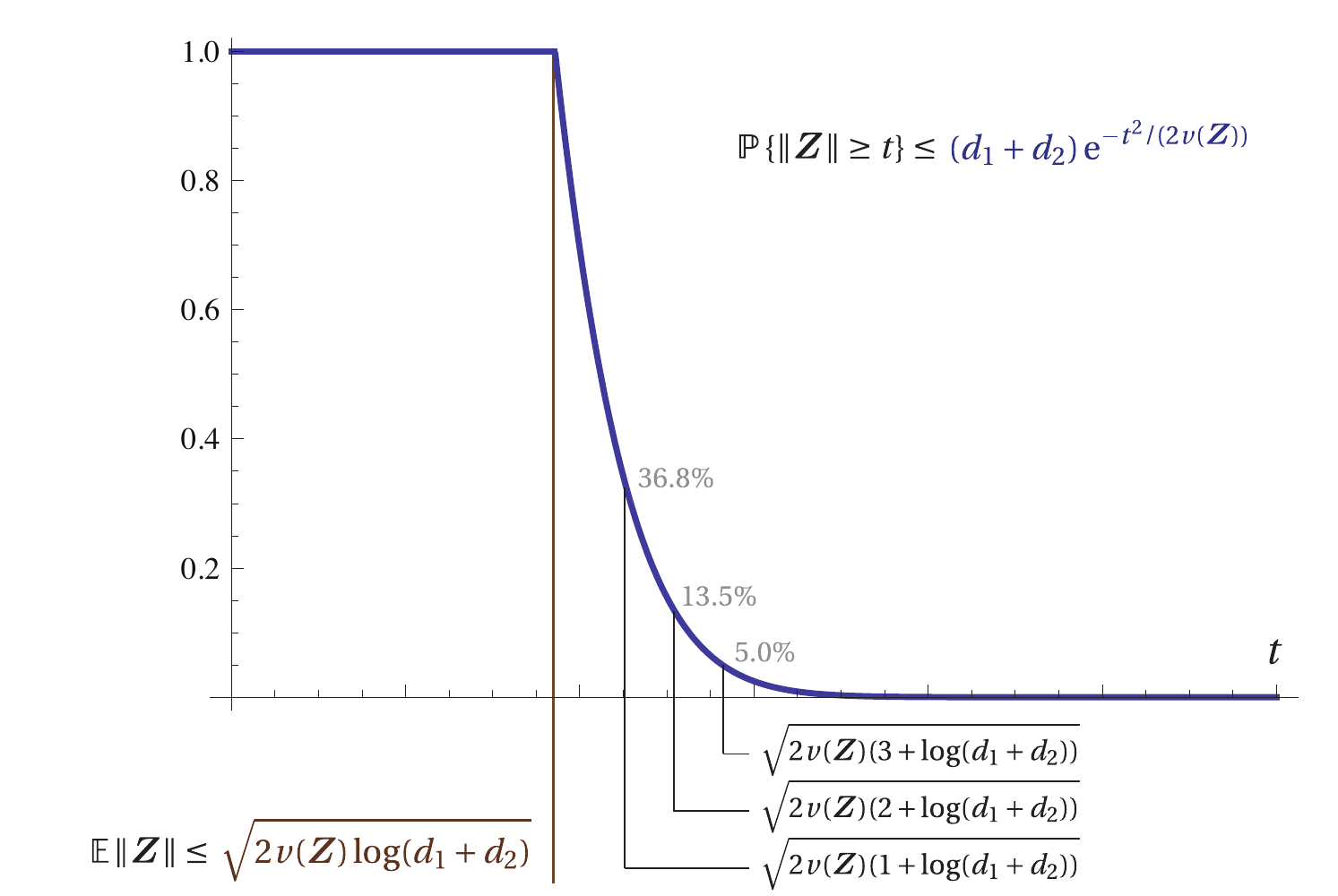}
\begin{caption}{
\textbf{Schematic of tail bound for matrix Gaussian series.}
Consider a matrix Gaussian series $\mtx{Z}$ with dimension $d_1 \times d_2$.
The tail probability $\Prob{ \norm{\mtx{Z}} \geq t }$ admits the upper bound
$(d_1 + d_2) \, \exp(-t^2/(2v(\mtx{Z})))$, marked as a dark blue curve.
This estimate provides no information below the level $t = \sqrt{2 v(\mtx{Z}) \log(d_1 + d_2)}$.
This value, the dark red vertical line,
coincides with the upper bound~\eqref{eqn:matrix-gauss-expect-rect} for $\Expect \norm{\mtx{Z}}$.
As $t$ increases beyond this point, the tail probability decreases at a subgaussian
rate with variance on the order of $v(\mtx{Z})$.}
\label{fig:gauss-tail-schema}
\end{caption}
\end{center}
\end{figure}

\subsection{Optimality of the Bounds for Matrix Gaussian Series} \label{sec:matrix-gauss-sharp}

One may wonder whether Theorem~\ref{thm:matrix-gauss-rect} provides accurate information about the behavior of a matrix Gaussian series.  The answer turns out to be complicated.  Here is the executive summary: the expectation bound~\eqref{eqn:matrix-gauss-expect-rect} is always quite good, but the tail bound~\eqref{eqn:matrix-gauss-tail-rect}
is sometimes quite bad.  The rest of this section expands on these claims.

\subsubsection{The Expectation Bound}

Let $\mtx{Z}$ be a matrix Gaussian series of the form~\eqref{eqn:matrix-gauss-series}.  We will argue that
\begin{equation} \label{eqn:matrix-gauss-expect-two-sided}
v(\mtx{Z})
	\quad\leq\quad
	\Expect \normsq{\mtx{Z}}
	\quad\leq\quad
	2 v(\mtx{Z}) (1 + \log(d_1 + d_2)).
\end{equation}
In other words, the matrix variance $v(\mtx{Z})$ is roughly the correct scale for $\normsq{\mtx{Z}}$.
This pair of estimates is a significant achievement because it is quite challenging to compute the
norm of a matrix Gaussian series in general.  Indeed, the literature contains very few examples
where explicit estimates are available, especially if one desires reasonable constants.

We begin with the lower bound in~\eqref{eqn:matrix-gauss-expect-two-sided},
which is elementary.  Indeed, since the spectral norm is convex,
Jensen's inequality ensures that
$$
\Expect \normsq{\mtx{Z}}
	= \Expect{} \max\big\{ \norm{\smash{\mtx{ZZ}^\adj}}, \ \norm{\smash{\mtx{Z}^\adj\mtx{Z}}} \big\}
	\geq \max\big\{ \norm{\smash{\Expect(\mtx{ZZ}^\adj)}}, \ \norm{\smash{\Expect(\mtx{Z}^\adj\mtx{Z})}} \big\}
	= v(\mtx{Z}).
$$
The first identity follows from~\eqref{eqn:spectral-norm-square}, and the last
is the definition~\eqref{eqn:matrix-variance-rect} of the matrix variance.

The upper bound in~\eqref{eqn:matrix-gauss-expect-two-sided} is
a consequence of the tail bound~\eqref{eqn:matrix-gauss-tail-rect}:
$$
\begin{aligned}
\Expect{} \normsq{\mtx{Z}}
	&= \int_0^\infty 2t \, \Prob{ \norm{\mtx{Z}} \geq t } \idiff{t} \\
	&\leq \int_0^E 2t \idiff{t} + 2 (d_1 + d_2) \int_E^\infty t \, \econst^{-t^2/(2v(\mtx{Z}))} \idiff{t}
	= E^2 + 2v(\mtx{Z}) \, (d_1 + d_2) \, \econst^{-E^2/(2v(\mtx{Z}))}. \phantom{\int_0^E}
\end{aligned}
$$
In the first step, rewrite the expectation using integration by parts, and then split the integral
at a positive number $E$.  In the first term, we bound the probability by one, while the second term
results from the tail bound~\eqref{eqn:matrix-gauss-tail-rect}.  Afterward, we compute the integrals
explicitly.  Finally, select $E^2 = 2v(\mtx{Z}) \log(d_1+d_2)$ to complete the proof
of~\eqref{eqn:matrix-gauss-expect-two-sided}.

\subsubsection{About the Dimensional Factor}

At this point, one may ask whether it is possible to improve either side of the inequality~\eqref{eqn:matrix-gauss-expect-two-sided}.  The answer is negative unless we have additional information about the Gaussian series
beyond the matrix variance statistic $v(\mtx{Z})$.

Indeed, for arbitrarily large dimensions $d_1$ and $d_2$, we can exhibit a matrix Gaussian series where
the left-hand inequality in~\eqref{eqn:matrix-gauss-expect-two-sided} is correct.
That is, $\Expect \normsq{\mtx{Z}} \approx v(\mtx{Z})$ with no additional dependence on the dimensions $d_1$ or $d_2$.
One such example appears below in~\S\ref{sec:marcenko-pastur}.

At the same time, for arbitrarily large dimensions $d_1$ and $d_2$,
we can construct a matrix Gaussian series where the right-hand
inequality in~\eqref{eqn:matrix-gauss-expect-two-sided} is correct.
That is, $\Expect \normsq{\mtx{Z}} \approx v(\mtx{Z}) \log(d_1 + d_2)$.
See~\S\ref{sec:toeplitz} for an example.

We can offer a rough intuition about how these two situations differ from each other.
The presence or absence of the dimensional factor $\log(d_1 + d_2)$
depends on how much the coefficients $\mtx{B}_k$ in the matrix Gaussian series $\mtx{Z}$ 
commute with each other.  More commutativity leads to a logarithm,
while less commutativity can sometimes result in cancelations that obliterate the logarithm.
It remains a major open question to find a simple quantity, computable from the coefficients $\mtx{B}_k$,
that decides whether $\Expect \normsq{\mtx{Z}}$ contains a dimensional factor or not.

In Chapter~\ref{chap:intrinsic}, we will describe a technique that allows us to moderate the dimensional factor in~\eqref{eqn:matrix-gauss-expect-two-sided} for some types of matrix series.  But we cannot remove the dimensional factor entirely with current technology.

\subsubsection{The Tail Bound}

What about the tail bound~\eqref{eqn:matrix-gauss-tail-rect} for the norm of the Gaussian series?  Here, our results are less impressive.  It turns out that the large-deviation behavior of the spectral norm of a matrix Gaussian series $\mtx{Z}$ is controlled by a
statistic $v_\star(\mtx{Z})$ called the \term{weak variance}:
$$
v_{\star}(\mtx{Z})
	= \sup_{\norm{\vct{u}}=\norm{\vct{w}}=1} \Expect{} \abssq{ \smash{\vct{u}^\adj \mtx{Z} \vct{w}} }
	= \sup_{\norm{\vct{u}}=\norm{\vct{w}}=1} \sum\nolimits_k \abssq{ \smash{\vct{u}^\adj \mtx{B}_k \vct{w}} }.
$$
The best general inequalities between the matrix variance statistic and the weak variance are
$$
v_{\star}(\mtx{Z}) \quad\leq\quad v(\mtx{Z}) \quad\leq\quad \min\{ d_1, d_2 \} \cdot v_{\star}(\mtx{Z})
$$
There are examples of matrix Gaussian series that saturate the lower or the upper inequality.

The classical concentration inequality~\cite[Thm.~5.6]{BLM13:Concentration-Inequalities}
for a function of independent Gaussian random variables implies that
\begin{equation} \label{eqn:gauss-concentration}
\Prob{ \norm{\mtx{Z}} \geq \Expect \norm{\mtx{Z}} + t }
	\leq \econst^{-t^2/(2v_{\star}(\mtx{Z}))}.
\end{equation}
Let us emphasize that the bound~\eqref{eqn:gauss-concentration} provides no information about $\Expect \norm{\mtx{Z}}$;
it only tells us about the probability that $\norm{\mtx{Z}}$ is larger than its mean.

Together, the last two displays indicate that the exponent in the
tail bound~\eqref{eqn:matrix-gauss-tail-rect} is sometimes too big by a factor $\min\{d_1,d_2\}$.
Therefore, a direct application of Theorem~\ref{thm:matrix-gauss-rect} can badly overestimate
the tail probability $\Prob{ \norm{\mtx{Z}} > t }$ when the level $t$ is large.
Fortunately, this problem is less pronounced with the matrix Chernoff inequalities of Chapter~\ref{chap:matrix-chernoff} and the matrix Bernstein inequalities of Chapter~\ref{chap:matrix-bernstein}.

\subsubsection{Expectations and Tails}

When studying concentration of random variables, it is quite common that we need to use one method
to assess the expected value of the random variable and a separate technique to determine the
probability of a large deviation.

\begin{quotation}
\noindent
\textbf{The primary value of matrix concentration inequalities inheres in the estimates
that they provide for the expectation of the spectral norm
(or maximum eigenvalue or minimum eigenvalue) of a random matrix.}
\end{quotation}

\noindent
In many cases, matrix concentration bounds provide reasonable information about
the tail decay, but there are other situations where the tail bounds are feeble.
In this event, we recommend applying a scalar concentration inequality to
control the tails.

\section{Example: Some Gaussian Matrices} \label{sec:gauss-matrices}

Let us try out our methods on two types of Gaussian matrices that have been studied extensively in the classical literature on random matrix theory.  In these cases, precise information about the spectral distribution is available, which provides a benchmark for assessing our results.  We find that bounds based on Theorem~\ref{thm:matrix-gauss-rect}
lead to very reasonable estimates, but they are not sharp.  The advantage of our approach is that it applies to every example, whereas we are making comparisons with specialized techniques that only illuminate individual cases.
Similar conclusions hold for matrices with independent Rademacher entries.

\subsection{Gaussian Wigner Matrices} \label{sec:wigner}

We begin with a family of Gaussian Wigner matrices.  A $d \times d$ matrix $\mtx{W}_d$ from this ensemble is real-symmetric with a zero diagonal; the entries above the diagonal are independent normal variables with mean zero and variance one:
$$
\mtx{W}_d = \begin{bmatrix}
	0 & \gamma_{12} & \gamma_{13} &  \dots & \gamma_{1d} \\
	\gamma_{12} & 0 & \gamma_{23} & \dots  & \gamma_{2d} \\
	\gamma_{13} & \gamma_{23} & 0 &  & \gamma_{3d} \\
	\vdots & \vdots && \ddots & \vdots \\
	\gamma_{1d} & \gamma_{2d} & \dots & \gamma_{d-1,d} & 0
\end{bmatrix}
$$
where $\{ \gamma_{jk} : 1 \leq j < k \leq d \}$ is an independent family of standard normal variables.  We can represent this matrix compactly as a Gaussian series:
\begin{equation} \label{eqn:gauss-wigner}
\mtx{W}_d =
\sum\limits_{1 \leq j < k \leq d} \gamma_{jk} (\mathbf{E}_{jk} + \mathbf{E}_{kj}).
\end{equation}
The norm of a Wigner matrix satisfies
\begin{equation} \label{eqn:gauss-wigner-true}
\frac{1}{\sqrt{d}} \, \norm{\mtx{W}_d} \longrightarrow 2
\quad\text{as $d \to \infty$, almost surely}.
\end{equation}
For example, see~\cite[Thm.~5.1]{BS10:Spectral-Analysis}.
To make~\eqref{eqn:gauss-wigner-true} precise,
we assume that $\{\mtx{W}_d\}$ is an independent sequence of Gaussian Wigner matrices, indexed by the dimension $d$.

Theorem~\ref{thm:matrix-gauss-herm} provides a simple way to bound the norm of a Gaussian Wigner matrix.  We just need to compute the matrix variance statistic $v(\mtx{W}_d)$.
The formula~\eqref{eqn:matrix-gauss-rect-var-calc} for $v(\mtx{W}_d)$ asks us to form the sum of the squared coefficients from the representation~\eqref{eqn:gauss-wigner}:
$$
\sum\limits_{1 \leq j < k \leq d} (\mathbf{E}_{jk} + \mathbf{E}_{kj})^2
	= \sum\limits_{1 \leq j < k \leq d} (\mathbf{E}_{jj} + \mathbf{E}_{kk})
	= (d-1) \, \Id_d.
$$
Since the terms in~\eqref{eqn:gauss-wigner} are Hermitian, we have only one sum of squares to consider.
We have also used the facts that $\mathbf{E}_{jk} \mathbf{E}_{kj} = \mathbf{E}_{jj}$ while $\mathbf{E}_{jk} \mathbf{E}_{jk} = \mtx{0}$ because of the condition $j < k$ in the limits of summation.
We see that
$$
v(\mtx{W}_d) 
	= \norm{ \sum\limits_{1 \leq j < k \leq d} (\mathbf{E}_{jk} + \mathbf{E}_{kj})^2 }
	= \norm{ (d-1) \, \Id_d } = d-1.
$$
The bound~\eqref{eqn:matrix-gauss-expect-rect} for the expectation of the norm gives
\begin{equation} \label{eqn:gauss-wigner-est}
\Expect \norm{\mtx{W}_d} \leq \sqrt{2 (d-1) \log(2d)}.
\end{equation}
In conclusion, our techniques overestimate $\norm{\mtx{W}_d}$ by a factor of about $\sqrt{0.5 \log d}$.
The result~\eqref{eqn:gauss-wigner-est} is not perfect, but it only takes two lines of work.
In contrast, the classical result~\eqref{eqn:gauss-wigner-true} depends on a long moment
calculation that involves challenging combinatorial arguments.

\subsection{Rectangular Gaussian Matrices} \label{sec:marcenko-pastur}

Next, we consider a $d_1 \times d_2$ rectangular matrix with independent standard normal entries:
$$
\mtx{G} = \begin{bmatrix}
	\gamma_{11} & \gamma_{12} & \gamma_{13} & \dots & \gamma_{1d_2} \\
	\gamma_{21} & \gamma_{22} & \gamma_{23} & \dots & \gamma_{2d_2} \\
	\vdots & \vdots &&\ddots & \vdots \\
	\gamma_{d_1 1} & \gamma_{d_1 2} & \gamma_{d_1 3} & \dots & \gamma_{d_1d_2} \\
\end{bmatrix}
$$
where $\{ \gamma_{jk} \}$
is an independent family of standard normal variables.  We can express this
matrix efficiently using a Gaussian series:
\begin{equation} \label{eqn:mp-gauss-series}
\mtx{G} = \sum_{j=1}^{d_1} \sum_{k=1}^{d_2} \gamma_{jk} \mathbf{E}_{jk},
\end{equation}
There is an elegant estimate~\cite[Thm.~2.13]{DS02:Local-Operator} for the norm of this matrix:
\begin{equation} \label{eqn:gauss-rect-true}
\Expect \norm{\mtx{G}} \leq \sqrt{d_1} + \sqrt{d_2}.
\end{equation}
The inequality~\eqref{eqn:gauss-rect-true} is sharp when $d_1$ and $d_2$ tend to infinity while the ratio $d_1/d_2 \to \mathrm{const}$.  See~\cite[Thm.~5.8]{BS10:Spectral-Analysis} for details.

Theorem~\ref{thm:matrix-gauss-rect} yields another bound on the expected norm of the matrix $\mtx{G}$.
In order to compute the matrix variance statistic $v(\mtx{G})$,
we calculate the sums of the squared coefficients
from the representation~\eqref{eqn:mp-gauss-series}:
$$
\begin{aligned}
\sum_{j=1}^{d_1} \sum_{k=1}^{d_2} \mathbf{E}_{jk} \mathbf{E}_{jk}^\adj
	&= \sum_{j=1}^{d_1} \sum_{k=1}^{d_2} \mathbf{E}_{jj}
	= d_2 \, \Id_{d_1}, \quad\text{and} \\
\sum_{j=1}^{d_1} \sum_{k=1}^{d_2} \mathbf{E}_{jk}^\adj \mathbf{E}_{jk}
	&= \sum_{j=1}^{d_1} \sum_{k=1}^{d_2} \mathbf{E}_{kk}
	= d_1 \, \Id_{d_2}.
\end{aligned}
$$
The matrix variance statistic~\eqref{eqn:matrix-gauss-sigma2-rect} satisfies
$$
v(\mtx{G}) = \max\big\{ \norm{ \smash{d_2 \, \Id_{d_1}} }, \ \norm{ \smash{d_1 \, \Id_{d_2}} } \big\}
	= \max\{ d_1, \ d_2 \}.
$$
We conclude that
\begin{equation} \label{eqn:gauss-rect-est}
\Expect \norm{ \mtx{G} } \leq \sqrt{2 \max\{d_1, \ d_2\} \log(d_1 + d_2)}.
\end{equation}
The leading term is roughly correct because
$$
\sqrt{d_1} + \sqrt{d_2}
\leq 2 \sqrt{\max\{d_1, \ d_2\}}
\leq 2 \left( \sqrt{d_1} + \sqrt{d_2} \right). 
$$
The logarithmic factor in~\eqref{eqn:gauss-rect-est} does not belong, but it is rather small in comparison with the leading terms.  Once again, we have produced a reasonable result with a short argument based on general principles.

\section{Example: Matrices with Randomly Signed Entries} \label{sec:rdm-sign-mtx}

Next, we turn to an example that is superficially similar with the matrix discussed in~\S\ref{sec:marcenko-pastur} but is less understood.  Consider a fixed $d_1 \times d_2$ matrix $\mtx{B}$ with real entries, and let $\{ \varrho_{jk} \}$ be an independent family of Rademacher random variables.  Consider the $d_1 \times d_2$ random matrix
$$
\mtx{B}_{\pm} = \sum_{j=1}^{d_1} \sum_{k=1}^{d_2} \varrho_{jk} b_{jk} \mathbf{E}_{jk}
$$
In other words, we obtain the random matrix $\mtx{B}_{\pm}$ by randomly flipping the sign of each entry of $\mtx{B}$.
The expected norm of this matrix satisfies the bound
\begin{equation} \label{eqn:rdm-sign-matrix-true}
\Expect \norm{\mtx{B}_{\pm}} \leq \textrm{Const} \cdot  v^{1/2} \cdot \log^{1/4} \min\{ d_1, \ d_2 \},
\end{equation}
where the leading factor $v^{1/2}$ satisfies
\begin{equation} \label{eqn:rdm-sign-matrix-var}
v = \max\left\{ \max\nolimits_j \normsq{ \smash{\vct{b}_{j:}} }, \ \max\nolimits_k \normsq{ \vct{b}_{:k} } \right\}.
\end{equation}
We have written $\mtx{b}_{j:}$ for the $j$th row of $\mtx{B}$ and $\mtx{b}_{:k}$ for the $k$th column of $\mtx{B}$.  In other words, the expected norm of a matrix with randomly signed entries is comparable with the maximum $\ell_2$ norm achieved by any row or column.  There are cases where the bound~\eqref{eqn:rdm-sign-matrix-true} admits a matching lower bound.
These results appear in~\cite[Thms.~3.1, 3.2]{Seg00:Expected-Norm} and~\cite[Cor.~4.7]{BV14:Sharp-Nonasymptotic}.

Theorem~\ref{thm:matrix-gauss-rect} leads to a quick proof of a slightly weaker result.  We simply need to compute the
matrix variance statistic $v(\mtx{B}_{\pm})$.  To that end, note that
$$
\sum_{j=1}^{d_1} \sum_{k=1}^{d_2} (b_{jk} \mathbf{E}_{jk})(b_{jk} \mathbf{E}_{jk})^\adj 
	= \sum_{j=1}^{d_1} \left(\sum_{k=1}^{d_2} \abssq{\smash{b_{jk}}}\right) \mathbf{E}_{jj}
	= \begin{bmatrix} \normsq{ \vct{b}_{1:}} && \\ & \ddots & \\ &&\norm{\smash{\vct{b}_{d_1:}}}^2
	\end{bmatrix}.
$$
Similarly,
$$
\sum_{j=1}^{d_1} \sum_{k=1}^{d_2} (b_{jk} \mathbf{E}_{jk})^{\adj}(b_{jk} \mathbf{E}_{jk}) 
	= \sum_{k=1}^{d_2} \left(\sum_{j=1}^{d_1} \abssq{\smash{b_{jk}}}\right) \mathbf{E}_{kk}
	= \begin{bmatrix} \normsq{ \vct{b}_{:1}} && \\ & \ddots & \\ &&\norm{\smash{\vct{b}_{:d_2}}}^2
	\end{bmatrix}.
$$
Therefore, using the formula~\eqref{eqn:matrix-gauss-rect-var-calc}, we find that
\begin{align*}
v( \mtx{B}_{\pm} ) &=
	\max\left\{ \norm{ \sum_{j=1}^{d_1} \sum_{k=1}^{d_2} (b_{jk} \mathbf{E}_{jk})(b_{jk} \mathbf{E}_{jk})^\adj }, \
	\norm{ \sum_{j=1}^{d_1} \sum_{k=1}^{d_2} (b_{jk} \mathbf{E}_{jk})^{\adj}(b_{jk} \mathbf{E}_{jk}) }
	\right\} \\
	&= \max\left\{ \max\nolimits_j \norm{\smash{\vct{b}_{j:}}}^2, \
		\max\nolimits_k \norm{\smash{\vct{b}_{:k}}}^2 \right\}.
\end{align*}
We see that $v(\mtx{B}_{\pm})$ coincides with $v$, the leading term~\eqref{eqn:rdm-sign-matrix-var} in the established estimate~\eqref{eqn:rdm-sign-matrix-true}!  Now, Theorem~\ref{thm:matrix-gauss-rect} delivers the bound
\begin{equation} \label{eqn:rdm-sign-matrix-est}
\Expect \norm{\mtx{B}_{\pm}}
	\leq \sqrt{2 v(\mtx{B}_{\pm}) \log(d_1 + d_2)}.
\end{equation}
Observe that the estimate~\eqref{eqn:rdm-sign-matrix-est} for the norm matches the correct bound~\eqref{eqn:rdm-sign-matrix-true} up to the logarithmic factor. Yet again, we obtain a result that is respectably close to the optimal one, even though it is not quite sharp.

The main advantage of using results like Theorem~\ref{thm:matrix-gauss-rect} to analyze this random matrix is that we can obtain a good result with a minimal amount of arithmetic.  The  analysis that leads to~\eqref{eqn:rdm-sign-matrix-true} involves a specialized combinatorial argument.

\section{Example: Gaussian Toeplitz Matrices} \label{sec:toeplitz}

Matrix concentration inequalities offer an effective tool for analyzing random matrices whose dependency structures are more complicated than those of the classical ensembles.  In this section, we consider Gaussian Toeplitz matrices, which have applications in signal processing.

We construct an (unsymmetric) $d \times d$ Gaussian Toeplitz matrix $\mtx{\Gamma}_d$ by populating the first row and first column of the matrix with independent standard normal variables; the entries along each diagonal of the matrix take the same value:
$$
\mtx{\Gamma}_d = \begin{bmatrix}
	\gamma_0 & \gamma_1 & & \dots &&  \gamma_{d-1} \\
	\gamma_{-1} & \gamma_0 & \gamma_1 & &&  \\
	 & \gamma_{-1} & \gamma_0 & \gamma_1 && \vdots \\
	 \vdots & & \ddots & \ddots & \ddots & \\
	 & & & \gamma_{-1} & \gamma_0 & \gamma_1 \\
	\gamma_{-(d-1)} & & \dots & & \gamma_{-1} & \gamma_0
\end{bmatrix}
$$
where $\{ \gamma_k \}$ 
is an independent family of standard normal variables.
As usual, we represent the Gaussian Toeplitz matrix as a matrix Gaussian series:
\begin{equation} \label{eqn:gauss-toeplitz-cpt}
\mtx{\Gamma}_d = \gamma_0 \, \Id + \sum_{k=1}^{d-1} \gamma_k \mtx{C}^k
	+ \sum_{k=1}^{d-1} \gamma_{-k} \big(\mtx{C}^k \big)^\adj,
\end{equation}
where $\mtx{C} \in \mathbb{M}_d$ denotes the shift-up operator acting on $d$-dimensional column vectors:
$$
\mtx{C} = \begin{bmatrix} 0 & 1 \\ & 0 & 1 \\ && \ddots & \ddots \\ &&& 0 & 1 \\
&&&& 0 \end{bmatrix}.
$$
It follows that $\mtx{C}^k$ shifts a vector up by $k$ places, introducing zeros at the bottom,
while $(\mtx{C}^k)^\adj$ shifts a vector down by $k$ places, introducing zeros at the top.

We can analyze this example quickly using Theorem~\ref{thm:matrix-gauss-rect}.  First, note that
$$
\big(\mtx{C}^k\big)\big(\mtx{C}^k\big)^\adj = \sum_{j=1}^{d-k} \mathbf{E}_{jj}
\quad\text{and}\quad
\big(\mtx{C}^k\big)^\adj \big(\mtx{C}^k\big) = \sum_{j=k+1}^d \mathbf{E}_{jj}.
$$
To obtain the matrix variance statistic~\eqref{eqn:matrix-gauss-rect-var-calc}, we calculate the sum of the squares of the coefficient matrices that appear in~\eqref{eqn:gauss-toeplitz-cpt}.  In this instance, the two terms in the variance are the same.  We find that
\begin{multline}
\Id^2 + \sum_{k=1}^{d-1} \big(\mtx{C}^k\big)\big(\mtx{C}^k\big)^{\adj} + \sum_{k=1}^{d-1} \big(\mtx{C}^k\big)^\adj \big(\mtx{C}^k\big)
	= \Id + \sum_{k=1}^{d-1} \left[ \sum_{j=1}^{d-k} \mathbf{E}_{jj} + \sum_{j=k+1}^{d} \mathbf{E}_{jj} \right] \\
	= \sum_{j=1}^d \left[ 1 + \sum_{k=1}^{d-j} 1 + \sum_{k=1}^{j-1} 1 \right] \mathbf{E}_{jj}
	= \sum_{j=1}^d (1 + (d-j) + (j-1)) \, \mathbf{E}_{jj}
	= d \, \Id_d.
\end{multline}
In the second line, we (carefully) switch the order of summation and rewrite the identity matrix as a sum of diagonal standard basis matrices.  We reach
$$
v(\mtx{\Gamma}_d) = \norm{ d \, \Id_d } = d.
$$
An application of Theorem~\ref{thm:matrix-gauss-rect} leads us to conclude that
\begin{equation} \label{eqn:gauss-toeplitz-est}
\Expect \norm{ \mtx{\Gamma}_d } \leq \sqrt{2d\log(2d)}.
\end{equation}
It turns out that the inequality~\eqref{eqn:gauss-toeplitz-est} is correct up to the precise value of the constant, which does not seem to be known.  Nevertheless, the limiting value is available for the top eigenvalue of a (scaled) symmetric Toeplitz matrix whose first row contains independent standard normal variables~\cite[Thm.~1]{SV13:Top-Eigenvalue}.  From this
result, we may conclude that
$$
0.8288 \quad\leq\quad
	\frac{\Expect \norm{ \mtx{\Gamma}_d }}{ \sqrt{2d\log(2d)} }
	\quad\leq\quad 1
	\quad\text{as $d \to \infty$.}
$$
Here, we take $\{\mtx{\Gamma}_d\}$ to be a sequence of unsymmetric Gaussian Toeplitz matrices, indexed by the ambient dimension $d$.  Our simple argument gives the right scaling for this problem, and our estimate for the constant lies within 21\% of the optimal value!

\section{Application: Rounding for the MaxQP Relaxation} \label{sec:maxqp}

Our final application involves a more substantial question from combinatorial optimization.  One of the methods that has been proposed for solving a certain optimization problem leads to a matrix Rademacher series, and the analysis of this method requires the spectral norm bounds from Theorem~\ref{thm:matrix-gauss-rect}.  A detailed treatment would take us too far afield, so we just sketch the context and indicate how the random matrix arises.

There are many types of optimization problems that are computationally difficult to solve exactly.  One approach to solving these problems is to enlarge the constraint set in such a way that the problem becomes tractable, a process called ``relaxation.''  After solving the relaxed problem, we can use a randomized ``rounding'' procedure to map the solution back to the constraint set for the original problem.  If we can perform the rounding step without changing the value of the objective function substantially, then the rounded solution is also a decent solution to the original optimization problem.

One difficult class of optimization problems has a matrix decision variable, and it requires us to maximize a quadratic form in the matrix variable subject to a set of convex quadratic constraints and a spectral norm constraint~\cite{Nem07:Sums-Random}.  This problem is referred to as \textsc{MaxQP}.  The desired solution $\mtx{B}$ to this problem is a $d_1 \times d_2$ matrix.  The solution needs to satisfy several different requirements, but we focus on the condition that $\norm{\mtx{B}} \leq 1$.

There is a natural relaxation of the \textsc{MaxQP} problem.  When we solve the relaxation, we obtain a family $\{ \mtx{B}_k : k = 1, 2, \dots, n \}$ of $d_1 \times d_2$ matrices that satisfy the constraints
\begin{equation} \label{eqn:maxqp-constraint}
\sum_{k=1}^n \mtx{B}_k \mtx{B}_k^\adj \psdle \Id_{d_1}
\quad\text{and}\quad
\sum_{k=1}^n \mtx{B}_k^\adj \mtx{B}_k \psdle \Id_{d_2}.
\end{equation}
In fact, these two bounds are part of the specification of the relaxed problem.  To round the family of matrices back to a solution of the original problem, we form the random matrix
$$
\mtx{Z} = \alpha \sum_{k=1}^n \varrho_k \mtx{B}_k,
$$
where $\{ \varrho_k \}$ is an independent family of Rademacher random variables.  The scaling factor $\alpha > 0$ can be adjusted to guarantee that the norm constraint $\norm{\mtx{Z}} \leq 1$ holds with high probability.

What is the expected norm of $\mtx{Z}$?  Theorem~\ref{thm:matrix-gauss-rect} yields
$$
\Expect \norm{ \mtx{Z} }
	\leq \sqrt{ 2 v(\mtx{Z}) \log(d_1 + d_2) }.
$$
Here, the matrix variance statistic satisfies
$$
v(\mtx{Z}) = \alpha^2 \, \max\left\{ \norm{ \sum_{k=1}^n \mtx{B}_k \mtx{B}_k^\adj }, \
	\norm{ \sum_{k=1}^n \mtx{B}_k^\adj \mtx{B}_k } \right\}
	\leq \alpha^2,
$$
owing to the constraint~\eqref{eqn:maxqp-constraint} on the matrices $\mtx{B}_1, \dots, \mtx{B}_n$.  It follows that the scaling parameter $\alpha$ should satisfy
$$
\alpha^2 = \frac{1}{2 \log(d_1 + d_2)}
$$
to ensure that $\Expect \norm{\mtx{Z}} \leq 1$.  For this choice of $\alpha$, the rounded solution $\mtx{Z}$ obeys the spectral norm constraint on average.  By using the tail bound~\eqref{eqn:matrix-gauss-tail-rect}, we can even obtain high-probability estimates for the norm of the rounded solution $\mtx{Z}$.

The important fact here is that the scaling parameter $\alpha$ is usually small as compared with the other parameters of the problem ($d_1, d_2$, $n$, and so forth).  Therefore, the scaling does not have a massive effect on the value of the objective function.  Ultimately, this approach leads to a technique for solving the \textsc{MaxQP} problem that produces a feasible point whose objective value is within a factor of $\sqrt{2 \log(d_1+d_2)}$ of the maximum objective value possible.

\section{Analysis of Matrix Gaussian \& Rademacher Series}
\label{sec:matrix-gauss-proof}

We began this chapter with a concentration inequality,
Theorem~\ref{thm:matrix-gauss-rect},
for the norm of a matrix Gaussian series,
and we have explored a number of different applications of this result.
This section contains a proof of this theorem.

\subsection{Random Series with Hermitian Coefficients}

As the development in Chapter~\ref{chap:matrix-lt} suggests,
random Hermitian matrices provide the natural setting for
establishing matrix concentration inequalities.
Therefore, we begin our treatment with a detailed statement of the matrix concentration
inequality for a Gaussian series with Hermitian matrix coefficients.

\begin{thm}[Matrix Gaussian \& Rademacher Series: The Hermitian Case] \label{thm:matrix-gauss-herm}
Consider a finite sequence $\{ \mtx{A}_k \}$ of fixed Hermitian matrices with dimension $d$,
and let $\{ \gamma_k \}$ be a finite sequence of independent standard normal variables.
Introduce the matrix Gaussian series
$$
\mtx{Y} = \sum\nolimits_k \gamma_k \mtx{A}_k.
$$
Let $v(\mtx{Y})$ be the matrix variance statistic of the sum:
\begin{equation} \label{eqn:matrix-gauss-sigma2}
v(\mtx{Y})
= \norm{ \smash{\Expect{} \mtx{Y}^2} }
 = \norm{ \sum\nolimits_k \mtx{A}_k^2 }.
\end{equation}
Then %
\begin{align}
\Expect \lambda_{\max}\left( \mtx{Y} \right)
	&\leq \sqrt{2 v(\mtx{Y}) \log d}.
	\label{eqn:matrix-gauss-upper-expect} %
\end{align}
Furthermore, for all $t \geq 0$,
\begin{align}
\Prob{ \lambda_{\max}\left( \mtx{Y} \right) \geq t }
	&\leq d \, \exp\left( \frac{-t^2}{2 v(\mtx{Y})} \right).
	 \label{eqn:matrix-gauss-upper-tail} %
\end{align}
The same bounds hold when we replace $\{\gamma_k\}$ by a finite sequence of independent Rademacher random variables.
\end{thm}

\noindent
The proof of this result occupies the rest of the section.

\subsection{Discussion}

Before we proceed to the analysis, let us take a moment to compare Theorem~\ref{thm:matrix-gauss-herm}
with the result for general matrix series, Theorem~\ref{thm:matrix-gauss-rect}.

First, we consider the matrix variance statistic $v(\mtx{Y})$ defined
in~\eqref{eqn:matrix-gauss-sigma2}.  Since $\mtx{Y}$ has zero mean,
this definition coincides with the general formula~\eqref{eqn:matrix-variance-herm}.
The second expression, in terms of the coefficient matrices, follows
from the additivity property~\eqref{eqn:indep-sum-herm}
for the variance of a sum of independent, random Hermitian matrices.

Next, bounds for the minimum eigenvalue $\lambda_{\min}(\mtx{Y})$
follow from the results for the maximum eigenvalue
because $-\mtx{Y}$ has the same distribution as $\mtx{Y}$.
Therefore, %
\begin{equation} \label{eqn:matrix-gauss-lower-expect}
\Expect \lambda_{\min}(\mtx{Y}) = \Expect \lambda_{\min}(-\mtx{Y})
	= - \Expect \lambda_{\max}(\mtx{Y})
	\geq - \sqrt{2 v(\mtx{Y}) \log d}.
\end{equation}
The second identity holds because of the relationship~\eqref{eqn:min-max-sign-eig}
between minimum and maximum eigenvalues.
Similar considerations lead to a lower tail bound for the minimum eigenvalue:
\begin{equation} \label{eqn:matrix-gauss-lower-tail}
\Prob{ \lambda_{\min}(\mtx{Y}) \leq -t }
	\leq d \, \exp\left( \frac{-t^2}{2v(\mtx{Y})} \right)
	\quad\text{for $t \geq 0$.}
\end{equation}
This result follows directly from the upper tail bound~\eqref{eqn:matrix-gauss-upper-tail}.

This observation points to the most important difference between the Hermitian case and the general
case. %
Indeed, Theorem~\ref{thm:matrix-gauss-herm} concerns the extreme eigenvalues
of the random series $\mtx{Y}$ instead of the norm.  This change amounts to producing
one-sided tail bounds instead of two-sided tail bounds.
For Gaussian and Rademacher series, this improvement is not really useful,
but there are random Hermitian
matrices whose minimum and maximum eigenvalues exhibit different types of behavior.
For these problems, it can be extremely valuable to examine the two tails separately.
See Chapter~\ref{chap:matrix-chernoff} and~\ref{chap:matrix-bernstein} for some results of this type.

\subsection{Analysis for Hermitian Gaussian Series}

We continue with the proof that matrix Gaussian series exhibit the behavior described in Theorem~\ref{thm:matrix-gauss-herm}.  Afterward, we show how to adapt the argument to address matrix Rademacher series.  Our main tool is Theorem~\ref{thm:master-ineq}, the set of master bounds for independent sums.  To use this result, we must identify the cgf of a fixed matrix modulated by a Gaussian random variable. %

\begin{lemma}[Gaussian $\times$ Matrix: Mgf and Cgf] \label{lem:matrix-gauss-mgf}
Suppose that $\mtx{A}$ is a fixed Hermitian matrix, and  
let $\gamma$ be a standard normal random variable.  Then
$$
\Expect \econst^{\gamma \theta \mtx{A}}
	= \econst^{\theta^2 \mtx{A}^2/2}
\quad\text{and}\quad
\log{} \Expect \econst^{\gamma \theta \mtx{A}}
	= \frac{\theta^2}{2} \mtx{A}^2
\quad\text{for $\theta \in \mathbb{R}$.}
$$
\end{lemma}

\begin{proof}
We may assume $\theta = 1$ by absorbing $\theta$ into the matrix $\mtx{A}$.
It is well known that the moments of a standard normal variable satisfy
$$
\Expect\big( \gamma^{2q+1} \big) = 0
\quad\text{and}\quad
\Expect \big( \gamma^{2q} \big) = \frac{(2q)!}{2^q \, q!}
\quad\text{for $q = 0, 1, 2, \dots$}.
$$
The formula for the odd moments holds because a standard normal variable is symmetric.
One way to establish the formula for the even moments is to use integration by parts
to obtain a recursion for the $(2q)$th moment in terms of the $(2q-2)$th moment.

Therefore, the matrix mgf satisfies
$$
\Expect \econst^{\gamma \mtx{A}}
	= \Id + \sum_{q=1}^\infty \frac{\Expect\big(\gamma^{2q}\big)}{(2q)!}\mtx{A}^{2q}
	= \Id + \sum_{q=1}^\infty \frac{1}{q!} \big(\mtx{A}^2/2 \big)^q
	= \econst^{ \mtx{A}^2 / 2 }.
$$
The first identity holds because the odd terms vanish from the series representation~\eqref{eqn:exp-series}
of the matrix exponential when we take the expectation.  To compute the cgf, we extract the logarithm of the mgf and recall~\eqref{eqn:log-defn}, which states that the matrix logarithm is the functional inverse of the matrix exponential.
\end{proof}

We quickly reach results on the maximum eigenvalue of a matrix Gaussian series with Hermitian coefficients.

\begin{proof}[Proof of Theorem~\ref{thm:matrix-gauss-herm}: Gaussian Case]
Consider a finite sequence $\{ \mtx{A}_k \}$ of Hermitian matrices with dimension $d$,
and let $\{ \gamma_k \}$ be a finite sequence of independent standard normal variables.
Define the matrix Gaussian series
$$
\mtx{Y} = \sum\nolimits_k \gamma_k \mtx{A}_k.
$$
We begin with the upper bound~\eqref{eqn:matrix-gauss-upper-expect} for $\Expect \lambda_{\max}(\mtx{Y})$.  The master expectation bound~\eqref{eqn:master-upper-expect} from Theorem~\ref{thm:master-ineq} implies that
\begin{align*}
\Expect \lambda_{\max}(\mtx{Y})
	&\leq \inf_{\theta > 0} \ \frac{1}{\theta} \log{} \trace \exp\left(
	\sum\nolimits_k \log{} \Expect \econst^{\gamma_k \theta \mtx{A}_k} \right) \\
	&= \inf_{\theta > 0} \ \frac{1}{\theta} \log{} \trace \exp\left(
	\frac{\theta^2}{2} \sum\nolimits_k \mtx{A}_k^2 \right) \\
	&\leq \inf_{\theta > 0} \ \frac{1}{\theta} \log{} \left[ d \, \lambda_{\max}\left(
	\exp\left( \frac{\theta^2}{2} \sum\nolimits_{k} \mtx{A}_k^2 \right) \right) \right] \\
	&= \inf_{\theta > 0} \ \frac{1}{\theta} \log{} \left[ d \, \exp\left(
	\frac{\theta^2}{2} \lambda_{\max}\left(\sum\nolimits_k \mtx{A}_k^2 \right) \right) \right] \\
	&= \inf_{\theta > 0} \ \frac{1}{\theta} \left[ \log d +
	\frac{\theta^2 v(\mtx{Y})}{2} \right]
\end{align*}
The second line follows when we introduce the cgf from Lemma~\ref{lem:matrix-gauss-mgf}.  To reach the third inequality, we bound the trace by the dimension times the maximum eigenvalue.  The fourth line is the Spectral Mapping Theorem, Proposition~\ref{prop:spectral-mapping}.  Use the formula~\eqref{eqn:matrix-gauss-sigma2} to identify the matrix variance statistic $v(\mtx{Y})$ in the exponent.
The infimum is attained at $\theta = \sqrt{2 v(\mtx{Y})^{-1} \log d}$.  This choice leads to~\eqref{eqn:matrix-gauss-upper-expect}.

Next, we turn to the proof of the upper tail bound~\eqref{eqn:matrix-gauss-upper-tail}
for $\lambda_{\max}(\mtx{Y})$.  Invoke the master tail bound~\eqref{eqn:master-upper-tail} from Theorem~\ref{thm:master-ineq}, and calculate that
\begin{align*}
\Prob{ \lambda_{\max}(\mtx{Y}) \geq t }
	&\leq \inf_{\theta > 0} \ \econst^{-\theta t} \,
	\trace \exp\left( \sum\nolimits_k \log{} \Expect \econst^{\gamma_k \theta \mtx{A}_k} \right)	\\
	&= \inf_{\theta > 0} \ \econst^{-\theta t} \,
	\trace \exp\left( \frac{\theta^2}{2} \sum\nolimits_k \mtx{A}_k^2 \right) \\
	&\leq  \inf_{\theta > 0 } \ \econst^{-\theta t} \cdot d \,
	\exp\left( \frac{\theta^2}{2} \lambda_{\max}\left(\sum\nolimits_k \mtx{A}_k^2 \right) \right) \\
	&= d \, \inf_{\theta > 0} \ \econst^{-\theta t + \theta^2 v(\mtx{Y}) / 2}.
\end{align*}
The steps here are the same as in the previous calculation.
The infimum is achieved at $\theta = t/v(\mtx{Y})$, which yields~\eqref{eqn:matrix-gauss-upper-tail}.
\end{proof}

\subsection{Analysis for Hermitian Rademacher Series}

The inequalities for matrix Rademacher series involve arguments closely related to the proofs for matrix Gaussian series, but we require one additional piece of reasoning to obtain the simplest results.  First, let us compute bounds for the matrix mgf and cgf of a Hermitian matrix modulated by a Rademacher random variable.

\begin{lemma}[Rademacher $\times$ Matrix: Mgf and Cgf] \label{lem:matrix-rad-mgf}
Suppose that $\mtx{A}$ is a fixed Hermitian~matrix, and let $\varrho$ be a Rademacher random variable.  Then
$$
\Expect \econst^{\varrho \theta \mtx{A}}
\psdle \econst^{\theta^2\mtx{A}^2/2}
\quad\text{and}\quad
\log{} \Expect \econst^{\varrho \theta \mtx{A}}
\psdle \frac{\theta^2}{2} \mtx{A}^2
\quad\text{for $\theta \in \mathbb{R}$.}
$$ 
\end{lemma}

\begin{proof}
First, we establish a scalar inequality.  Comparing Taylor series,
\begin{equation} \label{eqn:cosh-exp}
\cosh(a) = \sum_{q=0}^\infty \frac{a^{2q}}{(2q)!}
	\leq \sum_{q=0}^\infty \frac{a^{2q}}{2^q q!}
	= \econst^{a^2/2}
	\quad\text{for $a \in \R$.}
\end{equation}
The inequality holds because $(2q)! \geq (2q)(2q-2)\cdots (4)(2) = 2^q q!$.

To compute the matrix mgf, we may assume $\theta = 1$.  By direct calculation,
$$
\Expect \econst^{\varrho \mtx{A}}
	= \tfrac{1}{2} \econst^{\mtx{A}} + \tfrac{1}{2} \econst^{-\mtx{A}}
	= \cosh(\mtx{A})
	\psdle \econst^{\mtx{A}^2/2}.
$$
The semidefinite bound follows when we apply the Transfer Rule~\eqref{eqn:transfer-rule} to the inequality~\eqref{eqn:cosh-exp}.

To determine the matrix cgf, observe that
$$
\log{} \Expect \econst^{\varrho \mtx{A}}
	= \log \cosh(\mtx{A})
	\psdle \tfrac{1}{2} \mtx{A}^2.
$$
The semidefinite bound follows when we apply the Transfer Rule~\eqref{eqn:transfer-rule} to the scalar inequality $\log \cosh(a) \leq a^2 / 2$ for $a \in \R$, which is a consequence of~\eqref{eqn:cosh-exp}.
\end{proof}

We are prepared to develop some probability inequalities for the maximum eigenvalue of a
Rademacher series with Hermitian coefficients.

\begin{proof}[Proof of Theorem~\ref{thm:matrix-gauss-herm}: Rademacher Case]
Consider a finite sequence $\{ \mtx{A}_k \}$ of Hermitian matrices, and let $\{ \varrho_k \}$ be a finite sequence of independent Rademacher variables.  Define the matrix Rademacher series
$$
\mtx{Y} = \sum\nolimits_k \varrho_k \mtx{A}_k.
$$
The bounds for the extreme eigenvalues of $\mtx{Y}$ follow from an argument almost identical with the proof in the Gaussian case.  The only point that requires justification is the inequality
$$
\trace \exp\left( \sum\nolimits_k \log{} \Expect \econst^{\varrho_k \theta \mtx{A}_k} \right)
	\leq \trace \exp\left( \frac{\theta^2}{2} \sum\nolimits_k \mtx{A}_k^2 \right).
$$
To obtain this result, we introduce the semidefinite bound, Lemma~\ref{lem:matrix-rad-mgf}, for the Rademacher cgf into the trace exponential.  The left-hand side increases after this substitution because of the fact~\eqref{eqn:exp-trace-monotone} that the trace exponential function is monotone with respect to the semidefinite order.
\end{proof}

\subsection{Analysis of Matrix Series with Rectangular  Coefficients} \label{sec:matrix-gauss-proof-rect}

Finally, we consider a series with non-Hermitian matrix coefficients modulated by independent Gaussian or Rademacher random variables.  The bounds for the norm of a rectangular series follow instantly from the bounds for the norm of an Hermitian series because of a formal device.  We simply apply the Hermitian results to the Hermitian dilation~\eqref{eqn:herm-dilation} of the series.

\begin{proof}[Proof of Theorem~\ref{thm:matrix-gauss-rect}]
Consider a finite sequence $\{ \mtx{B}_k \}$ of $d_1 \times d_2$ complex matrices, and let $\{\zeta_k\}$ be a finite sequence of independent random variables, either standard normal or Rademacher.

Recall from Definition~\ref{def:herm-dilation} that the Hermitian dilation is the map
$$
\coll{H} : \mtx{B} \longmapsto \begin{bmatrix} \mtx{0} & \mtx{B} \\ \mtx{B}^\adj & \mtx{0} \end{bmatrix}.
$$
This leads us to form the two series
$$
\mtx{Z} = \sum\nolimits_k \zeta_k \mtx{B}_k
\quad\text{and}\quad
\mtx{Y} = \coll{H}(\mtx{Z}) = \sum\nolimits_k \zeta_k \coll{H}(\mtx{B}_k).
$$
The second expression for $\mtx{Y}$ holds because the Hermitian dilation is real-linear.
Since we have written $\mtx{Y}$ as a matrix series with Hermitian coefficients,
we may analyze it using Theorem~\ref{thm:matrix-gauss-herm}.  We just need to express
the conclusions in terms of the random matrix $\mtx{Z}$.

First, we employ the fact~\eqref{eqn:herm-dilation-norm} that the Hermitian dilation preserves spectral information:
$$
\norm{ \mtx{Z} } = \lambda_{\max}(\coll{H}(\mtx{Z})) = \lambda_{\max}(\mtx{Y}).
$$
Therefore, bounds on $\lambda_{\max}(\mtx{Y})$ deliver bounds on $\norm{\mtx{Z}}$.
In view of the calculation~\eqref{eqn:var-stat-dilation} for the variance statistic
of a dilation, we have
$$
v(\mtx{Y}) = v( \coll{H}(\mtx{Z}) ) = v(\mtx{Z}).
$$
Recall that the matrix variance statistic $v(\mtx{Z})$ defined in~\eqref{eqn:matrix-gauss-sigma2-rect}
coincides with the general definition from~\eqref{eqn:matrix-variance-rect}.
Now, invoke Theorem~\ref{thm:matrix-gauss-herm} to obtain Theorem~\ref{thm:matrix-gauss-rect}.
\end{proof}

\section{Notes}

We give an overview of research related to matrix Gaussian series, along with references for the specific random matrices that we have analyzed.

\subsection{Matrix Gaussian and Rademacher Series}

The main results, Theorem~\ref{thm:matrix-gauss-rect} and Theorem~\ref{thm:matrix-gauss-herm}, have an interesting history.  In the precise form presented here, these two statements first appeared in~\cite{Tro11:User-Friendly-FOCM}, but we can trace them back more than two decades.

In his work~\cite[Thm.~1]{Oli10:Sums-Random}, Oliveira established the mgf bounds presented in Lemma~\ref{lem:matrix-gauss-mgf} and Lemma~\ref{lem:matrix-rad-mgf}.  He also developed an ingenious improvement on the arguments of Ahlswede \& Winter~\cite[App.]{AW02:Strong-Converse}, and he obtained a bound similar with Theorem~\ref{thm:matrix-gauss-herm}.  The constants in Oliveira's result are worse, but the dependence on the dimension is better because it depends on the number of summands.  We do not believe that the approach Ahlswede \& Winter describe in~\cite{AW02:Strong-Converse} can deliver any of these results.

Recently, there have been some minor improvements to the dimensional factor that appears in Theorem~\ref{thm:matrix-gauss-herm}.  We discuss these results and give citations in Chapter~\ref{chap:intrinsic}.

\subsection{The Noncommutative Khintchine Inequality}
\label{sec:nc-khintchine}

Our theory about matrix Rademacher and Gaussian series should be compared with
a classic result, called the \term{noncommutative Khintchine inequality},
that was originally due to Lust-Piquard~\cite{LP86:Inegalites-Khintchine};
see also the follow-up work~\cite{LPP91:Noncommutative-Khintchine}.
In its simplest form, this inequality concerns a matrix Rademacher series
with Hermitian coefficients:
$$
\mtx{Y} = \sum\nolimits_k \varrho_k \mtx{A}_k
$$
The noncommutative Khintchine inequality states that
\begin{equation} \label{eqn:nc-khintchine}
\Expect{} \trace\big[ \mtx{Y}^{2q} \big]
	\leq C_{2q} \trace\big[ \big(\Expect \mtx{Y}^2 \big)^{q} \big]
	\quad\text{for $q = 1,2,3, \dots$.}
\end{equation}
The minimum value of the constant $C_{2q} = (2q)!/(2^q \, q!)$
was obtained in the two papers~\cite{Buc01:Operator-Khintchine,Buc05:Optimal-Constants}.
Traditional proofs of the noncommutative Khintchine inequality are quite involved,
but there is now an elementary argument available~\cite[Cor.~7.3]{MJCFT12:Matrix-Concentration}.

Theorem~\ref{thm:matrix-gauss-herm} is the exponential moment analog
of the polynomial moment bound~\eqref{eqn:nc-khintchine}.  
The polynomial moment inequality is somewhat stronger than the exponential
moment inequality.  Nevertheless, the exponential results are often more
useful in practice.  For a more thorough exploration of the relationships
between Theorem~\ref{thm:matrix-gauss-herm} and noncommutative moment inequalities,
such as~\eqref{eqn:nc-khintchine}, see the discussion in~\cite[\S4]{Tro11:User-Friendly-FOCM}.

\subsection{Application to Random Matrices}

It has also been known for a long time that results such as Theorem~\ref{thm:matrix-gauss-herm} and inequality~\eqref{eqn:nc-khintchine} can be used to study random matrices.

We believe that the geometric functional analysis literature contains the earliest applications of matrix concentration results to analyze random matrices.  In a well-known paper~\cite{Rud99:Random-Vectors}, Mark Rudelson---acting on a suggestion of Gilles Pisier---showed how to use the noncommutative Khintchine inequality~\eqref{eqn:nc-khintchine} to study covariance estimation.  This work led to a significant amount of activity in which researchers used variants of Rudelson's argument to prove other types of results.  See, for example, the paper~\cite{RV07:Sampling-Large}.  This approach is powerful, but it tends to require some effort to use.

In parallel, other researchers in noncommutative probability theory also came to recognize the power of noncommutative moment inequalities in random matrix theory.  The paper~\cite{JX08:Noncommutative-Burkholder-II} contains a specific example.  Unfortunately, this literature is technically formidable, which makes it difficult for outsiders to appreciate its achievements.

The work~\cite{AW02:Strong-Converse} of Ahlswede \& Winter led to the first ``packaged'' matrix concentration inequalities
of the type that we describe in these lecture notes.  For the first few
years after this work, most of the applications concerned quantum information theory and random graph theory.  The paper~\cite{Gro11:Recovering-Low-Rank} introduced the method of Ahlswede \& Winter to researchers in mathematical signal processing and statistics, and it served to popularize matrix concentration bounds.

At this point, the available matrix concentration inequalities were still significantly suboptimal.  The main advances, in~\cite{Oli10:Concentration-Adjacency,Tro11:User-Friendly-FOCM}, led to optimal matrix concentration results of the kind that we present in these lecture notes.  These results allow researchers to obtain reasonably accurate analyses of a wide variety of random matrices with very little effort.  %

\subsection{Wigner and Mar{\v c}enko--Pastur}

Wigner matrices first emerged in the literature on nuclear physics, where they were used to model the Hamiltonians of reactions involving heavy atoms~\cite[\S1.1]{Meh04:Random-Matrices}.  Wigner~\cite{Wig55:Characteristic-Vectors} showed that the limiting spectral distribution of a certain type of Wigner matrix follows the semicircle law.  See the book~\cite[\S2.4]{Tao12:Topics-Random} of Tao for an overview and the book~\cite[Chap.~2]{BS10:Spectral-Analysis} of Bai \& Silverstein for a complete treatment.  The Bai--Yin law~\cite{BY93:Limit-Smallest} states that, up to scaling, the maximum eigenvalue of a Wigner matrix converges almost surely to two.  See~\cite[\S2.3]{Tao12:Topics-Random} or~\cite[Chap.~5]{BS10:Spectral-Analysis} for more information.  The analysis of the Gaussian Wigner matrix that we present here, using Theorem~\ref{thm:matrix-gauss-herm}, is drawn from~\cite[\S4]{Tro11:User-Friendly-FOCM}.

The first rigorous work on a rectangular Gaussian matrix is due to Mar{\v c}enko \& Pastur~\cite{MP67:Distribution-Eigenvalues}, who established that the limiting distribution of the squared singular values follows a distribution that now bears their names.  The Bai--Yin law~\cite{BY93:Limit-Smallest} gives an almost-sure limit for the largest singular value of a rectangular Gaussian matrix.  The expectation bound~\eqref{eqn:gauss-rect-true} appears in a survey article~\cite{DS02:Local-Operator} by Davidson \& Szarek.  The latter result is ultimately derived from a comparison theorem for Gaussian processes due to F{\'e}rnique~\cite{Fer75:Regularite-Trajectoires} and amplified by Gordon~\cite{Gor85:Some-Inequalities}.
Our approach, using Theorem~\ref{thm:matrix-gauss-rect}, is based on~\cite[\S4]{Tro11:User-Friendly-FOCM}.

\subsection{Randomly Signed Matrices}

Matrices with randomly signed entries have not received much attention in the literature.  The result~\eqref{eqn:rdm-sign-matrix-true} is due to Yoav Seginer~\cite{Seg00:Expected-Norm}.  There is also a well-known paper~\cite{Lat05:Some-Estimates} by Rafa{\l} Lata{\l}a that provides a bound for the expected norm of a Gaussian matrix whose entries have nonuniform variance.  Riemer \& Sch{\"u}tt~\cite{RS13:Expectation-Norm} have extended the earlier results.
The very recent paper~\cite{BV14:Sharp-Nonasymptotic} of Afonso Bandeira and Ramon Van Handel contains
an elegant new proof of Seginer's result based on a general theorem for random matrices with independent
entries.  The analysis here, using Theorem~\ref{thm:matrix-gauss-rect}, is drawn from~\cite[\S4]{Tro11:User-Friendly-FOCM}.  

\subsection{Gaussian Toeplitz Matrices}

Research on random Toeplitz matrices is surprisingly recent, but there are now a number of papers available.  Bryc, Dembo, \& Jiang obtained the limiting spectral distribution of a symmetric Toeplitz matrix based on independent and identically distributed (iid) random variables~\cite{BDJ06:Spectral-Measure}.  Later, Mark Meckes established the first bound for the expected norm of a random Toeplitz matrix based on iid random variables~\cite{Mec07:Spectral-Norm}.  More recently, Sen \& Vir{\'a}g computed the limiting value of the expected norm of a random, symmetric Toeplitz matrix whose entries have identical second-order statistics~\cite{SV13:Top-Eigenvalue}. See the latter paper for additional references.  The analysis here, based on Theorem~\ref{thm:matrix-gauss-rect}, is new.
Our lower bound for the value of $\Expect \norm{\mtx{\Gamma}_d}$ follows from the results of Sen \& Vir{\'a}g.
We are not aware of any analysis for a random Toeplitz matrix whose entries have different variances,
but this type of result would follow from a simple modification of the argument in \S\ref{sec:toeplitz}.

\subsection{Relaxation and Rounding of \textsc{MaxQP}}

The idea of using semidefinite relaxation and rounding to solve the \textsc{MaxQP} problem is due to Arkadi Nemirovski~\cite{Nem07:Sums-Random}.  He obtained nontrivial results on the performance of his method using some matrix moment calculations, but he was unable to reach the sharpest possible bound.  Anthony So~\cite{So09:Moment-Inequalities} pointed out that matrix moment inequalities imply an optimal result; he also showed that matrix concentration inequalities have applications to robust optimization.  The presentation here, using Theorem~\ref{thm:matrix-gauss-rect}, is essentially equivalent with the approach in~\cite{So09:Moment-Inequalities}, but we have achieved slightly better bounds for the constants.

\makeatletter{}%

\chapter{A Sum of Random Positive-Semidefinite Matrices} \label{chap:matrix-chernoff}

This chapter presents matrix concentration inequalities that are analogous with the classical Chernoff bounds.  In the matrix setting, Chernoff-type inequalities allow us to control the extreme eigenvalues of a sum of independent, random, positive-semidefinite matrices.

More formally, we consider a finite sequence $\{ \mtx{X}_k \}$ of independent, random Hermitian matrices that satisfy
$$
0 \leq \lambda_{\min}(\mtx{X}_k)
\quad\text{and}\quad
\lambda_{\max}(\mtx{X}_k) \leq L
\quad\text{for each index $k$.}
$$
Introduce the sum $\mtx{Y} = \sum_k \mtx{X}_k$.  Our goal is to study the expectation and tail behavior of $\lambda_{\max}(\mtx{Y})$ and $\lambda_{\min}(\mtx{Y})$.  
Bounds on the maximum eigenvalue $\lambda_{\max}(\mtx{Y})$ give us information about the norm of the matrix $\mtx{Y}$, a measure of how much the action of the matrix can dilate a vector.  Bounds for the minimum eigenvalue $\lambda_{\min}(\mtx{Y})$ tell us when the matrix $\mtx{Y}$ is nonsingular; they also provide evidence about the norm of the inverse $\mtx{Y}^{-1}$, when it exists.

The matrix Chernoff inequalities are quite powerful, and they have numerous applications.  We demonstrate the relevance of this theory by considering two examples.  First, we show how to study the norm of a random submatrix drawn from a fixed matrix, and we explain how to check when the random submatrix has full rank.  Second, we develop an analysis to determine when a random graph is likely to be connected.  These two problems are closely related to basic questions in statistics and in combinatorics.

In contrast, the matrix Bernstein inequalities, appearing in Chapter~\ref{chap:matrix-bernstein}, describe how much a random matrix deviates from its mean value.  As such, the matrix Bernstein bounds are more suitable than the matrix Chernoff bounds for problems that concern matrix approximations.  Matrix Bernstein inequalities are also more appropriate when the variance $v(\mtx{Y})$ is small in comparison with the upper bound $L$ on the summands.

\subsubsection{Overview}

Section~\ref{sec:matrix-chernoff} presents the main results on the expectations and the tails of the extreme eigenvalues of a sum of independent, random, positive-semidefinite matrices.  Section~\ref{sec:rdm-submatrix} explains how the matrix Chernoff bounds provide spectral information about a random submatrix drawn from a fixed matrix.  In \S\ref{sec:erdos-renyi},
we use the matrix Chernoff bounds to study when a random graph is connected.  Afterward, in \S\ref{sec:matrix-chernoff-proof} we explain how to prove the main results.

\section{The Matrix Chernoff Inequalities} \label{sec:matrix-chernoff}

In the scalar setting, the Chernoff inequalities describe the behavior of a sum of independent, nonnegative random variables
that are subject to a uniform upper bound.  These results are often applied to study the number $Y$ of
successes in a sequence of independent---but not necessarily identical---Bernoulli trials with
small probabilities of success.  In this case, the Chernoff bounds show that $Y$ behaves
like a Poisson random variable.  The random variable $Y$ concentrates near the expected number
of successes.  Its lower tail has Gaussian decay, while its upper tail drops
off faster than that of an exponential random variable.
See~\cite[\S2.2]{BLM13:Concentration-Inequalities} for more background.

In the matrix setting, we encounter similar phenomena when we consider a sum of
independent, random, positive-semidefinite matrices whose eigenvalues meet
a uniform upper bound.  This behavior emerges from the next theorem,
which closely parallels the scalar Chernoff theorem.

\begin{thm}[Matrix Chernoff] \label{thm:matrix-chernoff}
Consider a finite sequence $\{ \mtx{X}_k \}$ of independent, random, Hermitian matrices
with common dimension $d$.  Assume that
$$
0 \leq \lambda_{\min}(\mtx{X}_k)
\quad\text{and}\quad
\lambda_{\max}(\mtx{X}_k) \leq L
\quad\text{for each index $k$.}
$$
Introduce the random matrix
$$
\mtx{Y} = \sum\nolimits_k \mtx{X}_k.
$$
Define the minimum eigenvalue $\mu_{\min}$ and maximum eigenvalue $\mu_{\max}$ of the expectation $\Expect \mtx{Y}$: 
\begin{align}
\mu_{\min} &= \lambda_{\min}(\Expect \mtx{Y} )
	= \lambda_{\min}\left( \sum\nolimits_{k} \Expect \mtx{X}_k \right),
	\quad\text{and}\quad
	\label{eqn:matrix-chernoff-mu-min} \\
\mu_{\max} &= \lambda_{\max}(\Expect \mtx{Y} )
	= \lambda_{\max}\left( \sum\nolimits_{k} \Expect \mtx{X}_k \right).
	\label{eqn:matrix-chernoff-mu-max}
\end{align}
Then, for $\theta > 0$,
\begin{align}
\Expect \lambda_{\min}(\mtx{Y}) &\geq \frac{1 - \econst^{-\theta}}{\theta} \, \mu_{\min} - \frac{1}{\theta} \, L \, \log d,
	\quad\text{and}
	\label{eqn:matrix-chernoff-lower-expect} \\
\Expect \lambda_{\max}(\mtx{Y}) &\leq \frac{\econst^{\theta} - 1}{\theta} \, \mu_{\max} + \frac{1}{\theta} \, L \, \log d.
	\label{eqn:matrix-chernoff-upper-expect}
\end{align}
Furthermore,
\begin{align}
\Prob{ \lambda_{\min}\left( \mtx{Y} \right) \leq (1 - \eps) \mu_{\min} }
	&\leq d \left[ \frac{\econst^{-\eps}}{(1 - \eps)^{1-\eps}} \right]^{\mu_{\min}/ L} 
	\quad\text{for $\eps \in [0, 1)$,\quad and} \label{eqn:matrix-chernoff-lower-tail} \\
\Prob{ \lambda_{\max}\left( \mtx{Y} \right) \geq (1 + \eps) \mu_{\max} }
	&\leq d \left[ \frac{\econst^{\eps}}{(1 + \eps)^{1+\eps}} \right]^{\mu_{\max} / L} 
	\quad\text{for $\eps \geq 0$.} \label{eqn:matrix-chernoff-upper-tail}
\end{align}
\end{thm}

\noindent
The proof of Theorem~\ref{thm:matrix-chernoff} appears below in \S\ref{sec:matrix-chernoff-proof}.

\subsection{Discussion}

Let us consider some facets of Theorem~\ref{thm:matrix-chernoff}.

\subsubsection{Aspects of the Matrix Chernoff Inequality}

In many situations, it is easier to work with streamlined versions of the expectation
bounds:
\begin{align}
\Expect \lambda_{\min}(\mtx{Y}) &\geq 0.63 \, \mu_{\min} - L \, \log d,
	\quad\text{and}
	\label{eqn:matrix-chernoff-lower-expect-simp} \\
\Expect \lambda_{\max}(\mtx{Y}) &\leq 1.72 \, \mu_{\max} + L \, \log d.
	\label{eqn:matrix-chernoff-upper-expect-simp}
\end{align}
We obtain these results by selecting $\theta = 1$ in both~\eqref{eqn:matrix-chernoff-lower-expect} and~\eqref{eqn:matrix-chernoff-upper-expect} and evaluating the numerical constants.

These simplifications also help to clarify the meaning of Theorem~\ref{thm:matrix-chernoff}.
On average, $\lambda_{\min}(\mtx{Y})$ is not much smaller than
$\lambda_{\min}(\Expect \mtx{Y})$, minus a fluctuation term that reflects the maximum size $L$
of a summand and the ambient dimension $d$.  Similarly, the average value of $\lambda_{\max}(\mtx{Y})$
is close to $\lambda_{\max}(\Expect \mtx{Y})$, plus the same fluctuation term.

We can also weaken the tail bounds~\eqref{eqn:matrix-chernoff-lower-tail} and~\eqref{eqn:matrix-chernoff-upper-tail}
to reach
\begin{align*}
\Prob{ \lambda_{\min}\left( \mtx{Y} \right)
	\leq t \mu_{\min} }
	&\leq d \, \econst^{-(1-t)^2 \mu_{\min}/2L}
	\quad\text{for $t \in [0, 1)$, and} \\
\Prob{ \lambda_{\max}\left(\mtx{Y} \right)
	\geq t \mu_{\max} }
	&\leq d \, \left( \frac{\econst}{t} \right)^{t \mu_{\max}/L}
	\quad\text{for $t \geq \econst$.}
\end{align*}
The first bound shows that the lower tail of $\lambda_{\min}(\mtx{Y})$ decays at a subgaussian rate
with variance $L/\mu_{\min}$.
The second bound manifests that the upper tail of $\lambda_{\max}(\mtx{Y})$ decays faster than that of an exponential random variable with mean $L/\mu_{\max}$.  
This is the same type of prediction we receive from the scalar Chernoff inequalities.

As with other matrix concentration results, the tail bounds~\eqref{eqn:matrix-chernoff-lower-tail}
and~\eqref{eqn:matrix-chernoff-upper-tail} can overestimate the actual tail probabilities for the
extreme eigenvalues of $\mtx{Y}$, especially at large deviations from the mean.
The value of the matrix Chernoff theorem derives from the estimates~\eqref{eqn:matrix-chernoff-lower-expect}
and~\eqref{eqn:matrix-chernoff-upper-expect} for the expectation of the minimum and maximum eigenvalue of $\mtx{Y}$.
Scalar concentration inequalities may provide better estimates for tail probabilities.

\subsubsection{Related Results}

We can moderate the dimensional factor $d$ in the bounds
for $\lambda_{\max}(\mtx{Y})$ from Theorem~\ref{thm:matrix-chernoff}
when the random matrix $\mtx{Y}$ has limited spectral content in most directions.
We take up this analysis in Chapter~\ref{chap:intrinsic}.

Next, let us present an important refinement~\cite[Thm.~A.1]{CGT12:Masked-Sample}
of the bound~\eqref{eqn:matrix-chernoff-upper-expect-simp}
that can be very useful in practice:
\begin{equation} \label{eqn:matrix-rosenthal}
\Expect \lambda_{\max}(\mtx{Y})
	\leq 2 \mu_{\max} + 8\econst\, \big( \Expect{} \max\nolimits_k \lambda_{\max}(\mtx{X}_k) \big) \log d.
\end{equation}
This estimate may be regarded as a matrix version of Rosenthal's inequality~\cite{Ros70:Subspaces-Lp}.
Observe that the uniform bound $L$ appearing in~\eqref{eqn:matrix-chernoff-upper-expect-simp}
always exceeds the large parenthesis on the right-hand side of~\eqref{eqn:matrix-rosenthal}.
Therefore, the estimate~\eqref{eqn:matrix-rosenthal} is valuable when the summands
are unbounded and, especially, when they have heavy tails.  See the notes at
the end of the chapter for more information.

\subsection{Optimality of the Matrix Chernoff Bounds}
\label{sec:matrix-chernoff-sharp}

In this section, we explore how well bounds such as Theorem~\ref{thm:matrix-chernoff}
and inequality~\eqref{eqn:matrix-rosenthal} describe the behavior of a random matrix
$\mtx{Y}$ formed as a sum of independent, random positive-semidefinite matrices.

\subsubsection{The Upper Chernoff Bounds}

We will demonstrate that both terms in the matrix Rosenthal
inequality~\eqref{eqn:matrix-rosenthal} are necessary.  More precisely,
\begin{equation} \label{eqn:rosenthal-two-sided}
\begin{aligned}
\mathrm{const} \cdot \big[ \mu_{\max} \ +\ \Expect{} \max\nolimits_k \lambda_{\max}(\mtx{X}_k) \big]
	\quad&\leq\quad \Expect \lambda_{\max}(\mtx{Y}) \\
	\quad&\leq\quad \mathrm{Const} \cdot \big[ \mu_{\max}
	\ + \ \big(\Expect{} \max\nolimits_k \lambda_{\max}(\mtx{X}_k)\big) \log d \big].
\end{aligned}
\end{equation}
Therefore, we have identified appropriate parameters for bounding $\Expect{} \lambda_{\max}(\mtx{Y})$,
although the constants and the logarithm may not be sharp in every case.

The appearance of $\mu_{\max}$ on the left-hand side of~\eqref{eqn:rosenthal-two-sided}
is a consequence of Jensen's inequality.  Indeed, the maximum eigenvalue is convex, so
$$
\Expect \lambda_{\max}(\mtx{Y}) \geq \lambda_{\max}(\Expect \mtx{Y}) = \mu_{\max}.
$$
To justify the other term, apply the fact that
the summands $\mtx{X}_k$ are positive semidefinite to conclude that
$$
\Expect \lambda_{\max}(\mtx{Y})
	= \Expect \lambda_{\max}\left( \sum\nolimits_k \mtx{X}_k \right)
	\geq \Expect{} \max\nolimits_k \lambda_{\max}(\mtx{X}_k).
$$
We have used the fact that $\lambda_{\max}(\mtx{A} + \mtx{H}) \geq \lambda_{\max}(\mtx{A})$
whenever $\mtx{H}$ is positive semidefinite.  Average the last two
displays to develop the left-hand side of~\eqref{eqn:rosenthal-two-sided}.
The right-hand side of~\eqref{eqn:rosenthal-two-sided}
is obviously just~\eqref{eqn:matrix-rosenthal}.  

A simple example suffices to show that the logarithm cannot always
be removed from the second term in~\eqref{eqn:matrix-chernoff-upper-expect-simp}
or from~\eqref{eqn:matrix-rosenthal}.
For each natural number $n$, consider the $d \times d$ random matrix
$$
\mtx{Y}_n = \sum_{i=1}^n \sum_{k=1}^d \delta_{ik}^{(n)} \,\mathbf{E}_{kk}
$$
where $\big\{\delta_{ik}^{(n)} \big\}$ is an independent family of $\textsc{bernoulli}(n^{-1})$
random variables and $\mathbf{E}_{kk}$ is the $d \times d$ matrix with a one
in the $(k, k)$ entry an zeros elsewhere.
An easy application of~\eqref{eqn:matrix-chernoff-upper-expect-simp} delivers
$$
\lambda_{\max}(\mtx{Y}_n) \leq 1.72 + \log d.
$$
Using the Poisson limit of a binomial random variable and the
Skorokhod representation, we can construct an independent family $\{Q_k\}$
of $\textsc{poisson}(1)$ random variables for which
$$
\mtx{Y}_n \to \sum_{k=1}^d Q_k \, \mathbf{E}_{kk}
\quad\text{almost surely as $n \to \infty$}.
$$
It follows that
$$
\Expect \lambda_{\max}(\mtx{Y}_n) \to \Expect{} \max\nolimits_k Q_k
	\approx \mathrm{const} \cdot \frac{\log d}{\log \log d}
	\quad\text{as $n \to \infty$.}
$$
Therefore, the logarithm on the second term in~\eqref{eqn:matrix-chernoff-upper-expect-simp}
cannot be reduced by a factor larger than the iterated logarithm $\log \log d$.
This modest loss comes from approximations we make when developing the estimate for the
mean.  The tail bound~\eqref{eqn:matrix-chernoff-upper-tail} accurately predicts the order
of $\lambda_{\max}(\mtx{Y}_n)$ in this example.

The latter example depends on the commutativity of the summands
and the infinite divisibility of the Poisson distribution,
so it may seem rather special.
Nevertheless, the logarithm really does belong in many (but not all!)
examples that arise in practice.  In particular, it is necessary in
the application to random submatrices in~\S\ref{sec:rdm-submatrix}.

\subsubsection{The Lower Chernoff Bounds}

The upper expectation bound~\eqref{eqn:matrix-chernoff-upper-expect}
is quite satisfactory, but the situation is murkier
for the lower expectation bound~\eqref{eqn:matrix-chernoff-lower-expect}.
The mean term appears naturally in the lower bound:
$$
\Expect \lambda_{\min}(\mtx{Y}) \leq \lambda_{\min}(\Expect \mtx{Y})
	= \mu_{\min}.
$$
This estimate is a consequence of Jensen's inequality and
the concavity of the minimum eigenvalue.
On the other hand, it is not clear what the correct form of the second
term in~\eqref{eqn:matrix-chernoff-lower-expect}
should be for a general sum of random positive-semidefinite matrices.

Nevertheless, a simple example demonstrates that the lower Chernoff bound
~\eqref{eqn:matrix-chernoff-lower-expect} is numerically sharp in some situations.
Let $\mtx{X}$ be a $d \times d$ random positive-semidefinite matrix that satisfies
$$
\mtx{X} = d \, \mathbf{E}_{ii}
\quad\text{with probability $d^{-1}$ for each index $i = 1, \dots, d$.}
$$
It is clear that $\Expect \mtx{X} = \Id_d$.  Form the random matrix
$$
\mtx{Y}_n = \sum_{k=1}^n \mtx{X}_k
\quad\text{where each $\mtx{X}_k$ is an independent copy of $\mtx{X}$.}
$$
The lower Chernoff bound~\eqref{eqn:matrix-chernoff-lower-expect} implies that
$$
\Expect \lambda_{\min}(\mtx{Y}_n) \geq \frac{1- \econst^{-\theta}}{\theta} \cdot n - \frac{1}{\theta} d \log d.
$$
The parameter $\theta > 0$ is at our disposal.  This analysis predicts that
$\Expect \lambda_{\min}(\mtx{Y}_n) > 0$ precisely when $n > d \log d$.

On the other hand, $\lambda_{\max}(\mtx{Y}_n) > 0$ if and only if
each diagonal matrix $d \, \mathbf{E}_{ii}$ appears at least once among the summands
$\mtx{X}_1, \dots, \mtx{X}_n$.  To determine the probability that this event occurs,
notice that this question is an instance of the coupon collector problem~\cite[\S3.6]{MR95:Randomized-Algorithms}.
The probability of collecting all $d$ coupons within $n$ draws undergoes a phase transition
from about zero to about one at $n = d \log d$.  By refining this argument~\cite{Tro11:Improved-Analysis},
we can verify that both lower Chernoff bounds~\eqref{eqn:matrix-chernoff-lower-expect}
and~\eqref{eqn:matrix-chernoff-lower-tail} provide a numerically
sharp lower bound for the value of $n$ where the phase transition
occurs.  In other words, the lower matrix Chernoff bounds are themselves sharp.

\section{Example: A Random Submatrix of a Fixed Matrix} \label{sec:rdm-submatrix}

The matrix Chernoff inequality can be used to bound the extreme singular values of a random submatrix drawn from a fixed matrix.  Theorem~\ref{thm:matrix-chernoff} might not seem suitable for this purpose because it deals with eigenvalues, but we can connect the method with the problem via a simple transformation.  The results in this section have found applications in randomized linear algebra, sparse approximation, machine learning, and other fields.  See the notes at the end of the chapter for some additional discussion and references.

\subsection{A Random Column Submatrix} \label{sec:column-submatrix}

Let $\mtx{B}$ be a fixed $d \times n$ matrix, and let $\vct{b}_{:k}$ denote the $k$th column of this matrix.  The matrix can be expressed as the sum of its columns:
$$
\mtx{B} = \sum_{k=1}^n \vct{b}_{:k} \mathbf{e}_k^\adj.
$$
The symbol $\mathbf{e}_k$ refers to the standard basis (column) vector with a one in the $k$th component and zeros elsewhere; the length of the vector $\mathbf{e}_k$ is determined by context.

We consider a simple model for a random column submatrix.  Let $\{ \delta_k \}$ be an independent sequence of $\textsc{bernoulli}(p/n)$ random variables.  Define the random matrix
$$
\mtx{Z} = \sum_{k=1}^n \delta_k \, \vct{b}_{:k} \mathbf{e}_k^\adj.
$$
That is, we include each column independently with probability $p/n$, which means that there are typically about $p$ nonzero columns in the matrix.  We do not remove the other columns; we just zero them out.

In this section, we will obtain bounds on the expectation of the extreme singular values $\sigma_1(\mtx{Z})$ and $\sigma_d(\mtx{Z})$ of the $d \times n$ random matrix $\mtx{Z}$.  More precisely,
\begin{equation} \label{eqn:column-submatrix}
\begin{aligned}
\Expect{} \sigma_1(\mtx{Z})^2 
	&\leq 1.72 \cdot \frac{p}{n} \cdot \sigma_1(\mtx{B})^2
		+ (\log d) \cdot \max\nolimits_k \normsq{\vct{b}_{:k}},
		\quad\text{and}\quad \\
\Expect{} \sigma_d(\mtx{Z})^2 
	&\geq 0.63 \cdot \frac{p}{n} \cdot \sigma_d(\mtx{B})^2
		- (\log d) \cdot \max\nolimits_k \normsq{\vct{b}_{:k}}.
\end{aligned}	
\end{equation}
That is, the random submatrix $\mtx{Z}$ gets its ``fair share''
of the squared singular values of the original matrix $\mtx{B}$.
There is a fluctuation term that depends on largest norm of a column of $\mtx{B}$ and the logarithm of the number $d$ of rows in $\mtx{B}$.  This result is very useful because a positive lower bound on $\sigma_d(\mtx{Z})$ ensures that the rows of the random submatrix $\mtx{Z}$ are linearly independent.

\subsubsection{The Analysis}

To study the singular values of $\mtx{Z}$, it is convenient to define a $d \times d$ random, positive-semidefinite matrix
$$
\mtx{Y} = \mtx{ZZ}^\adj = \sum_{j,k=1}^n \delta_j \delta_k \, (\vct{b}_{:j} \mathbf{e}_j^\adj)( \mathbf{e}_k \vct{b}_{:k}^\adj )
	= \sum\limits_{k=1}^n \delta_k \, \vct{b}_{:k} \vct{b}_{:k}^\adj.
$$
Note that $\delta_k^2 = \delta_k$ because $\delta_k$ only takes the values zero and one.  The eigenvalues of $\mtx{Y}$ determine the singular values of $\mtx{Z}$, and vice versa.  In particular,
$$
\lambda_{\max}(\mtx{Y}) = \lambda_{\max}(\mtx{ZZ}^\adj) = \sigma_1(\mtx{Z})^2
\quad\text{and}\quad
\lambda_{\min}(\mtx{Y}) = \lambda_{\min}(\mtx{ZZ}^\adj) = \sigma_d(\mtx{Z})^2,
$$
where we arrange the singular values of $\mtx{Z}$ in weakly decreasing order
$\sigma_1(\mtx{Z}) \geq \dots \geq \sigma_d(\mtx{Z})$.

The matrix Chernoff inequality provides bounds for the expectations of the eigenvalues of $\mtx{Y}$.  To apply the result, first calculate
$$
\Expect \mtx{Y} = \sum_{k=1}^n (\Expect \delta_k) \, \vct{b}_{:k}\vct{b}_{:k}^\adj
	= \frac{p}{n} \sum_{k=1}^n \vct{b}_{:k} \vct{b}_{:k}^\adj
	= \frac{p}{n} \cdot \mtx{BB}^\adj,
$$
so that
$$
\mu_{\max} = \lambda_{\max}(\Expect \mtx{Y}) = \frac{p}{n} \, \sigma_1(\mtx{B})^2
\quad\text{and}\quad
\mu_{\min} = \lambda_{\min}(\Expect \mtx{Y}) = \frac{p}{n} \, \sigma_d(\mtx{B})^2.
$$
Define $L = \max_k \normsq{\vct{b}_{:k}}$, and observe that $\norm{\smash{\delta_k \vct{b}_{:k} \vct{b}_{:k}^\adj }} \leq L$ for each index $k$.  The simplified matrix Chernoff bounds~\eqref{eqn:matrix-chernoff-lower-expect-simp} and~\eqref{eqn:matrix-chernoff-upper-expect-simp} now deliver the result~\eqref{eqn:column-submatrix}.

\subsection{A Random Row and Column Submatrix}

Next, we consider a model for a random set of rows and columns drawn from a fixed $d \times n$ matrix $\mtx{B}$.  In this case, it is helpful to use matrix notation to represent the extraction of a submatrix.
Define independent random projectors
$$
\mtx{P} = \diag( \delta_1, \dots, \delta_d )
\quad\text{and}\quad
\mtx{R} = \diag( \xi_1, \dots, \xi_n )
$$
where $\{\delta_k\}$ is an independent family of $\textsc{bernoulli}(p/d)$ random variables and
$\{\xi_k\}$ is an independent family of $\textsc{bernoulli}(r/n)$ random variables.  Then
$$
\mtx{Z} = \mtx{PBR}
$$
is a random submatrix of $\mtx{B}$ with about $p$ nonzero rows and $r$ nonzero columns.

In this section, we will show that
\begin{multline} \label{eqn:rdm-submatrix-normsq}
\Expect{} \normsq{\mtx{Z}}
	\leq
	3 \cdot \frac{p}{d} \cdot \frac{r}{n} \cdot \normsq{\mtx{B}}
	+ 2 \cdot \frac{p \log d}{d} \cdot \max\nolimits_k \norm{\smash{ \vct{b}_{:k}} }^2 \\
	+ 2 \cdot \frac{r \log n}{n} \cdot \max\nolimits_j \norm{\smash{ \vct{b}_{j:}} }^2
	+ (\log d)(\log n) \cdot \max\nolimits_{j,k} \abs{\smash{b_{jk}}}^2.
\end{multline}
The notations $\vct{b}_{j:}$ and $\vct{b}_{:k}$ refer to the $j$th row and $k$th column of the matrix $\mtx{B}$,
while $b_{jk}$ is the $(j, k)$ entry of the matrix.  In other words, the random submatrix $\mtx{Z}$ gets its share of the total squared norm of the matrix $\mtx{B}$.  The fluctuation terms reflect the maximum row norm and the maximum column norm of $\mtx{B}$, as well as the size of the largest entry.  There is also a weak dependence on the ambient dimensions $d$ and $n$.

\subsubsection{The Analysis}

The argument has much in common with the calculations for a random column submatrix, but we need to do some extra work to handle the interaction between the random row sampling and the random column sampling.

To begin, we express the squared norm $\normsq{\mtx{Z}}$ in terms of the maximum eigenvalue of a random positive-semidefinite matrix:
$$
\begin{aligned}\Expect{} \normsq{\mtx{Z}}
	&= \Expect \lambda_{\max}( (\mtx{PBR})(\mtx{PBR})^\adj ) \\
	&= \Expect \lambda_{\max}( (\mtx{PB}) \mtx{R} (\mtx{PB})^\adj )
	= \Expect\left[ \Expect\left[ \lambda_{\max}\left( \sum_{k=1}^n \xi_k \, (\mtx{PB})_{:k} (\mtx{PB})_{:k}^\adj \right)
	\, \bigg\vert \, \mtx{P} \right] \right]
\end{aligned}
$$
We have used the fact that $\mtx{RR}^\adj = \mtx{R}$,
and the notation $(\mtx{PB})_{:k}$ refers to the $k$th column of the matrix $\mtx{PB}$.
Observe that the random positive-semidefinite matrix on the right-hand side has dimension $d$.
Invoking the matrix Chernoff inequality~\eqref{eqn:matrix-chernoff-upper-expect-simp}, conditional
on the choice of $\mtx{P}$, we obtain
\begin{equation} \label{eqn:rdm-submatrix-1}
\Expect{}  \normsq{\mtx{Z}}
	\leq 1.72 \cdot \frac{r}{n} \cdot \Expect \lambda_{\max}((\mtx{PB})(\mtx{PB})^\adj)
	+ (\log d) \cdot \Expect \max\nolimits_k \normsq{(\mtx{PB})_{:k}}.
\end{equation}
The required calculation is analogous with the one in the Section~\ref{sec:column-submatrix},
so we omit the details.
To reach a deterministic bound, we still have two more expectations to control.

Next, we examine the term in~\eqref{eqn:rdm-submatrix-1} that involves the maximum eigenvalue:
$$
\Expect \lambda_{\max}((\mtx{PB})(\mtx{PB})^\adj)
	= \Expect \lambda_{\max}( \mtx{B}^\adj \mtx{P} \mtx{B} )
	= \Expect \lambda_{\max}\left( \sum_{j=1}^d \delta_j \vct{b}_{j:}^\adj \vct{b}_{j:} \right).
$$
The first identity holds because $\lambda_{\max}(\mtx{CC}^\adj) = \lambda_{\max}(\mtx{C}^\adj \mtx{C})$
for any matrix $\mtx{C}$, and $\mtx{PP}^\adj = \mtx{P}$.
Observe that the random positive-semidefinite matrix on the right-hand side has dimension $n$,
and apply the matrix Chernoff inequality~\eqref{eqn:matrix-chernoff-upper-expect-simp} again to reach
\begin{equation} \label{eqn:rdm-submatrix-2}
\Expect \lambda_{\max}((\mtx{PB})(\mtx{PB})^\adj)
	\leq 1.72 \cdot \frac{p}{d}\cdot \lambda_{\max}(\mtx{B}^\adj \mtx{B})
	+ (\log n) \cdot \max\nolimits_j \norm{ \smash{\vct{b}_{j:}} }^2.
\end{equation}
Recall that $\lambda_{\max}(\mtx{B}^\adj\mtx{B}) = \normsq{\mtx{B}}$ to simplify this expression slightly.

Last, we develop a bound on the maximum column norm in~\eqref{eqn:rdm-submatrix-1}.
This result also follows from the matrix Chernoff inequality, but we need to do a little
work to see why.  There are more direct proofs, but this approach is closer in spirit to
the rest of our proof.

We are going to treat the maximum column norm as the maximum eigenvalue
of a sum of independent, random diagonal matrices.  Observe that
$$
\normsq{ (\mtx{PB})_{:k} } = \sum_{j=1}^d \delta_j \,\abs{\smash{b_{jk}}}^2
\quad\text{for each $k = 1, \dots, n$.}
$$
Using this representation, we see that
$$
\begin{aligned}
\max\nolimits_k \normsq{ (\mtx{PB})_{:k} }
	&= \lambda_{\max} \begin{bmatrix} \sum_{j=1}^d \delta_j \, \abs{\smash{b_{j1}}}^2 \\ & \ddots \\ && \sum_{j=1}^d \delta_j \, \abs{\smash{b_{jn}}}^2 \end{bmatrix} \\
	&= \lambda_{\max}\left( \sum_{j=1}^d \delta_j\, \diag\left( \abs{\smash{b_{j1}}}^2,
	\dots,\abs{\smash{b_{jn}}}^2 \right) \right).
\end{aligned}
$$
To activate the matrix Chernoff bound, we need to compute the two parameters that appear
in~\eqref{eqn:matrix-chernoff-upper-expect-simp}.  First,
the uniform upper bound $L$ satisfies
$$
L = \max\nolimits_j \lambda_{\max}\left( \diag\left( \abs{\smash{b_{j1}}}^2,
	\dots,\abs{\smash{b_{jn}}}^2 \right) \right)
	= \max\nolimits_j \max\nolimits_k \abs{\smash{b_{jk}}}^2.
$$
Second, to compute $\mu_{\max}$, note that
$$
\begin{aligned}
\Expect \sum_{j=1}^d \delta_j \diag\left( \abs{\smash{b_{j1}}}^2, \dots, \abs{\smash{b_{jn}}}^2 \right)
	&= \frac{p}{d} \cdot \diag\left( \sum_{j=1}^d \abs{\smash{b_{j1}}}^2,
	\dots, \sum_{j=1}^d  \abs{\smash{b_{jn}}}^2 \right) \\
	&= \frac{p}{d} \cdot \diag\left( \normsq{\smash{\vct{b}_{:1}}},
	\dots, \normsq{\smash{\vct{b}_{:n}}} \right).
\end{aligned}
$$
Take the maximum eigenvalue of this expression to reach
$$
\mu_{\max} = \frac{p}{d} \cdot \max\nolimits_k \norm{\smash{ \vct{b}_{:k} }}^2.
$$
Therefore, the matrix Chernoff inequality implies
\begin{equation} \label{eqn:rdm-submatrix-3}
\Expect \max\nolimits_k \normsq{ (\mtx{PB})_{:k} }
	\leq 1.72 \cdot \frac{p}{d} \cdot \max\nolimits_k \norm{\smash{\vct{b}_{:k}} }^2
	+ (\log n) \cdot \max\nolimits_{j,k} \abs{\smash{b_{jk}}}^2.
\end{equation}
On average, the maximum squared column norm of a
random submatrix $\mtx{PB}$ with approximately $p$ nonzero rows
gets its share $p/d$ of the maximum squared column norm of $\mtx{B}$,
plus a fluctuation term that depends on the magnitude of
the largest entry of $\mtx{B}$ and the logarithm of the number $n$ of columns.

Combine the three bounds~\eqref{eqn:rdm-submatrix-1},~\eqref{eqn:rdm-submatrix-2}, and
~\eqref{eqn:rdm-submatrix-3} to reach the result~\eqref{eqn:rdm-submatrix-normsq}.  We
have simplified numerical constants to make the expression more compact.

\section{Application: When is an Erd{\H os}--R{\'e}nyi Graph Connected?}
\label{sec:erdos-renyi}

Random graph theory concerns probabilistic models for the interactions between pairs of objects.  One basic question about a random graph is to ask whether there is a path connecting every pair of vertices or whether there are vertices segregated in different parts of the graph.  It is possible to address this problem by studying the eigenvalues of random matrices, a challenge that we take up in this section.

\subsection{Background on Graph Theory}
\label{sec:graph-theory}

Recall that an \term{undirected graph} is a pair $G = (V, E)$.  The elements of the set $V$ are called \term{vertices}.  The set $E$ is a collection of unordered pairs $\{u, v\}$ of distinct vertices, called \term{edges}.  We say that the graph has an edge between vertices $u$ and $v$ in $V$ if the pair $\{u, v\}$ appears in $E$.  For simplicity, we assume that the vertex set $V = \{1, \dots, n\}$.  The \term{degree} $\textrm{deg}(k)$ of the vertex $k$ is the number of edges in $E$ that include the vertex $k$.

There are several natural matrices associated with an undirected graph.  The \term{adjacency matrix} of the graph $G$ is an $n \times n$ symmetric matrix $\mtx{A}$ whose entries indicate which edges are present:
$$
a_{jk} = \begin{cases}
	1, & \{j, k\} \in E \\
	0, & \{j, k\} \notin E.
\end{cases}
$$
We have assumed that edges connect distinct vertices, so the diagonal entries of the matrix $\mtx{A}$ equal zero.  Next, define a diagonal matrix $\mtx{D} = \diag( \textrm{deg}(1), \dots, \textrm{deg}(n) )$ whose entries list the degrees of the vertices.  The \term{Laplacian} $\mtx{\Delta}$ and \term{normalized Laplacian} $\mtx{H}$ of the graph are the matrices
$$
\mtx{\Delta} = \mtx{D} - \mtx{A}
\quad\text{and}\quad
\mtx{H} = \mtx{D}^{-1/2} \mtx{\Delta} \mtx{D}^{-1/2}.
$$
We place the convention that $\mtx{D}^{-1/2}(k,k) = 0$ when $\textrm{deg}(k) = 0$.  The Laplacian matrix $\mtx{\Delta}$ is always positive semidefinite.  The vector $\mathbf{e} \in \R^n$ of ones is always an eigenvector of $\mtx{\Delta}$ with eigenvalue zero.

These matrices and their spectral properties play a dominant role in modern graph theory.  For example, the graph $G$ is connected if and only if the second-smallest eigenvalue of $\mtx{\Delta}$ is strictly positive.  The second smallest eigenvalue of $\mtx{H}$ controls the rate at which a random walk on the graph $G$ converges to the stationary distribution (under appropriate assumptions).  See the book~\cite{GR01:Algebraic-Graph} for more information about these connections.

\subsection{The Model of Erd{\H o}s \& R{\'e}nyi}

The simplest possible example of a random graph is the independent model $G(n, p)$ of Erd{\H o}s and R{\'e}nyi~\cite{ER60:Evolution-Random}.  The number $n$ is the number of vertices in the graph, and $p \in (0,1)$ is the probability that two vertices are connected.  More precisely, here is how to construct a random graph in $G(n, p)$.  Between each pair of distinct vertices, we place an edge independently at random with probability $p$.  In other words, the adjacency matrix takes the form
\begin{equation} \label{eqn:er-graph-adjacency-entries}
a_{jk} = \begin{cases}
	\xi_{jk}, & 1 \leq j < k \leq n \\
	\xi_{kj}, & 1 \leq k < j \leq n \\
	0, & j = k.
\end{cases}
\end{equation}
The family $\{ \xi_{jk} : 1 \leq j < k \leq n \}$ consists of mutually independent $\textsc{bernoulli}(p)$ random variables.  Figure~\ref{fig:er-graph} shows one realization of the adjacency matrix of an Erd{\H o}s--R{\'e}nyi graph.

\begin{figure}
\begin{center}
\includegraphics[width=3in]{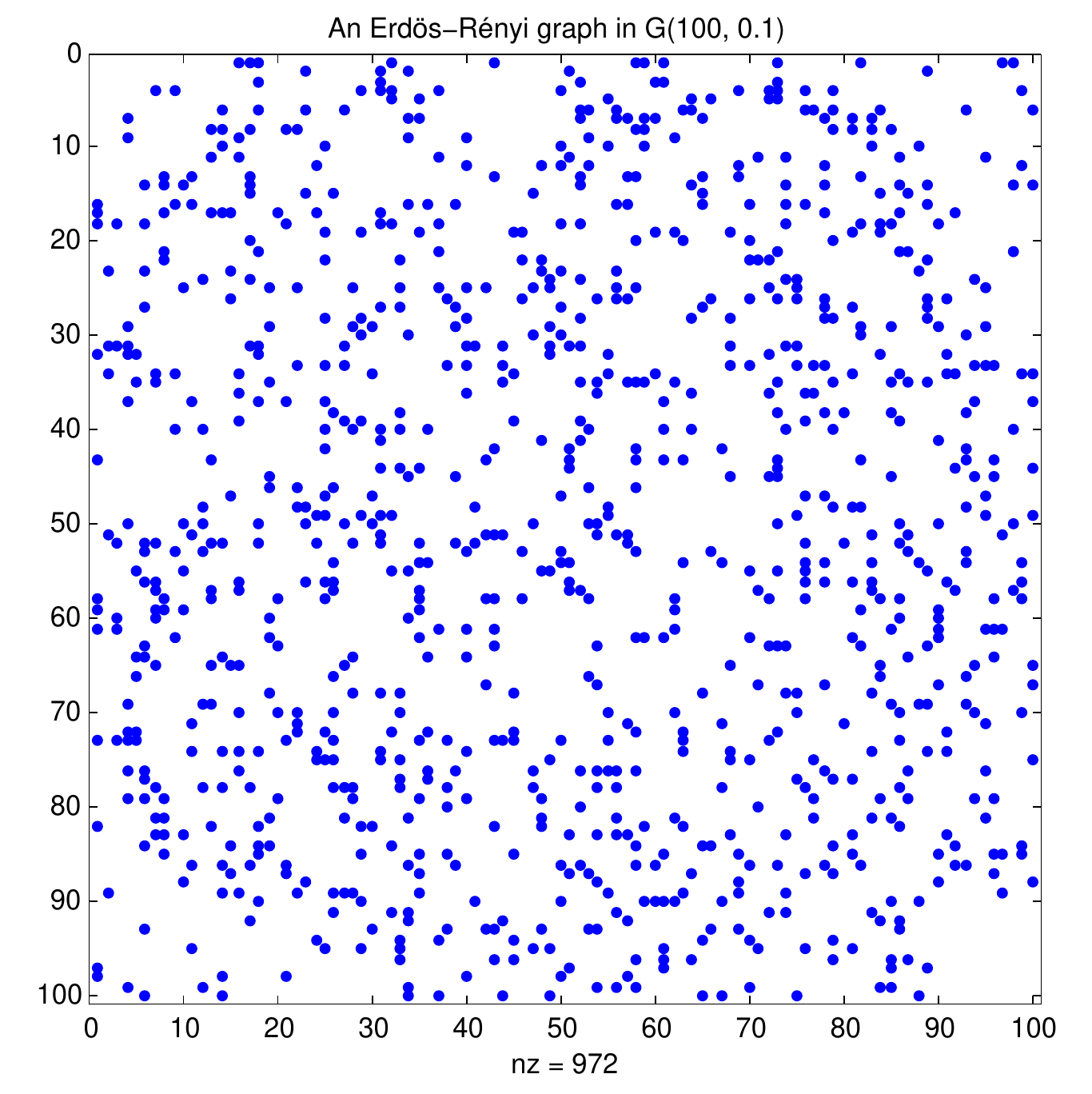}
\end{center}
\begin{caption} %
{\textbf{The adjacency matrix of an Erd{\H o}s--R{\'e}nyi graph.}  This figure shows the pattern of nonzero entries in the adjacency matrix $\mtx{A}$ of a random graph drawn from $G(100, 0.1)$.  Out of a possible 4,950 edges, there are 486 edges present.  A basic question is whether the graph is connected.  The graph is \emph{dis}connected if and only if there is a permutation of the vertices so that the adjacency matrix is block diagonal.  This property is reflected in the second-smallest eigenvalue of the Laplacian matrix $\mtx{\Delta}$.}
\end{caption} \label{fig:er-graph}
\end{figure}

Let us explain how to represent the adjacency matrix and Laplacian matrix of an Erd{\H o}s--R{\'e}nyi graph as a sum of independent random matrices.
The adjacency matrix $\mtx{A}$ of a random graph in $G(n, p)$ can be written as
\begin{equation} \label{eqn:er-graph-adjacency}
\mtx{A} = \sum_{1 \leq j < k \leq n} \xi_{jk}\, (\mathbf{E}_{jk} + \mathbf{E}_{kj}).
\end{equation}
This expression is a straightforward translation of the definition~\eqref{eqn:er-graph-adjacency-entries} into matrix form.  Similarly, the Laplacian matrix $\mtx{\Delta}$ of the random graph can be expressed as
\begin{equation} \label{eqn:er-graph-laplacian}
\mtx{\Delta} = \sum_{1 \leq j < k \leq n} \xi_{jk}\, (\mathbf{E}_{jj} + \mathbf{E}_{kk} - \mathbf{E}_{jk} - \mathbf{E}_{kj}).
\end{equation}
To verify the formula~\eqref{eqn:er-graph-laplacian}, observe that the presence of an edge between the vertices $j$ and $k$ increases the degree of $j$ and $k$ by one.  Therefore, when $\xi_{jk} = 1$, we augment the $(j, j)$ and $(k, k)$ entries of $\mtx{\Delta}$ to reflect the change in degree, and we mark the $(j, k)$ and $(k, j)$ entries with $-1$ to reflect the presence of the edge between $j$ and $k$.

\subsection{Connectivity of an Erd{\H o}s--R{\'e}nyi Graph}

We will obtain a near-optimal bound for the range of parameters where an Erd{\H o}s--R{\'e}nyi graph $G(n, p)$ is likely to be connected.  We can accomplish this goal by showing that the second smallest eigenvalue of the $n \times n$ random Laplacian matrix $\mtx{\Delta} = \mtx{D} - \mtx{A}$ is strictly positive.  We will solve the problem by using the matrix Chernoff inequality to study the
second-smallest eigenvalue of the random Laplacian $\mtx{\Delta}$.

We need to form a random matrix $\mtx{Y}$ that consists of independent positive-semidefinite terms and whose minimum eigenvalue coincides with the second-smallest eigenvalue of $\mtx{\Delta}$.  Our approach is to compress the matrix
$\mtx{Y}$ to the orthogonal complement of the vector $\mathbf{e}$ of ones.  To that end, we introduce an $(n - 1) \times n$ partial isometry $\mtx{R}$ that satisfies
\begin{equation} \label{eqn:er-restrict-prop}
\mtx{RR}^\adj = \Id_{n-1}
\quad\text{and}\quad
\mtx{R} \mathbf{e} = \vct{0}.
\end{equation}
Now, consider the $(n -1) \times (n-1)$ random matrix
\begin{equation} \label{eqn:er-Y-defn}
\mtx{Y} = \mtx{R \Delta R}^\adj = \sum_{1 \leq j < k \leq n} \xi_{jk} \cdot
	\mtx{R} \, (\mathbf{E}_{jj} + \mathbf{E}_{kk} - \mathbf{E}_{jk} - \mathbf{E}_{kj}) \mtx{R}^\adj.
\end{equation}
Recall that $\{ \xi_{jk} \}$ is an independent family of $\textsc{bernoulli}(p)$ random variables, so the summands are mutually independent.  The Conjugation Rule~\eqref{eqn:conjugation-rule} ensures that each summand remains positive semidefinite.  Furthermore, the Courant--Fischer theorem implies that the minimum eigenvalue of $\mtx{Y}$ coincides
with the second-smallest eigenvalue of $\mtx{\Delta}$ because the smallest eigenvalue of $\mtx{\Delta}$
has eigenvector $\mathbf{e}$.

To apply the matrix Chernoff inequality, we show that $L = 2$ is an upper bound for the eigenvalues of each summand in~\eqref{eqn:er-restrict-prop}.  We have
$$
\norm{ \smash{\xi_{jk} \cdot
	\mtx{R} \, (\mathbf{E}_{jj} + \mathbf{E}_{kk} - \mathbf{E}_{jk} - \mathbf{E}_{kj}) \mtx{R}^\adj } }
	\leq \abs{\smash{\xi_{jk}}} \cdot \norm{ \mtx{R} } \cdot
	\norm{ \smash{\mathbf{E}_{jj} + \mathbf{E}_{kk} - \mathbf{E}_{jk} - \mathbf{E}_{kj}} } \cdot \norm{ \smash{\mtx{R}^\adj} }
	\leq 2.	
$$
The first bound follows from the submultiplicativity of the spectral norm.  To obtain the second bound, note that $\xi_{jk}$ takes 0--1 values.  The matrix $\mtx{R}$ is a partial isometry so its norm equals one.  Finally, a direct calculation shows that $\mtx{T} = \mathbf{E}_{jj} + \mathbf{E}_{kk} - \mathbf{E}_{jk} - \mathbf{E}_{kj}$ satisfies the polynomial $\mtx{T}^2 = 2\mtx{T}$, so each eigenvalue of $\mtx{T}$ must equal zero or two.

Next, we compute the expectation of the matrix $\mtx{Y}$.
$$
\Expect \mtx{Y} = 
	p \cdot \mtx{R} \left[ \sum_{1 \leq j < k \leq n} (\mathbf{E}_{jj} + \mathbf{E}_{kk} - \mathbf{E}_{jk} - \mathbf{E}_{kj}) \right] \mtx{R}^\adj %
	= p \cdot \mtx{R} \bigl[ (n-1) \, \Id_n - (\mathbf{ee}^\adj - \Id_n) \bigr] \mtx{R}^\adj
	= pn \cdot \Id_{n-1}.
$$
The first identity follows when we apply linearity of expectation to~\eqref{eqn:er-Y-defn} and then use linearity of matrix multiplication to draw the sum inside the conjugation by $\mtx{R}$.  The term $(n-1) \, \Id_n$ emerges when we sum the diagonal matrices.  The term $\mathbf{ee}^\adj - \Id_n$ comes from the off-diagonal matrix units, once we note that the matrix $\mathbf{ee}^\adj$ has one in each component.  The last identity holds because of the properties of $\mtx{R}$ displayed in~\eqref{eqn:er-restrict-prop}.  We conclude that
$$
\lambda_{\min}(\Expect \mtx{Y}) = pn.
$$
This is all the information we need.

To arrive at a probability inequality for the second-smallest eigenvalue $\lambda_2^{\uparrow}(\mtx{\Delta})$ of the matrix $\mtx{\Delta}$, we apply the tail bound~\eqref{eqn:matrix-chernoff-lower-tail} to the matrix $\mtx{Y}$.  We obtain, for $t \in (0, 1)$,
$$
\Prob{ \lambda_{2}^\uparrow(\mtx{\Delta}) \leq t \cdot pn }
	= \Prob{ \lambda_{\min}(\mtx{Y}) \leq t \cdot pn }
	\leq (n-1) \left[ \frac{\econst^{t-1}}{t^t} \right]^{pn/2}.
$$
To appreciate what this means, we may think about the situation where $t \to 0$.  Then the bracket tends to $\econst^{-1}$, and we see that the second-smallest eigenvalue of $\mtx{\Delta}$ is unlikely to be zero when $\log(n-1) - pn / 2 < 0$.  Rearranging this expression, we obtain a sufficient condition
$$
p > \frac{2 \log(n-1)}{n}
$$
for an Erd{\H o}s--R{\'e}nyi graph $G(n,p)$ to be connected with high probability as $n \to \infty$.  This bound is quite close to the optimal result, which lacks the factor two on the right-hand side.  It is possible to make this reasoning more precise, but it does not seem worth the fuss.

\section{Proof of the Matrix Chernoff Inequalities} \label{sec:matrix-chernoff-proof}

The first step toward the matrix Chernoff inequalities is to develop an appropriate semidefinite bound for the mgf and cgf of a random positive-semidefinite matrix.  The method for establishing this result mimics the proof in the scalar case: we simply bound the exponential with a linear function.

\begin{lemma}[Matrix Chernoff: Mgf and Cgf Bound] \label{lem:matrix-chernoff-mgf}
Suppose that $\mtx{X}$ is a random matrix that satisfies $0 \leq \lambda_{\min}(\mtx{X})$ and  $\lambda_{\max}(\mtx{X}) \leq L$.  Then
$$
\Expect \econst^{\theta\mtx{X}}
	\psdle \exp\left( \frac{\econst^{\theta L} - 1}{L} \cdot \Expect \mtx{X} \right)
	\quad\text{and}\quad
\log{} \Expect \econst^{\theta\mtx{X}}
	\psdle \frac{\econst^{\theta L} - 1}{L} \cdot \Expect \mtx{X}
	\quad\text{for $\theta \in \mathbb{R}$}.
$$
\end{lemma}

\begin{proof}
Consider the function $f(x) = \econst^{\theta x}$.  Since $f$ is convex, its graph lies below the chord connecting any two points on the graph.  In particular,
$$
f(x) \leq f(0) + \frac{f(L) - f(0)}{L} \cdot x
\quad\text{for $x \in [0, L]$.}
$$
In detail,
$$
\econst^{\theta x} \leq 1 + \frac{\econst^{\theta L} - 1}{L} \cdot x
\quad\text{for $x \in [0, L]$.}
$$
By assumption, each eigenvalue of $\mtx{X}$ lies in the interval $[0, L]$.  Thus, the Transfer Rule~\eqref{eqn:transfer-rule} implies that
$$
\econst^{\theta \mtx{X}} \psdle \Id + \frac{\econst^{\theta L} - 1}{L} \cdot \mtx{X}.
$$
Expectation respects the semidefinite order, so
$$
\Expect \econst^{\theta \mtx{X}}
	\psdle \Id + \frac{\econst^{\theta L} - 1}{L} \cdot \Expect\mtx{X}
	\psdle \exp\left( \frac{\econst^{\theta L} - 1}{L} \cdot \Expect\mtx{X} \right).
$$
The second relation is a consequence of the fact that $\Id + \mtx{A} \psdle \econst^{\mtx{A}}$
for every matrix $\mtx{A}$, which we obtain by applying the Transfer Rule~\eqref{eqn:transfer-rule}
to the inequality $1 + a \leq \econst^a$, valid for all $a \in \mathbb{R}$.

To obtain the semidefinite bound for the cgf, we simply take the
logarithm of the semidefinite bound for the mgf.  This operation
preserves the semidefinite order because of
the property~\eqref{eqn:log-monotone} that
the logarithm is operator monotone.
\end{proof}

We break the proof of the matrix inequality into two pieces.  First, we establish
the bounds on the maximum eigenvalue, which are slightly easier.  Afterward, we
develop the bounds on the minimum eigenvalue.

\begin{proof}[Proof of Theorem~\ref{thm:matrix-chernoff}, Maximum Eigenvalue Bounds]
Consider a finite sequence $\{\mtx{X}_k\}$ of independent, random Hermitian matrices
with common dimension $d$.  Assume that
$$
0 \leq \lambda_{\min}(\mtx{X}_k)
\quad\text{and}\quad
\lambda_{\max}(\mtx{X}_k) \leq L
\quad\text{for each index $k$.}
$$
The cgf bound, Lemma~\ref{lem:matrix-chernoff-mgf}, states that
\begin{equation} \label{eqn:upper-chernoff-proof-cgf}
\log{} \Expect \econst^{\theta \mtx{X}_k}
	\psdle g(\theta) \cdot \Expect \mtx{X}_k
	\quad\text{where}\quad
	g(\theta) = \frac{\econst^{\theta L} - 1}{L}
	\quad\text{for $\theta > 0$.}
\end{equation}
We begin with the upper bound~\eqref{eqn:matrix-chernoff-upper-expect}
for $\Expect \lambda_{\max}(\mtx{Y})$.
Using the fact~\eqref{eqn:exp-trace-monotone} that the trace of the exponential
function is monotone with respect to the semidefinite order, we substitute these
cgf bounds into the master inequality~\eqref{eqn:master-upper-expect} for the
expectation of the maximum eigenvalue to reach
$$
\begin{aligned}
\Expect \lambda_{\max}(\mtx{Y})
	&\leq \inf_{\theta > 0} \ \frac{1}{\theta} \, \log{} \trace
	\exp\Big( g(\theta) \sum\nolimits_k \Expect \mtx{X}_k \Big) \\
	&\leq \inf_{\theta > 0} \ \frac{1}{\theta} \, \log{} \left[ d \, \lambda_{\max}\left(
	\exp\left( g(\theta) \cdot \Expect \mtx{Y} \right) \right) \right] \\
	&= \inf_{\theta > 0} \ \frac{1}{\theta} \, \log{} \left[ d \,
	\exp\left( \lambda_{\max}\left( g(\theta) \cdot \Expect \mtx{Y} \right) \right) \right]	\\
	&= \inf_{\theta > 0} \ \frac{1}{\theta} \, \log{} \left[ d \,
	\exp\left( g(\theta) \cdot \lambda_{\max}(\Expect \mtx{Y}) \right) \right] \\
	&= \inf_{\theta > 0} \ \frac{1}{\theta} \left[ \log d + g(\theta) \cdot \mu_{\max} \right].
\end{aligned}
$$
In the second line, we use the fact that the matrix exponential is positive definite to bound the trace by $d$ times the maximum eigenvalue; we have also identified the sum as $\Expect \mtx{Y}$.  The third line follows from the Spectral Mapping Theorem, Proposition~\ref{prop:spectral-mapping}.  Next, we use the fact~\eqref{eqn:eig-pos-homo} that the maximum eigenvalue is a positive-homogeneous map, which depends on the observation that $g(\theta) > 0$ for $\theta > 0$.
Finally, we identify the statistic~$\mu_{\max}$ defined in~\eqref{eqn:matrix-chernoff-mu-max}.
The infimum does not admit a closed form, but we can obtain the expression~\eqref{eqn:matrix-chernoff-upper-expect} by making the change of variables $\theta \mapsto \theta/L$.

Next, we turn to the upper bound~\eqref{eqn:matrix-chernoff-upper-tail} for the upper tail of the maximum eigenvalue.  Substitute the cgf bounds~\eqref{eqn:upper-chernoff-proof-cgf} into the master inequality~\eqref{eqn:master-upper-tail} to reach
$$
\begin{aligned}
\Prob{ \lambda_{\max}(\mtx{Y}) \geq t }
	&\leq \inf_{\theta > 0} \ \econst^{-\theta t} \,
	\trace \exp\Big( g(\theta) \sum\nolimits_k \Expect \mtx{X}_k \Big) \\
	&\leq \inf_{\theta > 0} \ \econst^{-\theta t} \cdot
	d \, \exp\left( g(\theta) \cdot \mu_{\max} \right).
\end{aligned}
$$
The steps here are identical with the previous argument.  To complete the proof, make the change of variables $t \mapsto (1+\eps) \mu_{\max}$.  Then the infimum is achieved at $\theta = L^{-1} \log(1+\eps)$, which leads to the tail bound~\eqref{eqn:matrix-chernoff-upper-tail}.
\end{proof}

The lower bounds follow from a related argument that is slightly more delicate.

\begin{proof}[Proof of Theorem~\ref{thm:matrix-chernoff}, Minimum Eigenvalue Bounds]
Once again, consider a finite sequence $\{\mtx{X}_k\}$ of independent, random Hermitian matrices
with dimension $d$.  Assume that
$$
0 \leq \lambda_{\min}(\mtx{X}_k)
\quad\text{and}\quad
\lambda_{\max}(\mtx{X}_k) \leq L
\quad\text{for each index $k$.}
$$
The cgf bound, Lemma~\ref{lem:matrix-chernoff-mgf}, states that
\begin{equation} \label{eqn:lower-chernoff-proof-cgf}
\log{} \Expect \econst^{\theta \mtx{X}_k}
	\psdle g(\theta) \cdot \Expect \mtx{X}_k
	\quad\text{where}\quad
	g(\theta) = \frac{\econst^{\theta L} - 1}{L}
	\quad\text{for $\theta < 0$.}
\end{equation}
Note that $g(\theta) < 0$ for $\theta < 0$, which
alters a number of the steps in the argument.

We commence with the lower bound~\eqref{eqn:matrix-chernoff-lower-expect}
for $\Expect \lambda_{\min}(\mtx{Y})$.
As stated in~\eqref{eqn:exp-trace-monotone},
the trace exponential function is monotone with respect to the semidefinite order,
so the master inequality~\eqref{eqn:master-lower-expect}
for the minimum eigenvalue delivers
$$
\begin{aligned}
\Expect \lambda_{\min}(\mtx{Y})
	&\geq \sup_{\theta < 0} \ \frac{1}{\theta} \, \log{} \trace
	\exp\Big( g(\theta) \sum\nolimits_k \Expect \mtx{X}_k \Big) \\
	&\geq \sup_{\theta < 0} \ \frac{1}{\theta} \, \log{} \left[ d \, \lambda_{\max}\left(
	\exp\left( g(\theta) \cdot \Expect \mtx{Y} \right) \right) \right] \\	
	&= \sup_{\theta < 0} \ \frac{1}{\theta} \, \log{} \left[ d \, \exp\left(
	\lambda_{\max}\left( g(\theta) \cdot \Expect \mtx{Y} \right) \right) \right] \\	
	&= \sup_{\theta < 0} \ \frac{1}{\theta} \, \log{} \left[ d \,
	\exp\left( g(\theta) \cdot \lambda_{\min}(\Expect \mtx{Y}) \right) \right] \\
	&= \sup_{\theta < 0} \ \frac{1}{\theta} \left[ \log d + g(\theta) \cdot \mu_{\min} \right].
\end{aligned}
$$
Most of the steps are the same as in the proof of the upper bound~\eqref{eqn:matrix-chernoff-upper-expect}, so we focus on the differences.  Since the factor $\theta^{-1}$ in the first and second lines is negative, upper bounds on the trace reduce the value of the expression.  We move to the fourth line by invoking the property $\lambda_{\max}(\alpha \mtx{A}) = \alpha \lambda_{\min}(\mtx{A})$ for $\alpha < 0$, which follows from~\eqref{eqn:eig-pos-homo} and~\eqref{eqn:min-max-sign-eig}.
This piece of algebra depends on the fact that $g(\theta) < 0$ when $\theta < 0$.
To obtain the result~\eqref{eqn:matrix-chernoff-lower-expect}, we change variables: $\theta \mapsto -\theta /L$.

Finally, we establish the bound~\eqref{eqn:matrix-chernoff-lower-tail} for the lower tail of the minimum eigenvalue.  Introduce the cgf bounds~\eqref{eqn:lower-chernoff-proof-cgf} into the master inequality~\eqref{eqn:master-lower-tail} to reach
$$
\begin{aligned}
\Prob{ \lambda_{\min}(\mtx{Y}) \leq t }
	&\leq \inf_{\theta < 0} \ \econst^{-\theta t} \,
	\trace \exp\Big( g(\theta) \sum\nolimits_k \Expect \mtx{X}_k \Big) \\
	&\leq \inf_{\theta < 0} \ \econst^{-\theta t} \cdot
	d \, \exp\left( g(\theta) \cdot \mu_{\min} \right).
\end{aligned}
$$
The justifications here match those in with the previous argument.  Finally, we make the change of variables $t \mapsto (1-\eps) \mu_{\min}$.  The infimum is attained at $\theta = L^{-1} \log(1-\eps)$, which yields the tail bound~\eqref{eqn:matrix-chernoff-lower-tail}.
\end{proof}

\section{Notes}

As usual, we continue with an overview of background references and related work.

\subsection{Matrix Chernoff Inequalities}

Scalar Chernoff inequalities date to the paper~\cite[Thm.~1]{Che52:Measure-Asymptotic} by Herman Chernoff.  The original result provides probability bounds for the number of successes in a sequence of independent but non-identical Bernoulli trials.  Chernoff's proof combines the scalar Laplace transform method with refined bounds on the mgf of a Bernoulli random variable.  It is very common to encounter simplified versions of Chernoff's result, such as~\cite[Exer.~8]{Lug09:Concentration-Measure} or~\cite[\S4.1]{MR95:Randomized-Algorithms}.

In their paper~\cite{AW02:Strong-Converse}, Ahlswede \& Winter developed a matrix version of the Chernoff inequality.  The matrix mgf bound, Lemma~\ref{lem:matrix-chernoff-mgf}, essentially appears in their work.  Ahlswede \& Winter focus on the case of independent and identically distributed random matrices, in which case their results are roughly equivalent with Theorem~\ref{thm:matrix-chernoff}.  For the general case, their approach leads to matrix expectation statistics of the form
$$
\mu_{\min}^{\textrm{AW}} = \sum\nolimits_k \lambda_{\min}(\Expect \mtx{X}_k)
\quad\text{and}\quad
\mu_{\max}^{\textrm{AW}} = \sum\nolimits_k \lambda_{\max}(\Expect \mtx{X}_k).
$$
It is clear that their $\mu_{\min}^{\textrm{AW}}$ may be substantially smaller than the quantity $\mu_{\min}$ we defined in Theorem~\ref{thm:matrix-chernoff}.  Similarly, their $\mu_{\max}^{\textrm{AW}}$ may be substantially larger than the quantity $\mu_{\max}$ that drives the upper Chernoff bounds.

The tail bounds from Theorem~\ref{thm:matrix-chernoff} are drawn from~\cite[\S5]{Tro11:User-Friendly-FOCM}, but the expectation bounds we present are new.  The technical report~\cite{GT11:Tail-Bounds} extends the matrix Chernoff inequality to provide upper and lower tail bounds for all eigenvalues of a sum of random, positive-semidefinite matrices.  Chapter~\ref{chap:intrinsic} contains a slight improvement of the bounds for the maximum eigenvalue in Theorem~\ref{thm:matrix-chernoff}.

Let us mention a few other results that are related to the matrix Chernoff inequality.
First, Theorem~\ref{thm:matrix-chernoff} has a lovely information-theoretic formulation
where the tail bounds are stated in terms of an information divergence.  To establish
this result, we must restructure the proof and eliminate some of the approximations.
See~\cite[Thm.~19]{AW02:Strong-Converse} or~\cite[Thm.~5.1]{Tro11:User-Friendly-FOCM}.

Second, the problem of bounding the minimum eigenvalue of a sum of random,
positive-semidefinite matrices has a special character.  The reason, roughly,
is that a sum of independent, nonnegative random variables cannot easily take the value zero.
A closely related phenomenon holds in the matrix setting,
and it is possible to develop estimates that exploit this observation.
See~\cite[Thm.~3.1]{Oli13:Lower-Tail} and~\cite[Thm.~1.3]{KM13:Bounding-Smallest}
for two wildly different approaches.

\subsection{The Matrix Rosenthal Inequality}

The matrix Rosenthal inequality~\eqref{eqn:matrix-rosenthal}
is one of the earliest matrix concentration bounds.  In his paper~\cite{Rud99:Random-Vectors},
Rudelson used the noncommutative Khintchine inequality~\eqref{eqn:nc-khintchine}
to establish a specialization of~\eqref{eqn:matrix-rosenthal} to rank-one summands.
A refinement appears in~\cite{RV07:Sampling-Large}, and explicit
constants were first derived in~\cite{Tro08:Random-Paving}.  We believe
that the paper~\cite{CGT12:Masked-Sample} contains the first complete
statement of the moment bound~\eqref{eqn:matrix-rosenthal}
for general positive-semidefinite summands;
see also the work~\cite{MZ11:Low-Rank-Matrix-Valued}.
The constants in~\cite[Thm.~A.1]{CGT12:Masked-Sample},
and hence in~\eqref{eqn:matrix-rosenthal}, can be improved
slightly by using the sharp version of the noncommutative Khintchine
inequality from~\cite{Buc01:Operator-Khintchine,Buc05:Optimal-Constants}.
Let us stress that all of these results follow from easy variations
of Rudelson's argument.

The work~\cite[Cor.~7.4]{MJCFT12:Matrix-Concentration} provides a self-contained
and completely elementary proof of a matrix Rosenthal inequality that is closely
related to~\eqref{eqn:matrix-rosenthal}.  This result depends on different
principles from the works mentioned in the last paragraph.

\subsection{Random Submatrices}

The problem of studying a random submatrix drawn from a fixed matrix has a long history.
An early example is the paving problem from operator theory, which asks for a maximal well-conditioned set of columns (or a well-conditioned submatrix) inside a fixed matrix.  Random selection provides a natural way to approach this question.  The papers of Bourgain \& Tzafriri~\cite{BT87:Invertibility-Large,BT91:Problem-Kadison-Singer} and Kashin \& Tzafriri~\cite{KT94:Some-Remarks} study random paving using sophisticated tools from functional analysis.  See the paper~\cite{NT12:Paved-Good} for a summary of research on randomized methods for constructing pavings.  Very recently, Adam Marcus, Dan Spielman, \& Nikhil Srivastava~\cite{MSS13:Interlacing-Families-II} have solved the paving problem completely.

Later, Rudelson and Vershynin~\cite{RV07:Sampling-Large} showed that the noncommutative Khintchine inequality provides a clean way to bound the norm of a random column submatrix (or a random row and column submatrix) drawn from a fixed matrix.  Their ideas have found many applications in the mathematical signal processing literature.  For example, the paper~\cite{Tro08:Conditioning-Random} uses similar techniques to analyze the perfomance of $\ell_1$ minimization for recovering a random sparse signal.  The same methods support the paper~\cite{Tro08:Random-Paving}, which contains a modern proof of the random paving result~\cite[Thm.~2.1]{BT91:Problem-Kadison-Singer} of Bourgain \& Tzafriri.

The article~\cite{Tro11:Improved-Analysis} contains the observation that the matrix Chernoff inequality is an ideal tool for studying random submatrices.  It applies this technique to study a random matrix that arises in numerical linear algebra~\cite{HMT11:Finding-Structure}, and it achieves an optimal estimate for the minimum singular value of the random matrix that arises in this setting.  Our analysis of a random column submatrix is based on this work.  The analysis of a random row and column submatrix is new.  The paper~\cite{CD12:Invertibility-Random}, by Chr{\'e}tien and Darses, uses matrix Chernoff bounds in a more sophisticated way to develop tail bounds for the norm of a random row and column submatrix.

\subsection{Random Graphs}

The analysis of random graphs and random hypergraphs appeared as one of the earliest applications of matrix concentration inequalities~\cite{AW02:Strong-Converse}.  Christofides and Markstr{\"o}m developed a matrix Hoeffding inequality to aid in this purpose~\cite{CM08:Expansion-Properties}.  Later, Oliveira wrote two papers~\cite{Oli10:Concentration-Adjacency,Oli11:Spectrum-Random} on random graph theory based on matrix concentration.  We recommend these works for further information.

To analyze the random graph Laplacian, we compressed the Laplacian to a subspace so that the minimum eigenvalue of the compression coincides with the second-smallest eigenvalue of the original Laplacian.  This device can be extended to
obtain tail bounds for all the eigenvalues of a sum of independent random matrices.
See the technical report~\cite{GT11:Tail-Bounds} for a development of this idea.

\makeatletter{}%

\chapter[A Sum of Bounded Random Matrices]{A Sum of Bounded \\ Random Matrices}
\label{chap:matrix-bernstein}

In this chapter, we describe matrix concentration inequalities that generalize the classical Bernstein bound.  The matrix Bernstein inequalities concern a random matrix formed as a sum of independent, random matrices that are bounded in spectral norm.  The results allow us to study how much this type of random matrix deviates from its mean value in the spectral norm.

Formally, we consider an finite sequence $\{ \mtx{S}_k \}$ of random matrices of the same dimension.
Assume that the matrices satisfy the conditions
$$
\Expect \mtx{S}_k = \mtx{0}
\quad\text{and}\quad
\norm{ \mtx{S}_k } \leq L
\quad\text{for each index $k$.}
$$
Form the sum $\mtx{Z} = \sum_k \mtx{S}_k$.  The matrix Bernstein inequality controls the expectation and tail behavior of $\norm{\mtx{Z}}$ in terms of the matrix variance statistic $v(\mtx{Z})$ and the uniform bound $L$.

The matrix Bernstein inequality is a powerful tool with a huge number of applications.  In these pages, we can only give a coarse indication of how researchers have used this result, so we have chosen to focus on problems that use random sampling to approximate a specified matrix.  This model applies to the sample covariance matrix in the introduction.  In this chapter, we outline several additional examples.  First, we consider the technique of randomized sparsification, in which we replace a dense matrix with a sparse proxy that has similar spectral behavior.  Second, we explain how to develop a randomized algorithm for approximate matrix multiplication, and we establish an error bound for this method.  Third, we develop an analysis of random features, a method for approximating kernel matrices that has become popular in contemporary machine learning.

As these examples suggest, the matrix Bernstein inequality is very effective for studying randomized approximations of a given matrix.  Nevertheless, when the matrix Chernoff inequality, Theorem~\ref{thm:matrix-chernoff}, happens to apply to a problem, it often delivers better results.

\subsubsection{Overview}

Section~\ref{sec:bernstein-rect} describes the matrix Bernstein inequality.  Section~\ref{sec:rdm-mtx-approx} explains how to use the Bernstein inequality to study randomized methods for matrix approximation.  In \S\S\ref{sec:rdm-sparse},~\ref{sec:rdm-mtx-mult}, and~\ref{sec:rdm-features}, we apply the latter result to three matrix approximation problems. %
We conclude with the proof of the matrix Bernstein inequality in~\S\ref{sec:bernstein-proof}.

\section{A Sum of Bounded Random Matrices} \label{sec:bernstein-rect}

In the scalar setting, the label ``Bernstein inequality'' applies to a very large number of concentration results.  Most of these bounds have extensions to matrices.  For simplicity, we focus on the most famous of the scalar results, a tail bound for the sum $Z$ of independent, zero-mean random variables that are subject to a uniform bound.  In this case, the Bernstein inequality shows that $Z$ concentrates around zero.  The tails of $Z$ make a transition from subgaussian decay at moderate deviations to subexponential decay at large deviations.  See~\cite[\S2.7]{BLM13:Concentration-Inequalities} for more information about Bernstein's inequality.

In analogy, the simplest matrix Bernstein inequality concerns a sum of independent, zero-mean random matrices whose norms are bounded above.  The theorem demonstrates that the norm of the sum acts much like the scalar random variable $Z$ that we discussed in the last paragraph.

\begin{thm}[Matrix Bernstein] \label{thm:matrix-bernstein-rect}
Consider a finite sequence $\{\mtx{S}_k\}$ of independent, random matrices with common dimension $d_1 \times d_2$.  Assume that
$$
\Expect \mtx{S}_k = \mtx{0}
\quad\text{and}\quad
\norm{ \mtx{S}_k } \leq L
\quad\text{for each index $k$.}
$$
Introduce the random matrix
$$
\mtx{Z} = \sum\nolimits_k \mtx{S}_k.
$$
Let $v(\mtx{Z})$ be the matrix variance statistic of the sum:
\begin{align} %
v(\mtx{Z}) &= \max\big\{ \norm{ \smash{\Expect(\mtx{ZZ}^\adj )} },\
	\norm{ \smash{\Expect(\mtx{Z}^\adj \mtx{Z})} } \big\}
	\label{eqn:matrix-bernstein-sigma2-rect} \\
	&= \max\left\{ \norm{\sum\nolimits_k \Expect \left(\mtx{S}_k\mtx{S}_k^\adj \right)}, \
	\norm{ \sum\nolimits_k \Expect\left( \mtx{S}_k^\adj \mtx{S}_k \right) } \right\}.
	\label{eqn:matrix-bernstein-rect-var-calc}
\end{align}
Then
\begin{equation} \label{eqn:matrix-bernstein-expect-rect}
\Expect \norm{\mtx{Z}} \leq \sqrt{2 v(\mtx{Z}) \log(d_1 + d_2)} + \frac{1}{3} L\, \log(d_1 + d_2).
\end{equation}
Furthermore, for all $t \geq 0$,
\begin{equation} \label{eqn:matrix-bernstein-tail-rect}
\Prob{ \norm{ \mtx{Z} } \geq t }
	\leq (d_1 + d_2)\, \exp\left( \frac{-t^2/2}{v(\mtx{Z}) + Lt/3} \right).
\end{equation}
\end{thm}

\noindent
The proof of Theorem~\ref{thm:matrix-bernstein-rect} appears in \S\ref{sec:bernstein-proof}.

\subsection{Discussion}

Let us spend a few moments to discuss the matrix Bernstein inequality, Theorem~\ref{thm:matrix-bernstein-rect},
its consequences, and some of the improvements that are available.

\subsubsection{Aspects of the Matrix Bernstein Inequality}

First, observe that the matrix variance statistic $v(\mtx{Z})$ appearing in~\eqref{eqn:matrix-bernstein-sigma2-rect}
coincides with the general definition~\eqref{eqn:matrix-variance-rect} because $\mtx{Z}$ has zero mean.
To reach~\eqref{eqn:matrix-bernstein-rect-var-calc}, we have used the additivity law~\eqref{eqn:indep-sum-rect}
for an independent sum to express the matrix variance statistic in terms of the summands.
Observe that, when the summands $\mtx{S}_k$ are Hermitian, the two terms in the maximum coincide.

The expectation bound~\eqref{eqn:matrix-bernstein-expect-rect} shows that $\Expect \norm{\mtx{Z}}$ is on the same scale as the root $\sqrt{v(\mtx{Z})}$ of the matrix variance statistic and the upper bound $L$ for the summands; there is also a weak dependence on the ambient dimension $d$.  In general, all three of these features are necessary.  Nevertheless, the bound may not be very tight for particular examples.  See Section~\ref{sec:bernstein-optimality} for some evidence.

Next, let us explain how to interpret the tail bound~\eqref{eqn:matrix-bernstein-tail-rect}.  The main difference between this result and the scalar Bernstein bound is the appearance of the dimensional factor $d_1 + d_2$, which reduces the range of $t$ where the inequality is informative.  To get a better idea of what this result means, it is helpful to make a further estimate:
\begin{equation} \label{eqn:split-bernstein}
\Prob{ \norm{\mtx{Z}} \geq t } \leq
\begin{dcases}
	(d_1+d_2) \cdot \econst^{ - 3t^2 / (8v(\mtx{Z})) }, & t \leq v(\mtx{Z})/L \\
	(d_1+d_2) \cdot \econst^{ - 3t / (8L) }, & t \geq v(\mtx{Z})/L.
\end{dcases}
\end{equation}
In other words, for moderate values of $t$, the tail probability decays as fast as the tail of a Gaussian random variable whose variance is comparable with $v(\mtx{Z})$.  For larger values of $t$, the tail probability decays at least as fast as that of an exponential random variable whose mean is comparable with $L$.
As usual, we insert a warning that the tail behavior reported by the matrix Bernstein inequality can overestimate the actual tail behavior.

Last, it is helpful to remember that the matrix Bernstein inequality extends to a sum of \emph{uncentered} random matrices.  In this case, the result describes the spectral-norm deviation of the random sum from its mean value.
For reference, we include the statement here.

\begin{cor}[Matrix Bernstein: Uncentered Summands] %
Consider a finite sequence $\{\mtx{S}_k\}$ of independent random matrices with common dimension $d_1 \times d_2$.
Assume that each matrix has uniformly bounded deviation from its mean:
$$
\norm{ \mtx{S}_k - \Expect \mtx{S}_k } \leq L
\quad\text{for each index $k$.}
$$
Introduce the sum
$$
\mtx{Z} = \sum\nolimits_k \mtx{S}_k,
$$
and let $v(\mtx{Z})$ denote the matrix variance statistic of the sum:
\begin{equation*} %
\begin{aligned}
v(\mtx{Z}) &= %
	\max\left\{ \norm{ \Expect\bigl[ (\mtx{Z} - \Expect \mtx{Z})(\mtx{Z} - \Expect \mtx{Z})^\adj \bigr] }, \ 
	\norm{ \Expect \bigl[ (\mtx{Z} - \Expect \mtx{Z})^\adj (\mtx{Z} - \Expect \mtx{Z}) \bigr] } \right\} \\
	&= \max\left\{ \norm{ \sum\nolimits_k \Expect\bigl[ (\mtx{S}_k - \Expect \mtx{S}_k)(\mtx{S}_k - \Expect \mtx{S}_k)^\adj \bigr] }, \ 
	\norm{ \sum\nolimits_k \Expect \bigl[ (\mtx{S}_k - \Expect \mtx{S}_k)^\adj (\mtx{S}_k - \Expect \mtx{S}_k) \bigr] } \right\}.
\end{aligned}
\end{equation*}
Then
\begin{equation*} %
\Expect \norm{ \mtx{Z} - \Expect \mtx{Z} }
	\leq \sqrt{ 2 v(\mtx{Z}) \log(d_1 + d_2) } + \frac{1}{3} L \,\log(d_1 + d_2).
\end{equation*}
Furthermore, for all $t \geq 0$,
\begin{equation*} %
\Prob{ \norm{ \mtx{Z} - \Expect \mtx{Z} } \geq t }
	\leq (d_1 + d_2) \cdot \exp\left( \frac{-t^2/2}{v(\mtx{Z}) + Lt/3} \right).
\end{equation*}
\end{cor}

\noindent
This result follows as an immediate corollary of Theorem~\ref{thm:matrix-bernstein-rect}.

\subsubsection{Related Results}

The bounds in Theorem~\ref{thm:matrix-bernstein-rect} are stated in
terms of the ambient dimensions $d_1$ and $d_2$ of the random matrix $\mtx{Z}$.
The dependence on the ambient dimension is not completely natural.  For example,
consider embedding the random matrix $\mtx{Z}$ into the top corner of a much
larger matrix which is zero everywhere else.  It turns out that we can achieve
results that reflect only the ``intrinsic dimension'' of $\mtx{Z}$.
We turn to this analysis in Chapter~\ref{chap:intrinsic}.

In addition, there are many circumstances where the uniform upper bound $L$
that appears in~\eqref{eqn:matrix-bernstein-expect-rect} does not accurately
reflect the tail behavior of the random matrix.
For instance, the summands themselves may have very heavy tails.
In such emergencies, the following expectation bound~\cite[Thm.~A.1]{CGT12:Masked-Sample}
can be a lifesaver.
\begin{equation} \label{eqn:matrix-moment-ineq}
\left( \Expect \normsq{\mtx{Z}} \right)^{1/2}
	\leq \sqrt{2 \econst v(\mtx{Z}) \log(d_1 + d_2)}
	+ 4\econst \, \big( \Expect{} \max\nolimits_k \normsq{\mtx{S}_k} \big)^{1/2} \log(d_1 + d_2).
\end{equation}
This result is a matrix formulation of the Rosenthal--Pinelis inequality~\cite[Thm.~4.1]{Pin94:Optimum-Bounds}.

Finally, let us reiterate that there are other types of matrix Bernstein inequalities.
For example, we can sharpen the tail bound~\eqref{eqn:matrix-bernstein-tail-rect} to obtain a matrix Bennett inequality.
We can also relax the boundedness assumption to a weaker hypothesis on the growth of the moments of each summand $\mtx{S}_k$.
In the Hermitian setting, the result can also discriminate the behavior of the upper and lower tails,
which is a consequence of Theorem~\ref{thm:matrix-bernstein-herm} below.
See the notes at the end of this chapter and the annotated bibliography for more information.

\subsection{Optimality of the Matrix Bernstein Inequality} \label{sec:bernstein-optimality}

To use the matrix Bernstein inequality, Theorem~\ref{thm:matrix-bernstein-rect}, and its relatives
with intelligence, one must appreciate their strengths and weaknesses.  We will focus on the matrix
Rosenthal--Pinelis inequality~\eqref{eqn:matrix-moment-ineq}.
Nevertheless, similar insights are relevant to the estimate~\eqref{eqn:matrix-bernstein-expect-rect}.

\subsubsection{The Expectation Bound}

Let us present lower bounds to demonstrate that the matrix
Rosenthal--Pinelis inequality~\eqref{eqn:matrix-moment-ineq} requires
both terms that appear.
First, the quantity $v(\mtx{Z})$ cannot be omitted because Jensen's inequality implies that
$$
\Expect{} \normsq{ \mtx{Z} }
	= \Expect{} \max\big\{ \norm{ \smash{\mtx{ZZ}^\adj} }, \ \norm{\smash{\mtx{Z}^\adj \mtx{Z}}} \big\}
	\geq \max\big\{ \norm{ \smash{\Expect(\mtx{ZZ}^\adj)} }, \ \norm{ \smash{\Expect(\mtx{Z}^\adj \mtx{Z})}} \big\}
	= v(\mtx{Z}).
$$
Under a natural hypothesis, the second term on the right-hand side of~\eqref{eqn:matrix-moment-ineq}
also is essential.  Suppose that each summand $\mtx{S}_k$ is a symmetric random variable;
that is, $\mtx{S}_k$ and $- \mtx{S}_k$ have the same distribution.
In this case, an involved argument~\cite[Prop.~6.10]{LT91:Probability-Banach} leads to the bound
\begin{equation} \label{eqn:mtx-bernstein-lower-unif}
\Expect{} \normsq{ \mtx{Z} }
	\geq \mathrm{const} \cdot \Expect{} \max\nolimits_k \normsq{ \mtx{S}_k }.
\end{equation}
There are examples where the right-hand side of~\eqref{eqn:mtx-bernstein-lower-unif}
is comparable with the uniform upper bound $L$ on the summands, but this is not always so.

In summary, when the summands $\mtx{S}_k$ are symmetric, we have matching estimates
\begin{multline*}
\mathrm{const} \cdot \left[ \sqrt{v(\mtx{Z})} + \big(\Expect{} \max\nolimits_k \normsq{ \mtx{S}_k } \big)^{1/2} \right]
	\leq \left( \Expect{} \normsq{\mtx{Z}} \right)^{1/2} \\
	\leq \mathrm{Const} \cdot \left[ \sqrt{v(\mtx{Z}) \log(d_1 + d_2)} 
	+ \big( \Expect{} \max\nolimits_k \normsq{ \mtx{S}_k }\big)^{1/2} \log(d_1 + d_2) \right].
\end{multline*}
We see that the bound~\eqref{eqn:matrix-moment-ineq} must include some version
of each term that appears, but the logarithms are not always necessary.

\subsubsection{Examples where the Logarithms Appear}

First, let us show that the variance term in~\eqref{eqn:matrix-moment-ineq} must contain
a logarithm.  For each natural number $n$,
consider the $d \times d$ random matrix $\mtx{Z}$ of the form
$$
\mtx{Z}_n =  \frac{1}{\sqrt{n}} \sum_{i=1}^n \sum_{k=1}^d \varrho_{ik} \, \mathbf{E}_{kk}
$$
where $\{\varrho_{ik}\}$ is an independent family of Rademacher random
variables.  An easy application of the bound~\eqref{eqn:matrix-moment-ineq}
implies that
$$
\Expect{} \norm{ \mtx{Z}_n } \leq \mathrm{Const} \cdot \left( \sqrt{\log(2d)} + \frac{1}{\sqrt{n}} \log(2d) \right)
\to \mathrm{Const} \cdot \sqrt{\log(2d)}
\quad\text{as $n \to \infty$.}
$$
Using the central limit theorem and the Skorokhod representation,
we can construct an independent family $\{ \gamma_k \}$ of standard
normal random variables for which
$$
\mtx{Z}_n \to \sum_{k=1}^d \gamma_k \, \mathbf{E}_{kk}
\quad\text{almost surely as $n \to \infty$.}
$$
But this fact ensures that
$$
\Expect{} \norm{ \mtx{Z}_n } \rightarrow \Expect{} \norm{ \sum_{k=1}^d \gamma_k \mathbf{E}_{kk} }
	= \Expect{} \max\nolimits_k \abs{\smash{\gamma_k}}
	\approx \sqrt{2 \log d}
	\quad\text{as $n \to \infty$.}
$$
Therefore, we cannot remove the logarithm from the variance term in~\eqref{eqn:matrix-moment-ineq}.

Next, let us justify the logarithm on the norm of the summands in~\eqref{eqn:matrix-moment-ineq}.
For each natural number $n$, consider a $d \times d$ random matrix $\mtx{Z}$ of the form
$$
\mtx{Z}_n = \sum_{i=1}^n \sum_{k=1}^d \big(\delta_{ik}^{(n)} - n^{-1} \big) \, \mathbf{E}_{kk}
$$
where $\big\{\delta_{ik}^{(n)} \big\}$ is an independent family of $\textsc{bernoulli}(n^{-1})$
random variables.  The matrix Rosenthal--Pinelis inequality~\eqref{eqn:matrix-moment-ineq} ensures that
$$
\Expect{} \norm{ \mtx{Z}_n } \leq \mathrm{Const} \cdot \left( \sqrt{\log(2d)} + \log(2d) \right).
$$
Using the Poisson limit of a binomial random variable and the Skorohod representation,
we can construct an independent family $\{Q_k\}$ of $\textsc{poisson}(1)$ random variables
for which
$$
\mtx{Z}_n \rightarrow \sum_{k=1}^d (Q_k - 1) \, \mathbf{E}_{kk}
\quad\text{almost surely as $n \to \infty$.}
$$
Therefore,
$$
\Expect{} \norm{\mtx{Z}_n} \rightarrow
	\Expect \norm{ \sum_{k=1}^d (Q_k - 1) \, \mathbf{E}_{kk} }
	= \Expect{} \max\nolimits_k \abs{ Q_k - 1 }
	\approx \mathrm{const} \cdot \frac{\log d}{\log \log d}
	\quad\text{as $n \to \infty$.}
$$
In short, the bound we derived from~\eqref{eqn:matrix-moment-ineq} requires the logarithm
on the second term, but it is suboptimal by a $\log \log$ factor.
The upper matrix Chernoff inequality~\eqref{eqn:matrix-chernoff-upper-tail}
correctly predicts the appearance of the iterated logarithm in this example,
as does the matrix Bennett inequality.

The last two examples rely heavily on the commutativity of the summands
as well as the infinite divisibility of the normal and Poisson distributions.
As a consequence, it may appear that the logarithms
only appear in very special contexts.  In fact, many (but not all!) examples that arise in practice
do require the logarithms that appear in the matrix Bernstein inequality.  It is a subject of ongoing
research to obtain a simple criterion for deciding when the logarithms belong.

\section{Example: Matrix Approximation by Random Sampling} \label{sec:rdm-mtx-approx}

In applied mathematics, we often need to approximate a complicated target object
by a more structured object.  In some situations, we can solve this problem using
a beautiful probabilistic approach called \term{empirical approximation}.
The basic idea is to construct a ``simple'' random object whose expectation equals the target.
We obtain the approximation by averaging several independent copies of the simple random object.
As the number of terms in this average increases, the approximation becomes
more complex, but it represents the target more faithfully.  The challenge is to
quantify this tradeoff.  %

In particular, we often encounter problems where we need to approximate
a matrix by a more structured matrix.  For example, we may wish to find a sparse matrix
that is close to a given matrix, or we may need to construct a low-rank matrix that is
close to a given matrix.  Empirical approximation provides a mechanism for obtaining
these approximations.  The matrix Bernstein inequality offers a natural tool for
assessing the quality of the randomized approximation.

This section develops a general framework for empirical approximation of matrices.
Subsequent sections explain how this technique applies to specific examples from
the fields of randomized linear algebra and machine learning.

\subsection{Setup} \label{sec:matrix-approx-setup}

Let $\mtx{B}$ be a target matrix that we hope to approximate by a more
structured matrix.
To that end, let us represent the target as a sum of ``simple'' matrices:
\begin{equation} \label{eqn:matrix-sum-decomp}
\mtx{B} = \sum_{i=1}^N \mtx{B}_i.
\end{equation}
The idea is to identify summands with desirable properties that we want our approximation to inherit.
The examples in this chapter depend on decompositions of the form~\eqref{eqn:matrix-sum-decomp}.

Along with the decomposition~\eqref{eqn:matrix-sum-decomp}, we need a set of sampling probabilities:
\begin{equation} \label{eqn:matrix-sampling-probs}
\sum_{i=1}^N p_i = 1
\quad\text{and}\quad
p_i > 0
\quad\text{for $i = 1, \dots, N$.}
\end{equation}
We want to ascribe larger probabilities to ``more important'' summands.
Quantifying what ``important'' means is the most difficult aspect of randomized
matrix approximation.  Choosing the right sampling distribution for a specific
problem requires insight and ingenuity.

Given the data~\eqref{eqn:matrix-sum-decomp} and~\eqref{eqn:matrix-sampling-probs},
we may construct a ``simple'' random matrix $\mtx{R}$ by sampling:
\begin{equation} \label{eqn:matrix-sampling-model}
\mtx{R} = p_i^{-1} \mtx{B}_i
\quad\text{with probability $p_i$.}
\end{equation}
This construction ensures that $\mtx{R}$ is an unbiased estimator of the target:
$\Expect \mtx{R} = \mtx{B}$.
Even so, the random matrix $\mtx{R}$ offers a poor approximation of the target $\mtx{B}$
because it has a lot more structure.  To improve the quality of the approximation,
we average $n$ independent copies of the random matrix $\mtx{R}$.
We obtain an estimator of the form
$$
\bar{\mtx{R}}_n = \frac{1}{n} \sum_{k=1}^n \mtx{R}_k
\quad\text{where each $\mtx{R}_k$ is an independent copy of $\mtx{R}$.}
$$
By linearity of expectation, this estimator is also unbiased:
$
\Expect \bar{\mtx{R}}_n = \mtx{B}.
$
The approximation $\bar{\mtx{R}}_n$ remains structured when the number $n$ of terms
in the approximation is small as compared with the number $N$ of terms
in the decomposition~\eqref{eqn:matrix-sum-decomp}.

Our goal is to quantify the approximation error as a function
of the complexity $n$ of the approximation:
$$
\Expect \norm{ \smash{\bar{\mtx{R}}_n - \mtx{B}} } \leq \mathrm{error}(n).
$$
There is a tension between the total number $n$ of terms in the approximation
and the error $\mathrm{error}(n)$ the approximation incurs.  In applications,
it is essential to achieve the right balance.

\subsection{Error Estimate for Matrix Sampling Estimators}

We can obtain an error estimate for the approximation scheme described
in Section~\ref{sec:matrix-approx-setup} as an immediate corollary of the
matrix Bernstein inequality, Theorem~\ref{thm:matrix-bernstein-rect}.

\begin{cor}[Matrix Approximation by Random Sampling] \label{cor:matrix-approx-sampling}
Let $\mtx{B}$ be a fixed $d_1 \times d_2$ matrix.
Construct a $d_1 \times d_2$ random matrix $\mtx{R}$ that satisfies
$$
\Expect \mtx{R} = \mtx{B}
\quad\text{and}\quad
\norm{\mtx{R}} \leq L.
$$
Compute the per-sample second moment:
\begin{equation} \label{eqn:matrix-second-moment}
m_2(\mtx{R}) = \max\big\{ \norm{ \smash{\Expect(\mtx{RR}^\adj)} }, \ \norm{ \smash{\Expect( \mtx{R}^\adj \mtx{R})} } \big\}.
\end{equation}
Form the matrix sampling estimator
$$
\bar{\mtx{R}}_n = \frac{1}{n} \sum_{k=1}^n \mtx{R}_k
\quad\text{where each $\mtx{R}_k$ is an independent copy of $\mtx{R}$.}
$$
Then the estimator satisfies
\begin{equation} \label{eqn:matrix-approx-err-expect}
\Expect \norm{ \smash{\bar{\mtx{R}}_n - \mtx{B}} }
	\leq \sqrt{\frac{2 m_2(\mtx{R}) \log(d_1 + d_2)}{n}} + \frac{2L \log(d_1 + d_2)}{3n}.
\end{equation}
Furthermore, for all $t \geq 0$,
\begin{equation} \label{eqn:matrix-approx-err-tail}
\Prob{ \norm{ \smash{\bar{\mtx{R}}_n - \mtx{B}} } \geq t }
	\leq (d_1 + d_2) \exp\left( \frac{-nt^2/2}{m_2(\mtx{R}) + 2Lt/3} \right).
\end{equation}
\end{cor}

\begin{proof}
Since $\mtx{R}$ is an unbiased estimator of the target matrix $\mtx{B}$,
we can write
$$
\mtx{Z} =  \bar{\mtx{R}}_n - \mtx{B} = \frac{1}{n} \sum_{k=1}^n (\mtx{R}_k - \Expect \mtx{R})
	= \sum_{k=1}^n \mtx{S}_k.
$$
We have defined the summands $\mtx{S}_k = n^{-1} (\mtx{R}_k - \Expect \mtx{R})$.
These random matrices form an independent and identically distributed family,
and each $\mtx{S}_k$ has mean zero.

Now, each of the summands is subject to an upper bound:
$$
\norm{\mtx{S}_k} \leq \frac{1}{n} \left( \norm{ \mtx{R}_k } + \norm{ \Expect \mtx{R} } \right)
	\leq \frac{1}{n} \left( \norm{\mtx{R}_k} + \Expect \norm{\mtx{R}} \right)
	\leq \frac{2L}{n}.
$$
The first relation is the triangle inequality; the second is Jensen's inequality.  The last estimate
follows from our assumption that $\norm{\mtx{R}} \leq L$.

To control the matrix variance statistic $v(\mtx{Z})$, first note that
$$
v(\mtx{Z}) = \max\left\{ \norm{ \sum_{k=1}^n \Expect \big( \mtx{S}_k \mtx{S}_k^\adj  \big) }, \
	\norm{ \sum_{k=1}^n \Expect  \big( \mtx{S}_k^\adj \mtx{S}_k  \big) } \right\}
	= n \cdot \max\left\{ \norm{ \smash{\Expect( \mtx{S}_1 \mtx{S}_1^\adj )} }, \ 
	\norm{ \smash{\Expect( \mtx{S}_1^\adj \mtx{S}_1 )} } \right\}.
$$
The first identity follows from the expression~\eqref{eqn:matrix-bernstein-rect-var-calc}
for the matrix variance statistic, and the
second holds because the summands $\mtx{S}_k$ are identically distributed.  We may calculate that
\begin{align*}
\mtx{0} \psdle \Expect( \mtx{S}_1 \mtx{S}_1^\adj )
	&= n^{-2} \Expect \big[ (\mtx{R} - \Expect \mtx{R})(\mtx{R} - \Expect \mtx{R})^\adj \big] \\
	&= n^{-2} \big[ \Expect( \mtx{R} \mtx{R}^\adj) - (\Expect \mtx{R})(\Expect \mtx{R})^{\adj} \big]
	\psdle n^{-2} \Expect (\mtx{RR}^\adj).
\end{align*}
The first relation holds because the expectation of the random positive-semidefinite
matrix $\mtx{S}_1 \mtx{S}_1^\adj$ is positive semidefinite.  The first identity follows from the
definition of $\mtx{S}_1$ and the fact that $\mtx{R}_1$ has the same distribution as $\mtx{R}$.
The second identity is a direct calculation.
The last relation holds because $(\Expect \mtx{R})(\Expect \mtx{R})^\adj$ is positive semidefinite.
As a consequence,
$$
\norm{ \smash{\Expect( \mtx{S}_1 \mtx{S}_1^\adj )} } \leq \frac{1}{n^2} \norm{ \smash{\Expect (\mtx{RR}^\adj)} }.
$$
Likewise,
$$
\norm{ \smash{\Expect( \mtx{S}_1^\adj \mtx{S}_1 )} } \leq \frac{1}{n^2} \norm{ \smash{\Expect (\mtx{R}^\adj \mtx{R})} }.
$$
In summary,
$$
v(\mtx{Z}) \leq \frac{1}{n} \max\left\{ \norm{ \smash{\Expect (\mtx{RR}^\adj)} }, \ 
	\norm{ \smash{\Expect (\mtx{R}^\adj \mtx{R})} } \right\}
	= \frac{ m_2(\mtx{R})}{n}.
$$
The last line follows from the definition~\eqref{eqn:matrix-second-moment} of $m_2(\mtx{R})$.

We are prepared to apply the matrix Bernstein inequality, Theorem~\ref{thm:matrix-bernstein-rect},
to the random matrix $\mtx{Z} = \sum_k \mtx{S}_k$.  This operation results in the statement of the corollary.
\end{proof}

\subsection{Discussion}

One of the most common applications of the matrix Bernstein inequality is to
analyze empirical matrix approximations.  As a consequence,
Corollary~\ref{cor:matrix-approx-sampling} is one of the most useful
forms of the matrix Bernstein inequality.
Let us discuss some of the important aspects of this result.

\subsubsection{Understanding the Bound on the Approximation Error}

First, let us examine how many samples $n$ suffice to bring the approximation error bound
in Corollary~\ref{cor:matrix-approx-sampling} below a specified positive tolerance $\eps$.
Examining inequality~\eqref{eqn:matrix-approx-err-expect}, we find that
\begin{equation} \label{eqn:matrix-approx-sample-cost}
n \geq \frac{2 m_2(\mtx{R}) \log(d_1 + d_2)}{\eps^2} + \frac{2L \log(d_1 + d_2)}{3\eps}
\quad\text{implies}\quad
\Expect \norm{ \smash{\bar{\mtx{R}}_n - \mtx{B}} } \leq 2 \eps.
\end{equation}
Roughly, the number $n$ of samples should be on the scale of the per-sample second moment $m_2(\mtx{R})$
and the uniform upper bound $L$.

The bound~\eqref{eqn:matrix-approx-sample-cost} also reveals an unfortunate aspect of empirical
matrix approximation.  To make the tolerance $\eps$ small, the number $n$ of samples must
increase proportional with $\eps^{-2}$.  In other words, it takes many samples to achieve
a highly accurate approximation.  We cannot avoid this phenomenon, which ultimately is a
consequence of the central limit theorem.

On a more positive note, it is quite valuable that the error bounds~\eqref{eqn:matrix-approx-err-expect}
and~\eqref{eqn:matrix-approx-err-tail} involve the spectral norm.
This type of estimate simultaneously controls the error in every
linear function of the approximation:
$$
\norm{ \smash{\bar{\mtx{R}}_n - \mtx{B}} } \leq \eps
\quad\text{implies}\quad
\abs{ \trace( \bar{\mtx{R}}_n \mtx{C}) - \trace( \mtx{BC} ) } \leq \eps 
\quad\text{when $\pnorm{S_1}{\mtx{C}} \leq 1$.}
$$
The Schatten $1$-norm $\pnorm{S_1}{\cdot}$ is defined in~\eqref{eqn:schatten-1-norm}.
These bounds also control the error in each singular value $\sigma_j(\bar{\mtx{R}}_n)$
of the approximation:
$$
\norm{ \smash{\bar{\mtx{R}}_n - \mtx{B}} } \leq \eps
\quad\text{implies}\quad
\abs{ \sigma_j(\bar{\mtx{R}}_n) - \sigma_j(\mtx{B}) } \leq \eps
\quad\text{for each $j = 1, 2, 3, \dots, \min\{d_1, d_2\}$.}
$$
When there is a gap between two singular values of $\mtx{B}$, we can also obtain bounds for the
discrepancy between the associated singular vectors of $\bar{\mtx{R}}_n$ and $\mtx{B}$
using perturbation theory.

To construct a good sampling estimator $\mtx{R}$,
we ought to control both $m_2(\mtx{R})$ and $L$.
In practice, this demands considerable creativity.
This observation hints at the possibility of achieving
a bias--variance tradeoff when approximating $\mtx{B}$.
To do so, we can drop all of the ``unimportant''
terms in the representation~\eqref{eqn:matrix-sum-decomp},
i.e., those whose sampling probabilities are small.
Then we construct a random approximation $\mtx{R}$ only for the
``important'' terms that remain.  Properly executed,
this process may decrease both the per-sample second moment
$m_2(\mtx{R})$ and the upper bound $L$.
The idea is analogous with shrinkage in statistical estimation.

\subsubsection{A General Sampling Model}

Corollary~\ref{cor:matrix-approx-sampling} extends beyond the
sampling model based on the finite expansion~\eqref{eqn:matrix-sum-decomp}.
Indeed, we can consider a more general decomposition of the target matrix $\mtx{B}$:
$$
\mtx{B} = \int_{\Omega} \mtx{B}(\omega) \idiff{\mu}(\omega)
$$
where $\mu$ is a probability measure on a sample space $\Omega$.  As before, the idea is
to represent the target matrix $\mtx{B}$ as an average of ``simple'' matrices $\mtx{B}(\omega)$.
The main difference is that the family of simple matrices may now be infinite.
In this setting, we construct the random approximation $\mtx{R}$ so that
$$
\Prob{ \mtx{R} \in E } = \mu \{ \omega : \mtx{B}(\omega) \in E \}
\quad\text{for $E \subset \mathbb{M}^{d_1 \times d_2}$}
$$
In particular, it follows that
$$
\Expect \mtx{R} = \mtx{B}
\quad\text{and}\quad
\norm{\mtx{R}} \leq \sup_{\omega \in \Omega} \norm{\mtx{B}(\omega)}.
$$
As we will discuss, this abstraction is important for applications in machine learning.

\subsubsection{Suboptimality of Sampling Estimators}

Another fundamental point about sampling estimators is that they are usually suboptimal.
In other words, the matrix sampling estimator may incur an error substantially worse
than the error in the best structured approximation of the target matrix.

To see why, let us consider a simple form of low-rank approximation by random sampling.
The method here does not have practical value, but it highlights the reason that sampling
estimators usually do not achieve ideal results.
Suppose that $\mtx{B}$ has singular value decomposition
$$
\mtx{B} = \sum_{i=1}^{N} \sigma_i \vct{u}_i \vct{v}_i^\adj
\quad\text{where}\quad
\sum_{i=1}^{N} \sigma_i = 1
\quad\text{and}\quad
N = \min\{d_1, d_2\}.
$$
Given the SVD, we can construct a random rank-one approximation $\mtx{R}$ of the form
$$
\mtx{R} = \vct{u}_i \vct{v}_i^\adj
\quad\text{with probability}\quad
\sigma_i.
$$
Per Corollary~\ref{cor:matrix-approx-sampling}, the error in the associated sampling estimator
$\bar{\mtx{R}}_n$ of $\mtx{B}$ satisfies
$$
\norm{ \smash{\bar{\mtx{R}}_n - \mtx{B}} }
	\leq \sqrt{ \frac{ 2 \log(d_1 + d_2) }{ n } } + \frac{2 \log(d_1 + d_2)}{n}
$$
On the other hand, a best rank-$n$ approximation of $\mtx{B}$
takes the form $\mtx{B}_n = \sum_{j=1}^n \sigma_j \vct{u}_j \vct{v}_j^\adj$, and it incurs error
$$
\norm{ \mtx{B}_n - \mtx{B} } = \sigma_{n+1} \leq \frac{1}{n+1}.
$$
The second relation is Markov's inequality, which provides an accurate estimate
only when the singular values $\sigma_1, \dots, \sigma_{n+1}$ are comparable.  In
that case, the sampling estimator arrives within a logarithmic factor of the optimal error.
But there are many matrices whose singular values decay quickly, so that $\sigma_{n+1} \ll (n+1)^{-1}$.
In the latter situation, the error in the sampling estimator is much worse than the optimal error.

\subsubsection{Warning: Frobenius-Norm Bounds}

We often encounter papers that develop Frobenius-norm error bounds for matrix approximations,
perhaps because the analysis is more elementary.  But one must recognize that
Frobenius-norm error bounds are not acceptable in most cases of practical interest:

\begin{quotation}
\textbf{Frobenius-norm error bounds are typically vacuous.}
\end{quotation}

\noindent
In particular, this phenomenon occurs in data analysis whenever we try to approximate a matrix
that contains white or pink noise.

To illustrate this point, let us consider the ubiquitous problem of approximating a low-rank
matrix corrupted by additive white Gaussian noise:
\begin{equation} \label{eqn:matrix-approx-bad-ex}
\mtx{B} = \vct{xx}^\adj + \alpha \mtx{E} \in \mathbb{M}_d.
\quad\text{where}\quad
\normsq{\vct{x}} = 1.
\end{equation}
The desired approximation of the matrix $\mtx{B}$ is the rank-one matrix
$\mtx{B}_{\mathrm{opt}} = \vct{xx}^\adj$.
For modeling purposes, %
we assume that $\mtx{E}$ has independent $\normal(0, d^{-1})$ entries.  As a consequence,
$$
\norm{ \mtx{E} } \approx 2
\quad\text{and}\quad
\fnorm{ \mtx{E} } \approx \sqrt{d}.
$$
Now, the spectral-norm error in the desired approximation satisfies
$$
\norm{\smash{\mtx{B}_{\mathrm{opt}} - \mtx{B}}} = \alpha \norm{ \mtx{E} } \approx 2 \alpha.
$$
On the other hand, the Frobenius-norm error in the desired approximation satisfies
$$
\fnorm{\smash{\mtx{B}_{\mathrm{opt}} - \mtx{B}} } = \alpha \fnorm{ \mtx{E} } \approx \alpha \sqrt{ d }.
$$
We see that the Frobenius-norm error can be quite large, even when we find the required approximation.

Here is another way to look at the same fact.  Suppose we construct an approximation
$\widehat{\mtx{B}}$ of the matrix $\mtx{B}$ from~\eqref{eqn:matrix-approx-bad-ex} whose
Frobenius-norm error is comparable with the optimal error:
$$
\fnorm{ \smash{\widehat{\mtx{B}} - \mtx{B}} } \leq \eps \sqrt{ d }.
$$
There is no reason for the approximation $\widehat{\mtx{B}}$ to have any relationship with
the desired approximation $\mtx{B}_{\mathrm{opt}}$.  For example, the approximation
$\widehat{\mtx{B}} = \alpha{\mtx{E}}$ satisfies this error bound with $\eps = d^{-1/2}$
even though $\widehat{\mtx{B}}$ consists only of noise.

\section{Application: Randomized Sparsification of a Matrix} \label{sec:rdm-sparse}

Many tasks in data analysis involve large, dense matrices
that contain a lot of redundant information.  For example,
an experiment that tabulates many variables about a large
number of subjects typically results in a low-rank data
matrix because subjects are often similar with each other.
Many questions that we pose about these data
matrices can be addressed by spectral computations.
In particular, factor analysis involves a singular value decomposition.

When the data matrix is approximately low rank, it
has fewer degrees of freedom than its ambient dimension.
Therefore, we can construct a simpler approximation that
still captures most of the information in the matrix.
One method for finding this approximation is to replace the
dense target matrix by a sparse matrix that is close in
spectral-norm distance.  An elegant way to identify this
sparse proxy is to randomly select a small number of 
entries from the original matrix to retain.
This is a type of empirical approximation.

Sparsification has several potential advantages.  First,
it is considerably less expensive to store a sparse matrix
than a dense matrix.  Second, many algorithms for spectral
computation operate more efficiently on sparse matrices.

In this section, we examine a very recent approach to randomized
sparsification due to Kundu \& Drineas~\cite{KD14:Note-Randomized}.
The analysis is an immediate consequence of Corollary~\ref{cor:matrix-approx-sampling}.
See the notes at the end of the chapter for history and references.

\subsection{Problem Formulation \& Randomized Algorithm}

Let $\mtx{B}$ be a fixed $d_1 \times d_2$ complex matrix.  The sparsification problem
requires us to find a sparse matrix $\widehat{B}$ that has small distance from $\mtx{B}$
with respect to the spectral norm.
We can achieve this goal using an empirical approximation strategy.

First, let us express the target matrix as a sum of its entries:
$$
\mtx{B} = \sum_{i = 1}^{d_1} \sum_{j=1}^{d_2} b_{ij} \, \mathbf{E}_{ij}.
$$
Introduce sampling probabilities
\begin{equation} \label{eqn:sparsify-prob}
p_{ij} = \frac{1}{2} \left[ \frac{\abssq{\smash{b_{ij}}}}{\fnormsq{\mtx{B}}} + \frac{\abs{\smash{b_{ij}}}}{\pnorm{\ell_1}{\mtx{B}}} \right]
\quad\text{for $i = 1, \dots, d_1$ and $j = 1, \dots, d_2$.}
\end{equation}
The Frobenius norm is defined in~\eqref{eqn:frobenius-norm},
and the entrywise $\ell_1$ norm is defined in~\eqref{eqn:matrix-l1-norm}.
It is easy to check that the numbers $p_{ij}$ form a probability distribution.
Let us emphasize that the non-obvious form of the
distribution~\eqref{eqn:sparsify-prob} represents a decade of research.

Now, we introduce a $d_1 \times d_2$ random matrix $\mtx{R}$ that has exactly one nonzero entry:
$$
\mtx{R} = \frac{1}{p_{ij}} \cdot b_{ij} \, \mathbf{E}_{ij}
\quad\text{with probability $p_{ij}$.}
$$
We use the convention that $0/0 = 0$ so that we do not need to treat zero entries separately.
It is immediate that
$$
\Expect \mtx{R} = \sum_{i=1}^{d_1} \sum_{j=1}^{d_2} \frac{1}{p_{ij}} \cdot b_{ij}\, \mathbf{E}_{ij} \cdot p_{ij}
	= \sum_{i=1}^{d_1} \sum_{j=1}^{d_2} b_{ij} \, \mathbf{E}_{ij} = \mtx{B}.
$$
Therefore, $\mtx{R}$ is an unbiased estimate of $\mtx{B}$.

Although the expectation of $\mtx{R}$ is correct, its variance is quite high.  Indeed, $\mtx{R}$
has only one nonzero entry, while $\mtx{B}$ typically has many nonzero entries.  To reduce the
variance, we combine several independent copies of the simple estimator:
$$
\bar{\mtx{R}}_n = \frac{1}{n} \sum_{k=1}^n \mtx{R}_k
\quad\text{where each $\mtx{R}_k$ is an independent copy of $\mtx{R}$.}
$$
By linearity of expectation, $\Expect \bar{\mtx{R}}_n = \mtx{B}$.
Therefore, the matrix $\bar{\mtx{R}}_n$ has at most $n$ nonzero entries,
and its also provides an unbiased estimate of the target.  The challenge
is to quantify the error $\norm{ \smash{\bar{\mtx{R}}_n - \mtx{B}} }$ as
a function of the sparsity level $n$.

\subsection{Performance of Randomized Sparsification}

The randomized sparsification method is clearly a type of empirical approximation,
so we can use Corollary~\ref{cor:matrix-approx-sampling} to perform the analysis.
We will establish the following error bound.
\begin{equation} \label{eqn:sparsify-bound}
\Expect \norm{ \smash{\bar{\mtx{R}}_n - \mtx{B}} }
	\leq \sqrt{ \frac{4 \fnormsq{\mtx{B}} \cdot \max\{ d_1, d_2 \} \log(d_1 + d_2) }{n} }
	+ \frac{4 \norm{\mtx{B}}_{\ell_1} \log(d_1 + d_2)}{3n}.
\end{equation}
The short proof of~\eqref{eqn:sparsify-bound} appears below in
Section~\ref{sec:sparsification-analysis}.

Let us explain the content of the estimate~\eqref{eqn:sparsify-bound}.
First, the bound~\eqref{eqn:matrix-l1-l2} allows us to replace the $\ell_1$ norm by the Frobenius norm:
$$
\pnorm{\ell_1}{\mtx{B}} \leq \sqrt{d_1 d_2} \cdot \fnorm{\mtx{B}} \leq \max\{ d_1, \ d_2 \} \cdot \fnorm{\mtx{B}}.
$$
Placing the error~\eqref{eqn:sparsify-bound} on a relative scale, we see that
$$
\begin{aligned}
\frac{\Expect \norm{ \smash{\bar{\mtx{R}}_n - \mtx{B} }}}{\norm{\mtx{B}}}
	\leq \frac{\fnorm{\mtx{B}}}{\norm{\mtx{B}}} \cdot \left[ \sqrt{\frac{4 \max\{d_1, \ d_2\} \log(d_1 + d_2)}{n}}
		+ \frac{4 \max\{d_1, \ d_2\} \log(d_1+d_2)}{3n} \right] %
\end{aligned}
$$
The stable rank $\strank(\mtx{B})$, defined in~\eqref{eqn:stable-rank},
emerges naturally as a quantity of interest.

Now, suppose that the sparsity level $n$ satisfies
$$
n \geq \eps^{-2} \cdot \strank(\mtx{B}) \cdot \max\{d_1, \ d_2\} \log(d_1 + d_2)
$$
where the tolerance $\eps \in (0, 1]$.  We determine that
$$
\frac{\Expect \norm{ \smash{\bar{\mtx{R}}_n - \mtx{B}} }}{\norm{\mtx{B}}}
	\leq 2\eps + \frac{4}{3} \cdot \frac{\eps^2}{\sqrt{\strank(\mtx{B})}}.
$$
Since the stable rank always exceeds one and we have assumed that $\eps \leq 1$,
this estimate implies that
$$
\frac{\Expect \norm{ \smash{\bar{\mtx{R}}_n - \mtx{B}} }}{\norm{\mtx{B}}}
	\leq 4 \eps.
$$
We discover that it is possible to replace the matrix $\mtx{B}$
by a matrix with at most $n$ nonzero entries while achieving a small
relative error in the spectral norm.  When
$\strank(\mtx{B}) \ll \min\{ d_1, \ d_2 \}$,
we can achieve a dramatic reduction in the
number of nonzero entries needed to carry
the spectral information in the matrix $\mtx{B}$.

\subsection{Analysis of Randomized Sparsification}
\label{sec:sparsification-analysis}

Let us proceed with the analysis of randomized sparsification.
To apply Corollary~\ref{cor:matrix-approx-sampling},
we need to obtain bounds for the per-sample variance $m_2(\mtx{R})$
and the uniform upper bound $L$.
The key to both calculations is to obtain appropriate \emph{lower} bounds
on the sampling probabilities $p_{ij}$.   Indeed,
\begin{equation} \label{eqn:sparsify-prob-lower}
p_{ij} \geq \frac{1}{2} \cdot \frac{\abs{\smash{b_{ij}}}}{\pnorm{\ell_1}{\mtx{B}}}
\quad\text{and}\quad
p_{ij} \geq \frac{1}{2} \cdot \frac{\abssq{\smash{b_{ij}}}}{\fnormsq{\mtx{B}}}.
\end{equation}
Each estimate follows by neglecting one term in~\eqref{eqn:sparsify-prob-lower}.

First, we turn to the uniform bound on the random matrix $\mtx{R}$.  We have
$$
\norm{ \mtx{R} } \leq \max_{ij} \norm{ \smash{p_{ij}^{-1} b_{ij} \mathbf{E}_{ij}} }
	= \max_{ij} \frac{1}{p_{ij}} \cdot \abs{\smash{b_{ij}}}
	\leq 2 \pnorm{\ell_1}{\mtx{B}}.
$$
The last inequality depends on the first bound in~\eqref{eqn:sparsify-prob-lower}.
Therefore, we may take $L = 2 \pnorm{\ell_1}{\mtx{B}}$.

Second, we turn to the computation of the per-sample second moment $m_2(\mtx{R})$.  We have
$$
\begin{aligned}
\Expect( \mtx{RR}^\adj )
	&= \sum_{i=1}^{d_1} \sum_{j=1}^{d_2} \frac{1}{p_{ij}^2} \cdot (b_{ij} \mathbf{E}_{ij})(b_{ij} \mathbf{E}_{ij})^\adj p_{ij} \\
	&= \sum_{i=1}^{d_1} \sum_{j=1}^{d_2} \frac{\abssq{\smash{b_{ij}}}}{p_{ij}} \cdot \mathbf{E}_{ii} \\
	&\psdle 2 \fnormsq{\mtx{B}} \sum_{i=1}^{d_1} \sum_{j=1}^{d_2} \mathbf{E}_{ii}
	= 2 d_2 \fnormsq{\mtx{B}} \cdot \Id_{d_1}.
\end{aligned}
$$
The semidefinite inequality holds because each matrix $\abssq{\smash{b_{ij}}} \mathbf{E}_{ii}$ is positive semidefinite
and because of the second bound in~\eqref{eqn:sparsify-prob-lower}.  Similarly,
$$
\Expect( \mtx{R}^\adj \mtx{R} ) \psdle 2d_1 \fnormsq{\mtx{B}} \cdot \Id_{d_2}.
$$
In summary,
$$
m_2(\mtx{R}) = \max\big\{ \norm{\smash{\Expect( \mtx{RR}^\adj )}}, \ \norm{\smash{ \Expect( \mtx{R}^\adj \mtx{R} ) }} \big\}
	\leq 2 \max\{ d_1, d_2 \}.
$$
This is the required estimate for the per-sample second moment.

Finally, to reach the advertised error bound~\eqref{eqn:sparsify-bound},
we invoke Corollary~\ref{cor:matrix-approx-sampling} with the parameters
$L = \pnorm{\ell_1}{\mtx{B}}$ and $m_2(\mtx{R}) \leq 2 \max\{ d_1, d_2 \}$.

\section{Application: Randomized Matrix Multiplication} \label{sec:rdm-mtx-mult}

Numerical linear algebra (NLA) is a well-established and important part of computer
science. Some of the basic problems in this area include multiplying matrices, solving
linear systems, computing eigenvalues and eigenvectors, and solving linear least-squares
problems. Historically, the NLA community has focused on developing highly accurate
deterministic methods that require as few floating-point operations as possible. Unfortunately,
contemporary applications can strain standard NLA methods because problems
have continued to become larger. Furthermore, on modern computer architectures,
computational costs depend heavily on communication and other resources that the
standard algorithms do not manage very well.

In response to these challenges, researchers have started to develop randomized
algorithms for core problems in NLA. In contrast to the classical algorithms, these new
methods make random choices during execution to achieve computational efficiencies.
These randomized algorithms can also be useful for large problems or for modern
computer architectures. On the other hand, randomized methods can fail with some
probability, and in some cases they are less accurate than their classical competitors.

Matrix concentration inequalities are one of the key tools used to design and analyze
randomized algorithms for NLA problems.  In this section, we will describe a randomized
method for matrix multiplication developed by
Magen \& Zouzias~\cite{MZ11:Low-Rank-Matrix-Valued,Zou12:Randomized-Primitives}.
We will analyze this algorithm using Corollary~\ref{cor:matrix-approx-sampling}.
Turn to the notes at the end of the chapter for more information about the history.

\subsection{Problem Formulation \& Randomized Algorithm}

One of the basic tasks in numerical linear algebra is to multiply two matrices with compatible dimensions.
Suppose that $\mtx{B}$ is a $d_1 \times N$ complex matrix and that $\mtx{C}$ is an $N \times d_2$ complex matrix,
and we wish to compute the product $\mtx{BC}$.  The straightforward algorithm forms the product entry by entry:
\begin{equation} \label{eqn:matrix-mult-ip}
(\mtx{BC})_{ik} = \sum_{j=1}^N b_{ij} c_{jk}
\quad\text{for each $i = 1, \dots, d_1$ and $k = 1, \dots, d_2$.}
\end{equation}
This approach takes $\bigO(N \cdot d_1 d_2)$ arithmetic operations.  There are algorithms, such as Strassen's divide-and-conquer method, that can reduce the cost, but these approaches are not considered practical for most applications.

Suppose that the inner dimension $N$ is substantially larger than the outer dimensions $d_1$ and $d_2$.
In this setting, both matrices $\mtx{B}$ and $\mtx{C}$ are rank-deficient, so the columns of $\mtx{B}$
contain a lot of linear dependencies, as do the rows of $\mtx{C}$.
As a consequence, a random sample of columns from $\mtx{B}$
(or rows from $\mtx{C}$) can be used as a proxy for the full matrix.
Formally, the key to this approach is to view
the matrix product as a sum of outer products:
\begin{equation} \label{eqn:matrix-mult-outer-product}
\mtx{BC} = \sum_{j=1}^N \vct{b}_{:j} \vct{c}_{j:}.
\end{equation}
As usual, $\vct{b}_{:j}$ denotes the $j$th column of $\mtx{B}$,
while $\vct{c}_{j:}$ denotes the $j$th row of $\mtx{C}$.
We can approximate this sum using the empirical method.

To develop an algorithm, the first step is to construct a simple
random matrix $\mtx{R}$ that provides an unbiased
estimate for the matrix product.  To that end, we pick
a random index and form a rank-one matrix from
the associated columns of $\mtx{B}$ and row of $\mtx{C}$.
More precisely, define
\begin{equation} \label{eqn:matrix-mult-prob}
p_j = \frac{\normsq{\smash{\vct{b}_{:j}}} + \normsq{\smash{\vct{c}_{j:}}}}{\fnormsq{\mtx{B}} + \fnormsq{\mtx{C}}}
	\quad\text{for $j = 1, 2, 3, \dots, N$.}
\end{equation}
The Frobenius norm is defined in~\eqref{eqn:frobenius-norm}.
Using the properties of the norms, we can easily check
that $(p_1, p_2, p_3, \dots, p_N)$ forms a bonafide probability
distribution.  The cost of computing these probabilities
is at most $O(N \cdot (d_1 + d_2))$ arithmetic operations,
which is much smaller than the cost of forming the product
$\mtx{BC}$ when $d_1$ and $d_2$ are large.

We now define a $d_1 \times d_2$ random matrix $\mtx{R}$ by
the expression
$$
\mtx{R} = \frac{1}{p_j} \cdot \vct{b}_{:j} \vct{c}_{j:}
\quad\text{with probability $p_j$.}
$$
We use the convention that $0/0 = 0$ so we do not have to treat zero rows and
columns separately.  It is straightforward to compute the expectation
of $\mtx{R}$:
$$
\Expect \mtx{R} = \sum_{j=1}^N \frac{1}{p_j} \cdot \vct{b}_{:j} \vct{c}_{j:} \cdot p_j
	= \sum_{j=1}^N \vct{b}_{:j} \vct{c}_{j:}
	= \mtx{BC}.
$$
As required, $\mtx{R}$ is an unbiased estimator for the product $\mtx{BC}$.

Although the expectation of $\mtx{R}$ is correct, its variance is quite high.
Indeed, $\mtx{R}$ has rank one, while the rank of $\mtx{BC}$ is usually larger!
To reduce the variance, we combine several independent copies of the simple estimator:
\begin{equation} \label{eqn:matrix-mult-approx}
\bar{\mtx{R}}_n = \frac{1}{n} \sum_{k=1}^n \mtx{R}_k
\quad\text{where}\quad
\text{each $\mtx{R}_k$ is an independent copy of $\mtx{R}$.}
\end{equation}
By linearity of expectation, $\Expect \bar{\mtx{R}}_n = \mtx{BC}$, so we
imagine that $\bar{\mtx{R}}_n$ approximates the product well.

To see whether this heuristic holds true, we need to understand how
the error $\Expect \norm{ \smash{\bar{\mtx{R}}_n - \mtx{BC}} }$ depends on the number $n$ of samples.
It costs $O(n \cdot d_1 d_2)$ floating-point operations to determine
all the entries of $\bar{\mtx{R}}_n$.  Therefore, when the number $n$ of samples
is much smaller than the inner dimension $N$ of the matrices, we can achieve
significant economies over the na{\"i}ve matrix multiplication algorithm.

In fact, it requires no computation beyond sampling the row/column indices
to express $\bar{\mtx{R}}_n$ in the form~\eqref{eqn:matrix-mult-approx}.
This approach gives an inexpensive way to represent the product approximately.

\subsection{Performance of Randomized Matrix Multiplication}

To simplify our presentation, we will assume that both matrices have been scaled
so that their spectral norms are equal to one:
$$
\norm{\mtx{B}} = \norm{\mtx{C}} = 1.
$$
It is relatively inexpensive to compute the spectral norm of a matrix accurately, so this
preprocessing step is reasonable.

Let $\textsf{asr} = \half (\strank(\mtx{B}) + \strank(\mtx{C}))$ be the
average stable rank of the two factors; see~\eqref{eqn:stable-rank}
for the definition of the stable rank.  In~\S\ref{sec:rdm-mtx-mult-anal},
we will prove that
\begin{equation} \label{eqn:matrix-mult-error}
\Expect \norm{\smash{\bar{\mtx{R}}_n - \mtx{BC}}}
	\leq \sqrt{ \frac{4 \cdot \textsf{asr} \cdot \log(d_1 + d_2)}{n} } + \frac{2 \cdot \textsf{asr} \cdot \log(d_1 + d_2)}{3n}.
\end{equation}
To appreciate what this estimate means, suppose that the number $n$ of samples satisfies
$$
n \geq \eps^{-2} \cdot \textsf{asr} \cdot \log(d_1 + d_2)
$$
where $\eps$ is a positive tolerance.  Then we obtain a relative error bound
for the randomized matrix multiplication method
$$
\frac{\Expect \norm{\smash{\bar{\mtx{R}}_n - \mtx{BC}}}}{\norm{\mtx{B}} \norm{\mtx{C}}}
	\leq 2 \eps + \frac{2}{3} \eps^2.
$$
This expression depends on the normalization of $\mtx{B}$ and $\mtx{C}$.  The computational
cost of forming the approximation is
$$
O( \eps^{-2} \cdot \textsf{asr} \cdot d_1 d_2 \log(d_1 + d_2) )
\quad\text{arithmetic operations}.
$$
In other words, when the average stable rank $\textsf{asr}$ is substantially smaller
than the inner dimension $N$ of the two matrices $\mtx{B}$ and $\mtx{C}$,
the random estimate $\bar{\mtx{R}}_n$ for the product $\mtx{BC}$ achieves
a small error relative to the scale of the factors.

\subsection{Analysis of Randomized Matrix Multiplication}
\label{sec:rdm-mtx-mult-anal}

The randomized matrix multiplication method is just a specific example
of empirical approximation, and the error bound~\eqref{eqn:matrix-mult-error}
is an immediate consequence of Corollary~\ref{cor:matrix-approx-sampling}.

To pursue this approach, we need to establish a uniform bound on the norm
of the estimator $\mtx{R}$ for the product.  Observe that
$$
\norm{\mtx{R}} \leq \max\nolimits_j \norm{ \smash{p_j^{-1} \vct{b}_{:j} \vct{c}_{j:}} }
	= \max\nolimits_j \frac{\norm{ \smash{ \vct{b}_{:j}} } \norm{ \smash{\vct{c}_{j:}}}}{p_j}.
$$
To obtain a bound, recall the value~\eqref{eqn:matrix-mult-prob} of the probability $p_j$,
and invoke the inequality between geometric and arithmetic means:
$$
\norm{\mtx{R}} \leq \big(\fnormsq{\mtx{B}} + \fnormsq{\mtx{C}} \big)
	\cdot \max\nolimits_j \frac{\norm{ \smash{ \vct{b}_{:j}} } \norm{ \smash{\vct{c}_{j:}}}}
	{\normsq{\smash{\vct{b}_{:j}}} + \normsq{\smash{\vct{c}_{j:}}}}
	\leq \frac{1}{2} \big(\fnormsq{\mtx{B}} + \fnormsq{\mtx{C}} \big).
$$
Since the matrices $\mtx{B}$ and $\mtx{C}$ have unit spectral norm,
we can express this inequality in terms of the average stable rank:
$$
\norm{\mtx{R}} \leq \frac{1}{2} \big( \strank(\mtx{B}) + \strank(\mtx{C}) \big) = \textsf{asr}.
$$
This is the exactly kind of bound that we need.

Next, we need an estimate for the per-sample second moment $m_2(\mtx{R})$.  By
direct calculation,
$$
\begin{aligned}
\Expect( \mtx{RR}^\adj )
	&= \sum_{j=1}^N \frac{1}{p_j^2} \cdot (\vct{b}_{:j} \vct{c}_{j:})(\vct{b}_{:j} \vct{c}_{j:})^\adj \cdot p_j \\
	&= \big( \fnormsq{\mtx{B}} + \fnormsq{\mtx{C}} \big) \cdot
	\sum_{j=1}^n \frac{\normsq{\smash{\vct{c}_{j:}}}}{\normsq{\smash{\vct{b}_{:j}}} + \normsq{\smash{\vct{c}_{j:}}}}
	\cdot \vct{b}_{:j} \vct{b}_{:j}^\adj \\
	&\psdle \big( \fnormsq{\mtx{B}} + \fnormsq{\mtx{C}} \big) \cdot \mtx{BB}^\adj.
\end{aligned}
$$
The semidefinite relation holds because each fraction lies between zero and one,
and each matrix $\vct{b}_{:j} \vct{b}_{:j}^\adj$ is positive semidefinite.  Therefore,
increasing the fraction to one only increases in the matrix in the semidefinite order.
Similarly,
$$
\Expect( \mtx{RR}^\adj )
	\psdle \big( \fnormsq{\mtx{B}} + \fnormsq{\mtx{C}} \big) \cdot \mtx{C}^\adj \mtx{C}.
$$
In summary,
$$
\begin{aligned}
m_2(\mtx{R}) &= \max\big\{ \norm{ \smash{\Expect(\mtx{RR}^\adj)} },
	\ \norm{ \smash{\Expect(\mtx{R}^\adj \mtx{R})} } \\
	&\leq \big( \fnormsq{\mtx{B}} + \fnormsq{\mtx{C}} \big)
	\cdot \max\big\{ \norm{ \smash{\mtx{BB}^\adj}}, \ \norm{ \smash{\mtx{C}^\adj \mtx{C}}} \big\} \\
	&= \big( \fnormsq{\mtx{B}} + \fnormsq{\mtx{C}} \big) \\
	&= 2 \cdot \textsf{asr}.
\end{aligned}
$$
The penultimate line depends on the identity~\eqref{eqn:spectral-norm-square}
and our assumption that both matrices $\mtx{B}$ and $\mtx{C}$ have norm one.

Finally, to reach the stated estimate~\eqref{eqn:matrix-mult-error},
we apply Corollary~\ref{cor:matrix-approx-sampling} with the parameters
$L = \textsf{asr}$ and $m_2(\mtx{R}) \leq 2 \cdot \textsf{asr}$.

\section{Application: Random Features}
\label{sec:rdm-features}

As a final application of empirical matrix approximation,
let us discuss a contemporary idea from machine learning
called \term{random features}.  Although this technique
may appear more sophisticated than randomized sparsification
or randomized matrix multiplication, it depends on
exactly the same principles.
Random feature maps were proposed by Ali Rahimi
and Ben Recht~\cite{RR07:Random-Features}.
The analysis in this section is due to David Lopez-Paz
et al.~\cite{LSS+14:Randomized-Nonlinear}.

\subsection{Kernel Matrices}

Let $\coll{X}$ be a set.  We think about the elements of the set $\coll{X}$
as (potential) observations that we would like to use to perform learning
and inference tasks.  Let us introduce a bounded measure $\Phi$ of similarity
between pairs of points in the set:
$$
\Phi : \coll{X} \times \coll{X} \to [-1, +1].
$$
The similarity measure $\Phi$ is often called a \term{kernel}.
We assume that the kernel returns the value $+1$ when its arguments
are identical, and it returns smaller values when its arguments are dissimilar.
We also assume that the kernel is symmetric;
that is, $\Phi(\vct{x}, \vct{y}) = \Phi(\vct{y}, \vct{x})$
for all arguments $\vct{x}, \vct{y} \in \coll{X}$.

A simple example of a kernel is the angular similarity between a pair of points
in a Euclidean space:
\begin{equation} \label{eqn:angular-similarity}
\Phi( \vct{x}, \vct{y} ) = \frac{2}{\pi} \arcsin \frac{ \ip{ \vct{x} }{ \smash{\vct{y}}} }{\norm{\vct{x}}{\norm{\smash{\vct{y}}}}}
= 1 - \frac{2 \measuredangle (\vct{x}, \vct{y})}{ \pi }
\quad\text{for $\vct{x}, \vct{y} \in \R^d$.}
\end{equation}
We write $\measuredangle( \cdot, \cdot )$ for the planar angle between two vectors, measured
in radians.  As usual, we instate the convention that $0/0=0$.  See Figure~\ref{fig:angle-sim-new}
for an illustration.  %

Suppose that $\vct{x}_1, \dots, \vct{x}_N \in \coll{X}$ are observations.
The kernel matrix $\mtx{G} = [ g_{ij} ] \in \mathbb{M}_N$ just tabulates
the values of the kernel function for each pair of data points:
$$
g_{ij} = \Phi( \vct{x}_i, \vct{x}_j )
\quad\text{for $i, j = 1, \dots, N$.}
$$
It may be helpful to think about the kernel matrix $\mtx{G}$ as a generalization of the
Gram matrix of a family of points in a Euclidean space.
We say that the kernel $\Phi$ is \term{positive definite} if the
kernel matrix $\mtx{G}$ is positive semidefinite for any choice of observations
$\{ \vct{x}_i \} \subset \coll{X}$.  We will be concerned only with
positive-definite kernels in this discussion.

In the Euclidean setting, there are statistical learning methods that only
require the inner product between each pair of observations.  These algorithms can
be extended to the kernel setting by replacing each inner product with a kernel
evaluation.  As a consequence, kernel matrices can be used for classification,
regression, and feature selection.
In these applications, kernels are advantageous because they work outside
the Euclidean domain, and they allow task-specific measures of similarity.
This idea, sometimes called the \term{kernel trick}, is one of the major
insights in modern machine learning.

A significant challenge for algorithms based on kernels is that the kernel matrix is big.
Indeed, $\mtx{G}$ contains $O( N^2 )$ entries, where $N$ is the number
of data points.  Furthermore, the cost of constructing the kernel matrix
is $O( d N^2 )$ where $d$ is the number of parameters required to specify
a point in the universe $\coll{X}$.

Nevertheless, there is an opportunity.  Large data sets tend
to be redundant, so the kernel matrix also tends to be redundant.
This manifests in the kernel matrix being close to a low-rank matrix.
As a consequence, we may try to replace the kernel matrix by a
low-rank proxy.  For some similarity measures, we can accomplish
this task using empirical approximation.

\subsection{Random Features and Low-Rank Approximation of the Kernel Matrix}

In certain cases, a positive-definite kernel can be written as an expectation,
and we can take advantage of this representation to construct an empirical
approximation of the kernel matrix.  Let us begin with the general construction,
and then we will present a few examples in Section~\ref{sec:random-feature-examples}.

Let $\coll{W}$ be a sample space equipped with a sigma-algebra and a probability measure $\mu$.
Introduce a bounded \term{feature map}:
$$
\psi : \coll{X} \times \coll{W} \to [-b, +b]
\quad\text{where $b \geq 0$.}
$$
Consider a random variable $\vct{w}$ taking values in $\coll{W}$ 
and distributed according to the measure $\mu$.  We assume that
this random variable satisfies the \term{reproducing property}
\begin{equation} \label{eqn:reproducing-property}
\Phi( \vct{x}, \vct{y} )
	= \Expect_{\vct{w}} \big[ \psi( \vct{x}; \vct{w} ) \cdot \psi( \vct{y}; \vct{w} ) \big]
	\quad\text{for all $\vct{x}, \vct{y} \in \coll{X}$.}
\end{equation}
The pair $(\psi, \vct{w})$ is called a \term{random feature map} for the kernel $\Phi$.

We want to approximate the kernel matrix with a set
$\{ \vct{x}_1, \dots, \vct{x}_N \} \subset \coll{X}$
of observations.  To do so, we draw a random vector
$\vct{w} \in \coll{W}$ distributed according to $\mu$.
Form a random vector $\vct{z} \in \R^N$
by applying the feature map to each data point with the
\emph{same} choice of the random vector $\vct{w}$.  That is,
$$
\vct{z} = \begin{bmatrix} z_1 \\ \vdots \\ z_N \end{bmatrix}
	= \begin{bmatrix} \psi( \vct{x}_1; \vct{w} ) \\ \vdots \\ \psi(\vct{x}_N; \vct{w}) \end{bmatrix}.
$$
The vector $\vct{z}$ is sometimes called a \term{random feature}.
By the reproducing property~\eqref{eqn:reproducing-property}
for the random feature map,
$$
g_{ij} = \Phi(\vct{x}_i, \vct{x}_j)
	= \Expect_{\vct{w}} \big[\psi( \vct{x}_i; \vct{w} ) \cdot \psi( \vct{x}_j; \vct{w} ) \big]
	= \Expect_{\vct{w}} \big[ z_i \cdot z_j \big]
\quad\text{for $i, j = 1, 2, 3, \dots, N$.}
$$
We can write this relation in matrix form as $\mtx{G} = \Expect( \vct{zz}^\adj )$.
Therefore, the random matrix $\mtx{R} = \vct{zz}^\adj$ is
an unbiased rank-one estimator for the kernel matrix $\mtx{G}$.
This representation demonstrates that random feature maps, as defined here,
only exist for positive-definite kernels.

As usual, we construct a better empirical approximation of the kernel
matrix $\mtx{G}$ by averaging several realizations of the simple estimator $\mtx{R}$:
\begin{equation} \label{eqn:kernel-approx}
\bar{\mtx{R}}_n = \frac{1}{n} \sum_{k=1}^n \mtx{R}_k
\quad\text{where each $\mtx{R}_k$ is an independent copy of $\mtx{R}$.}
\end{equation}
In other words, we are using $n$ independent random features
$\vct{z}_1, \dots, \vct{z}_n$ to approximate the kernel matrix.
The question is how many random features are needed before our
estimator is accurate.

\subsection{Examples of Random Feature Maps}
\label{sec:random-feature-examples}

Before we continue with the analysis, let us describe some random feature
maps.  This discussion is tangential to our theme of matrix concentration,
but it is valuable to understand why random feature maps exist.

First, let us consider the angular similarity~\eqref{eqn:angular-similarity}
defined on $\R^d$.  We can construct a random feature map using a classical
result from plane geometry.  If we draw $\vct{w}$ uniformly from
the unit sphere $\mathbb{S}^{d-1} \subset \R^d$, then
\begin{equation} \label{eqn:angular-similarity-avg}
\Phi( \vct{x}; \vct{y} ) = 1 - \frac{2 \measuredangle(\vct{x}, \vct{y})}{\pi}
	= \Expect_{\vct{w}} \big[ \sgn{ \ip{ \vct{x} }{ \vct{w} } } \cdot \sgn{ \ip{\smash{\vct{y}}}{\vct{w}} }\big]
	\quad\text{for all $\vct{x}, \vct{y} \in \coll{X}$.}
\end{equation}
The easy proof of this relation should be visible from
the diagram in Figure~\ref{fig:angle-sim-new}.  In light
of the formula~\eqref{eqn:angular-similarity-avg},
we set $\coll{W} = \mathbb{S}^{d-1}$ with the uniform
measure, and we define the feature map
$$
\psi( \vct{x}; \vct{w} ) = \sgn{ \ip{\vct{x}}{\vct{w}} }.
$$
The reproducing property~\eqref{eqn:reproducing-property}
follows immediately from~\eqref{eqn:angular-similarity-avg}.
Therefore, the pair $(\psi, \vct{w})$ is a random feature map
for the angular similarity kernel.

\begin{figure}
\begin{center}
\begin{tikzpicture}[scale=2]
	\fill[fill=red!25!white] (2.5,0.1) rectangle (2.8,0.4)
	node[below right,scale=1]{$\sgn{ \ip{\vct{x}}{\vct{u}} } \cdot \sgn{ \ip{\smash{\vct{y}}}{\vct{u}} } = +1$};

	\fill[fill=blue!40!white] (2.5,-0.4) rectangle (2.8,-0.1)
	node[below right,scale=1]{$\sgn{ \ip{\vct{x}}{\vct{u}} } \cdot \sgn{ \ip{\smash{\vct{y}}}{\vct{u}} } = -1$};

	\fill[fill=blue!40!white] (0,0) -- (90:1) arc (90:120:1) -- cycle;
	\fill[fill=red!25!white] (0,0) -- (120:1) arc (120:270:1) -- cycle;
	\fill[fill=blue!40!white] (0,0) -- (270:1) arc (270:300:1) -- cycle;
	\fill[fill=red!25!white] (0,0) -- (300:1) arc (300:450:1) -- cycle;
	
	\draw (270:1.5) -- (90:1.75) node[below right,scale=1]{$\ip{\vct{x}}{\vct{u}} = 0$};
	\draw (120:1.5) -- (300:1.75) node[above right,scale=1]{$\ip{\smash{\vct{y}}}{\vct{u}} = 0$};
	
	\draw[very thick] (0,0) circle (1);
	\filldraw (0,0) circle (.02);
	\draw[->,very thick] (0,0) -- (1.75,0) node[right]{$\vct{x}$};
	\draw[->,very thick] (0,0) -- +(30:1.75) node[right]{$\vct{y}$};
	
	\draw[thick] (0:0.25) arc (0:30:0.25); \draw[thin] (15:0.23) -- (15:0.27);
	\draw[thick] (90:0.25) arc (90:120:0.25); \draw[thin] (105:0.23) -- (105:0.27);
	\draw[thick] (270:0.25) arc (270:300:0.25); \draw[thin] (285:0.23) -- (285:0.27);
\end{tikzpicture}
\end{center} 
\begin{caption} 
{\textbf{The angular similarity between two vectors.}  Let $\vct{x}$ and $\vct{y}$ be nonzero vectors
in $\R^2$ with angle $\measuredangle(\vct{x}, \vct{y})$.
The red region contains the directions $\vct{u}$ where the product $\sgn{ \ip{ \vct{x} }{\vct{u} }} \cdot \sgn{ \ip{ \smash{\vct{y}} }{ \vct{u} }}$ equals $+1$, and the blue region contains the directions $\vct{u}$
where the same product equals $-1$.  The blue region subtends a total angle of $2\measuredangle(\vct{x}, \vct{y})$,
and the red region subtends a total angle of $2 \pi - 2\measuredangle(\vct{x}, \vct{y})$.} \label{fig:angle-sim-new}
\end{caption}
\end{figure} 

Next, let us describe an important class of kernels that
can be expressed using random feature maps.  A kernel on $\R^d$
is \term{translation invariant} if there is a function $\phi: \R^d \to \R$
for which
$$
\Phi(\vct{x}, \vct{y}) = \phi( \vct{x} - \vct{y} )
\quad\text{for all $\vct{x}, \vct{y} \in \R^d$.}
$$
B{\^o}chner's Theorem, a classical result from harmonic analysis, gives
a representation for each continuous, positive-definite, translation-invariant
kernel:
\begin{equation} \label{eqn:bochner}
\Phi( \vct{x}, \vct{y} )
	= \phi(\vct{x} - \vct{y})
	= c \int_{\R^d} \econst^{\iunit \,\ip{ \vct{x} }{ \vct{w} }}
	\cdot \econst^{-\iunit \, \ip{\smash{\vct{y}}}{\vct{w}}} \idiff{\mu}(\vct{w})
	\quad\text{for all $\vct{x}, \vct{y} \in \R^d$.}
\end{equation}
In this expression, the positive scale factor $c$ and the probability measure $\mu$
depend only on the function $\phi$.  The formula~\eqref{eqn:bochner} yields
a (complex-valued) random feature map:
$$
\psi_{\C}(\vct{x}; \vct{w}) = \sqrt{c} \, \econst^{\iunit\, \ip{\vct{x}}{\vct{w}}}
\quad\text{where $\vct{w}$ has distribution $\mu$ on $\R^d$.}
$$
This map satisfies a complex variant of the reproducing
property~\eqref{eqn:reproducing-property}:
$$
\Phi(\vct{x}, \vct{y}) = \Expect_{\vct{w}} \big[ \psi_{\C}(\vct{x}; \vct{w})
	\cdot \psi_{\C}(\vct{y}; \vct{w})^\adj \big]
	\quad\text{for all $\vct{x}, \vct{y} \in \R^d$},
$$
where we have written ${}^\adj$ for complex conjugation.

With a little more work, we can construct a real-valued random feature map.
Recall that the kernel $\Phi$ is symmetric, so the complex exponentials
in~\eqref{eqn:bochner} can be written in terms of cosines.  This observation
leads to the random feature map
\begin{equation} \label{eqn:bochner-real}
\psi(\vct{x}; \vct{w}, U) = \sqrt{2c} \cos\big( \ip{\vct{x}}{\vct{w}} + U \big)
\quad\text{where}\quad
\text{$\vct{w} \sim \mu$ and $U \sim \uniform[0,2\pi]$.}
\end{equation}
To verify that $(\psi, (\vct{w}, U))$ reproduces the kernel $\Phi$,
as required by~\eqref{eqn:reproducing-property},
we just make a short calculation using the angle-sum formula for the cosine. 

We conclude this section with the most important example of a random feature map
from the class we have just described.  Consider the Gaussian radial basis function kernel:
$$
\Phi( \vct{x}, \vct{y} ) = \econst^{-\alpha \normsq{\smash{\vct{x} - \vct{y}}}/ 2}
\quad\text{for all $\vct{x}, \vct{y} \in \R^d$.}
$$
The positive parameter $\alpha$ reflects how close two points must be before they are
regarded as ``similar.''  For the Gaussian kernel, B{\^o}chner's Theorem~\eqref{eqn:bochner}
holds with the scaling factor $c = 1$ and the probability measure
$\mu = \normal( \vct{0}, \alpha\, \Id_d )$. In summary, we define
$$
\psi( \vct{x}; \vct{w}, U ) = \sqrt{2} \cos\big( \ip{ \vct{x} }{ \vct{w} } + U \big)
\quad\text{where}\quad
\text{$\vct{w} \sim \normal(\vct{0}, \alpha \,\Id_d)$ and $U \sim \uniform[0,2\pi]$.}
$$
This random feature map reproduces the Gaussian radial basis function kernel.

\subsection{Performance of the Random Feature Approximation}

We will demonstrate that the approximation $\bar{\mtx{R}}_n$ of the $N \times N$
kernel matrix $\mtx{G}$ using $n$ random features, constructed in~\eqref{eqn:kernel-approx},
leads to an estimate of the form
\begin{equation} \label{eqn:random-feature-bound}
\Expect \norm{ \smash{\bar{\mtx{R}}_n - \mtx{G}} }
	\leq \sqrt{\frac{2bN \norm{\mtx{G}} \log(2N)}{n}} + \frac{2bN \log(2N)}{3n}.
\end{equation}
In this expression, $b$ is the uniform bound on the magnitude of the feature map $\psi$.
The short proof of~\eqref{eqn:random-feature-bound} appears in~\S\ref{sec:analysis-rdm-features}.

To clarify what this result means, we introduce the \term{intrinsic dimension}
of the $N \times N$ kernel matrix $\mtx{G}$:
$$
\intdim(\mtx{G}) = \strank( \mtx{G}^{1/2} )
	= \frac{\trace \mtx{G}}{ \norm{\mtx{G}}}
	= \frac{N}{\norm{\mtx{G}}}.
$$
The stable rank is defined in Section~\ref{sec:stable-rank}.
We have used the assumption that the similarity measure is positive definite
to justify the computation of the square root of the kernel matrix,
and $\trace \mtx{G} = N$ because of the requirement that
$\Phi(\vct{x}, \vct{x}) = +1$ for all $\vct{x} \in \coll{X}$.
See \S\ref{sec:int-dim} for further discussion
of the intrinsic dimension

Now, assume that the number $n$ of random features satisfies the bound
$$
n \geq 2 b \eps^{-2} \cdot \intdim(\mtx{G}) \cdot \log(2N),
$$
In view of~\eqref{eqn:random-feature-bound}, the relative error in the
empirical approximation of the kernel matrix satisfies
$$
\frac{\Expect \norm{ \smash{\bar{\mtx{R}}_n - \mtx{G}} }}{\norm{\mtx{G}}}
	\leq \eps + \eps^{-2}.
$$
We learn that the randomized approximation of the kernel matrix $\mtx{G}$ is accurate
when its intrinsic dimension is much smaller than the number of data points.  That is,
$\intdim(\mtx{G}) \ll N$.

\subsection{Analysis of the Random Feature Approximation}
\label{sec:analysis-rdm-features}

The analysis of random features is based on Corollary~\ref{cor:matrix-approx-sampling}.
To apply this result, we need the per-sample second-moment $m_2(\mtx{R})$ and the uniform
upper bound $L$.  Both are easy to come by.

First, observe that
$$
\norm{ \mtx{R} } = \norm{ \smash{\mtx{zz}^\adj} } = \normsq{\vct{z}}
	\leq b N
$$
Recall that $b$ is the uniform bound on the feature map $\psi$,
and $N$ is the number of components in the random feature vector $\vct{z}$.

Second, we calculate that
$$
\Expect \mtx{R}^2 = \Expect\big( \normsq{\vct{z}} \vct{zz}^\adj \big)
	\psdle bN \cdot \Expect( \vct{zz}^\adj )
	= bN \cdot \mtx{G}.
$$
Each random matrix $\vct{zz}^\adj$ is positive semidefinite, so we can introduce
the upper bound $\normsq{\vct{z}} \leq bN$.  The last identity holds because $\mtx{R}$
is an unbiased estimator of the kernel matrix $\mtx{G}$.  It follows that
$$
m_2(\mtx{R}) = \norm{ \smash{\Expect \mtx{R}^2} } \leq bN \cdot \norm{\mtx{G}}.
$$
This is our bound for the per-sample second moment.

Finally, we invoke Corollary~\ref{cor:matrix-approx-sampling} with
parameters  $L = bN$ and $m_2(\mtx{R}) \leq bN \norm{\mtx{G}}$ to
arrive at the estimate~\eqref{eqn:random-feature-bound}.

\section{Proof of the Matrix Bernstein Inequality}
\label{sec:bernstein-proof}

Now, let us turn to the proof of the matrix Bernstein inequality, Theorem~\ref{thm:matrix-bernstein-rect}.
This result is a corollary of a matrix concentration inequality for a sum of 
bounded random Hermitian matrices.  We begin with a statement and discussion
of the Hermitian result, and then we explain how the general result follows.

\subsection{A Sum of Bounded Random Hermitian Matrices}

The first result is a Bernstein inequality for a sum of independent,
random Hermitian matrices whose eigenvalues are bounded above.

\begin{thm}[Matrix Bernstein: Hermitian Case] \label{thm:matrix-bernstein-herm}
Consider a finite sequence $\{ \mtx{X}_k \}$ of independent, random, Hermitian matrices with dimension $d$.  Assume that
$$
\Expect \mtx{X}_k = \mtx{0}
\quad\text{and}\quad
\lambda_{\max}(\mtx{X}_k) \leq L
\quad\text{for each index $k$.}
$$
Introduce the random matrix
$$
\mtx{Y} = \sum\nolimits_k \mtx{X}_k.
$$
Let $v(\mtx{Y})$ be the matrix variance statistic of the sum:
\begin{equation} \label{eqn:matrix-bernstein-sigma2}
v(\mtx{Y}) = \norm{ \smash{\Expect{} \mtx{Y}^2} }
	= \norm{ \sum\nolimits_k \Expect{} \mtx{X}_k^2 }.
\end{equation}
Then
\begin{equation} \label{eqn:matrix-bernstein-expect}
\Expect \lambda_{\max}(\mtx{Y})
	\leq \sqrt{2v(\mtx{Y}) \log d} + \frac{1}{3} L \,\log d.
\end{equation}
Furthermore, for all $t \geq 0$.
\begin{equation} \label{eqn:matrix-bernstein-tail}
\Prob{ \lambda_{\max}\left( \mtx{Y} \right) \geq t }
	\leq d \cdot \exp\left( \frac{-t^2/2}{v(\mtx{Y}) + Lt/3} \right).
\end{equation}
\end{thm}

\noindent
The proof of Theorem~\ref{thm:matrix-bernstein-herm} appears below in \S\ref{sec:bernstein-proof}.

\subsection{Discussion}

Theorem~\ref{thm:matrix-bernstein-herm} also yields information about the minimum eigenvalue of an independent sum of $d$-dimensional Hermitian matrices.  Suppose that the independent random matrices satisfy
$$
\Expect \mtx{X}_k = \mtx{0}
\quad\text{and}\quad
\lambda_{\min}(\mtx{X}_k) \geq - \underline{L}
\quad\text{for each index $k$.}
$$
Applying the expectation bound~\eqref{eqn:matrix-bernstein-expect} to $-\mtx{Y}$, we obtain
\begin{equation} \label{eqn:matrix-bernstein-lower-expect}
\Expect \lambda_{\min}(\mtx{Y}) \geq - \sqrt{2 v(\mtx{Y}) \log d} - \frac{1}{3} \underline{L} \, \log d.
\end{equation}
We can use~\eqref{eqn:matrix-bernstein-tail} to develop a tail bound.
For $t \geq 0$,
$$
\Prob{ \lambda_{\min}(\mtx{Y}) \leq - t }
	\leq d \cdot \exp\left( \frac{-t^2/2}{v(\mtx{Y}) + \underline{L} t/3} \right).
$$
Let us emphasize that the bounds for $\lambda_{\max}(\mtx{Y})$ and $\lambda_{\min}(\mtx{Y})$
may diverge because the two parameters $L$ and $\underline{L}$ can take sharply different values.
This fact indicates that the maximum eigenvalue bound in Theorem~\ref{thm:matrix-bernstein-herm}
is a less strict assumption than the spectral norm bound in Theorem~\ref{thm:matrix-bernstein-rect}.

\subsection{Bounds for the Matrix Mgf and Cgf}

In establishing the matrix Bernstein inequality, the main challenge is to obtain
an appropriate bound for the matrix mgf and cgf of a zero-mean random matrix
whose norm satisfies a uniform bound.  We do not present the sharpest
estimate possible, but rather the one that leads most directly to the useful
results stated in Theorem~\ref{thm:matrix-bernstein-herm}.

\begin{lemma}[Matrix Bernstein: Mgf and Cgf Bound] \label{lem:matrix-bernstein-mgf}
Suppose that $\mtx{X}$ is a random Hermitian matrix that satisfies
$$
\Expect \mtx{X} = \mtx{0}
\quad\text{and}\quad
\lambda_{\max}( \mtx{X} ) \leq L.
$$
Then, for $0 < \theta < 3/L$,
$$
\begin{aligned}
\Expect \econst^{\theta \mtx{X}}
	\psdle \exp\left( \frac{\theta^2/2}{1-\theta L/3}
	\cdot \Expect{} \mtx{X}^2 \right)
\quad\text{and}\quad
\log{} \Expect \econst^{\theta \mtx{X}}
	\psdle \frac{\theta^2/2}{1-\theta L/3}
	\cdot \Expect{} \mtx{X}^2.
\end{aligned}
$$
\end{lemma}

\begin{proof}
Fix the parameter $\theta > 0$.  In the exponential $\econst^{\theta \mtx{X}}$, we would like to expose the random matrix $\mtx{X}$ and its square $\mtx{X}^2$ so that we can exploit information about the mean and variance.  To that end, we write
\begin{equation} \label{eqn:bernstein-mgf-pf-1}
\econst^{\theta \mtx{X}}
	= \Id + \theta \mtx{X} + (\econst^{\theta \mtx{X}} - \theta \mtx{X} - \Id)
	= \Id + \theta \mtx{X} + \mtx{X} \cdot f(\mtx{X}) \cdot \mtx{X},
\end{equation}
where $f$ is a function on the real line:
$$
f(x) = \frac{\econst^{\theta x} - \theta x - 1}{x^2}
\quad\text{for $x \neq 0$}\quad\text{and}\quad
f(0) = \frac{\theta^2}{2}.
$$
The function $f$ is increasing because its derivative is positive.  Therefore, $f(x) \leq f(L)$ when $x \leq L$.  By assumption, the eigenvalues of $\mtx{X}$ do not exceed $L$, so the Transfer Rule~\eqref{eqn:transfer-rule} implies that
\begin{equation} \label{eqn:bernstein-mgf-pf-2}
f(\mtx{X}) \psdle f(L) \cdot \Id.
\end{equation}
The Conjugation Rule~\eqref{eqn:conjugation-rule} allows us to introduce the relation~\eqref{eqn:bernstein-mgf-pf-2} into our expansion~\eqref{eqn:bernstein-mgf-pf-1} of the matrix exponential:
$$
\econst^{\theta \mtx{X}}
	\psdle \Id + \theta \mtx{X} + \mtx{X}( f(L) \cdot \Id ) \mtx{X}
	= \Id + \theta \mtx{X} + f(L) \cdot \mtx{X}^2.
$$
This relation is the basis for our matrix mgf bound.

To obtain the desired result, we develop a further estimate for $f(L)$.
This argument involves a clever application of Taylor series:
\begin{equation*} \label{eqn:bernstein-mgf-pf-4}
f(L) = \frac{\econst^{\theta L} - \theta L - 1}{L^2}
	= \frac{1}{L^2} \sum_{q=2}^\infty \frac{(\theta L)^q}{q!}
	\leq \frac{\theta^{2}}{2} \sum_{q=2}^\infty \frac{(\theta L)^{q-2}}{3^{q-2}}
	= \frac{\theta^2/2}{1- \theta L/3}.
\end{equation*}
The second expression is simply the Taylor expansion of the fraction, viewed as a function of $\theta$.  We obtain the inequality by factoring out $(\theta L)^2/2$ from each term in the series and invoking the bound $q! \geq 2\cdot 3^{q-2}$, valid for each $q = 2, 3, 4,\dots$.  Sum the geometric series to obtain the final identity.

To complete the proof of the mgf bound, we combine the last two displays:
$$
\econst^{\theta \mtx{X}}
	\psdle \Id + \theta \mtx{X} + \frac{\theta^2/2}{1- \theta L/3} \cdot \mtx{X}^2.
$$
This estimate is valid because $\mtx{X}^2$ is positive semidefinite.
Expectation preserves the semidefinite order, so
$$
\Expect \econst^{\theta \mtx{X}}
	\psdle \Id +\frac{\theta^2/2}{1-\theta L/3} \cdot \Expect{} \mtx{X}^2
	\psdle \exp\left( \frac{\theta^2/2}{1 - \theta L/3} \cdot \Expect{} \mtx{X}^2 \right).
$$
We have used the assumption that $\mtx{X}$ has zero mean.
The second semidefinite relation follows when we apply the Transfer Rule~\eqref{eqn:transfer-rule}
to the inequality $1 + a \leq \econst^a$, which holds for $a \in \mathbb{R}$.

To obtain the semidefinite bound for the cgf, we extract the logarithm of the mgf bound using the fact~\eqref{eqn:log-monotone} that the logarithm is operator monotone.
\end{proof}

\subsection{Proof of the Hermitian Case}

We are prepared to establish the matrix Bernstein inequalities for random Hermitian matrices.

\begin{proof}[Proof of Theorem~\ref{thm:matrix-bernstein-herm}]
Consider a finite sequence $\{ \mtx{X}_k \}$ of random Hermitian matrices with dimension $d$.  Assume that
$$
\Expect \mtx{X}_k = \mtx{0}
\quad\text{and}\quad
\lambda_{\max}( \mtx{X}_k ) \leq L
\quad\text{for each index $k$.}
$$
The matrix Bernstein cgf bound, Lemma~\ref{lem:matrix-bernstein-mgf}, provides that
\begin{equation} \label{eqn:bernstein-pf-cgf}
\log{} \Expect \econst^{\theta \mtx{X}_k}
	\psdle g(\theta) \cdot \Expect{}\mtx{X}_k^2
	\quad\text{where}\quad
	g(\theta) = \frac{\theta^2/2}{1-\theta L/3}
	\quad\text{for $0 < \theta < 3/L$.}
\end{equation}
Introduce the sum $\mtx{Y} = \sum_k \mtx{X}_k$. %

We begin with the bound~\eqref{eqn:matrix-bernstein-expect} for the expectation $\Expect \lambda_{\max}(\mtx{Y})$.
Invoke the master inequality, relation~\eqref{eqn:master-upper-expect} in Theorem~\ref{thm:master-ineq}, to find that
$$
\begin{aligned}
\Expect \lambda_{\max}(\mtx{Y})
	&\leq \inf_{\theta > 0} \ \frac{1}{\theta} \log{} \trace \exp\left(
		\sum\nolimits_k \log{} \Expect \econst^{\theta \mtx{X}_k} \right) \\
	&\leq \inf_{0 < \theta < 3/L} \ \frac{1}{\theta} \log{} \trace \exp\bigg(
		g(\theta) \sum\nolimits_k \Expect{} \mtx{X}_k^2 \bigg) \\
	&= \inf_{0 < \theta < 3/L} \ \frac{1}{\theta} \log{} \trace \exp\left(
	g(\theta) \cdot \Expect{} \mtx{Y}^2 \right).
\end{aligned}
$$
As usual, to move from the first to the second line,
we invoke the fact~\eqref{eqn:exp-trace-monotone} that the trace exponential is
monotone to introduce the semidefinite bound~\eqref{eqn:bernstein-pf-cgf}
for the cgf.  Then we use the additivity rule~\eqref{eqn:matrix-variance-add}
for the variance of an independent sum to identify
$\Expect{} \mtx{Y}^2$.
The rest of the argument glides along a well-oiled track:
$$
\begin{aligned}
\Expect \lambda_{\max}(\mtx{Y})
	&\leq \inf_{0 < \theta < 3/L} \ \frac{1}{\theta} \log{} \left[ d \, \lambda_{\max}\left( \exp\left(
		g(\theta) \cdot \Expect{} \mtx{Y}^2 \right) \right) \right] \\
	&= \inf_{0 < \theta < 3/L} \ \frac{1}{\theta} \log{} \left[ d \, \exp\left(
		g(\theta) \cdot \lambda_{\max}\left( \Expect{} \mtx{Y}^2 \right) \right)\right] \\
	&\leq \inf_{0 < \theta < 3/L} \ \frac{1}{\theta} \log{} \left[ d \, \exp\left(
		g(\theta) \cdot v(\mtx{Y}) \right)\right] \\
	&= \inf_{0 < \theta < 3/L} \ \left[ \frac{\log d}{\theta} +
		\frac{\theta/2}{1 - \theta L/3} \cdot v(\mtx{Y}) \right].
\end{aligned}
$$
In the first inequality, we bound the trace of the exponential by the dimension $d$ times the maximum eigenvalue.  The next line follows from the Spectral Mapping Theorem, Proposition~\ref{prop:spectral-mapping}.  In the third line, we identify the matrix variance statistic $v(\mtx{Y})$ from~\eqref{eqn:matrix-bernstein-sigma2}.  Afterward, we extract the logarithm and simplify.  Finally, we compute the infimum to complete the proof of~\eqref{eqn:matrix-bernstein-expect}.  For reference, the optimal argument is
$$
\theta = \frac{6 L \log d + 9 \sqrt{2v(\mtx{Y}) \log d}}{2 L^2 t + 9 (\mtx{Y}) + 6L\sqrt{2(\mtx{Y}) \log d}}.
$$
We recommend using a computer algebra system to confirm this point.

Next, we develop the tail bound~\eqref{eqn:matrix-bernstein-tail} for $\lambda_{\max}(\mtx{Y})$.  Owing to the master tail inequality~\eqref{eqn:master-upper-tail}, we have
$$
\begin{aligned}
\Prob{ \lambda_{\max}(\mtx{Y}) \geq t }
	&\leq \inf_{\theta > 0} \ \econst^{-\theta t} \, \trace \exp\left(
	\sum\nolimits_k \log{} \Expect \econst^{\theta \mtx{X}_k} \right) \\
	&\leq \inf_{0 < \theta < 3/L} \ \econst^{-\theta t} \, \trace \exp\bigg(
		g(\theta) \sum\nolimits_k \Expect{} \mtx{X}_k^2 \bigg) \\
	&\leq \inf_{0 < \theta < 3/L} \ d \, \econst^{-\theta t} \, \exp\left(
		g(\theta) \cdot v(\mtx{Y}) \right).
\end{aligned}
$$
The justifications are the same as before.  The exact value of the infimum is messy, so we proceed with the inspired choice $\theta = t/(v(\mtx{Y}) + L t / 3)$, which results in the elegant bound~\eqref{eqn:matrix-bernstein-tail}.
\end{proof} 

\subsection{Proof of the General Case}

Finally, we explain how to derive Theorem~\ref{thm:matrix-bernstein-rect}, for general matrices, from Theorem~\ref{thm:matrix-bernstein-herm}.  This result follows immediately when we apply the matrix Bernstein bounds for Hermitian matrices to the Hermitian dilation of a sum of general matrices.

\begin{proof}[Proof of Theorem~\ref{thm:matrix-bernstein-rect}]
Consider a finite sequence $\{ \mtx{S}_k \}$ of $d_1 \times d_2$ random matrices, and assume that
$$
\Expect \mtx{S}_k = \mtx{0}
\quad\text{and}\quad
\norm{\mtx{S}_k} \leq L
\quad\text{for each index $k$.}
$$
We define the two random matrices
$$
\mtx{Z} = \sum\nolimits_k \mtx{S}_k
\quad\text{and}\quad
\mtx{Y} = \coll{H}(\mtx{Z})
	= \sum\nolimits_k \coll{H}(\mtx{S}_k)
$$
where $\coll{H}$ is the Hermitian dilation~\eqref{eqn:herm-dilation}.
The second expression for $\mtx{Y}$ follows from the property that the
dilation is a real-linear map.

We will apply Theorem~\ref{thm:matrix-bernstein-herm} to analyze $\norm{ \mtx{Z} }$.
First, recall the fact~\eqref{eqn:herm-dilation-norm} that
$$
\norm{ \mtx{Z} } = \lambda_{\max}(\coll{H}(\mtx{Z})) = \lambda_{\max}(\mtx{Y}).
$$
Next, we express the variance~\eqref{eqn:matrix-bernstein-sigma2} of the random Hermitian matrix $\mtx{Y}$ in terms of the general matrix $\mtx{Z}$.  Indeed, the calculation~\eqref{eqn:var-stat-dilation} of the variance statistic
of a dilation shows that
$$
v(\mtx{Y}) = v(\coll{H}(\mtx{Y})) = v(\mtx{Z}).
$$
Recall that the matrix variance statistic $v(\mtx{Z})$ defined in~\eqref{eqn:matrix-bernstein-sigma2-rect}
coincides with the general definition from~\eqref{eqn:matrix-variance-rect}.
Finally, we invoke Theorem~\ref{thm:matrix-bernstein-herm} to establish Theorem~\ref{thm:matrix-bernstein-rect}.
\end{proof}

\section{Notes}

The literature contains a wide variety of Bernstein-type inequalities in the scalar case, and the matrix case is no different.  The applications of the matrix Bernstein inequality are also numerous.  We only give a brief summary here.

\subsection{Matrix Bernstein Inequalities}

David Gross~\cite{Gro11:Recovering-Low-Rank} and Ben Recht~\cite{Rec11:Simpler-Approach} used the approach of Ahlswede \& Winter~\cite{AW02:Strong-Converse} to develop two different versions of the matrix Bernstein inequality.
These papers helped to popularize the use matrix concentration inequalities in mathematical
signal processing and statistics.  Nevertheless, their results involve a suboptimal variance parameter of the form
$$
v_{\textrm{AW}}(\mtx{Y}) = \sum\nolimits_k \norm{\Expect{} \mtx{X}_k^2 }.
$$
This parameter can be significantly larger than the matrix variance statistic~\eqref{eqn:matrix-bernstein-sigma2} that appears in Theorem~\ref{thm:matrix-bernstein-herm}.  They do coincide in some special cases, such as when the summands are independent and identically distributed.

Oliveira~\cite{Oli10:Concentration-Adjacency} established the first version of the matrix Bernstein inequality that yields the correct matrix variance statistic~\eqref{eqn:matrix-bernstein-sigma2}.  He accomplished this task with an elegant application of the Golden--Thompson inequality~\eqref{eqn:golden-thompson}.  His method even gives a result, called the matrix Freedman inequality, that holds for matrix-valued martingales.  His bound is roughly equivalent with Theorem~\ref{thm:matrix-bernstein-herm}, up to the precise value of the constants.

The matrix Bernstein inequality we have stated here, Theorem~\ref{thm:matrix-bernstein-herm}, first appeared in the paper~\cite[\S6]{Tro11:User-Friendly-FOCM} by the author of these notes.
The bounds for the expectation are new.  The argument is based on Lieb's Theorem, and it also delivers a matrix Bennett inequality. %
This paper also describes how to establish matrix Bernstein inequalities for sums of unbounded random matrices, given some control over the matrix moments.

The research in~\cite{Tro11:User-Friendly-FOCM} is independent from Oliveira's work~\cite{Oli10:Concentration-Adjacency}, although Oliveira's paper motivated the subsequent article~\cite{Tro11:Freedmans-Inequality} and the technical report~\cite{Tro11:User-Friendly-Martingale-TR}, which explain how to use Lieb's Theorem to study matrix martingales.  The technical report~\cite{GT11:Tail-Bounds} develops a Bernstein inequality for interior eigenvalues using the Lieb--Seiringer Theorem~\cite{LS05:Stronger-Subadditivity}.

For more versions of the matrix Bernstein inequality, see Vladimir Koltchinskii's lecture notes from Saint-Flour~\cite{Kol11:Oracle-Inequalities}.  In Chapter~\ref{chap:intrinsic}, we present another extension of the matrix Bernstein
inequality that involves a smaller dimensional parameter.

\subsection{The Matrix Rosenthal--Pinelis Inequality}

The matrix Rosenthal--Pinelis inequality~\eqref{eqn:matrix-moment-ineq} is a close cousin of
the matrix Rosenthal inequality~\eqref{eqn:matrix-rosenthal}.  Both results are derived
from the noncommutative Khintchine inequality~\eqref{eqn:nc-khintchine} using the same pattern
of argument~\cite[Thm.~A.1]{CGT12:Masked-Sample}.  We believe that~\cite{CGT12:Masked-Sample}
is the first paper to recognize and state the result~\eqref{eqn:matrix-moment-ineq},
even though it is similar in spirit with the work in~\cite{Rud99:Random-Vectors}.
A self-contained, elementary proof of a related matrix Rosenthal-Pinelis inequality
appears in~\cite[Cor.~7.4]{MJCFT12:Matrix-Concentration}.

Versions of the matrix Rosenthal--Pinelis inequality first appeared
in the literature~\cite{JX03:Noncommutative-Burkholder}
on noncommutative martingales, where they were called \term{noncommutative
Burkholder inequalities}.  For an application to random matrices,
see the follow-up work~\cite{JX08:Noncommutative-Burkholder-II} by the same authors.
Subsequent papers~\cite{JZ12:Noncommutative-Martingale,JZ13:Noncommutative-Bennett}
contain related noncommutative martingale inequalities inspired
by the research in~\cite{Oli10:Sums-Random,Tro11:User-Friendly-FOCM}.

\subsection{Empirical Approximation}

Matrix approximation by random sampling is a special case of
a general method that Bernard Maurey developed to compute
entropy numbers of convex hulls.  Let us give a short presentation
of the original context, along with references to some other applications.

\subsubsection{Empirical Bounds for Covering Numbers}

Suppose that $X$ is a Banach space.  Consider the convex hull
$E = \conv\{ \vct{e}_1, \dots, \vct{e}_N \}$ of a set of $N$ points
in $X$, and assume that $\norm{ \vct{e}_k } \leq L$.
We would like to give an upper bound for the number of balls of radius
$\eps$ it takes to cover this set.

Fix a point $\vct{u} \in E$,
and express $\vct{u}$ as a convex combination:
$$
\vct{u} = \sum_{i=1}^N p_i \vct{e}_i
\quad\text{where}\quad
\sum_{i=1}^N p_i = 1
\quad\text{and}\quad
p_i \geq 0.
$$
Let $\vct{x}$ be the random vector in $X$ that takes value $\vct{e}_k$
with probability $p_k$.  We can approximate the point $\vct{u}$ as
an average $\bar{\vct{x}} = n^{-1} \sum_{k=1}^n \vct{x}_k$
of independent copies $\vct{x}_1, \dots, \vct{x}_n$ of
the random vector $\vct{x}$.  Then
$$
\Expect{} \norm{ \bar{\vct{x}}_n - \vct{u} }_X
	= \frac{1}{n} \Expect{} \norm{ \sum\nolimits_{k=1}^n (\vct{x}_k - \Expect \vct{x}) }_X
	\leq \frac{2}{n} \Expect{} \norm{ \sum\nolimits_{k=1}^n \varrho_k \vct{x}_k }_X
	\leq \frac{2}{n} \left(\Expect{} \normsq{ \sum\nolimits_{k=1}^n \varrho_k \vct{x}_k }_{X} \right)^{1/2}.
$$
The family $\{\varrho_k\}$ consists of independent Rademacher random variables.
The first inequality depends on the symmetrization procedure~\cite[Lem.~6.3]{LT91:Probability-Banach},
and the second is H{\"o}lder's.  In certain Banach spaces, a Khintchine-type inequality holds:
$$
\Expect{} \norm{ \bar{\vct{x}}_n - \vct{u} }_X
	\leq \frac{2  T_2(X)}{n} \left( \sum\nolimits_{k=1}^n \Expect{} \normsq{\vct{x}_k}_X \right)^{1/2}
	\leq \frac{2 T_2(X) L}{\sqrt{n}}.
$$
The last inequality depends on the uniform bound $\norm{\vct{e}_k} \leq L$.
This estimate controls the expected error in approximating an arbitrary point
in $E$ by randomized sampling.

The number $T_2(X)$ is called the \term{type two constant} of the Banach space $X$,
and it can be estimated in many concrete instances; see~\cite[Chap.~9]{LT91:Probability-Banach}
or~\cite[Chap.~11]{Pis89:Volume-Convex}.
For our purposes, the most relevant example is the Banach space $\mathbb{M}^{d_1 \times d_2}$
consisting of $d_1 \times d_2$ matrices equipped with the spectral norm.  Its type two constant
satisfies
$$
T_2\big(\mathbb{M}^{d_1 \times d_2}\big) \leq \mathrm{Const} \cdot \sqrt{\log(d_1 + d_2)}.
$$
This result follows from work of Tomczak--Jaegermann~\cite[Thm.~3.1(ii)]{TJ74:Moduli-Smoothness}.  In fact,
the space $\mathbb{M}^{d_1 \times d_2}$ enjoys an even stronger property with respect to 
averages, namely the noncommutative Khintchine inequality~\eqref{eqn:nc-khintchine}.

Now, suppose that the number $n$ of samples in our empirical approximation $\bar{\vct{x}}_n$
of the point $\vct{u} \in E$ satisfies
$$
n \geq \left(\frac{ 2 T_2(X) L}{\eps} \right)^2.
$$
Then the probabilistic method ensures that there is a some collection
of $\vct{u}_1, \dots, \vct{u}_n$ of points drawn with repetition from
the set $\{\vct{e}_1, \dots, \vct{e}_N\}$ that satisfies
$$
\norm{ \left( n^{-1} \sum\nolimits_{k=1}^n \vct{u}_k \right) - \vct{u} }_{X}
	\leq \eps.
$$
There are at most $N^n$ different ways to select the points $\vct{u}_k$.
It follows that we can cover the convex hull $E = \conv\{\vct{e}_1, \dots, \vct{e}_N\}$
in $X$ with at most $N^n$ norm balls of radius $\eps$.

\subsubsection{History and Applications of Empirical Approximation}

Maurey did not publish his ideas, and the method was first broadcast in a paper
of Pisier~\cite[Lem.~1]{Pis81:Remarques-Resultat}.  Another early reference
is the work of Carl~\cite[Lem.~1]{Car85:Inequalities-Bernstein}.  More recently,
this covering argument has been used to study the restricted isomorphism
behavior of a random set of rows drawn from a discrete Fourier
transform matrix~\cite{RV06:Sparse-Reconstruction}.

By now, empirical approximation has appeared in a wide range of applied contexts,
although many papers do not recognize the provenance of the method.
Let us mention some examples in machine learning.  Empirical approximation
has been used to study what functions can be approximated
by neural networks~\cite{Bar93:Universal-Approximation,LBW96:Efficient-Agnostic}.
The same idea appears in papers on sparse modeling, such as~\cite{SS08:Low-l1-Norm},
and it supports the method of random features~\cite{RR07:Random-Features}.
Empirical approximation also stands at the core of a recent algorithm
for constructing approximate Nash equilibria~\cite{Bar14:Approximate-Version}.

It is difficult to identify the earliest work in computational mathematics
that invoked the empirical method to approximate matrices.  The
paper of Achlioptas \& McSherry~\cite{AM01:Fast-Computation} on
randomized sparsification is one possible candidate.

Corollary~\ref{cor:matrix-approx-sampling}, which we use to perform the analysis
of matrix approximation by sampling, does not require the full power of
the matrix Bernstein inequality, Theorem~\ref{thm:matrix-bernstein-rect}.
Indeed, Corollary~\ref{cor:matrix-approx-sampling} can be derived from
the weaker methods of Ahlswede \& Winter~\cite{AW02:Strong-Converse};
for example, see the papers~\cite{Gro11:Recovering-Low-Rank,Rec11:Simpler-Approach}.

\subsection{Randomized Sparsification}

The idea of using randomized sparsification to accelerate spectral computations appears in a paper of Achlioptas \& McSherry~\cite{AM01:Fast-Computation,AM07:Fast-Computation}.  d'Aspr{\'e}mont~\cite{dAs11:Subsampling-Algorithms}
proposed to use sparsification to accelerate algorithms for semidefinite programming.
The paper~\cite{AKL13:Near-Optimal-Entrywise} by Achlioptas, Karnin, \& Liberty recommends sparsification
as a mechanism for data compression.

After the initial paper~\cite{AM01:Fast-Computation}, several other researchers developed sampling schemes for randomized sparsification~\cite{AHK06:Fast-Random,GT09:Error-Bounds}.  Later, Drineas \& Zouzias~\cite{DZ11:Note-Elementwise} pointed out that matrix concentration inequalities can be used to analyze this type of algorithm.  The paper~\cite{AKL13:Near-Optimal-Entrywise} refined this analysis to obtain sharper bounds.  The simple analysis here is drawn from a recent note by Kundu \& Drineas~\cite{KD14:Note-Randomized}.

\subsection{Randomized Matrix Multiplication}

The idea of using random sampling to accelerate matrix multiplication appeared in nascent form
in a paper of Frieze, Kannan, \& Vempala~\cite{FKV98:Fast-Monte-Carlo}.  The
paper~\cite{DK01:Fast-Monte-Carlo} of Drineas \& Kannan develops this idea in full generality,
and the article~\cite{DKM06:Fast-Monte-Carlo-I} of Drineas, Kannan, \& Mahoney contains a more
detailed treatment.  Subsequently, Tam{\'a}s Sarl{\'os} obtained a significant improvement in the performance of this algorithm~\cite{Sar06:Improved-Approximation}.  Rudelson \& Vershynin~\cite{RV07:Sampling-Large}
obtained the first error bound for approximate matrix multiplication with respect to the spectral norm.
The analysis that we presented is adapted from the dissertation~\cite{Zou12:Randomized-Primitives}
of Tassos Zouzias, which refines an earlier treatment by Magen \& Zouzias~\cite{MZ11:Low-Rank-Matrix-Valued}.
See the monographs of Mahoney~\cite{Mah11:Randomized-Algorithms} and Woodruff~\cite{Woo14:Sketching-Tool}
for a more extensive discussion.

\subsection{Random Features}

Our discussion of kernel methods is adapted from the book~\cite{SS98:Learning-Kernels}.
The papers~\cite{RR07:Random-Features,RR08:Weighted-Sums} of Ali Rahimi and Ben Recht
proposed the idea of using random features to summarize data for large-scale kernel machines.
The construction~\eqref{eqn:bochner-real} of a random feature map for a translation-invariant, positive-definite
kernel appears in their work.  This approach has received a significant amount of attention
over the last few years, and there has been a lot of subsequent development.
For example, the paper~\cite{KK12:Random-Feature} of Kar \& Karnick shows how to construct random features
for inner-product kernels, and the paper~\cite{HXGD14:Compact-Random} of Hamid et al.~develops random features
for polynomial kernels.  Our analysis of random features using the matrix Bernstein inequality
is drawn from the recent article~\cite{LSS+14:Randomized-Nonlinear} of Lopez-Paz et al.
The presentation here is adapted from the author's tutorial on randomized matrix approximation,
given at ICML 2014 in Beijing.  We recommend the two
papers~\cite{HXGD14:Compact-Random,LSS+14:Randomized-Nonlinear} for an up-to-date bibliography.  

\makeatletter{}%

\chapter[Results Involving the Intrinsic Dimension]{Results Involving \\ the Intrinsic Dimension} \label{chap:intrinsic}

A minor shortcoming of our matrix concentration results is the dependence on the ambient dimension of the matrix.  In this chapter, we show how to obtain a dependence on an intrinsic dimension parameter, which occasionally is much smaller than the ambient dimension.  In many cases, intrinsic dimension bounds offer only a modest improvement.  Nevertheless, there are examples where the benefits are significant enough that we can obtain nontrivial results for infinite-dimensional random matrices.

In this chapter, present a version of the matrix Chernoff inequality that involves an intrinsic dimension parameter.  We also describe a version of the matrix Bernstein inequality that involves an intrinsic dimension parameter.  The intrinsic Bernstein result usually improves on Theorem~\ref{thm:matrix-bernstein-rect}.  These results depend on a new argument that distills ideas from a paper~\cite{Min11:Some-Extensions} of Stanislav Minsker.  We omit intrinsic dimension bounds for matrix series, which the reader may wish to develop as an exercise.

To give a sense of what these new results accomplish, we revisit some of the examples from earlier chapters.  We apply the intrinsic Chernoff bound to study a random column submatrix of a fixed matrix.  We also reconsider the randomized matrix multiplication algorithm in light of the intrinsic Bernstein bound.  In each case, the intrinsic dimension parameters have an attractive interpretation in terms of the problem data.

\subsubsection{Overview}

We begin our development in \S\ref{sec:int-dim} with the definition of the intrinsic dimension of a matrix.  In \S\ref{sec:int-chernoff}, we present the intrinsic Chernoff bound and some of its consequences.  In \S\ref{sec:int-bernstein}, we describe the intrinsic Bernstein inequality and its applications.  Afterward, we describe the new ingredients that are required in the proofs.  Section~\ref{sec:new-laplace} explains how to extend the matrix Laplace transform method beyond the exponential function, and \S\ref{sec:int-dim-lemma} describes a simple but powerful lemma that allows us to obtain the dependence on the intrinsic dimension.  Section~\ref{sec:int-chernoff-proof} contains the proof of the intrinsic Chernoff bound, and \S\ref{sec:int-bernstein-proof} develops the proof of the intrinsic Bernstein bound.

\section{The Intrinsic Dimension of a Matrix} \label{sec:int-dim}

Some types of random matrices are concentrated in a small number of dimensions,
while they have little content in other dimensions.  So far,
our bounds do not account for the difference.  We need to introduce
a more refined notion of dimension that will help us to discriminate
among these examples.

\begin{defn}[Intrinsic Dimension] \label{def:int-dim}
For a positive-semidefinite matrix $\mtx{A}$,
the \term{intrinsic dimension} is the quantity
$$
\intdim(\mtx{A}) = \frac{\trace \mtx{A}}{\norm{\mtx{A}}}.
$$
\end{defn}

\noindent
We interpret the intrinsic dimension as a measure
of the number of dimensions where $\mtx{A}$ has
significant spectral content.

Let us make a few observations that support this view.
By expressing the trace and the norm in terms of the eigenvalues,
we can verify that
$$
1 \leq \intdim(\mtx{A}) \leq \rank(\mtx{A}) \leq \dim(\mtx{A}).
$$
The first inequality is attained precisely when $\mtx{A}$ has rank one, while
the second inequality is attained precisely when $\mtx{A}$ is a multiple of the
identity.
The intrinsic dimension is 0-homogeneous,
so it is insensitive to changes in the scale of the matrix $\mtx{A}$.
The intrinsic dimension is \textit{not} monotone with respect to the semidefinite order.
Indeed, we can drive the intrinsic dimension to one by increasing one eigenvalue of $\mtx{A}$
substantially.

\section{Matrix Chernoff with Intrinsic Dimension} \label{sec:int-chernoff}

Let us present an extension of the matrix Chernoff inequality.
This result controls the maximum eigenvalue of a sum of random,
positive-semidefinite matrices in terms of the intrinsic dimension of the
expectation of the sum.

\begin{thm}[Matrix Chernoff: Intrinsic Dimension] \label{thm:intdim-chernoff}
Consider a finite sequence $\{ \mtx{X}_k \}$ of random, Hermitian matrices
of the same size, and assume that
$$
0 \leq \lambda_{\min}(\mtx{X}_k)
\quad\text{and}\quad
\lambda_{\max}(\mtx{X}_k) \leq L
\quad\text{for each index $k$.}
$$
Introduce the random matrix
$$
\mtx{Y} = \sum\nolimits_k \mtx{X}_k.
$$
Suppose that we have a semidefinite upper bound $\mtx{M}$ for the expectation $\Expect \mtx{Y}$:
$$
\mtx{M} \psdge \Expect \mtx{Y} = \sum\nolimits_k \Expect \mtx{X}_k.
$$
Define an intrinsic dimension bound and a mean bound:
$$
d = \intdim\left( \mtx{M} \right)
\quad\text{and}\quad
\mu_{\max} = \lambda_{\max}( \mtx{M} ).
$$
Then, for $\theta > 0$,
\begin{equation} \label{eqn:intdim-chernoff-expect}
\Expect \lambda_{\max}(\mtx{Y}) \leq \frac{\econst^{\theta} - 1}{\theta} \cdot \mu_{\max}
	+ \frac{1}{\theta} \cdot L \, \log(2d).
\end{equation}
Furthermore,
\begin{equation} \label{eqn:intdim-chernoff-upper}
\Prob{ \lambda_{\max}(\mtx{Y}) \geq (1+\eps) \mu_{\max} }
	\leq 2 d \cdot \left[ \frac{\econst^{\eps}}{(1+\eps)^{1+\eps}} \right]^{\mu_{\max}/L}
	\quad\text{for $\eps \geq L/\mu_{\max}$.}
\end{equation}
\end{thm}

\noindent
The proof of this result appears below in~\S\ref{sec:int-chernoff-proof}.

\subsection{Discussion}

Theorem~\ref{thm:intdim-chernoff} is almost identical with the parts of
the basic matrix Chernoff inequality that concern the maximum eigenvalue
$\lambda_{\max}(\mtx{Y})$.
Let us call attention to the differences.  The key advantage is that the current
result depends on the intrinsic dimension of the matrix $\mtx{M}$
instead of the ambient dimension.  When the eigenvalues of $\mtx{M}$ decay,
the improvement can be dramatic.  We do suffer a small cost in the
extra factor of two, and the tail bound is restricted to a smaller range
of the parameter $\eps$.  Neither of these limitations is particularly
significant.

We have chosen to frame the result in terms of the upper bound $\mtx{M}$
because it can be challenging to calculate the mean $\Expect \mtx{Y}$
exactly.  The statement here allows us to draw conclusions directly from the
upper bound $\mtx{M}$.  These estimates do not follow formally from a
result stated for $\Expect \mtx{Y}$ because the intrinsic dimension
is not monotone with respect to the semidefinite order.

A shortcoming of Theorem~\ref{thm:intdim-chernoff} is that it does
not provide any information about $\lambda_{\min}(\mtx{Y})$.
Curiously, the approach we use to prove the result just does not work for
the minimum eigenvalue.

\subsection{Example: A Random Column Submatrix}

To demonstrate the value of Theorem~\ref{thm:intdim-chernoff},
let us return to one of the problems we studied in \S\ref{sec:rdm-submatrix}.
We can now develop a refined estimate for the expected norm of a random
column submatrix drawn from a fixed matrix.

In this example, we consider a fixed $m \times n$ matrix $\mtx{B}$,
and we let $\{ \delta_k \}$ be an independent family of $\textsc{bernoulli}(p/n)$
random variables.  We form the random submatrix
$$
\mtx{Z} = \sum\nolimits_k \delta_k \, \vct{b}_{:k} \mathbf{e}_k^\adj
$$
where $\vct{b}_{:k}$ is the $k$th column of $\mtx{B}$.  This random submatrix
contains an average of $p$ nonzero columns from $\mtx{B}$.
To study the norm of $\mtx{Z}$, we consider the positive-semidefinite
random matrix
$$
\mtx{Y} = \mtx{ZZ}^\adj = \sum_{k=1}^n \delta_k \, \vct{b}_{:k} \vct{b}_{:k}^\adj.
$$
This time, we invoke Theorem~\ref{thm:intdim-chernoff} to obtain a new estimate
for the maximum eigenvalue of $\mtx{Y}$.

We need a semidefinite bound $\mtx{M}$ for the mean $\Expect \mtx{Y}$ of the random matrix.
In this case, the exact value is available:
$$
\mtx{M} = \Expect \mtx{Y} = \frac{p}{n} \, \mtx{BB}^\adj.
$$
We can easily calculate the intrinsic dimension of this matrix:
$$
d = \intdim(\mtx{M})
	= \intdim \left( \frac{p}{n} \mtx{BB}^\adj \right)
	= \intdim( \mtx{BB}^\adj )
	= \frac{\trace(\mtx{BB}^\adj)}{\norm{\mtx{BB}^\adj}}
	= \frac{\fnormsq{\mtx{B}}}{\normsq{\mtx{B}}}
	= \strank(\mtx{B}).
$$
The second identity holds because the intrinsic dimension is scale invariant.
The last relation is simply the definition~\eqref{eqn:stable-rank} of the stable rank.
The maximum eigenvalue of $\mtx{M}$ verifies
$$
\lambda_{\max}(\mtx{M}) = \frac{p}{n} \lambda_{\max}(\mtx{BB}^\adj)
	= \frac{p}{n} \normsq{\mtx{B}}.
$$
The maximum norm $L$ of any term in the sum $\mtx{Y}$ satisfies $L = \max\nolimits_k \normsq{\smash{\vct{b}_{:k}}}$.

We may now apply the intrinsic Chernoff inequality.
The expectation bound~\eqref{eqn:intdim-chernoff-expect} with $\theta = 1$ delivers
$$
\Expect{} \normsq{\mtx{Z}}
	= \Expect \lambda_{\max}(\mtx{Y})
	\leq 1.72 \cdot \frac{p}{n} \cdot \normsq{\mtx{B}}
	+ \log( 2 \strank(\mtx{B}) ) \cdot \max\nolimits_k \normsq{\smash{\vct{b}_{:k}}}.
$$
In the earlier analysis, we obtained a similar bound~\eqref{eqn:column-submatrix}.
The new result depends on the logarithm of the stable rank instead of $\log m$,
the logarithm of the number of rows of $\mtx{B}$.  When the stable rank of $\mtx{B}$
is small---meaning that many rows are almost collinear---then the revised estimate
can result in a substantial improvement.

\section{Matrix Bernstein with Intrinsic Dimension} \label{sec:int-bernstein}

Next, we present an extension of the matrix Bernstein inequality.  These results
provide tail bounds for an independent sum of bounded random matrices that
depend on the intrinsic dimension of the variance.  This theorem is essentially
due to Stanislav Minsker.

\begin{thm}[Intrinsic Matrix Bernstein] \label{thm:intdim-bernstein-rect}
Consider a finite sequence $\{ \mtx{S}_k \}$ of random complex matrices with the same size,
and assume that
$$
\Expect \mtx{S}_k = \mtx{0}
\quad\text{and}\quad
\norm{\mtx{S}_k} \leq L.
$$
Introduce the random matrix
$$
\mtx{Z} = \sum\nolimits_k \mtx{S}_k.
$$
Let $\mtx{V}_1$ and $\mtx{V}_2$ be semidefinite upper bounds for the
matrix-valued variances $\mVar_1(\mtx{Z})$ and $\mVar_2(\mtx{Z})$:
$$
\begin{aligned}
\mtx{V}_1 &\psdge \mVar_1(\mtx{Z}) = \Expect( \mtx{ZZ}^\adj ) = \sum\nolimits_k \Expect \big( \mtx{S}_k\mtx{S}_k^\adj \big),
\quad\text{and} \\
\mtx{V}_2 &\psdge \mVar_2(\mtx{Z}) = \Expect( \mtx{Z}^\adj \mtx{Z} ) = \sum\nolimits_k \Expect\big( \mtx{S}_k^\adj \mtx{S}_k \big).
\end{aligned}
$$
Define an intrinsic dimension bound and a variance bound
\begin{equation} \label{eqn:intdim-var-bernstein-rect}
d = \intdim \begin{bmatrix} \mtx{V}_1 & \mtx{0} \\
	\mtx{0} & \mtx{V}_2 \end{bmatrix}
\quad\text{and}\quad
v = \max\big\{ \norm{\mtx{V}_1}, \ \norm{\mtx{V}_2} \big\}.
\end{equation}
Then, for $t \geq \sqrt{v} + L/3$,
\begin{equation} \label{eqn:intdim-bernstein-tail-rect}
\Prob{ \norm{\mtx{Z}} \geq t }
	\leq 4 d \, \exp\left( \frac{-t^2/2}{v + Lt/3} \right).
\end{equation}
\end{thm}

\noindent
The proof of this result appears below in \S\ref{sec:int-bernstein-proof}.

\subsection{Discussion}

Theorem~\ref{thm:intdim-bernstein-rect} is quite similar to Theorem~\ref{thm:matrix-bernstein-rect},
so we focus on the differences.
Although the statement of Theorem~\ref{thm:intdim-bernstein-rect} may seem circumspect,
it is important to present the result in terms of upper bounds $\mtx{V}_1$ and $\mtx{V}_2$
for the matrix-valued variances.  Indeed, it can be challenging to calculate the
matrix-valued variances exactly.  The fact that the intrinsic dimension
is not monotone interferes with our ability to use a simpler result.  %

Note that the tail bound~\eqref{eqn:intdim-bernstein-tail-rect}
now depends on the intrinsic dimension of the block-diagonal matrix $\diag(\mtx{V}_1, \mtx{V}_2)$.
This intrinsic dimension quantity never exceeds the total of the two side lengths of the random
matrix $\mtx{Z}$.
As a consequence, the new tail bound always has a better dimensional dependence than the
earlier result.
The costs of this improvement are small: We pay an extra factor of four in the probability bound,
and we must restrict our attention to a more limited
range of the parameter $t$.  Neither of these changes is significant.

The result does not contain an explicit estimate for $\Expect \norm{\mtx{Z}}$,
but we can obtain such a bound by integrating the tail
inequality~\eqref{eqn:intdim-bernstein-tail-rect}.
This estimate is similar with the earlier bound~\eqref{eqn:matrix-bernstein-expect-rect},
but it depends on the intrinsic dimension instead of the ambient dimension.

\begin{cor}[Intrinsic Matrix Bernstein: Expectation Bound] \label{cor:intdim-bernstein-expect-rect}
Instate the notation and hypotheses of Theorem~\ref{thm:intdim-bernstein-rect}.
Then
\begin{equation} \label{eqn:intdim-bernstein-tail-int}
\Expect \norm{\mtx{Z}} \leq \mathrm{Const} \cdot
	\left( \sqrt{v \log(1+ d)} + L \, \log(1+ d) \right).
\end{equation}
\end{cor}

\noindent
See \S\ref{sec:int-bernstein-expect-pf} for the proof.

Next, let us have a closer look at the intrinsic dimension quantity defined in~\eqref{eqn:intdim-var-bernstein-rect}.
$$
d = \frac{ \trace \mtx{V}_1 + \trace \mtx{V}_2  }
	{\max\big\{ \norm{ \mtx{V}_1 }, \ \norm{ \mtx{V}_2 } \big\}}.
$$
We can make a further bound on the denominator to obtain an estimate in terms
of the intrinsic dimensions of the two blocks:
\begin{equation} \label{eqn:intdim-bernstein-intdim-bounds}
\min\big\{ \intdim(\mtx{V}_1), \ \intdim(\mtx{V}_2) \big\}
	\quad\leq\quad d \quad\leq\quad \intdim(\mtx{V}_1) + \intdim(\mtx{V}_2).
\end{equation}
This bound reflects a curious phenomenon: the intrinsic dimension parameter
$d$ is not necessarily comparable with the larger
of $\intdim(\mtx{V}_1)$ or $\intdim(\mtx{V}_2)$.

The other commentary about the original matrix Bernstein inequality, Theorem~\ref{thm:matrix-bernstein-rect}, also applies to the intrinsic dimension result.  For example, we can adapt the result to a sum of uncentered, independent, random, bounded matrices.  In addition, the theorem becomes somewhat simpler for a Hermitian random matrix because there is only one matrix-valued variance to deal with.  The modifications required in these cases are straightforward.

\subsection{Example: Matrix Approximation by Random Sampling}

We can apply the intrinsic Bernstein inequality to study the behavior of
randomized methods for matrix approximation.  The following result
is an immediate consequence of Theorem~\ref{thm:intdim-bernstein-rect}
and Corollary~\ref{cor:intdim-bernstein-expect-rect}.

\begin{cor}[Matrix Approximation by Random Sampling: Intrinsic Dimension Bounds]
\label{cor:matrix-approx-sampling-intdim}
Let $\mtx{B}$ be a fixed $d_1 \times d_2$ matrix.
Construct a $d_1 \times d_2$ random matrix $\mtx{R}$ that satisfies
$$
\Expect \mtx{R} = \mtx{B}
\quad\text{and}\quad
\norm{\mtx{R}} \leq L.
$$
Let $\mtx{M}_1$ and $\mtx{M}_2$ be semidefinite upper bounds for the expected squares:
$$
\mtx{M}_1 \psdge \Expect( \mtx{RR}^\adj )
\quad\text{and}\quad
\mtx{M}_2 \psdge \Expect( \mtx{R}^\adj \mtx{R}).
$$
Define the quantities
$$
d = \intdim \begin{bmatrix} \mtx{M}_1 & \mtx{0} \\ \mtx{0} & \mtx{M}_2 \end{bmatrix}
\quad\text{and}\quad
m = \max\big\{ \norm{ \mtx{M}_1 }, \ \norm{\mtx{M}_2} \big\}.
$$
Form the matrix sampling estimator
$$
\bar{\mtx{R}}_n = \frac{1}{n} \sum_{k=1}^n \mtx{R}_k
\quad\text{where each $\mtx{R}_k$ is an independent copy of $\mtx{R}$.}
$$
Then the estimator satisfies
\begin{equation} \label{eqn:matrix-approx-err-expect}
\Expect \norm{ \smash{\bar{\mtx{R}}_n - \mtx{B}} }
	\leq \mathrm{Const} \cdot \left[ \sqrt{\frac{m \log(1 + d)}{n}} + \frac{L \log(1 + d)}{n} \right].
\end{equation}
Furthermore, for all $t \geq \sqrt{m} + L/3$,
\begin{equation} \label{eqn:matrix-approx-err-tail}
\Prob{ \norm{ \smash{\bar{\mtx{R}}_n - \mtx{B}} } \geq t }
	\leq 4d \exp\left( \frac{-nt^2/2}{m + 2Lt/3} \right).
\end{equation}
\end{cor}

\noindent
The proof is similar with that of Corollary~\ref{cor:matrix-approx-sampling},
so we omit the details.

\subsection{Application: Randomized Matrix Multiplication}

We will apply Corollary~\ref{cor:matrix-approx-sampling-intdim} to study
the randomized matrix multiplication algorithm from~\S\ref{sec:rdm-mtx-mult}.
This method results in a small, but very appealing, improvement in the
number of samples that are required.  This argument is essentially
due to Tassos Zouzias~\cite{Zou12:Randomized-Primitives}.

Our goal is to approximate the product of a $d_1 \times N$ matrix $\mtx{B}$
and an $N \times d_2$ matrix $\mtx{C}$.  We assume that both matrices
$\mtx{B}$ and $\mtx{C}$ have unit spectral norm.  The results are stated
in terms of the average stable rank
$$
\mathsf{asr} = \half( \strank(\mtx{B}) + \strank(\mtx{C})).
$$
The stable rank was introduced in~\eqref{eqn:stable-rank}.
To approximate the product $\mtx{BC}$, we constructed a simple
random matrix $\mtx{R}$ whose mean $\Expect \mtx{R} = \mtx{BC}$,
and then we formed the estimator
$$
\bar{\mtx{R}}_n = \frac{1}{n} \sum_{k=1}^n \mtx{R}_k
\quad\text{where each $\mtx{R}_k$ is an independent copy of $\mtx{R}$.}
$$
The challenge is to bound the error $\norm{ \smash{\bar{\mtx{R}}_n - \mtx{BC}} }$.

To do so, let us refer back to our calculations from \S\ref{sec:rdm-mtx-mult}.
We find that
$$
\begin{aligned}
\norm{\mtx{R}} &\leq \mathsf{asr}, \\
\Expect(\mtx{RR}^\adj) &\psdle 2 \cdot \mathsf{asr} \cdot \mtx{BB}^\adj, \quad\text{and} \\
\Expect(\mtx{R}^\adj \mtx{R}) &\psdle 2 \cdot \mathsf{asr} \cdot \mtx{C}^\adj \mtx{C}.
\end{aligned}
$$
Starting from this point, we can quickly improve on our earlier
analysis by incorporating the intrinsic dimension bounds.

It is natural to set $\mtx{M}_1 = 2 \cdot \mathsf{asr} \cdot \mtx{BB}^\adj$
and $\mtx{M}_2 = 2 \cdot \mathsf{asr} \cdot \mtx{C}^\adj \mtx{C}$.
We may now bound the intrinsic dimension parameter
$$
\begin{aligned}
d &= \intdim \begin{bmatrix} \mtx{M}_1 & \mtx{0} \\ \mtx{0} & \mtx{M}_2 \end{bmatrix}
	\leq \intdim(\mtx{M}_1) + \intdim(\mtx{M}_2) \\
	&= \frac{\trace(\mtx{BB}^\adj)}{\norm{\mtx{BB}^\adj}} + \frac{\trace(\mtx{C}^\adj \mtx{C})}{\norm{\mtx{C}^\adj \mtx{C}}} 
	= \frac{\fnormsq{\mtx{B}}}{\normsq{\mtx{B}}} + \frac{\fnormsq{\mtx{C}}}{\normsq{\mtx{C}}} \\
	&= \strank(\mtx{B}) + \strank(\mtx{C}) = 2 \cdot \mathsf{asr}.
\end{aligned}
$$
The first inequality follows from~\eqref{eqn:intdim-bernstein-intdim-bounds},
and the second is Definition~\ref{def:int-dim}, of the intrinsic dimension.
The third relation depends on the norm identities~\eqref{eqn:trace-Frobenius}
and~\eqref{eqn:spectral-norm-square}.  Finally, we identify
the stable ranks of $\mtx{B}$ and $\mtx{C}$ and the average stable rank.
The calculation of the quantity $m$ proceeds from the same considerations
as in \S\ref{sec:rdm-mtx-mult}.  Thus,
$$
m = \max\big\{ \norm{\mtx{M}_1}, \ \norm{\mtx{M}_2} \big\}
	= 2 \cdot \mathsf{asr}.
$$
This is all the information we need to collect.

Corollary~\ref{cor:matrix-approx-sampling-intdim} now implies that
$$
\Expect \norm{\smash{\bar{\mtx{R}}_n - \mtx{BC}}}
	\leq \mathrm{Const} \cdot \left( \sqrt{\frac{\mathsf{asr} \cdot \log(1 + \mathsf{asr})}{n}}
	+ \frac{ \mathsf{asr} \cdot \log(1 + \mathsf{asr}) }{ n } \right).
$$
In other words, if the number $n$ of samples satisfies
$$
n \geq \eps^{-2} \cdot \mathsf{asr} \cdot \log(1 + \mathsf{asr}),
$$
then the error satisfies
$$
\Expect{} \norm{ \smash{\bar{\mtx{R}}_n - \mtx{BC}}}
	\leq \mathrm{Const} \cdot \big( \eps + \eps^2 \big).
$$
In the original analysis from \S\ref{sec:rdm-mtx-mult},
our estimate for the number $n$ of samples contained the
term $\log(d_1 + d_2)$ instead of $\log(1 + \mathsf{asr})$.
We have replaced the dependence on the
ambient dimension of the product $\mtx{BC}$ by a measure of the
stable rank of the two factors.
When the average stable rank is small in comparison with the dimension of the product,
the analysis based on the intrinsic dimension offers an improvement in 
the bound on the number of samples required to approximate the product.

\section{Revisiting the Matrix Laplace Transform Bound} \label{sec:new-laplace}

Let us proceed with the proofs of the matrix concentration inequalities
based on intrinsic dimension.  The challenge is to identify and
remedy the weak points in the arguments from Chapter~\ref{chap:matrix-lt}.

After some reflection, we can trace the dependence on the ambient dimension
in our earlier results to the
proof of Proposition~\ref{prop:matrix-lt}.
In the original argument, we used an exponential function to transform the tail event
before applying Markov's inequality.
This approach leads to trouble for the simple reason that the exponential function
does not pass through the origin, which gives undue weight to eigenvalues that are
close to zero.

We can resolve this problem by using other types of maps to
transform the tail event.
The functions we have in mind are adjusted versions of the exponential.
In particular, for fixed $\theta > 0$, we can consider
$$
\psi_1(t) = \max\bigl\{ 0, \ \econst^{\theta t} - 1 \bigr\}
\quad\text{and}\quad
\psi_2(t) = \econst^{\theta t} - \theta t - 1.
$$
Both functions are nonnegative and convex, and they are nondecreasing
on the positive real line.  In each case, $\psi_i(0) = 0$.
At the same time, the presence of the exponential function allows
us to exploit our bounds for the trace mgf.

\begin{prop}[Generalized Matrix Laplace Transform Bound] \label{prop:matrix-lt-general}
Let $\mtx{Y}$ be a random Hermitian matrix.
Let $\psi : \R \to \R_+$ be a nonnegative function
that is nondecreasing on $[0,\infty)$.
For each $t \geq 0$,
$$
\Prob{ \lambda_{\max}(\mtx{Y}) \geq t }
	\leq \frac{1}{\psi(t)} \, \Expect \trace \psi(\mtx{Y}).
$$
\end{prop}

\begin{proof}
The proof follows the same lines as the proof of Proposition~\ref{prop:matrix-lt},
but it requires some additional finesse.  Since $\psi$ is nondecreasing
on $[0, \infty)$, the bound $a \geq t$ implies that $\psi(a) \geq \psi(t)$.
As a consequence,
$$
\lambda_{\max}(\mtx{Y}) \geq t
\quad\text{implies}\quad
\lambda_{\max}(\psi(\mtx{Y})) \geq \psi(t).
$$
Indeed, on the tail event $\lambda_{\max}(\mtx{Y}) \geq t$, we must have $\psi(\lambda_{\max}(\mtx{Y})) \geq \psi(t)$.  The Spectral Mapping Theorem, Proposition~\ref{prop:spectral-mapping},
indicates that the number $\psi(\lambda_{\max}(\mtx{Y}))$ is one of the eigenvalues of the matrix $\psi(\mtx{Y})$,
so we determine that $\lambda_{\max}(\psi(\mtx{Y}))$ also exceeds $\psi(t)$.

Returning to the tail probability, we discover that
$$
\Prob{ \lambda_{\max}(\mtx{Y}) \geq t }
	\leq \Prob{ \lambda_{\max}(\psi(\mtx{Y})) \geq \psi(t) }
	\leq \frac{1}{\psi(t)} \, \Expect \lambda_{\max}(\psi(\mtx{Y})).
$$
The second bound is Markov's inequality~\eqref{eqn:markov}, which is valid because $\psi$ is nonnegative.
Finally,
$$
\Prob{ \lambda_{\max}(\mtx{Y}) \geq t }
	\leq \frac{1}{\psi(t)} \, \Expect \trace \psi(\mtx{Y}).
$$
The inequality holds because of the fact~\eqref{eqn:maxeig-trace}
that the trace of $\psi(\mtx{Y})$, a positive-semidefinite matrix,
must be at least as large as its maximum eigenvalue.
\end{proof}

\section{The Intrinsic Dimension Lemma} \label{sec:int-dim-lemma}

The other new ingredient is a simple observation that allows us to control
a trace function applied to a positive-semidefinite matrix
in terms of the intrinsic dimension of the matrix.

\begin{lemma}[Intrinsic Dimension] \label{lem:intrinsic-dim}
Let $\phi$ be a convex function on the interval $[0, \infty)$,
and assume that $\phi(0) = 0$.
For any positive-semidefinite matrix $\mtx{A}$, it holds that
$$
\trace \phi(\mtx{A}) \leq \intdim(\mtx{A}) \cdot \phi(\norm{\mtx{A}}).
$$
\end{lemma}

\begin{proof}
Since the function $a \mapsto \phi(a)$ is convex on the interval $[0, L]$, it is bounded above by the chord
connecting the graph at the endpoints.  That is, for $a \in [0,L]$, 
$$
\phi(a) \leq \left(1 - \frac{a}{L} \right) \cdot \phi(0) + \frac{a}{L} \cdot \phi(L)
	= \frac{a}{L} \cdot \phi(L).
$$
The eigenvalues of $\mtx{A}$ fall in the interval $[0, L]$, where $L = \norm{\mtx{A}}$.
As an immediate consequence of the Transfer Rule~\eqref{eqn:transfer-rule}, we find that
$$
\trace \phi(\mtx{A}) \leq \frac{\trace \mtx{A}}{\norm{\mtx{A}}} \cdot \phi(\norm{\mtx{A}}).
$$
Identify the intrinsic dimension of $\mtx{A}$ to complete the argument.
\end{proof}

\section{Proof of the Intrinsic Chernoff Bound} \label{sec:int-chernoff-proof}

With these results at hand, we are prepared to prove our first intrinsic dimension
result, which extends the matrix Chernoff inequality.

\begin{proof}[Proof of Theorem~\ref{thm:intdim-chernoff}]

Consider a finite sequence $\{\mtx{X}_k\}$ of independent, random Hermitian matrices with
$$
0 \leq \lambda_{\min}(\mtx{X}_k)
\quad\text{and}\quad
\lambda_{\max}(\mtx{X}_k) \leq L
\quad\text{for each index $k$.}
$$
Introduce the sum
$$
\mtx{Y} = \sum\nolimits_k \mtx{X}_k.
$$
The challenge is to establish bounds for $\lambda_{\max}(\mtx{Y})$
that depend on the intrinsic dimension of a matrix $\mtx{M}$ that
satisfies $\mtx{M} \psdge \Expect \mtx{Y}$.
We begin the argument with the proof of the tail bound~\eqref{eqn:intdim-chernoff-upper}.
Afterward, we show how to extract the expectation bound~\eqref{eqn:intdim-chernoff-expect}.

Fix a number $\theta > 0$, and define the function $\psi(t) = \max\big\{0, \ \econst^{\theta t} - 1 \big\}$
for $t \in \R$.
For $t \geq 0$,
the general version of the matrix Laplace transform bound, Proposition~\ref{prop:matrix-lt-general},
states that
\begin{equation} \label{eqn:intdim-chernoff-pf-1}
\Prob{ \lambda_{\max}(\mtx{Y}) \geq t }
	\leq \frac{1}{\psi(t)} \Expect \trace \psi(\mtx{Y})
	= \frac{1}{\econst^{\theta t} - 1} \Expect \trace\bigl( \econst^{\theta \mtx{Y}} - \Id \bigr).
\end{equation}
We have exploited the fact that $\mtx{Y}$ is positive semidefinite and the assumption that $t \geq 0$.
The presence of the identity matrix on the right-hand side allows us to draw stronger
conclusions than we could before.

Let us study the expected trace term on the right-hand side of~\eqref{eqn:intdim-chernoff-pf-1}.
As in the proof of the original matrix Chernoff bound, Theorem~\ref{thm:matrix-chernoff},
we have the estimate
$$
\Expect \trace \econst^{\theta \mtx{Y}}
	\leq \trace \exp\left(  g(\theta) \cdot \Expect \mtx{Y} \right)
	\quad\text{where}\quad
	g(\theta) = \frac{\econst^{\theta L} - 1}{L}.
$$
Introduce the function $\phi(a) = \econst^{a} - 1$, and observe that
$$
\Expect \trace \bigl( \econst^{\theta \mtx{Y}} - \Id \bigr)
	\leq \trace \bigl( \econst^{g(\theta) \cdot \Expect \mtx{Y}} - \Id \bigr)
	\leq \trace \bigl( \econst^{g(\theta) \cdot \mtx{M}} - \Id \bigr)
	= \trace \phi( g(\theta) \cdot \mtx{M} ).
$$
The second inequality follows from the assumption that $\Expect \mtx{Y} \psdle \mtx{M}$
and the monotonicity~\eqref{eqn:exp-trace-monotone} of the trace exponential.
Now, apply the intrinsic dimension bound, Lemma~\ref{lem:intrinsic-dim}, to reach
$$
\Expect \trace\bigl( \econst^{\theta \mtx{Y}} - \Id \bigr) 
	\leq \intdim(\mtx{M}) \cdot \phi\left( g(\theta) \norm{ \mtx{M} } \right).
$$
We have used the fact that the intrinsic dimension does not depend on the scaling factor $g(\theta)$.
Recalling the notation $d = \intdim(\mtx{M})$ and $\mu_{\max} = \norm{\mtx{M}}$,
we continue the calculation:
\begin{equation} \label{eqn:intdim-chernoff-pf-2}
\Expect \trace\bigl( \econst^{\theta \mtx{Y}} - \Id \bigr)
	\leq d \cdot \phi\left( g(\theta) \cdot \mu_{\max} \right)
	\leq d \cdot \econst^{ g(\theta) \cdot \mu_{\max} }.
\end{equation}
We have invoked the trivial inequality $\phi(a) \leq \econst^a$, which holds for $a \in \R$.

Next, introduce the bound~\eqref{eqn:intdim-chernoff-pf-2} on the expected trace into
the probability bound~\eqref{eqn:intdim-chernoff-pf-1} to obtain
\begin{equation} \label{eqn:intdim-chernoff-pf-2.5}
\Prob{ \lambda_{\max}(\mtx{Y}) \geq t }
	\leq d \cdot \frac{\econst^{\theta t}}{\econst^{\theta t} - 1}
	\cdot \econst^{ -\theta t + g(\theta) \cdot \mu_{\max} }
	\leq d \cdot \left(1+ \frac{1}{\theta t}\right) \cdot
	\econst^{-\theta t + g(\theta) \cdot \mu_{\max}}.
\end{equation}
To control the fraction, we have observed that
$$
\frac{\econst^a}{\econst^a - 1}
	= 1 + \frac{1}{\econst^a - 1}
	\leq 1 + \frac{1}{a}
	\quad\text{for $a \geq 0$.}
$$
We obtain the latter inequality by replacing the convex function
$a \mapsto \econst^a - 1$ with its tangent line at $a = 0$.

In the estimate~\eqref{eqn:intdim-chernoff-pf-2.5},
we make the change of variables $t \mapsto (1+\eps) \mu_{\max}$.
The bound is valid for all $\theta > 0$,
so we can select $\theta = L^{-1} \log(1+\eps)$ to minimize the exponential.
Altogether, these steps lead to the estimate
\begin{equation} \label{eqn:intdim-chernoff-pf-3}
\Prob{ \lambda_{\max}(\mtx{Y}) \geq (1+\eps)\mu_{\max} }
	\leq d \cdot \left( 1 + \frac{L/\mu_{\max}}{(1+\eps) \log(1+\eps)} \right)
	\cdot \left[ \frac{\econst^\eps}{(1+\eps)^{1+\eps}} \right]^{\mu_{\max}/L}.
\end{equation}
Now, instate the assumption that $\eps \geq L/\mu_{\max}$.  The function
$a \mapsto (1 + a) \log(1+a)$ is convex when $a \geq -1$,
so we can bound it below using its tangent at $\eps = 0$.  Thus,
$$
(1 + \eps) \log (1 + \eps) \geq \eps \geq \frac{L}{\mu_{\max}}.
$$
It follows that the parenthesis in~\eqref{eqn:intdim-chernoff-pf-3} is bounded by two,
which yields the conclusion~\eqref{eqn:intdim-chernoff-upper}.

Now, we turn to the expectation bound~\eqref{eqn:intdim-chernoff-expect}.
Observe that the functional inverse of $\psi$ is the increasing
concave function
$$
\psi^{-1}(u) = \frac{1}{\theta} \log(1 + u)
\quad\text{for $u \geq 0$.}
$$
Since $\mtx{Y}$ is a positive-semidefinite matrix, we can calculate that
\begin{align} \label{eqn:intdim-chernoff-pf-4}
\Expect \lambda_{\max}(\mtx{Y})
	= \Expect \psi^{-1}( \psi( \lambda_{\max}(\mtx{Y}) ) )
	&\leq \psi^{-1}( \Expect \psi(\lambda_{\max}(\mtx{Y}) ) ) \notag \\
	&= \psi^{-1}( \Expect \lambda_{\max}(\psi(\mtx{Y})))
	\leq \psi^{-1}( \Expect \trace \psi(\mtx{Y}) ).
\end{align}
The second relation is Jensen's inequality~\eqref{eqn:jensen},
which is valid because $\psi^{-1}$ is concave.
The third relation follows from the Spectral Mapping Theorem,
Proposition~\ref{prop:spectral-mapping}, because the function $\psi$
is increasing.  We can bound the maximum eigenvalue by the
trace because $\psi(\mtx{Y})$ is positive semidefinite and
$\psi^{-1}$ is an increasing function.

Now, substitute the bound~\eqref{eqn:intdim-chernoff-pf-2} into
the last display~\eqref{eqn:intdim-chernoff-pf-4}
to reach
\begin{align*}
\Expect \lambda_{\max}(\mtx{Y})
	\leq \psi^{-1}( d \cdot \exp( g(\theta) \cdot \mu_{\max} ) )
	&= \frac{1}{\theta} \log\bigl(1 + d \cdot \econst^{g(\theta) \cdot \mu_{\max}} \bigr) \\
	&\leq \frac{1}{\theta} \log\bigl( 2d \cdot \econst^{g(\theta) \cdot \mu_{\max}} \bigr)
	= \frac{1}{\theta} \left( \log(2d) + g(\theta) \cdot \mu_{\max} \right).
\end{align*}
The first inequality again requires the property that $\psi^{-1}$ is increasing.
The second inequality follows because $1 \leq d \cdot \econst^{g(\theta) \cdot \mu_{\max}}$,
which is a consequence of the fact that the intrinsic dimension exceeds one and the exponent is nonnegative.
To complete the argument,
introduce the definition of $g(\theta)$, and make the change of variables
$\theta \mapsto \theta / L$.  These steps yield~\eqref{eqn:intdim-chernoff-expect}.
\end{proof}

\section{Proof of the Intrinsic Bernstein Bounds} \label{sec:int-bernstein-proof}

In this section, we present the arguments that lead up to the intrinsic Bernstein
bounds.  That is, we develop tail inequalities for an independent sum of bounded
random matrices that depend on the intrinsic dimension of the variance.

\subsection{The Hermitian Case} %

As usual, Hermitian matrices provide the natural setting for matrix concentration.
We begin with an explicit statement and proof of a bound for the Hermitian case.

\begin{thm}[Matrix Bernstein: Hermitian Case with Intrinsic Dimension] \label{thm:intdim-bernstein-herm}
Consider a finite sequence $\{ \mtx{X}_k \}$ of random Hermitian matrices of the same size,
and assume that
$$
\Expect \mtx{X}_k = \mtx{0}
\quad\text{and}\quad
\lambda_{\max}(\mtx{X}_k) \leq L
\quad\text{for each index $k$.}
$$
Introduce the random matrix
$$
\mtx{Y} = \sum\nolimits_k \mtx{X}_k.
$$
Let $\mtx{V}$ be a semidefinite upper bound for the matrix-valued variance $\mVar(\mtx{Y})$:
$$
\mtx{V} \psdge \mVar(\mtx{Y}) = \Expect{} \mtx{Y}^2 = \sum\nolimits_k \Expect{} \mtx{X}_k^2.
$$
Define the intrinsic dimension bound and variance bound
$$
d = \intdim(\mtx{V})
\quad\text{and}\quad
v = \norm{\mtx{V}}.
$$
Then, for $t \geq \sqrt{v} + L/3$,
\begin{equation} \label{eqn:intdim-bernstein-tail}
\Prob{ \lambda_{\max}(\mtx{Y}) \geq t }
	\leq 4 d \cdot \exp\left( \frac{-t^2/2}{v + Lt/3} \right).
\end{equation}
\end{thm}

\noindent
The proof of this result appears in the next section. %

\subsection{Proof of the Hermitian Case}

We commence with the results for an independent sum of random Hermitian matrices
whose eigenvalues are subject to an upper bound.

\begin{proof}[Proof of Theorem~\ref{thm:intdim-bernstein-herm}]
Consider a finite sequence $\{\mtx{X}_k\}$ of independent, random, Hermitian matrices with
$$
\Expect \mtx{X}_k = \mtx{0}
\quad\text{and}\quad
\lambda_{\max}(\mtx{X}_k) \leq L
\quad\text{for each index $k$.}
$$
Introduce the random matrix
$$
\mtx{Y} = \sum\nolimits_k \mtx{X}_k.
$$
Our goal is to obtain a tail bound for $\lambda_{\max}(\mtx{Y})$
that reflects the intrinsic dimension of a matrix $\mtx{V}$
that satisfies $\mtx{V} \psdge \mVar(\mtx{Y})$.

Fix a number $\theta > 0$, and define the function $\psi(t) = \econst^{\theta t} - \theta t - 1$ for $t \in \R$.
The general version of the matrix Laplace transform bound, Proposition~\ref{prop:matrix-lt-general}, implies that
\begin{equation} \label{eqn:intdim-gauss-pf-1}
\begin{aligned}
\Prob{ \lambda_{\max}(\mtx{Y}) \geq t }
	&\leq \frac{1}{\psi(t)} \Expect \trace \psi(\mtx{Y}) \\
	&= \frac{1}{\psi(t)}
	\Expect \trace\bigl( \econst^{\theta \mtx{Y}} - \theta \mtx{Y} - \Id \bigr) \\
	&= \frac{1}{\econst^{\theta t} - \theta t - 1}
	\Expect \trace\bigl( \econst^{\theta \mtx{Y}} - \Id \bigr).
\end{aligned}
\end{equation}
The last identity holds because the random matrix $\mtx{Y}$ has zero mean.

Let us focus on the expected trace on the right-hand side of~\eqref{eqn:intdim-gauss-pf-1}.
Examining the proof of the original matrix Bernstein bound for Hermitian matrices,
Theorem~\ref{thm:matrix-bernstein-herm}, we see that
$$
\Expect \trace \econst^{\theta \mtx{Y}}
	\leq \trace \exp\left( g(\theta) \cdot \Expect{} \mtx{Y}^2 \right)
	\quad\text{where}\quad
	g(\theta) = \exp\left( \frac{\theta^2/2}{1 - \theta L/3} \right).
$$
Introduce the function $\phi(a) = \econst^a - 1$, and observe that
\begin{equation*} %
\Expect\trace \bigl( \econst^{\theta \mtx{Y}} - \Id \bigr)
	\leq \trace \bigl( \econst^{g(\theta) \cdot \Expect{} \mtx{Y}^2} - \Id \bigr)
	\leq \trace \bigl( \econst^{g(\theta) \cdot \mtx{V}} - \Id \bigr)
	= \trace \phi\left(g(\theta)  \cdot \mtx{V} \right).
\end{equation*}
The second inequality depends on the assumption that $\Expect{}\mtx{Y}^2 = \mVar(\mtx{Y}) \psdle\mtx{V}$
and the monotonicity property~\eqref{eqn:exp-trace-monotone} of the trace exponential.
Apply the intrinsic dimension bound, Lemma~\ref{lem:intrinsic-dim}, to reach
\begin{equation} \label{eqn:intdim-gauss-pf-2}
\Expect\trace \bigl( \econst^{\theta \mtx{Y}} - \Id \bigr)
	\leq \intdim(\mtx{V}) \cdot \phi\bigl( g(\theta) \cdot \norm{\mtx{V}} \bigr)
	= d \cdot \phi \bigl( g(\theta) \cdot v \bigr)
	\leq d \cdot \econst^{g(\theta) \cdot v}.
\end{equation}
We have used the fact that the intrinsic dimension is scale invariant.
Then we identified $d = \intdim(\mtx{V})$ and $v = \norm{\mtx{V}}$.
The last inequality depends on the trivial estimate $\phi(a) \leq \econst^a$,
valid for all $a \in \R$.

Substitute the bound~\eqref{eqn:intdim-gauss-pf-2} into the probability inequality~\eqref{eqn:intdim-gauss-pf-1}
to discover that
\begin{equation} \label{eqn:intdim-gauss-pf-3}
\Prob{ \lambda_{\max}(\mtx{Y}) \geq t }
	\leq d \cdot \frac{\econst^{\theta t}}{\econst^{\theta t} - \theta t - 1}
	\cdot \econst^{- \theta t + g(\theta) \cdot v}
	\leq d \cdot \left( 1 + \frac{3}{\theta^2 t^2} \right)
	\cdot \econst^{- \theta t + g(\theta) \cdot v}.
\end{equation}
This estimate holds for any positive value  of $\theta$.
To control the fraction, we have observed that
\begin{equation*} %
\frac{\econst^{a}}{\econst^{a} - a - 1}
	= 1 + \frac{1 + a}{\econst^{a} - a - 1}
	\leq 1 + \frac{3}{a^2}
	\quad\text{for all $a \geq 0$.}
\end{equation*}
The inequality above is a consequence of the numerical fact
$$
\frac{\econst^a - a - 1}{a^2} - \frac{1+a}{3} > 0
\quad\text{for all $a \in \R$.} 
$$
Indeed, the left-hand side of the latter expression defines a convex function of $a$, whose minimal value, attained near $a \approx 1.30$, is strictly positive.

In the tail bound~\eqref{eqn:intdim-gauss-pf-3}, we select $\theta = t/(v + Lt/3)$ to reach
$$
\Prob{ \lambda_{\max}(\mtx{Y}) \geq t }
	\leq d \cdot \left(1 + 3 \cdot \frac{(v + Lt/3)^2}{t^4} \right) \cdot 
	\exp\left( \frac{-t^2/2}{v + Lt/3} \right).
$$
This probability inequality is typically vacuous when $t^2 < v + Lt/3$, so we may as
well limit our attention to the case where $t^2 \geq v + Lt/3$.  Under this assumption,
the parenthesis is bounded by four, which gives the tail bound~\eqref{eqn:intdim-bernstein-tail}.
We can simplify the restriction on $t$ by solving the quadratic inequality to obtain the sufficient condition
$$
t \geq \frac{1}{2} \left[ \frac{L}{3} + \sqrt{\frac{L^2}{9} + 4 v} \right].
$$
We develop an upper bound for the right-hand side of this inequality as follows.
$$
\begin{aligned}
\frac{1}{2} \left[ \frac{L}{3} + \sqrt{\frac{L^2}{9} + 4 v} \right]
	= \frac{L}{6} \left[ 1 + \sqrt{1 + \frac{36v}{L^2}} \right]
	\leq \frac{L}{6} \left[ 1 + 1 + \frac{6\sqrt{v}}{L} \right]
	= \sqrt{v} + \frac{L}{3}.
\end{aligned}
$$
We have used the numerical fact $\sqrt{a + b} \leq \sqrt{a} + \sqrt{b}$ for all $a, b \geq 0$.
Therefore, the tail bound~\eqref{eqn:intdim-bernstein-tail} is valid when $t \geq \sqrt{v} + L/3$.
\end{proof}

\subsection{Proof of the Rectangular Case}

Finally, we present the proof of the intrinsic Bernstein inequality, Theorem~\ref{thm:intdim-bernstein-rect},
for general random matrices.

\begin{proof}[Proof of Theorem~\ref{thm:intdim-bernstein-rect}]
Suppose that $\{\mtx{S}_k\}$ is a finite sequence of independent random matrices that satisfy
$$
\Expect \mtx{S}_k = \mtx{0}
\quad\text{and}\quad
\norm{ \mtx{S}_k } \leq L
\quad\text{for each index $k$.}
$$
Form the sum $\mtx{Z} = \sum_k \mtx{S}_k$.
As in the proof of Theorem~\ref{thm:matrix-bernstein-rect}, we derive the result by applying
Theorem~\ref{thm:intdim-bernstein-herm} to the Hermitian dilation $\mtx{Y} = \coll{H}(\mtx{Z})$.
The only new point that requires attention is the modification to the intrinsic dimension and
variance terms.

Recall the calculation of the variance of the dilation from~\eqref{eqn:var-dilation}:
$$
\Expect{} \mtx{Y}^2 = \Expect{} \coll{H}(\mtx{Z})^2
	= \begin{bmatrix} \mVar_1(\mtx{Z}) & \mtx{0} \\ \mtx{0} & \mVar_2(\mtx{Z}) \end{bmatrix}
	\psdle \begin{bmatrix} \mtx{V}_1 & \mtx{0} \\ \mtx{0} & \mtx{V}_2 \end{bmatrix}
	= \mtx{V}.
$$
The semidefinite inequality follows from our assumptions on $\mtx{V}_1$ and $\mtx{V}_2$.  Therefore,
the intrinsic dimension quantity in Theorem~\ref{thm:intdim-bernstein-herm} induces the definition in the general case:
$$
d = \intdim(\mtx{V}) = \intdim \begin{bmatrix} \mtx{V}_1 & \mtx{0} \\ \mtx{0} & \mtx{V}_2 \end{bmatrix}
$$
Similarly,
$$
v = \norm{ \mtx{V} } = \max\big\{ \norm{\mtx{V}_1}, \ \norm{\mtx{V}_2} \big\}.
$$
This point completes the argument.
\end{proof}

\subsection{Proof of the Intrinsic Bernstein Expectation Bound}
\label{sec:int-bernstein-expect-pf}

Finally, let us establish the expectation bound, Corollary~\ref{cor:intdim-bernstein-expect-rect},
that accompanies Theorem~\ref{thm:intdim-bernstein-rect}.

\begin{proof}[Proof of Corollary~\ref{cor:intdim-bernstein-expect-rect}]
Fix a number $\mu \geq \sqrt{v}$.
We may rewrite the expectation of $\norm{\mtx{Z}}$ as an integral:
\begin{align*}
\Expect{} \norm{\mtx{Z}} = \int_0^\infty \Prob{ \norm{\mtx{Z}} \geq t } \idiff{t}
	&\leq \mu + 4d \int_{\mu}^\infty \exp\left( \frac{-t^2/2}{v + Lt /3} \right) \idiff{t} \\
	&\leq \mu + 4d \int_{\mu}^\infty \econst^{ -t^2/(2v) } \idiff{t}
	+ 4d \int_{\mu}^\infty \econst^{ -3t/(2L) } \idiff{t} \\
	&\leq \mu + 4d \sqrt{v} \, \econst^{-\mu^2/(2v)} + \frac{8}{3} d L \, \econst^{-3\mu/(2L)}.
\end{align*}
To obtain the first inequality, we split the integral at $\mu$, and we bound the probability
by one on the domain of integration $[0, \mu]$.  The second inequality holds because
$$
\exp\left( \frac{-t^2/2}{v + Lt /3} \right)
	\leq \max\left\{ \econst^{-t^2/(2v)}, \ \econst^{-3t/(2L)} \right\}
	\leq \econst^{-t^2/(2v)} + \econst^{-3t/(2L)}.
$$
We controlled the Gaussian integral by inserting the factor $\sqrt{v} \, (t/v) \geq 1$
into the integrand:
$$
\int_{\mu}^\infty \econst^{ -t^2/(2v) } \idiff{t}
	\leq \sqrt{v} \int_{\mu}^{\infty} (t/v) \, \econst^{-t^2/(2v)} \idiff{t}
	= \sqrt{v}\, \econst^{-\mu^2/(2v)}.
$$
To complete the argument, select $\mu = \sqrt{2v \log(1 + d)} + \frac{2}{3} L \log(1+ d)$
to reach
$$
\begin{aligned}
\Expect{} \norm{\mtx{Z}} &\leq \mu %
	+ 4d \sqrt{v} \, \econst^{-(2v \log(1+d))/(2v)}
	+ \frac{8}{3} d L \, \econst^{-3 ((2/3) L \log(1 + d))/(2L)} \\
	&\leq \sqrt{2v \log(1 + d)} + \frac{2}{3} L \log(1+ d) + 4 \sqrt{v} + \frac{8}{3} L.
\end{aligned}
$$
The stated bound~\eqref{eqn:intdim-bernstein-tail-int} follows after we combine
terms and agglomerate constants.
\end{proof}

\section{Notes}

At present, there are two different ways to improve the dimensional factor that appears in matrix concentration inequalities.

First, there is a sequence of matrix concentration results where the dimensional parameter is bounded by the maximum rank of the random matrix.  The first bound of this type is due to Rudelson~\cite{Rud99:Random-Vectors}.  Oliveira's results in~\cite{Oli10:Sums-Random} also exhibit
this reduced dimensional dependence.  A subsequent paper~\cite{MZ11:Low-Rank-Matrix-Valued}
by Magen \& Zouzias contains a related argument that gives similar results.  We do
not discuss this class of bounds here.

The idea that the dimensional factor should depend on metric properties of the random matrix
appears in a paper of Hsu, Kakade, \&  Zhang~\cite{HKZ12:Tail-Inequalities}.  They obtain a
bound that is similar with Theorem~\ref{thm:intdim-bernstein-herm}.  Unfortunately,
their argument is complicated, and the results it delivers are suboptimal.

Theorem~\ref{thm:intdim-bernstein-herm} is essentially due to Stanislav
Minsker~\cite{Min11:Some-Extensions}.  His approach leads to somewhat
sharper bounds than the approach in the paper of Hsu, Kakade, \& Zhang,
and his method is easier to understand.

These notes contain another approach to intrinsic dimension bounds.
The intrinsic Chernoff bounds that emerge from our framework
are new.  The proof of the intrinsic Bernstein bound,
Theorem~\ref{thm:intdim-bernstein-herm}, can be interpreted
as a distillation of Minsker's argument.
Indeed, many of the specific calculations
already appear in Minsker's paper.
We have obtained constants that are marginally better.

\makeatletter{}%

\chapter{A Proof of Lieb's Theorem}
\label{chap:lieb}

Our approach to random matrices depends on some sophisticated ideas that are not usually presented in linear algebra courses.  This chapter contains a complete derivation of the results that undergird our matrix concentration inequalities.  We begin with a short argument that explains how Lieb's Theorem follows from deep facts about a function called the matrix relative entropy.  The balance of the chapter is devoted to an analysis of the matrix relative entropy.  Along the way, we establish the core properties of the trace exponential function and the matrix logarithm.  This discussion may serve as an introduction to the advanced techniques of matrix analysis.

\section{Lieb's Theorem} \label{sec:lieb-chap}

In his 1973 paper on trace functions, Lieb established an important concavity theorem~\cite[Thm.~6]{Lie73:Convex-Trace} for the trace exponential function.
As we saw in Chapter~\ref{chap:matrix-lt}, this result animates all of our matrix concentration inequalities.

\begin{thm}[Lieb] \label{thm:lieb-chap}
Let $\mtx{H}$ be a fixed Hermitian matrix with dimension $d$.  The map
\begin{equation} \label{eqn:trace-fn}
\mtx{A} \longmapsto \trace \exp\left( \mtx{H} + \log \mtx{A} \right)
\end{equation}
is concave on the convex cone of $d \times d$ positive-definite Hermitian matrices.
\end{thm}

\noindent
Section~\ref{sec:lieb-chap} contains an overview of the proof
of Theorem~\ref{thm:lieb-chap}.
First, we state the background material that we require,
and then we show how the theorem follows.
Some of the supporting results are major theorems in their own
right, and the details of their proofs will consume the rest of the
chapter.

\subsection{Conventions}

The symbol $\R_{++}$ refers to the set of positive real numbers.
We remind the reader of our convention that bold capital letters that are
symmetric about the vertical axis $(\mtx{A}, \mtx{H}, \mtx{M}, \mtx{T}, \mtx{U},
\mtx{\Psi})$ always refer to Hermitian matrices.
We reserve the letter $\Id$ for the identity
matrix, while the letter $\mtx{Q}$ always
refers to a unitary matrix.
Other bold capital letters
$(\mtx{B}, \mtx{K}, \mtx{L})$ denote rectangular matrices.

Unless stated otherwise, the results in this chapter hold for all matrices
whose dimensions are compatible.  For example, any result that involves
a sum $\mtx{A} + \mtx{H}$ includes the implicit constraint
that the two matrices are the same size.

Throughout this chapter, we assume that the parameter $\tau \in [0, 1]$,
and we use the shorthand $\bar{\tau} = 1 - \tau$ to make formulas
involving convex combinations more legible.

\subsection{Matrix Relative Entropy}

The proof of Lieb's Theorem depends on the properties of a bivariate function called the \term{matrix relative entropy}.

\begin{defn}[Matrix Relative Entropy]
Let $\mtx{A}$ and $\mtx{H}$ be positive-definite matrices of the same size.
The entropy of $\mtx{A}$ relative to $\mtx{H}$ is
$$
\mathrm{D}(\mtx{A}; \mtx{H}) = \trace \big[ \mtx{A} \, (\log \mtx{A} - \log \mtx{H}) - (\mtx{A} - \mtx{H}) \big].
$$
\end{defn}

\noindent
The relative entropy can be viewed as a measure of the difference between the matrix $\mtx{A}$ and the matrix $\mtx{H}$,
but it is not a metric.  Related functions arise in quantum statistical mechanics and quantum information theory.

We need two facts about the matrix relative entropy.

\begin{prop}[Matrix Relative Entropy is Nonnegative] \label{prop:mre-nonnegative}
For positive-definite matrices $\mtx{A}$ and $\mtx{H}$ of the same size,
the matrix relative entropy $\mathrm{D}(\mtx{A}; \mtx{H}) \geq 0$.
\end{prop}

\noindent
Proposition~\ref{prop:mre-nonnegative} is easy to prove;
see Section~\ref{sec:mre-nonnegative} for the short argument.

\begin{thm}[The Matrix Relative Entropy is Convex] \label{thm:mre-convex}
The map $(\mtx{A}, \mtx{H}) \mapsto \mathrm{D}(\mtx{A}; \mtx{H})$
is convex.  That is, for positive-definite $\mtx{A}_i$ and $\mtx{H}_i$ of the same size,
$$
{\mathrm D}\big(\tau \mtx{A}_1 + \bar{\tau} \mtx{A}_2; \, \tau \mtx{H}_1 + \bar{\tau} \mtx{H}_2 \big)
	\leq \tau \cdot {\mathrm D}\big(\mtx{A}_1; \mtx{H}_1 \big)
	+ \bar{\tau} \cdot {\mathrm D}\big(\mtx{A}_2; \mtx{H}_2 \big)
\quad\text{for $\tau \in [0,1]$}.
$$
\end{thm}

\noindent
Theorem~\ref{thm:mre-convex} is one of the crown jewels of matrix analysis.
The supporting material for this result occupies the bulk of this chapter;
the argument culminates in Section~\ref{sec:mre-convex}.

\subsection{Partial Maximization}

We also require a basic fact from convex analysis which states that partial maximization of a concave function produces a concave function.  We include the simple proof.

\begin{fact}[Partial Maximization] \label{fact:partial-max} %
Let $f$ be a concave function of two variables.  Then the function $y \mapsto \sup\nolimits_{x} f(x; y)$ obtained by partial maximization is concave.
\end{fact}

\begin{proof}
Fix $\eps > 0$.  For each pair of points $y_1$ and $y_2$, there are points $x_1$ and $x_2$ that satisfy
$$
f(x_1; y_1) \geq \sup\nolimits_x f(x; y_1) - \eps
\quad\text{and}\quad
f(x_2; y_2) \geq \sup\nolimits_x f(x; y_2) - \eps.
$$
For each $\tau \in [0,1]$, the concavity of $f$ implies that
\begin{align*}
\sup\nolimits_x f(x; \ \tau y_1 + {\bar \tau} y_2)
	&\geq f(\tau x_1 + \bar{\tau} x_2;\ \tau y_1 + \bar{\tau} y_2) \\
	&\geq \tau \cdot f(x_1; y_1) + \bar{\tau} \cdot f(x_2; y_2) \\
	&\geq \tau \cdot \sup\nolimits_x f(x; y_1) + \bar{\tau} \cdot \sup\nolimits_x f(x; y_2) - \eps.
\end{align*}
Take the limit as $\eps \downarrow 0$ to see that the partial supremum is a concave function of $y$.
\end{proof}

\subsection{A Proof of Lieb's Theorem}

Taking the results about the matrix relative entropy for granted, it is not hard to prove
Lieb's Theorem.  We begin with a variational representation of the trace, which restates
the fact that matrix relative entropy is nonnegative.

\begin{lemma}[Variational Formula for Trace] \label{lem:variation}
Let $\mtx{M}$ be a positive-definite matrix.  Then
$$
\trace \mtx{M} = \sup_{\mtx{T} \psdgt \mtx{0}} \ \trace\big[ \mtx{T} \log \mtx{M} - \mtx{T} \log \mtx{T} + \mtx{T} \big].
$$
\end{lemma}

\begin{proof}
Proposition~\eqref{prop:mre-nonnegative} states that $\mathrm{D}(\mtx{T}; \mtx{M}) \geq 0$.
Introduce the definition of the matrix relative entropy, and rearrange to reach
$$
\trace \mtx{M} \geq \trace\big[ \mtx{T} \log \mtx{M} - \mtx{T} \log \mtx{T} + \mtx{T} \big].
$$
When $\mtx{T} = \mtx{M}$, both sides are equal, which yields the advertised identity.
\end{proof}

To establish Lieb's Theorem, we use the variational formula to represent the 
trace exponential.  Then we use the partial maximization result to condense
the desired concavity property from the convexity of the matrix relative entropy.

\begin{proof}[Proof of Theorem~\ref{thm:lieb-chap}]
In the variational formula, Lemma~\ref{lem:variation}, select $\mtx{M} = \exp(\mtx{H} + \log \mtx{A})$ to obtain
$$
\trace \exp( \mtx{H} + \log \mtx{A} )
	= \sup_{\mtx{T} \psdgt \mtx{0}} \
	\trace\big[ \mtx{T}(\mtx{H} + \log \mtx{A}) - \mtx{T} \log \mtx{T} + \mtx{T} \big]
$$
The latter expression can be written compactly using the matrix relative entropy:
\begin{equation} \label{eqn:tr-exp-div}
\trace \exp ( \mtx{H} + \log \mtx{A} )
	= \sup_{\mtx{T} \psdgt \mtx{0}} \
	\big[ \trace(\mtx{TH}) + \trace \mtx{A} - \mathrm{D}(\mtx{T}; \mtx{A}) \big] 
\end{equation}
For each Hermitian matrix $\mtx{H}$, the bracket is a concave function of the pair $(\mtx{T}, \mtx{A})$ because of Theorem~\ref{thm:mre-convex}.  We see that the right-hand side of~\eqref{eqn:tr-exp-div} is the partial maximum of a concave function, and Fact~\ref{fact:partial-max} ensures that this expression defines a concave function of $\mtx{A}$.  This observation establishes the theorem.
\end{proof}

\section{Analysis of the Relative Entropy for Vectors}
\label{sec:vre}

Many deep theorems about matrices have analogies for vectors.
This observation is valuable because we can usually
adapt an analysis from the vector setting to establish the parallel
result for matrices.  In the matrix setting, however, it may be
necessary to install a significant amount of extra machinery.
If we keep the simpler structure of the vector argument
in mind, we can avoid being crushed in the gears.

\subsection{The Relative Entropy for Vectors}

The goal of \S\ref{sec:vre} is to introduce the relative entropy function for positive vectors
and to derive some key properties of this function.  Later we will analyze the matrix
relative entropy by emulating these arguments.

\begin{defn}[Relative Entropy]
Let $\vct{a}$ and $\vct{h}$ be positive vectors of the same size.
The entropy of $\vct{a}$ relative to $\vct{h}$ is defined as
$$
\mathrm{D}( \vct{a}; \vct{h} ) =
\sum\nolimits_k \big[ a_k \, (\log a_k - \log h_k) - (a_k - h_k) \big].
$$
\end{defn}

\noindent
A variant of the relative entropy arises in information theory and statistics as a
measure of the discrepancy between two probability distributions on a finite set.
We will show that the relative entropy is nonnegative and convex.

It may seem abusive to recycle the notation for the relative entropy
on matrices.  To justify this decision, we observe that
$$
\mathrm{D}(\vct{a}; \vct{h}) =
\mathrm{D}\big(\diag(\vct{a});\ \diag(\vct{h})\big)
$$
where $\diag(\cdot)$ maps a vector to a diagonal matrix in the natural way.
In other words, the vector relative entropy is a special case of the matrix relative entropy.
Ultimately, the vector case is easier to understand because diagonal matrices commute.

\subsection{Relative Entropy is Nonnegative}

As we have noted, the relative entropy measures the difference between two positive vectors.
This interpretation is supported by the fact that the relative entropy is nonnegative.

\begin{prop}[Relative Entropy is Nonnegative] \label{prop:vre-nonnegative}
For positive vectors $\vct{a}$ and $\vct{h}$ of the same size,
the relative entropy $\mathrm{D}(\vct{a}; \vct{h}) \geq 0$.
\end{prop}

\begin{proof}
Let $f : \R_{++} \to \R$ be a differentiable convex function on the positive real line.
The function $f$ lies above its tangent lines, so
$$
f(a) \geq f(h) + f'(h) \cdot (a - h)
\quad\text{for positive $a$ and $h$.}
$$
Instantiate this result for the convex function $f(a) = a \log a - a$,
and rearrange to obtain the numerical inequality
$$
a \, (\log a - \log h) - (a - h) \geq 0
\quad\text{for positive $a$ and $h$.}
$$
Sum this expression over the components of the vectors $\vct{a}$ and $\vct{h}$ to complete
the argument.
\end{proof}

Proposition~\ref{prop:mre-nonnegative} states that the matrix relative entropy satisfies
the same nonnegativity property as the vector relative entropy.
The argument for matrices relies on the same ideas as Proposition~\ref{prop:vre-nonnegative},
and it is hardly more difficult.  See \S\ref{sec:mre-nonnegative} for the details.

\subsection{The Perspective Transformation}

Our next goal is to prove that the relative entropy is a convex function.
To establish this claim, we use an elegant technique from convex analysis.
The approach depends on the \term{perspective transformation}, a method
for constructing a bivariate convex function from a univariate convex
function.

\begin{defn}[Perspective Transformation]
Let $f : \R_{++} \to \R$ be a convex function on the positive real line.
The perspective $\psi_f$ of the function $f$ is defined as
$$
\psi_f : \R_{++} \times \R_{++} \to \R
\quad\text{where}\quad
\psi_f(a; h) = a \cdot f( h / a ).
$$
\end{defn}

\noindent
The perspective transformation has an interesting geometric interpretation.
If we trace the ray from the origin $(0,0,0)$ in $\R^3$ through the point $(a, h, \psi_f(a;h))$,
it pierces the plane $(1, \cdot, \cdot)$ at the point $(1, h/a, f(h/a))$.
Equivalently, for each positive $a$, the epigraph of $f$ is the ``shadow''
of the epigraph of $\psi_f(a, \cdot)$ on the plane $(1, \cdot, \cdot)$.

The key fact is that the perspective of a convex function is convex.
This point follows from the geometric reasoning in the last paragraph;
we also include an analytic proof.

\begin{fact}[Perspectives are Convex] \label{fact:perspective-convex}
Let $f : \R_{++} \to \R$ be a convex function.  Then the perspective $\psi_f$
is convex.  That is, for positive numbers $a_i$ and $h_i$,
$$
\psi_f(\tau a_1 + \bar{\tau} a_2; \ \tau h_1 + \bar{\tau} h_2)
	\leq \tau \cdot \psi_f(a_1; h_1)
	+ \bar{\tau} \cdot \psi_f(a_2; h_2)
\quad\text{for $\tau \in [0,1]$}.
$$
\end{fact}

\begin{proof}
Fix two pairs $(a_1, h_1)$ and $(a_2, h_2)$ of positive numbers
and an interpolation parameter $\tau \in [0,1]$.
Form the convex combinations
$$
a = \tau a_1 + \bar{\tau} a_2
\quad\text{and}\quad
h = \tau h_1 + \bar{\tau} h_2.
$$
We need to bound the perspective $\psi_f(a; h)$ as the convex combination
of its values at $\psi_f(a_1; h_1)$ and $\psi_f(a_2; h_2)$.
The trick is to introduce another pair of interpolation parameters:
$$
s = \frac{\tau a_1}{a}
\quad\text{and}\quad
\bar{s} = \frac{\bar{\tau} a_2}{a}.
$$
By construction, $s \in [0, 1]$ and $\bar{s} = 1 - s$.  We quickly determine that
$$
\begin{aligned}
\psi_f(a; h)
	&= a \cdot f( h/a ) \\
	&= a \cdot f( \tau h_1 / a + \bar{\tau} h_2 / a )  \\
	&= a \cdot f( s \cdot h_1 / a_1 + \bar{s} \cdot h_2 / a_2 ) \\
	&\leq a \left[ s \cdot f( h_1 / a_1 ) + \bar{s} \cdot f( h_2 / a_2 ) \right]\\ 
	&= \tau \cdot a_1 \cdot f(h_1 / a_1) + \bar{\tau} \cdot a_2 \cdot f(h_2 / a_2) \\
	&= \tau \cdot \psi_f(a_1; h_1) + \bar{\tau} \cdot \psi_f(a_2; h_2).
\end{aligned}
$$
To obtain the second identity, we write $h$ as a convex combination.  The third identity
follows from the definitions of $s$ and $\bar{s}$.  The inequality depends on the
fact that $f$ is convex.  Afterward, we invoke the definitions of $s$ and $\bar{s}$
again.  We conclude that $\psi_f$ is convex.
\end{proof}

When we study standard matrix functions, it is sometimes necessary
to replace a convexity assumption by a stricter property called operator convexity.
There is a remarkable extension of the perspective transform that constructs a bivariate
matrix function from an operator convex function.  The matrix perspective has a powerful
convexity property analogous with the result in Fact~\ref{fact:perspective-convex}.
The analysis of the matrix perspective depends on a far-reaching generalization of the
Jensen inequality for operator convex functions.
We develop these ideas in \S\S\ref{sec:operator-convex},~\ref{sec:operator-jensen},
and~\ref{sec:matrix-perspective}.

\subsection{The Relative Entropy is Convex}

To establish that the relative entropy is convex, we simply need to 
represent it as the perspective of a convex function.

\begin{prop}[Relative Entropy is Convex] \label{prop:vre-convex}
The map $(\vct{a}, \vct{h}) \mapsto \mathrm{D}(\vct{a}; \vct{h})$
is convex.  That is, for positive vectors $\vct{a}_i$ and $\vct{h}_i$ of the same size,
$$
{\mathrm D}(\tau \vct{a}_1 + \bar{\tau} \vct{a}_2; \, \tau \vct{h}_1 + \bar{\tau} \vct{h}_2)
	\leq \tau \cdot {\mathrm D}(\vct{a}_1; \vct{h}_1)
	+ \bar{\tau} \cdot {\mathrm D}(\vct{a}_2; \vct{h}_2)
\quad\text{for $\tau \in [0,1]$}.
$$
\end{prop}

\begin{proof}
Consider the convex function $f(a) = a - 1 - \log a$, defined on the positive real line.
By direct calculation, the perspective transformation satisfies
$$
\psi_f(a; h) = a \, (\log a - \log h) - (a - h)
\quad\text{for positive $a$ and $h$.}
$$
Fact~\ref{fact:perspective-convex} states that $\psi_f$ is a convex function.
For positive vectors $\vct{a}$ and $\vct{h}$,
we can express the relative entropy as
$$
\mathrm{D}(\vct{a}; \vct{h})
	= \sum\nolimits_k \psi_f(a_k; h_k),
$$
It follows that the relative entropy is convex.
\end{proof}

Similarly, we can express the matrix relative entropy using the
matrix perspective transformation.
The analysis for matrices is substantially more involved.
But, as we will see in \S\ref{sec:mre-convex}, the argument ultimately
follows the same pattern as the proof of Proposition~\ref{prop:vre-convex}.

\section{Elementary Trace Inequalities}
\label{sec:trace-ineqs}

It is time to begin our investigation into the properties of matrix functions.
This section contains some simple inequalities for the trace of a matrix
function that we can establish by manipulating eigenvalues
and eigenvalue decompositions.
These techniques are adequate to explain why the matrix relative entropy
is nonnegative.  In contrast, we will need more subtle arguments
to study the convexity properties of the matrix relative entropy.

\subsection{Trace Functions}

We can construct a real-valued function on Hermitian matrices by composing
the trace with a standard matrix function.  This type of map is called a \term{trace function}.

\begin{defn}[Trace function]
Let $f : I \to \R$ be a function on an interval $I$ of the real line, and let $\mtx{A}$ be an Hermitian matrix whose eigenvalues are contained in $I$.  We define the \term{trace function} $\trace f$ by the rule
$$
\trace f(\mtx{A}) = \sum\nolimits_i f(\lambda_i(\mtx{A})),
$$
where $\lambda_i(\mtx{A})$ denotes the $i$th largest eigenvalue of $\mtx{A}$.  This formula gives the same result as composing the trace with the standard matrix function $f$.
\end{defn}

\noindent
Our first goal is to demonstrate that a trace function $\trace f$ inherits
a monotonicity property from the underlying scalar function $f$.

\subsection{Monotone Trace Functions}
\label{sec:monotone-trace}

Let us demonstrate that the trace of a weakly increasing scalar function induces a trace function that preserves the semidefinite order.  To that end, recall that the relation $\mtx{A} \psdle \mtx{H}$ implies that each eigenvalue of $\mtx{A}$ is dominated by the corresponding eigenvalue of $\mtx{H}$.

\begin{fact}[Semidefinite Order implies Eigenvalue Order] \label{fact:sdp-eig-order}
For Hermitian matrices $\mtx{A}$ and $\mtx{H}$,
$$
\mtx{A} \psdle \mtx{H}
\quad\text{implies}\quad
\lambda_i(\mtx{A}) \leq \lambda_i(\mtx{H})
\quad\text{for each index $i$}.
$$
\end{fact}

\begin{proof}
This result follows instantly from the Courant--Fischer Theorem:
$$
\lambda_{i}(\mtx{A}) = \max_{\dim L = i} \ \min_{\vct{u} \in L}\ \frac{\vct{u}^\adj \mtx{A} \vct{u}}{ \vct{u}^\adj \vct{u} }
	\leq \max_{\dim L = i} \ \min_{\vct{u} \in L}\ \frac{\vct{u}^\adj \mtx{H} \vct{u}}{ \vct{u}^\adj \vct{u} }
	= \lambda_i(\mtx{H}).
$$
The maximum ranges over all $i$-dimensional linear subspaces $L$ in the domain of $\mtx{A}$,
and we use the convention that $0/0 = 0$.  The inequality follows from the
definition~\eqref{eqn:semidefinite-order} of the semidefinite order $\psdle$.
\end{proof}

With this fact at hand, the claim follows quickly.

\begin{prop}[Monotone Trace Functions] \label{prop:trace-monotone}
Let $f : I \to \R$ be a weakly increasing function on an interval $I$ of the real line, and let $\mtx{A}$ and $\mtx{H}$ be Hermitian matrices whose eigenvalues are contained in $I$.  Then
$$
\mtx{A} \psdle \mtx{H}
\quad\text{implies}\quad
\trace f(\mtx{A}) \leq \trace f(\mtx{H}).
$$
\end{prop}

\begin{proof}
In view of Fact~\ref{fact:sdp-eig-order},
$$
\trace f(\mtx{A}) = \sum\nolimits_i f(\lambda_i(\mtx{A})) \leq \sum\nolimits_i f(\lambda_i(\mtx{H})) = \trace f(\mtx{H}).
$$
The inequality depends on the assumption that $f$ is weakly increasing.
\end{proof}

Our approach to matrix concentration relies on a special case of Proposition~\ref{prop:trace-monotone}.

\begin{example}[Trace Exponential is Monotone]
The trace exponential map is monotone:
$$
\mtx{A} \psdle \mtx{H}
\quad\text{implies}\quad
\trace \econst^{\mtx{A}} \leq \trace \econst^{ \mtx{H} }
$$
for all Hermitian matrices $\mtx{A}$ and $\mtx{H}$.
\end{example}

\subsection{Eigenvalue Decompositions, Redux}

Before we continue, let us introduce a style for writing
eigenvalue decompositions that will make the next argument more transparent.
Each $d \times d$ Hermitian matrix $\mtx{A}$ can be expressed as
$$
\mtx{A} = \sum_{i=1}^d \lambda_i \, \vct{u}_i \vct{u}_i^\adj.
$$
The eigenvalues $\lambda_1 \geq \dots \geq \lambda_d$ of $\mtx{A}$ are real numbers,
listed in weakly decreasing order.
The family $\{ \vct{u}_1, \dots, \vct{u}_d \}$
of eigenvectors of $\mtx{A}$ forms an orthonormal basis for $\C^d$
with respect to the standard inner product.

\subsection{A Trace Inequality for Bivariate Functions}

In general, it is challenging to study functions of two or more matrices because the eigenvectors can interact in complicated ways.  Nevertheless, there is one type of relation that always transfers from the scalar setting to the matrix setting.

\begin{prop}[Generalized Klein Inequality] \label{prop:gen-klein}
Let $f_i : I \to \R$ and $g_i : I \to \R$ be functions on an interval $I$ of the real line, and suppose that
$$
\sum\nolimits_i f_i(a) \, g_i(h) \geq 0
\quad\text{for all $a, h \in I$.}
$$
If $\mtx{A}$ and $\mtx{H}$ are Hermitian matrices whose eigenvalues are contained in $I$, then
$$
\sum\nolimits_i \trace\big[  f_i(\mtx{A}) \, g_i(\mtx{H}) \big] \geq 0.
$$
\end{prop}

\begin{proof}
Consider eigenvalue decompositions $\mtx{A} = \sum\nolimits_j \lambda_j \vct{u}_j \vct{u}_j^\adj$
and $\mtx{H} = \sum\nolimits_k \mu_k \vct{v}_k \vct{v}_k^\adj$.  Then
$$
\begin{aligned}
\sum\nolimits_i \trace \big[ f_i(\mtx{A}) \, g_i(\mtx{H}) \big]
	&= \sum\nolimits_i \trace \left[ \big( \sum\nolimits_{j} f_i(\lambda_j)  \vct{u}_j \vct{u}_j^\adj \big)
		\big( \sum\nolimits_k g_i( \mu_k ) \, \vct{v}_k \vct{v}_k^\adj \big) \right] \\
	&= \sum\nolimits_{j,k}  \left[\sum\nolimits_i f_i(\lambda_j) \, g_i(\mu_k) \right] \cdot \abssqip{ \smash{\vct{u}_j} }{ \vct{v}_k }
	\geq 0.
\end{aligned}
$$
We use the definition of a standard matrix function, we apply linearity of the trace to reorder the sums, and we identify the trace as a squared inner product.  The inequality follows from our assumption on the scalar functions.
\end{proof}

\subsection{The Matrix Relative Entropy is Nonnegative}
\label{sec:mre-nonnegative}

Using the generalized Klein inequality, it is easy to prove Proposition~\ref{prop:mre-nonnegative},
which states that the matrix relative entropy is nonnegative.  The argument echoes the
analysis in Proposition~\ref{prop:vre-nonnegative} for the vector case.

\begin{proof}[Proof of Proposition~\ref{prop:mre-nonnegative}]
Suppose that $f : \R_{++} \to \R$ is a differentiable, convex function on the positive real line.  Since $f$ is convex,
the graph of $f$ lies above its tangents:
$$
f(a) \geq f(h) + f'(h) (a-h)
\quad\text{for positive $a$ and $h$.}
$$
Using the generalized Klein inequality, Proposition~\ref{prop:gen-klein},
we can lift this relation to matrices:
$$
\trace f(\mtx{A}) \geq \trace\big[ f(\mtx{H}) + f'(\mtx{H})(\mtx{A} - \mtx{H}) \big]
\quad\text{for all positive-definite $\mtx{A}$ and $\mtx{H}$.}
$$
This formula is sometimes called the (ungeneralized) Klein inequality.

Instantiate the latter result for the function $f(a) = a \log a - a$,
and rearrange to see that
$$
\mathrm{D}(\mtx{A}; \mtx{H})
	= \trace \big[ \mtx{A} \, (\log \mtx{A} - \log \mtx{H}) - (\mtx{A} - \mtx{H}) \big]
	\geq 0
	\quad\text{for all positive-definite $\mtx{A}$ and $\mtx{H}$.}
$$
In other words, the matrix relative entropy is nonnegative.
\end{proof}

\section{The Logarithm of a Matrix}

In this section, we commence our journey toward the proof that the matrix relative entropy is convex.
The proof of Proposition~\ref{prop:vre-convex} indicates that the convexity of the logarithm plays
an important role in the convexity of the vector relative entropy.  As a first step, we will demonstrate that the
matrix logarithm has a striking convexity property with respect to the semidefinite order.
Along the way, we will also develop a monotonicity property of the matrix logarithm.

\subsection{An Integral Representation of the Logarithm}

Initially, we defined the logarithm of a $d \times d$ positive-definite
matrix $\mtx{A}$ using an eigenvalue decomposition:
$$
\log \mtx{A} = \mtx{Q} \begin{bmatrix} \log \lambda_1  \\ &\ddots \\ && \log \lambda_d \end{bmatrix}  \mtx{Q}^\adj
\quad\text{where}\quad
\mtx{A} = \mtx{Q} \begin{bmatrix} \lambda_1 \\ &\ddots \\ && \lambda_d \end{bmatrix} \mtx{Q}^\adj.
$$
To study how the matrix logarithm interacts with the semidefinite order,
we will work with an alternative presentation based on an integral formula.

\begin{prop}[Integral Representation of the Logarithm] \label{prop:log-integral}
The logarithm of a positive number $a$ is given by the integral
$$
\log a = \int_0^\infty \left[ \frac{1}{1+u} - \frac{1}{a+u} \right] \idiff{u}.
$$
Similarly, the logarithm of a positive-definite matrix $\mtx{A}$ is given by the integral
$$
\log \mtx{A} = \int_0^\infty \left[ (1+u)^{-1} \Id - (\mtx{A} + u\Id)^{-1} \right] \idiff{u}.
$$
\end{prop}

\begin{proof}
To verify the scalar formula, we simply use the definition of the improper integral:
$$
\begin{aligned}
\int_0^\infty \left[ \frac{1}{1+u} - \frac{1}{a+u} \right] \idiff{u}
	&= \lim_{L\to \infty} \ \int_0^L \left[ \frac{1}{1+u} - \frac{1}{a+u} \right] \idiff{u} \\
	&= \lim_{L\to \infty} \left[ \log(1+u) - \log(a+u) \right]_{u=0}^L \\
	&= \log a + \lim_{L\to\infty} \log\left( \frac{1 + L}{a + L} \right)
	= \log a.
\end{aligned}
$$
We obtain the matrix formula by applying the scalar formula to each eigenvalue
of $\mtx{A}$ and then expressing the result in terms of the original matrix. 
\end{proof}

The integral formula from Proposition~\ref{prop:log-integral}
is powerful because it expresses the logarithm in terms of the matrix inverse,
which is much easier to analyze.  Although it
may seem that we have pulled this representation from thin air,
the approach is motivated by a wonderful theory of matrix functions
initiated by L{\"o}wner in the 1930s.

\subsection{Operator Monotone Functions}
\label{sec:operator-monotone}

Our next goal is to study the monotonicity properties of the matrix logarithm.
To frame this discussion properly, we need to introduce an abstract definition.

\begin{defn}[Operator Monotone Function]
Let $f : I \to \R$ be a function on an interval $I$ of the real line.  The function $f$ is \term{operator monotone} on $I$ when
$$
\mtx{A} \psdle \mtx{H}
\quad\text{implies}\quad
f(\mtx{A}) \psdle f(\mtx{H})
$$
for all Hermitian matrices $\mtx{A}$ and $\mtx{H}$ whose eigenvalues are contained in $I$.
\end{defn}

Let us state some basic facts about operator monotone functions.
Many of these points follow easily from the definition.

\begin{itemize}
\item	When $\beta \geq 0$, the weakly increasing affine function $t \mapsto \alpha + \beta t$ is operator monotone
on each interval $I$ of the real line.

\item	The quadratic function $t \mapsto t^2$ is \textbf{not} operator monotone
on the positive real line.

\item	The exponential map $t \mapsto \econst^t$ is \textbf{not} operator monotone
on the real line.

\item	When $\alpha \geq 0$ and $f$ is operator monotone on $I$,
the function $\alpha f$ is operator monotone on $I$.

\item	If $f$ and $g$ are operator monotone on an interval $I$, then $f + g$ is operator monotone
on $I$.
\end{itemize}

\noindent
These properties imply that the operator monotone functions form a convex cone.
It also warns us that the class of operator monotone functions
is somewhat smaller than the class of weakly increasing functions.

\subsection{The Negative Inverse is Operator Monotone}

Fortunately, interesting operator monotone functions do exist.
Let us present an important example related to the matrix inverse.

\begin{prop}[Negative Inverse is Operator Monotone] \label{prop:inv-monotone}
For each number $u \geq 0$, the function $a \mapsto - (a+ u)^{-1}$ is operator
monotone on the positive real line.
That is, for positive-definite matrices $\mtx{A}$ and $\mtx{H}$,
$$
\mtx{A} \psdle \mtx{H}
\quad\text{implies}\quad
-(\mtx{A} + u \Id)^{-1} \psdle -(\mtx{H} + u \Id)^{-1}.
$$
\end{prop}

\begin{proof}
Define the matrices $\mtx{A}_u = \mtx{A} + u\Id$ and $\mtx{H}_u = \mtx{H} + u \Id$.
The semidefinite relation $\mtx{A} \psdle \mtx{H}$ implies that
$\mtx{A}_u \psdle \mtx{H}_u$.  Apply the Conjugation Rule~\eqref{eqn:conjugation-rule}
to see that
$$
\mtx{0} \psdlt \mtx{H}_u^{-1/2} \mtx{A}_u \mtx{H}_u^{-1/2} \psdle \Id.
$$
When a positive-definite matrix has eigenvalues bounded above by one, its inverse has eigenvalues bounded below by one.  Therefore,
$$
\Id \psdle \left(\mtx{H}_u^{-1/2} \mtx{A}_u \mtx{H}_u^{-1/2}\right)^{-1}
	= \mtx{H}_u^{1/2} \mtx{A}_u^{-1} \mtx{H}_u^{1/2}.
$$
Another application of the Conjugation Rule~\eqref{eqn:conjugation-rule}
delivers the inequality $\mtx{H}_u^{-1} \psdle \mtx{A}_u^{-1}$.
Finally, we negate this semidefinite relation, which reverses its direction.
\end{proof}

\subsection{The Logarithm is Operator Monotone}
\label{sec:log-monotone}

Now, we are prepared to demonstrate that the logarithm is an operator monotone function.
The argument combines the integral representation from Proposition~\ref{prop:log-integral}
with the monotonicity of the inverse map from Proposition~\ref{prop:inv-monotone}.

\begin{prop}[Logarithm is Operator Monotone] \label{prop:log-monotone}
The logarithm is an operator monotone function on the positive real line.
That is, for positive-definite matrices $\mtx{A}$ and $\mtx{H}$,
$$
\mtx{A} \psdle \mtx{H}
\quad\text{implies}\quad
\log \mtx{A} \psdle \log \mtx{H}.
$$
\end{prop}

\begin{proof}
For each $u \geq 0$, Proposition~\ref{prop:inv-monotone} demonstrates that
$$
(1+u)^{-1} \Id - (\mtx{A}+u\Id)^{-1}
\psdle (1+u)^{-1} \Id - (\mtx{H}+u\Id)^{-1}.
$$
The integral representation of the logarithm, Proposition~\ref{prop:log-integral},
allows us to calculate that
$$
\log \mtx{A} 
	= \int_0^{\infty} \left[ (1+u)^{-1} \Id - (\mtx{A}+u\Id)^{-1} \right] \idiff{u}
	\psdle \int_0^{\infty} \left[ (1+u)^{-1} \Id - (\mtx{H}+u\Id)^{-1} \right] \idiff{u}
	= \log \mtx{H}.
$$
We have used the fact that the semidefinite order is preserved by integration
against a positive measure.
\end{proof}

\subsection{Operator Convex Functions}
\label{sec:operator-convex}

Next, let us investigate the convexity properties of the matrix logarithm.  As before,
we start with an abstract definition.

\begin{defn}[Operator Convex Function] \label{def:operator-convex}
Let $f : I \to \R$ be a function on an interval $I$ of the real line.  The function $f$ is \term{operator convex} on $I$ when
$$
f(\tau \mtx{A} + \bar{\tau} \mtx{H}) \psdle \tau \cdot f(\mtx{A}) + \bar{\tau} \cdot f(\mtx{H})
\quad\text{for all $\tau \in [0,1]$}
$$
and for all Hermitian matrices $\mtx{A}$ and $\mtx{H}$ whose eigenvalues are contained in $I$.  A function $g : I \to \R$ is \term{operator concave} when $-g$ is operator convex on $I$. 
\end{defn}

We continue with some important facts about operator convex functions.
Most of these claims can be derived easily.

\begin{itemize}
\item	When $\gamma \geq 0$, the quadratic function $t \mapsto \alpha + \beta t + \gamma t^2$ is operator convex on the real line.

\item	The exponential map $t \mapsto \econst^t$ is \textbf{not} operator convex on
the real line.

\item	When $\alpha \geq 0$ and $f$ is operator convex on $I$, the function $\alpha f$ is
operator convex in $I$.

\item	If $f$ and $g$ are operator convex on $I$, then $f + g$ is operator convex on $I$.
\end{itemize}

\noindent
The operator monotone functions form a convex cone.
We also learn that the family of operator convex functions
is somewhat smaller than the family of convex functions.

\subsection{The Inverse is Operator Convex}

The inverse provides a very important example of an operator convex function.

\begin{prop}[Inverse is Operator Convex] \label{prop:inv-convex}
For each $u \geq 0$, the function $a \mapsto (a + u)^{-1}$ is
operator convex on the positive real line.
That is, for positive-definite matrices $\mtx{A}$ and $\mtx{H}$, 
$$
(\tau \mtx{A} + \bar{\tau} \mtx{H} + u \Id)^{-1} \psdle
\tau \cdot (\mtx{A} + u \Id)^{-1} + \bar{\tau} \cdot (\mtx{H} + u\Id)^{-1}
\quad\text{for $\tau \in [0,1]$.}
$$
\end{prop}

To establish Proposition~\ref{prop:inv-convex}, we use an argument based on the Schur complement lemma.
For completeness, let us state and prove this important fact.

\begin{fact}[Schur Complements] \label{fact:schur-complement}
Suppose that $\mtx{T}$ is a positive-definite matrix.  Then
\begin{equation} \label{eqn:schur-complement}
\mtx{0} \psdle \begin{bmatrix} \mtx{T} & \mtx{B} \\ \mtx{B}^\adj & \mtx{M} \end{bmatrix} %
\quad\text{if and only if}\quad
\mtx{B}^\adj \mtx{T}^{-1} \mtx{B} \psdle \mtx{M}.
\end{equation}
\end{fact}

\begin{proof}[Proof of Fact~\ref{fact:schur-complement}]
To see why this is true, just calculate that
$$
\begin{bmatrix} \Id & \mtx{0} \\ -\mtx{B}^\adj \mtx{T}^{-1} & \Id \end{bmatrix}
	\begin{bmatrix} \mtx{T} & \mtx{B} \\ \mtx{B}^\adj & \mtx{M} \end{bmatrix}
	\begin{bmatrix} \Id & -\mtx{T}^{-1} \mtx{B} \\ \mtx{0} & \Id \end{bmatrix}
= \begin{bmatrix} \mtx{T} & \mtx{0} \\ \mtx{0} & \mtx{M} - \mtx{B}^\adj \mtx{T}^{-1} \mtx{B} \end{bmatrix}
$$
In essence, we are performing block Gaussian elimination to bring the original
matrix into block-diagonal form.
Now, the Conjugation Rule~\eqref{eqn:conjugation-rule} ensures that the central matrix
on the left is positive semidefinite together with the matrix on the right.
From this equivalence, we extract the result~\eqref{eqn:schur-complement}.   
\end{proof}

We continue with the proof that the inverse is operator convex.

\begin{proof}[Proof of Proposition~\ref{prop:inv-convex}]
The Schur complement lemma, Fact~\ref{fact:schur-complement}, provides that
\begin{equation*} \label{eqn:schur-simple}
\mtx{0} \psdle \begin{bmatrix} \mtx{T} & \Id \\ \Id & \mtx{T}^{-1} \end{bmatrix}
\quad\text{whenever $\mtx{T}$ is positive definite.}
\end{equation*}
Applying this observation to the positive-definite matrices $\mtx{A} + u\Id$
and $\mtx{H} + u\Id$, we see that
$$
\begin{aligned}
\mtx{0} &\psdle
\tau \cdot \begin{bmatrix} \mtx{A} + u\Id & \Id \\ \Id & (\mtx{A} + u\Id)^{-1} \end{bmatrix}
	+ \bar{\tau} \cdot \begin{bmatrix} \mtx{H} + u \Id & \Id \\ \Id & (\mtx{H} + u \Id)^{-1} \end{bmatrix} \\
	&= \begin{bmatrix} \tau \mtx{A} + \bar{\tau} \mtx{H} + u \Id & \Id \\
	\Id & \tau \cdot (\mtx{A}+u\Id)^{-1} + \bar{\tau} \cdot (\mtx{H}+u\Id)^{-1} \end{bmatrix}.
\end{aligned}
$$
Since the top-left block of the latter matrix is positive definite,
another application of Fact~\ref{fact:schur-complement} delivers the relation
$$
(\tau \mtx{A} + \bar{\tau} \mtx{H} + u \Id)^{-1}
\psdle \tau \cdot (\mtx{A}+u\Id)^{-1} + \bar{\tau} \cdot (\mtx{H}+u\Id)^{-1}.
$$
This is the advertised conclusion.
\end{proof}

\subsection{The Logarithm is Operator Concave}

We are finally prepared to verify that the logarithm is operator concave.  The argument is based on the integral representation from Proposition~\ref{prop:log-monotone} and the convexity of the inverse map from Proposition~\ref{prop:inv-convex}.

\begin{prop}[Logarithm is Operator Concave] \label{prop:log-concave}
The logarithm is operator concave on the positive real line.
That is, for positive-definite matrices $\mtx{A}$ and $\mtx{H}$,
$$
\tau \cdot \log \mtx{A} + \bar{\tau} \cdot \log \mtx{H} \psdle \log( \tau \mtx{A} + \bar{\tau} \mtx{H} )
\quad\text{for $\tau \in [0, 1]$.}
$$
\end{prop}

\begin{proof}
For each $u \geq 0$, Proposition~\ref{prop:inv-convex} demonstrates that
$$
- \tau \cdot (\mtx{A} + u \Id)^{-1} - \bar{\tau} \cdot (\mtx{H} + u\Id)^{-1}
	\psdle -(\tau \mtx{A} + \bar{\tau}\mtx{H} + u \Id)^{-1}.
$$
Invoke the integral representation of the logarithm from Proposition~\ref{prop:log-integral}
to see that
$$
\begin{aligned}
\tau \cdot \log \mtx{A} + \bar{\tau} \cdot \log \mtx{H}
	&= \tau \cdot \int_0^\infty \left[ (1+u)^{-1}\Id - (\mtx{A} + u \Id)^{-1} \right] \idiff{u}
	+ \bar{\tau} \cdot \int_0^\infty \left[ (1+u)^{-1}\Id - (\mtx{H} + u \Id)^{-1} \right] \idiff{u} \\
	&= \int_0^\infty \left[ (1+u)^{-1} \Id
	- \big( \tau \cdot (\mtx{A} + u \Id)^{-1} + \bar{\tau} \cdot (\mtx{H} + u\Id)^{-1} \big) \right] \idiff{u} \\
	&\psdle \int_0^\infty \left[ (1+u)^{-1} \Id  - (\tau \mtx{A} + \bar{\tau} \mtx{H} + u\Id)^{-1}
	\right] \idiff{u}
	= \log(\tau \mtx{A} + \bar{\tau} \mtx{H}).
\end{aligned}
$$
Once again, we have used the fact that integration preserves the semidefinite order.
\end{proof}

\section{The Operator Jensen Inequality}
\label{sec:operator-jensen}

Convexity is a statement about how a function interacts with averages.
By definition, a function $f : I \to \R$ is convex when
\begin{equation} \label{eqn:cvx-op-jensen}
f( \tau a + \bar{\tau} h) \leq \tau \cdot f(a) + \bar{\tau} \cdot f(h)
\quad\text{for all $\tau \in [0,1]$ and all $a, h \in I$.}
\end{equation}
The convexity inequality~\eqref{eqn:cvx-op-jensen} automatically extends
from an average involving two terms to an arbitrary average.
This is the content of Jensen's inequality.

Definition~\ref{def:operator-convex}, of an operator convex function $f : I \to \R$, is similar in spirit:
\begin{equation} \label{eqn:op-cvx-op-jensen}
f(\tau \mtx{A} + \bar{\tau} \mtx{H})
	\psdle  \tau \cdot f(\mtx{A}) + \bar{\tau} \cdot f(\mtx{H})
	\quad\text{for all $\tau \in [0,1]$}
\end{equation}
and all Hermitian matrices $\mtx{A}$ and $\mtx{H}$ whose eigenvalues are contained in $I$.
Surprisingly, the semidefinite relation~\eqref{eqn:op-cvx-op-jensen} automatically extends
to a large family of matrix averaging operations.
This remarkable property is called the \term{operator Jensen inequality}.

\subsection{Matrix Convex Combinations}

In a vector space, convex combinations provide a natural method of averaging.
But matrices have a richer structure, so we can consider a more general class
of averages.

\begin{defn}[Matrix Convex Combination] \label{def:matrix-convex-comb}
Let $\mtx{A}_1$ and $\mtx{A}_2$ be Hermitian matrices.
Consider a decomposition of the identity of the form
$$
\mtx{K}_1^\adj \mtx{K}_1 + \mtx{K}_2^\adj \mtx{K}_2 = \Id.
$$
Then the Hermitian matrix
\begin{equation} \label{eqn:matrix-convex-comb}
\mtx{K}_1^\adj \mtx{A}_1 \mtx{K}_1 +
\mtx{K}_2^\adj \mtx{A}_2 \mtx{K}_2
\end{equation}
is called a \term{matrix convex combination} of $\mtx{A}_1$ and $\mtx{A}_2$.
\end{defn}

To see why it is reasonable to call~\eqref{eqn:matrix-convex-comb} an averaging
operation on Hermitian matrices, let us note a few of its properties.

\begin{itemize}
\item	Definition~\ref{def:matrix-convex-comb} encompasses scalar convex combinations
because we can take $\mtx{K}_1 = \tau^{1/2} \Id$
and $\mtx{K}_2 = \bar{\tau}^{1/2} \Id$.

\item	The matrix convex combination preserves the identity matrix:
$\mtx{K}_1^\adj \Id \mtx{K}_1 +
\mtx{K}_2^\adj \Id \mtx{K}_2 = \Id.$

\item	The matrix convex combination preserves positivity:
$$
\mtx{K}_1^\adj \mtx{A}_1 \mtx{K}_1 + \mtx{K}_2^\adj \mtx{A}_2 \mtx{K}_2 \psdge \mtx{0}
\quad\text{for all positive-semidefinite $\mtx{A}_1$ and $\mtx{A}_2$.}
$$

\item	If the eigenvalues of $\mtx{A}_1$ and $\mtx{A}_2$ are contained in an
interval $I$, then the eigenvalues of the matrix convex combination~\eqref{eqn:matrix-convex-comb}
are also contained in $I$.
\end{itemize}

\noindent
We will encounter a concrete example of a matrix convex combination
later when we prove Theorem~\ref{thm:mtx-perspective-convex}.

\subsection{Jensen's Inequality for Matrix Convex Combinations}

Operator convexity is a self-improving property.
Even though the definition of an operator convex function
only involves a scalar convex combination, it actually
contains an inequality for matrix convex combinations.
This is the content of the operator Jensen inequality.

\begin{thm}[Operator Jensen Inequality] \label{thm:operator-jensen}
Let $f$ be an operator convex function on an interval $I$ of the real line,
and let $\mtx{A}_1$ and $\mtx{A}_2$ be Hermitian matrices with eigenvalues in $I$.  
Consider a decomposition of the identity
\begin{equation} \label{eqn:nc-sum-to-one}
\mtx{K}_1^\adj \mtx{K}_1 + \mtx{K}_2^\adj \mtx{K}_2 = \Id.
\end{equation}
Then
$$
f\left( \mtx{K}_1^\adj \mtx{A}_1 \mtx{K}_1 + \mtx{K}_2^\adj \mtx{A}_2 \mtx{K}_2 \right)
	\psdle \mtx{K}_1^\adj f(\mtx{A}_1) \mtx{K}_1 + \mtx{K}_2^\adj f(\mtx{A}_2) \mtx{K}_2.
$$
\end{thm}

\begin{proof}
Let us introduce a block-diagonal matrix:
$$
\mtx{A} = \begin{bmatrix} \mtx{A}_1 & \mtx{0} \\ \mtx{0} & \mtx{A}_2 \end{bmatrix}
\quad\text{for which}\quad
f(\mtx{A}) = \begin{bmatrix} f(\mtx{A}_1) & \mtx{0} \\ \mtx{0} & f(\mtx{A}_2) \end{bmatrix}.
$$
Indeed, the matrix $\mtx{A}$ lies in the domain of $f$ because
its eigenvalues fall in the interval $I$.
We can apply a standard matrix function to a
block-diagonal matrix by applying the function to each block.

There are two main ingredients in the argument.
The first idea is to realize the matrix convex combination
of $\mtx{A}_1$ and $\mtx{A}_2$ by conjugating the block-diagonal
matrix $\mtx{A}$ with an appropriate unitary matrix.  
To that end, let us construct a unitary matrix
$$
\mtx{Q} = \begin{bmatrix} \mtx{K}_1 & \mtx{L}_1 \\ \mtx{K}_2 & \mtx{L}_2 \end{bmatrix}
\quad\text{where $\mtx{Q}^\adj \mtx{Q} = \Id$ and $\mtx{QQ}^\adj = \Id$.}
$$
To see why this is possible, note that the first block of columns is orthonormal:
$$
\begin{bmatrix} \mtx{K}_1 \\ \mtx{K}_2 \end{bmatrix}^\adj \begin{bmatrix} \mtx{K}_1 \\ \mtx{K}_2 \end{bmatrix}
	= \mtx{K}_1^\adj \mtx{K}_1 + \mtx{K}_2^\adj \mtx{K}_2
	= \Id.
$$
As a consequence, we can choose $\mtx{L}_1$ and $\mtx{L}_2$ to complete the unitary matrix $\mtx{Q}$.
By direct computation, we find that
\begin{equation} \label{eqn:operator-jensen-pf-1}
\mtx{Q}^\adj \mtx{AQ}
	= \begin{bmatrix} \mtx{K}_1^\adj \mtx{A}_1 \mtx{K}_1
		+ \mtx{K}_2^\adj \mtx{A}_2 \mtx{K}_2 & * & \\ * & * &
	\end{bmatrix}
\end{equation}
We have omitted the precise values of the entries labeled $*$ because they do not play a role in our argument.

The second idea is to restrict the block matrix in~\eqref{eqn:operator-jensen-pf-1}
to its diagonal.  %
To perform this maneuver, we express the diagonalizing operation
as a \emph{scalar convex combination} of two unitary conjugations, which gives us access to the
operator convexity of $f$.  Let us see how this works.  Define the unitary matrix
$$
\mtx{U} = \begin{bmatrix} \Id & \mtx{0} \\ \mtx{0} & - \Id \end{bmatrix}.
$$
The key observation is that, for any block matrix,
\begin{equation} \label{eqn:operator-jensen-pf-2}
\frac{1}{2} \begin{bmatrix} \mtx{T} & \mtx{B} \\ \mtx{B}^\adj & \mtx{M} \end{bmatrix}
	+ \frac{1}{2} \mtx{U}^\adj
	\begin{bmatrix} \mtx{T} & \mtx{B} \\ \mtx{B}^\adj & \mtx{M} \end{bmatrix} \mtx{U}
	= \begin{bmatrix} \mtx{T} & \mtx{0} \\ \mtx{0} & \mtx{M} \end{bmatrix}.
\end{equation}
Another advantage of this construction is that we can easily apply
a standard matrix function to the block-diagonal matrix.

Together, these two ideas lead to a succinct proof of the operator Jensen inequality.
Write $[\cdot]_{11}$ for the operation that returns the $(1, 1)$ block of a
block matrix.  We may calculate that
$$
\begin{aligned}
f\left( \mtx{K}_1^\adj \mtx{A}_1 \mtx{K}_1 + \mtx{K}_2^\adj \mtx{A}_2 \mtx{K}_2 \right)
	&= f\left( [\mtx{Q}^\adj \mtx{A} \mtx{Q}]_{11} \right) \\
	&= f\left( \left[\frac{1}{2} \mtx{Q}^\adj \mtx{A} \mtx{Q}
		+ \frac{1}{2}\mtx{U}^\adj (\mtx{Q}^\adj \mtx{A} \mtx{Q}) \mtx{U} \right]_{11} \right) \\
	&= \left[ f\left(\frac{1}{2} \mtx{Q}^\adj \mtx{AQ}
		+ \frac{1}{2} (\mtx{QU})^\adj \mtx{A} (\mtx{QU}) \right) \right]_{11} \\
	&\psdle \left[ \frac{1}{2} f(\mtx{Q}^\adj \mtx{AQ})
		+ \frac{1}{2} f( (\mtx{QU})^{\adj} \mtx{A} (\mtx{QU}) ) \right]_{11}.
\end{aligned}
$$
The first identity depends on the representation~\eqref{eqn:operator-jensen-pf-1}
of the matrix convex combination as the $(1, 1)$ block of $\mtx{Q}^\adj \mtx{AQ}$.
The second line follows because the averaging
operation presented in~\eqref{eqn:operator-jensen-pf-2}
does not alter the $(1,1)$ block of the matrix.
In view of~\eqref{eqn:operator-jensen-pf-2},
we are looking at the $(1,1)$ block of the matrix
obtained by applying $f$ to a block-diagonal matrix.
This is equivalent to applying the function $f$
inside the $(1, 1)$ block, which gives the third line.
Last, the semidefinite relation follows from the operator convexity of
$f$ on the interval $I$.

We complete the argument by reversing the steps we have taken so far.
$$
\begin{aligned}
f\left( \mtx{K}_1^\adj \mtx{A}_1 \mtx{K}_1 + \mtx{K}_2^\adj \mtx{A}_2 \mtx{K}_2 \right)
	&\psdle \left[ \frac{1}{2} \mtx{Q}^\adj f(\mtx{A}) \mtx{Q}
		+ \frac{1}{2} \mtx{U}^\adj (\mtx{Q}^\adj f(\mtx{A}) \mtx{Q}) \mtx{U} \right]_{11} \\
	&= \left[ \mtx{Q}^\adj f(\mtx{A}) \mtx{Q} \right]_{11} \\
	&= \mtx{K}_1^\adj f(\mtx{A}_1) \mtx{K}_1 + \mtx{K}_2^\adj f(\mtx{A}_2) \mtx{K}_2.
\end{aligned}
$$
To obtain the first relation, recall that a standard matrix function commutes with unitary
conjugation.  The second identity follows from the formula~\eqref{eqn:operator-jensen-pf-2}
because diagonalization preserves the $(1,1)$ block.
Finally, we identify the $(1, 1)$ block of $\mtx{Q}^\adj f(\mtx{A})\mtx{Q}$
just as we did in~\eqref{eqn:operator-jensen-pf-1}.  This step depends on
the fact that the diagonal blocks of $f(\mtx{A})$ are simply $f(\mtx{A}_1)$ and $f(\mtx{A}_2)$.
\end{proof}

\section{The Matrix Perspective Transformation}
\label{sec:matrix-perspective}

To show that the vector relative entropy is convex,
we represented it as the perspective of a convex function.
To demonstrate that the matrix relative entropy is convex,
we are going to perform a similar maneuver.  This section
develops an extension of the perspective transformation
that applies to operator convex functions.  Then we
demonstrate that this matrix perspective has a strong
convexity property with respect to the semidefinite order.

\subsection{The Matrix Perspective}

In the scalar setting, the perspective transformation converts
a convex function into a bivariate convex function.  There is
a related construction that applies to an operator convex function.

\begin{defn}[Matrix Perspective]
Let $f : \R_{++} \to \R$ be an operator convex function, and let
$\mtx{A}$ and $\mtx{H}$ be positive-definite matrices of the same size.
Define the perspective map
$$
\mtx{\Psi}_f(\mtx{A}; \mtx{H})
	= \mtx{A}^{1/2} \cdot f\big(\mtx{A}^{-1/2} \mtx{H} \mtx{A}^{-1/2} \big) \cdot \mtx{A}^{1/2}.
$$
The notation $\mtx{A}^{1/2}$ refers to the unique positive-definite square root of $\mtx{A}$,
and $\mtx{A}^{-1/2}$ denotes the inverse of this square root.
\end{defn}

\noindent
The Conjugation Rule~\eqref{eqn:conjugation-rule} ensures that all the matrices involved
remain positive definite, so this definition makes sense.  To see why the matrix perspective
extends the scalar perspective, notice that
\begin{equation} \label{eqn:perspective-commute}
\mtx{\Psi}_f(\mtx{A}; \mtx{H}) = \mtx{A} \cdot f(\mtx{H}\mtx{A}^{-1})
\quad\text{when $\mtx{A}$ and $\mtx{H}$ commute.}
\end{equation}
This formula is valid because commuting matrices are simultaneously diagonalizable.
We will use the matrix perspective in a case where the matrices commute,
but it is no harder to analyze the perspective without this assumption.

\subsection{The Matrix Perspective is Operator Convex}

The key result is that the matrix perspective is an operator convex
map on a pair of positive-definite matrices.  This theorem follows
from the operator Jensen inequality in much the same way that
Fact~\ref{fact:perspective-convex} follows from scalar convexity.

\begin{thm}[Matrix Perspective is Operator Convex] \label{thm:mtx-perspective-convex}
Let $f : \R_{++} \to \R$ be an operator convex function.  Let
$\mtx{A}_i$ and $\mtx{H}_i$ be positive-definite matrices of the same size.
Then
$$
\mtx{\Psi}_f\big( \tau \mtx{A}_1 + \bar{\tau} \mtx{A}_2;\ \tau \mtx{H}_1 + \bar{\tau} \mtx{H}_2\big)
	\psdle \tau \cdot \mtx{\Psi}_f\big(\mtx{A}_1;\ \mtx{H}_1\big)
	+ \bar{\tau} \cdot \mtx{\Psi}_f\big(\mtx{A}_2;\ \mtx{H}_2\big)
	\quad\text{for $\tau \in [0,1]$.}
$$ 
\end{thm}

\begin{proof}
Let $f$ be an operator convex function, and let $\mtx{\Psi}_f$ be its perspective transform.
Fix pairs $(\mtx{A}_1, \mtx{H}_1)$ and $(\mtx{A}_2, \mtx{H}_2)$ of positive-definite matrices,
and choose an interpolation parameter $\tau \in [0,1]$.
Form the scalar convex combinations
$$
\mtx{A} = \tau \mtx{A}_1 + \bar{\tau} \mtx{A}_2
\quad\text{and}\quad
\mtx{H} = \tau \mtx{H}_1 + \bar{\tau} \mtx{H}_2.
$$
Our goal is to bound the perspective $\mtx{\Psi}_f(\mtx{A}; \mtx{H})$
as a scalar convex combination of its values
$\mtx{\Psi}_f(\mtx{A}_1; \mtx{H}_1)$ and $\mtx{\Psi}_f(\mtx{A}_2; \mtx{H}_2)$.
The idea is to introduce matrix interpolation parameters:
$$
\mtx{K}_1 = \tau^{1/2} \mtx{A}_1^{1/2} \mtx{A}^{-1/2}
\quad\text{and}\quad
\mtx{K}_2 = \bar{\tau}^{1/2} \mtx{A}_2^{1/2} \mtx{A}^{-1/2}.
$$
Observe that these two matrices decompose the identity:
$$
\mtx{K}_1^\adj \mtx{K}_1 + \mtx{K}_2^\adj \mtx{K}_2
	= \tau \cdot \mtx{A}^{-1/2} \mtx{A}_1 \mtx{A}^{-1/2} + \bar{\tau} \cdot \mtx{A}^{-1/2} \mtx{A}_2 \mtx{A}^{-1/2}
	= \mtx{A}^{-1/2} \mtx{A} \mtx{A}^{-1/2}
	= \Id.
$$
This construction allows us to express the perspective using a matrix convex combination, which
gives us access to the operator Jensen inequality.

We calculate that
$$
\begin{aligned}
\mtx{\Psi}_f(\mtx{A}; \mtx{H})
	&= \mtx{A}^{1/2} \cdot f\big(\mtx{A}^{-1/2} \mtx{H} \mtx{A}^{-1/2} \big) \cdot \mtx{A}^{1/2} \\
	&= \mtx{A}^{1/2} \cdot f\big(\tau \cdot \mtx{A}^{-1/2} \mtx{H}_1 \mtx{A}^{-1/2}
		+ \bar{\tau} \cdot \mtx{A}^{-1/2} \mtx{H}_2 \mtx{A}^{-1/2} \big) \cdot \mtx{A}^{1/2} \\
	&= \mtx{A}^{1/2} \cdot f\big( \mtx{K}_1^\adj \mtx{A}_1^{-1/2} \mtx{H}_1 \mtx{A}_1^{-1/2} \mtx{K}_1
		+ \mtx{K}_2^\adj \mtx{A}_2^{-1/2} \mtx{H}_2 \mtx{A}_2^{-1/2} \mtx{K}_2 \big) \cdot \mtx{A}^{1/2}.
\end{aligned}
$$
The first line is simply the definition of the matrix perspective.
In the second line, we use the definition of $\mtx{H}$ as a scalar convex combination.
Third, we introduce the matrix interpolation parameters through the expressions
$\tau^{1/2} \mtx{A}^{-1/2} = \mtx{A}_1^{1/2} \mtx{K}_1$ and $\bar{\tau}^{1/2} \mtx{A}^{-1/2} = \mtx{A}_2^{1/2} \mtx{K}_2$
and their conjugate transposes.
To continue the calculation, we apply the operator Jensen inequality, Theorem~\ref{thm:operator-jensen}, to reach
$$
\begin{aligned}
\mtx{\Psi}_f(\mtx{A}; \mtx{H})
	&\psdle \mtx{A}^{1/2} \cdot \big[ \mtx{K}_1^\adj \cdot f\big(\mtx{A}_1^{-1/2} \mtx{H}_1 \mtx{A}_1^{-1/2}\big) \cdot \mtx{K}_1
		+ \mtx{K}_2^\adj \cdot f\big(\mtx{A}_2^{-1/2} \mtx{H}_2 \mtx{A}_2^{-1/2}\big) \cdot \mtx{K}_2 \big] \cdot \mtx{A}^{1/2} \\
	&= \tau \cdot \mtx{A}_1^{1/2} \cdot f\big(\mtx{A}_1^{-1/2} \mtx{H}_1 \mtx{A}_1^{-1/2}\big) \cdot \mtx{A}_1^{1/2}
		+ \bar{\tau} \cdot \mtx{A}_2^{1/2} \cdot f\big(\mtx{A}_2^{-1/2} \mtx{H}_2 \mtx{A}_2^{-1/2}\big) \cdot \mtx{A}_2^{1/2} \\
	&= \tau \cdot \mtx{\Psi}_f(\mtx{A}_1; \mtx{H}_1) + \bar{\tau} \cdot \mtx{\Psi}_f(\mtx{A}_2; \mtx{H}_2).
\end{aligned}
$$
We have also used the Conjugation Rule~\eqref{eqn:conjugation-rule} to support the first relation.  Finally, we recall the definitions of $\mtx{K}_1$ and $\mtx{K}_2$, and we identify the two matrix perspectives.
\end{proof}

\section{The Kronecker Product}

The matrix relative entropy is a function of two matrices.
One of the difficulties of analyzing this type of
function is that the two matrix arguments do not generally
commute with each other.  As a consequence, the
behavior of the matrix relative entropy depends
on the interactions between the eigenvectors of
the two matrices.  To avoid this problem, we will
build matrices that do commute with each other,
which simplifies our task considerably.

\subsection{The Kronecker Product}

Our approach is based on an fundamental object
from linear algebra.  We restrict our attention to
the simplest version here.

\begin{defn}[Kronecker Product]
Let $\mtx{A}$ and $\mtx{H}$ be Hermitian matrices with dimension $d \times d$.
The \term{Kronecker product} $\mtx{A} \otimes \mtx{H}$ is the $d^2 \times d^2$ Hermitian matrix
$$
\mtx{A} \otimes \mtx{H} =
\begin{bmatrix} a_{11} \mtx{H} & \dots & a_{1d} \mtx{H} \\
\vdots & \ddots & \vdots \\
a_{d1} \mtx{H} & \dots & a_{dd} \mtx{H}
\end{bmatrix}
$$
\end{defn}

\noindent
At first sight, the definition of the Kronecker product may seem strange,
but it has many delightful properties.  The rest of the section develops
the basic facts about this construction.

\subsection{Linearity Properties}

First of all, a Kronecker product with the zero matrix is always zero:
$$
\mtx{A} \otimes \mtx{0} = \mtx{0} \otimes \mtx{0} = \mtx{0} \otimes \mtx{H}
\quad\text{for all $\mtx{A}$ and $\mtx{H}$.}
$$
Next, the Kronecker product is homogeneous in each factor:
$$
(\alpha \mtx{A}) \otimes \mtx{H} = \alpha \, (\mtx{A} \otimes \mtx{H}) = \mtx{A} \otimes (\alpha \mtx{H})
\quad\text{for $\alpha \in \R$.}
$$
Furthermore, the Kronecker product is additive in each coordinate:
$$
(\mtx{A}_1 + \mtx{A}_2) \otimes \mtx{H} = \mtx{A}_1 \otimes \mtx{H} + \mtx{A}_2 \otimes \mtx{H}
\quad\text{and}\quad
\mtx{A} \otimes (\mtx{H}_1 + \mtx{H}_2) = \mtx{A} \otimes \mtx{H}_1 + \mtx{A} \otimes \mtx{H}_2.
$$
In other words, the Kronecker product is a bilinear operation.

\subsection{Mixed Products}

The Kronecker product interacts beautifully with the usual product of matrices.
By direct calculation, we obtain a simple rule for mixed products:
\begin{equation} \label{eqn:kron-prod-rule}
(\mtx{A}_1 \otimes \mtx{H}_1)(\mtx{A}_2 \otimes \mtx{H}_2)
	= (\mtx{A}_1\mtx{A}_2) \otimes (\mtx{H}_1\mtx{H}_2).
\end{equation}
Since $\Id \otimes \Id$ is the identity matrix, the identity~\eqref{eqn:kron-prod-rule}
leads to a formula for the inverse of a Kronecker product:
\begin{equation} \label{eqn:kron-inv-rule}
(\mtx{A} \otimes \mtx{H})^{-1} = \big(\mtx{A}^{-1}\big) \otimes \big(\mtx{H}^{-1}\big)
\quad\text{when $\mtx{A}$ and $\mtx{H}$ are invertible.}
\end{equation}
Another important consequence of the rule~\eqref{eqn:kron-prod-rule} is the following commutativity relation:
\begin{equation} \label{eqn:kron-comm}
(\mtx{A} \otimes \Id)(\Id \otimes \mtx{H}) = (\Id \otimes \mtx{H})(\mtx{A} \otimes \Id)
\quad\text{for all Hermitian matrices $\mtx{A}$ and $\mtx{H}$.}
\end{equation}
This simple fact has great importance for us.

\subsection{The Kronecker Product of Positive Matrices}

As we have noted, the Kronecker product of two Hermitian matrices is itself an Hermitian matrix.
In fact, the Kronecker product preserves positivity as well.

\begin{fact}[Kronecker Product Preserves Positivity]
Let $\mtx{A}$ and $\mtx{H}$ be positive-definite matrices.
Then $\mtx{A} \otimes \mtx{H}$ is positive definite.
\end{fact}

\begin{proof}
To see why, observe that
$$
\mtx{A} \otimes \mtx{H}
	= \big(\mtx{A}^{1/2} \otimes \mtx{H}^{1/2}\big)\big(\mtx{A}^{1/2} \otimes \mtx{H}^{1/2}\big).
$$
As usual, $\mtx{A}^{1/2}$ refers to the unique positive-definite square root of the positive-definite matrix $\mtx{A}$.  We have expressed $\mtx{A} \otimes \mtx{H}$ as the square of an Hermitian matrix, so it must be a positive-semidefinite matrix.  To see that it is actually positive definite, we simply apply the inversion formula~\eqref{eqn:kron-inv-rule} to discover that $\mtx{A} \otimes \mtx{H}$ is invertible.
\end{proof}

\subsection{The Logarithm of a Kronecker Product}

As we have discussed, the matrix logarithm plays a central role in our analysis.
There is an elegant formula for the logarithm of a Kronecker product that will
be valuable to us.

\begin{fact}[Logarithm of a Kronecker Product] \label{fact:kron-log}
Let $\mtx{A}$ and $\mtx{H}$ be positive-definite matrices.  Then
$$
\log(\mtx{A} \otimes \mtx{H}) = (\log \mtx{A}) \otimes \Id + \Id \otimes (\log \mtx{H}).
$$
\end{fact}

\begin{proof}
The argument is based on the fact that the matrix logarithm is the functional inverse of the matrix exponential.
Since the exponential of a sum of commuting matrices equals the product of the exponentials, we have
$$
\exp\big( \mtx{M} \otimes \Id + \Id \otimes \mtx{T} \big)
	= \exp( \mtx{M} \otimes \Id ) \cdot \exp( \Id \otimes \mtx{T} ).
$$
This formula relies on the commutativity relation~\eqref{eqn:kron-comm}.  Applying the power series representation of the exponential, we determine that
$$
\exp( \mtx{M} \otimes \Id )
	 = \sum_{q=0}^\infty \frac{1}{q!} (\mtx{M} \otimes \Id )^q
	 = \sum_{q=0}^\infty \frac{1}{q!} \big(\mtx{M}^q \big) \otimes \Id
	 = \econst^{\mtx{M}} \otimes \Id.
$$
The second identity depends on the rule~\eqref{eqn:kron-prod-rule} for mixed products, and the last identity follows from the linearity of the Kronecker product.  A similar calculation shows that $\exp(\Id \otimes \mtx{T} ) = \Id \otimes \econst^{\mtx{T}}$.  In summary,
$$
\exp\big( \mtx{M} \otimes \Id + \Id \otimes \mtx{T} \big)
	= \big( \econst^{\mtx{M}} \otimes \Id \big) \big( \Id \otimes \econst^{\mtx{T}} \big)
	= \econst^{\mtx{M}} \otimes \econst^{\mtx{T}}.
$$
We have used the product rule~\eqref{eqn:kron-prod-rule} again.  To complete the argument, simply choose $\mtx{M} = \log \mtx{A}$ and $\mtx{T} = \log \mtx{H}$ and take the logarithm of the last identity.
\end{proof}

\subsection{A Linear Map}

Finally, we claim that there is a linear map $\phi$ that extracts the trace of the matrix product from the Kronecker product.
Let $\mtx{A}$ and $\mtx{H}$ be $d\times d$ Hermitian matrices.  Then we define
\begin{equation} \label{eqn:phi-def}
\phi(\mtx{A} \otimes \mtx{H})
	= \trace(\mtx{AH} ).
\end{equation}
The map $\phi$ is linear because the Kronecker product $\mtx{A} \otimes \mtx{H}$
tabulates all the pairwise products of the entries of $\mtx{A}$ and $\mtx{H}$,
and $\trace(\mtx{AH})$ is a sum of certain of these pairwise products.
For our purposes, the key fact is that the map $\phi$ preserves the semidefinite order:
\begin{equation} \label{eqn:phi-monotone}
\sum\nolimits_i \mtx{A}_i \otimes \mtx{H}_i \psdge \mtx{0}
\quad\text{implies}\quad
\sum\nolimits_i \phi( \mtx{A}_i \otimes \mtx{H}_i ) \geq 0.
\end{equation}
This formula is valid for all Hermitian matrices $\mtx{A}_i$ and $\mtx{H}_i$.
To see why~\eqref{eqn:phi-monotone} holds, simply note that the map can be represented as an inner product:
$$
\phi(\mtx{A} \otimes \mtx{H}) = \vct{\iota}^\adj (\mtx{A} \otimes \mtx{H}) \, \vct{\iota}
\quad\text{where}\quad
\vct{\iota} := \mathrm{vec}(\Id_d).
$$
The $\mathrm{vec}$ operation stacks the columns of a $d \times d$ matrix on top of each other, moving from left to right,
to form a column vector of length $d^2$.

\section{The Matrix Relative Entropy is Convex}
\label{sec:mre-convex}

We are finally prepared to establish Theorem~\ref{thm:mre-convex}, which states that the matrix relative entropy is a convex function.  This argument draws on almost all of the ideas we have developed over the course of this chapter.

Consider the function $f(a) = a - 1 - \log a$, defined on the positive real line.  This function is operator convex because it is the sum of the affine function $a \mapsto a - 1$ and the operator convex function $a \mapsto - \log a$. 
The negative logarithm is operator convex because of Proposition~\ref{prop:log-concave}.

Let $\mtx{A}$ and $\mtx{H}$ be positive-definite matrices.  Consider the matrix perspective $\mtx{\Psi}_f$ evaluated at the commuting positive-definite matrices $\mtx{A} \otimes \Id$ and $\Id \otimes \mtx{H}$:
$$
\mtx{\Psi}_f(\mtx{A} \otimes \Id; \ \Id \otimes \mtx{H})
	= (\mtx{A} \otimes \Id) \cdot f\big( (\Id \otimes \mtx{H})(\mtx{A} \otimes \Id)^{-1} \big)
	= (\mtx{A} \otimes \Id) \cdot f\big( \mtx{A}^{-1} \otimes \mtx{H} \big).
$$
We have used the simplified definition~\eqref{eqn:perspective-commute} of the perspective for commuting matrices, and we have invoked the rules~\eqref{eqn:kron-prod-rule} and~\eqref{eqn:kron-inv-rule} for arithmetic with Kronecker products.  Introducing the definition of the function $f$, we find that
$$
\begin{aligned}
\mtx{\Psi}_f(\mtx{A} \otimes \Id;\ \Id \otimes \mtx{H})
	&= (\mtx{A} \otimes \Id) \cdot \big[ \mtx{A}^{-1} \otimes \mtx{H} - \Id \otimes \Id
	- \log\big(\mtx{A}^{-1} \otimes \mtx{H}\big) \big] \\
	&= \Id \otimes \mtx{H} - \mtx{A} \otimes \Id
	- (\mtx{A} \otimes \Id) \cdot \big[ \big( \log \mtx{A}^{-1} \big) \otimes \Id + \Id \otimes\big( \log \mtx{H} \big) \big]
	\\
	&= ( \mtx{A} \log \mtx{A} ) \otimes \Id - \mtx{A} \otimes ( \log\mtx{H} )
	- (\mtx{A} \otimes \Id - \Id \otimes \mtx{H} ) 
\end{aligned}
$$
To reach the second line, we use more Kronecker product arithmetic, along with Fact~\ref{fact:kron-log}, the law for calculating the logarithm of the Kronecker product.  The last line depends on the property that $\log\big(\mtx{A}^{-1}\big) = - \log\mtx{A}$.  Applying the linear map $\phi$ from~\eqref{eqn:phi-def} to both sides, we reach
\begin{equation} \label{eqn:phi-Psi}
\big(\phi \circ \mtx{\Psi}_f \big)\big(\mtx{A} \otimes \Id;\ \Id \otimes \mtx{H} \big)
	= \trace \big[ \mtx{A} \log \mtx{A} - \mtx{A}\log \mtx{H} - (\mtx{A} - \mtx{H}) \big]
	= \mathrm{D}(\mtx{A}; \mtx{H}).
\end{equation}
We have represented the matrix relative entropy in terms of a matrix perspective.

Let $\mtx{A}_i$ and $\mtx{H}_i$ be positive-definite matrices, and fix a parameter $\tau \in [0,1]$.
Theorem~\ref{thm:mtx-perspective-convex} tells us that the matrix perspective is operator convex:
\begin{multline*}
\mtx{\Psi}_f\big( \tau \cdot (\mtx{A}_1 \otimes \Id) + \bar{\tau} \cdot (\mtx{A}_2 \otimes \Id);\
	\tau \cdot (\Id \otimes \mtx{H}_1) + \bar{\tau} \cdot (\Id \otimes \mtx{H}_2) \big) \\
	\psdle \tau \cdot \mtx{\Psi}_f\big(\mtx{A}_1 \otimes \Id;\ \Id \otimes \mtx{H}_1\big)
	+ \bar{\tau} \cdot \mtx{\Psi}_f\big(\mtx{A}_2 \otimes \Id;\ \Id \otimes \mtx{H}_2\big).
\end{multline*}
The inequality~\eqref{eqn:phi-monotone} states that the linear map $\phi$ preserves the semidefinite order.  
\begin{multline*}
\big(\phi \circ \mtx{\Psi}_f\big)\big( \tau \cdot (\mtx{A}_1 \otimes \Id) + \bar{\tau} \cdot (\mtx{A}_2 \otimes \Id);\
	\tau \cdot (\Id \otimes \mtx{H}_1) + \bar{\tau} \cdot (\Id \otimes \mtx{H}_2) \big) \\
	\leq \tau \cdot \big(\phi \circ \mtx{\Psi}_f\big)\big(\mtx{A}_1 \otimes \Id;\ \Id \otimes \mtx{H}_1\big)
	+ \bar{\tau} \cdot \big(\phi \circ \mtx{\Psi}_f\big)\big(\mtx{A}_2 \otimes \Id;\ \Id \otimes \mtx{H}_2\big).
\end{multline*}
Introducing the formula~\eqref{eqn:phi-Psi}, we conclude that
$$
\mathrm{D}\big(\tau \mtx{A}_1 + \bar{\tau} \mtx{A}_2; \ \tau \mtx{H}_1 + \bar{\tau} \mtx{H}_2 \big)
	\leq \tau \cdot \mathrm{D}\big(\mtx{A}_1; \mtx{H}_1 \big) + \bar{\tau} \cdot \mathrm{D}\big(\mtx{A}_2; \mtx{H}_2\big).
$$
The matrix relative entropy is convex.

\section{Notes}

The material in this chapter is drawn from a variety of sources,
ranging from textbooks to lecture notes to contemporary research articles.
The best general sources include the books on matrix analysis by
Bhatia~\cite{Bha97:Matrix-Analysis,Bha07:Positive-Definite}
and by Hiai \& Petz~\cite{HP14:Introduction-Matrix}.  We also
recommend a set of notes~\cite{Car10:Trace-Inequalities}
by Eric Carlen.  More specific references appear below.

\subsection{Lieb's Theorem}

Theorem~\ref{thm:lieb-chap} is one of the major results in the
important paper~\cite{Lie73:Convex-Trace} of Elliott Lieb on convex trace functions.
Lieb wrote this paper to resolve a conjecture of Wigner, Yanase, \& Dyson about the
concavity properties of a certain measure of information in a quantum system.  He
was also motivated by a conjecture that quantum mechanical entropy satisfies a
strong subadditivity property.  The latter result states that our uncertainty about a
partitioned quantum system is controlled by the uncertainty about smaller
parts of the system.  See Carlen's notes~\cite{Car10:Trace-Inequalities} for
a modern presentation of these ideas.

Lieb derived Theorem~\ref{thm:lieb-chap} as a corollary of another difficult
concavity theorem that he developed~\cite[Thm.~1]{Lie73:Convex-Trace}.
The most direct proof of Lieb's Theorem is probably
Epstein's argument, which is based on methods from complex analysis~\cite{Eps73:Remarks-Two};
see Ruskai's papers~\cite{Rus02:Inequalities-Quantum,
Rus05:Erratum-Inequalities} for a condensed version of Epstein's approach.
The proof that appears in Section~\ref{sec:lieb-chap} is due to
the author of these notes~\cite{Tro12:Joint-Convexity};
this technique depends on ideas developed by Carlen \& Lieb
to prove some other convexity theorems~\cite[\S 5]{CL08:Minkowski-Type}.

In fact, many deep convexity and concavity theorems for trace functions are equivalent with each other,
in the sense that the mutual implications follow from relatively easy arguments.
See~\cite[\S5]{Lie73:Convex-Trace} and \cite[\S 5]{CL08:Minkowski-Type} for discussion of this point.

\subsection{The Matrix Relative Entropy}

Our definition of matrix relative entropy differs slightly from the usual definition
in the literature on quantum statistical mechanics and quantum information theory
because we have included an additional linear term.  This alteration does not lead to
substantive changes in the analysis.

The fact that matrix relative entropy is nonnegative is a classical
result attributed to Klein.  See~\cite[\S2]{Pet94:Survey-Certain}
or~\cite[\S2.3]{Car10:Trace-Inequalities}.

Lindblad~\cite{Lin73:Entropy-Information} is credited with the result
that matrix relative entropy is convex, as stated in Theorem~\ref{thm:mre-convex}.
Lindblad derived this theorem as a corollary of Lieb's results from~\cite{Lie73:Convex-Trace}.
Bhatia~\cite[Chap.~IX]{Bha97:Matrix-Analysis} gives two alternative proofs,
one due to Connes \& St{\o}rmer~\cite{CS75:Entropy-Automorphisms} and another
due to Petz~\cite{Pet86:Quasi-Entropies}.  There is also a remarkable proof
due to Ando~\cite[Thm.~7]{And79:Concavity-Certain}.

Our approach to Theorem~\ref{thm:mre-convex} is adapted directly from a recent
paper of Effros~\cite{Eff09:Matrix-Convexity}.  Nevertheless, many of the ideas
date back to the works cited in the last paragraph.

\subsection{The Relative Entropy for Vectors}

The treatment of the relative entropy for vectors in Section~\ref{sec:vre}
is based on two classical methods for constructing divergences.
To show that the relative entropy is nonnegative, we represent it as a Bregman divergence~\cite{Bre67:Relaxation-Method}.
To show that the relative entropy is convex, we represent it as an $f$-divergence~\cite{AS66:General-Class,Csi67:Information-Type-Measures}.
Let us say a few more words about these constructions.

Suppose that $f$ is a differentiable convex function on $\R^d$.
Bregman considered divergences of the form
$$
\mathrm{B}_f(\vct{a}; \vct{h}) := f(\vct{a}) - \big[ f(\vct{h}) - \ip{ \grad \smash{f}(\vct{h}) }{ \vct{a} - \vct{h} } \big].
$$
Since $f$ is convex, the Bregman divergence $\mathrm{B}_f$ is always nonnegative.
In the vector setting, there are two main examples of Bregman divergences.
The function $f(\vct{a}) = \half \enormsq{\vct{a}}$ leads to the squared Euclidean distance,
and the function $f(\vct{a}) = \sum_i (a_i \log a_i - a_i)$ leads to the vector relative entropy.
Bregman divergences have many geometric properties in common with these two functions.
For an introduction to Bregman divergences for matrices, see~\cite{DT06:Matrix-Nearness}.

Suppose that $f : \R_{++} \to \R$ is a convex function.
Ali \& Silvey~\cite{AS66:General-Class} and Csisz{\'a}r~\cite{Csi67:Information-Type-Measures} considered divergences of the form
$$
\mathrm{C}_f(\vct{a}; \vct{h}) := \sum\nolimits_i a_i \cdot f(h_i/a_i).
$$
We recognize this expression as a perspective transformation,
so the $f$-divergence $\mathrm{C}_f$ is always convex.
The main example is based on the Shannon entropy $f(a) = a \log a$,
which leads to a cousin of the vector relative entropy.
The paper~\cite{RW11:Information-Divergence} contains a recent discussion
of $f$-divergences and their applications in machine learning.
Petz has studied functions related to $f$-divergences in the matrix
setting~\cite{Pet86:Quasi-Entropies,Pet10:From-f-Divergence}.

\subsection{Elementary Trace Inequalities}

The material in Section~\ref{sec:trace-ineqs} on trace functions
is based on classical results in quantum statistical mechanics.
We have drawn the arguments in this section from Petz's survey~\cite[Sec.~2]{Pet94:Survey-Certain}
and Carlen's lecture notes~\cite[Sec.~2.2]{Car10:Trace-Inequalities}.

\subsection{Operator Monotone \& Operator Convex Functions}

The theory of operator monotone functions was initiated by L{\"o}wner~\cite{Loe34:Ueber-Monotone}.
He developed a characterization of an operator monotone function in terms of divided differences.
For a function $f$, the \term{first divided difference} is the
quantity
$$
f[a,h] := \begin{dcases} \frac{f(a) - f(h)}{a - h}, & h \neq a \\
	f'(a), & h = a.
	\end{dcases}
$$
L{\"o}wner proved that $f$ is operator monotone on an interval $I$ if and only
we have the semidefinite relation
$$
\begin{bmatrix} f[a_1, a_1] & \dots & f[a_1, a_d] \\
	\vdots & \ddots & \vdots \\
	f[a_d, a_1] & \dots & f[a_d, a_d] \end{bmatrix}
	\psdge \vct{0}
	\quad\text{for all $\{a_i\} \subset I$ and all $d \in \mathbb{N}$.}
$$
This result is analogous with the fact that a smooth, monotone scalar function has a nonnegative
derivative.  L{\"o}wner also established a connection between operator monotone functions
and Pick functions from the theory of complex variables.
A few years later, Kraus introduced the concept of an operator convex function
in~\cite{Kra36:Ueber-Konvexe},
and he developed some results that parallel L{\"o}wner's theory for operator monotone functions.

Somewhat later, Bendat \& Sherman~\cite{BS55:Monotone-Convex} developed characterizations
of operator monotone and operator convex functions based on integral formulas.  For example,
$f$ is an operator monotone function on $(0, \infty)$ if and only if it can be written in the form
$$
f(t) = \alpha + \beta t + \int_0^\infty \frac{ut}{u+t} \idiff{\rho}(u)
\quad\text{where}\quad
\beta \geq 0
\quad\text{and}\quad
\int_0^\infty \frac{u}{1+u} \idiff{\rho}(u) < \infty.
$$
Similarly, $f$ is an operator convex function on $[0,\infty)$ if and only if it can be written in the form
$$
f(t) = \alpha + \beta t + \gamma t^2 + \int_0^\infty \frac{ut^2}{u+t} \idiff{\rho}(u)
\quad\text{where}\quad
\gamma \geq 0
\quad\text{and}\quad
\int_0^\infty \diff{\rho}(u) < \infty.
$$
In both cases, $\diff{\rho}$ is a nonnegative measure.
The integral representation of the logarithm in Proposition~\ref{prop:log-integral}
is closely related to these formulas.

We have taken the proof that the matrix inverse is monotone from Bhatia's book~\cite[Prop.~V.1.6]{Bha97:Matrix-Analysis}.
The proof that the matrix inverse is convex appears in Ando's paper~\cite{And79:Concavity-Certain}.
Our treatment of the matrix logarithm was motivated by a conversation with Eric Carlen
at an IPAM workshop at Lake Arrowhead in December 2010.

For more information about operator monotonicity and operator convexity,
we recommend Bhatia's books~\cite{Bha97:Matrix-Analysis,Bha07:Positive-Definite},
Carlen's lecture notes~\cite{Car10:Trace-Inequalities},
and the book of Hiai \& Petz~\cite{HP14:Introduction-Matrix}.

\subsection{The Operator Jensen Inequality}

The paper~\cite{HP82:Jensens-Inequality} of Hansen \& Pedersen contains another
treatment of operator monotone and operator convex functions.  The highlight
of this work is a version of the operator Jensen inequality.  Theorem~\ref{thm:operator-jensen}
is a refinement of this result that was established by the same authors two decades
later~\cite{HP03:Jensens-Operator}.
Our proof of the operator Jensen inequality is drawn from
Petz's book~\cite[Thm.~8.4]{Pet11:Matrix-Analysis};
see also Carlen's lecture notes~\cite[Thm.~4.20]{Car10:Trace-Inequalities}.

\subsection{The Matrix Perspective \& the Kronecker Product}

We have been unable to identify the precise source of the idea that
a bivariate matrix function can be represented in terms of a
matrix perspective.  Two important results in this direction appear
in Ando's paper~\cite[Thms.~6 and 7]{And79:Concavity-Certain}.
\begin{multline*}
\text{$f$ positive and operator~concave on $(0,\infty)$}
\quad\text{implies}\quad \\
(\mtx{A}, \mtx{H}) \mapsto (\mtx{A} \otimes \Id) \cdot f\big( \mtx{A}^{-1} \otimes \mtx{H} \big)
\quad \text{is operator concave}
\end{multline*}
on pairs of positive-definite matrices.  Similarly,
\begin{multline*}
\text{$f$ operator~monotone~on $(0,\infty)$}
\quad\text{implies}\quad
(\mtx{A}, \mtx{H}) \mapsto (\mtx{A} \otimes \Id) \cdot f\big( \mtx{A} \otimes \mtx{H}^{-1} \big)
\quad \text{is operator~convex}
\end{multline*}
on pairs of positive-definite matrices.  Ando proves that
the matrix relative entropy is convex by applying the
latter result to the matrix logarithm.
We believe that Ando was the first author to appreciate the value
of framing results of this type in terms of the Kronecker product,
and we have followed his strategy here.
On the other hand, Ando's analysis is different in spirit
because he relies on integral representations of operator monotone
and convex functions.

In a subsequent paper~\cite{KA80:Means-Positive},
Kubo \& Ando constructed operator means using a related approach.
They show that
\begin{multline*}
\text{$f$ positive and operator~monotone~on $(0,\infty)$}
\quad\text{implies}\quad \\
(\mtx{A}, \mtx{H}) \mapsto \mtx{A}^{1/2} \cdot f\big( \mtx{A}^{-1/2} \mtx{H} \mtx{A}^{-1/2} \big) \cdot \mtx{A}^{1/2}
\quad \text{is operator concave}
\end{multline*}
on pairs of positive-definite matrices.
Kubo \& Ando point out that particular cases of this construction appear in
the work of Pusz \& Woronowicz~\cite{PW75:Functional-Calculus}.  This is the earliest citation
where we have seen the matrix perspective black-on-white.

A few years later, Petz introduced a class of quasi-entropies for
matrices~\cite{Pet86:Quasi-Entropies}.  These functions also involve
a perspective-like construction, and Petz was clearly influenced
by Csisz{\'a}r's work on $f$-divergences.
See~\cite{Pet10:From-f-Divergence} for a contemporary treatment.

The presentation in these notes is based on a recent paper~\cite{Eff09:Matrix-Convexity}
of Effros.  He showed that convexity properties of the matrix perspective
follow from the operator Jensen inequality, and he derived
the convexity of the matrix relative entropy as a consequence.
Our analysis of the matrix perspective in Theorem~\ref{thm:mtx-perspective-convex}
is drawn from a subsequent paper~\cite{ENG11:Perspectives-Matrix}, which removes some
commutativity assumptions from Effros's argument.

The proof in \S\ref{sec:mre-convex} that the matrix relative entropy is convex, Theorem~\ref{thm:mre-convex},
recasts Effros's argument~\cite[Cor.~2.2]{Eff09:Matrix-Convexity} in the language
of Kronecker products.  In his paper, Effros works with left- and
right-multiplication operators.  To appreciate the connection, simply
note the identities
$$
(\mtx{A} \otimes \Id) \, \mathrm{vec}(\mtx{M})
	= \mathrm{vec}( \mtx{M} \mtx{A} ).
\quad\text{and}\quad	
(\Id \otimes \mtx{H}) \, \mathrm{vec}(\mtx{M})
	= \mathrm{vec}( \mtx{H}\mtx{M} ).
$$
In other words, the matrix $\mtx{A} \otimes \Id$ can be interpreted as
right-multiplication by $\mtx{A}$, while the matrix $\Id \otimes \mtx{H}$
can be interpreted as left-multiplication by $\mtx{H}$.  (The change in
sense is an unfortunate consequence of the definition of the Kronecker product.)

%
%
%
%
%
%
%
%
%
%

%
%
%
%
%
%
%
%
%
%
%
%
%

%
%
%
%
%
%
%
%

%
%
%
%
%
%
%
%
%
%
%
%
%
%
%
%
%
%
%
%
%
%

\backmatter

\makeatletter{}%

\chapter{Matrix Concentration: Resources}

This annotated bibliography describes some papers that involve matrix concentration inequalities.
Right now, this presentation is heavily skewed toward theoretical results,
rather than applications of matrix concentration.

\subsection{Exponential Matrix Concentration Inequalities}

We begin with papers that contain the most current results on matrix concentration.

\begin{itemize}
\item	\cite{Tro11:User-Friendly-FOCM}.  These lecture notes are based heavily on the research described in this paper.  This work identifies Lieb's Theorem~\cite[Thm.~6]{Lie73:Convex-Trace} as the key result that animates exponential moment bounds for random matrices.  Using this technique, the paper develops the bounds for matrix Gaussian and Rademacher series, the matrix Chernoff inequalities, and several versions of the matrix Bernstein inequality.  In addition, it contains a matrix Hoeffding inequality (for sums of bounded random matrices), a matrix Azuma inequality (for matrix martingales with bounded differences), and a matrix bounded difference inequality (for matrix-valued functions of independent random variables).

\item	\cite{Tro12:Joint-Convexity}.  This note describes a simple proof of Lieb's Theorem that is based on the joint convexity of quantum relative entropy.  This reduction, however, still involves a deep convexity theorem.  Chapter~\ref{chap:lieb} contains an explication of this paper.

\item	\cite{Oli10:Concentration-Adjacency}.  Oliveira's paper uses an ingenious argument, based on the Golden--Thompson inequality~\eqref{eqn:golden-thompson}, to establish a matrix version of Freedman's inequality.  This result is, roughly, a martingale version of Bernstein's inequality.  This approach has the advantage that it extends to the fully noncommutative setting~\cite{JZ12:Noncommutative-Martingale}.  Oliveira applies his results to study some problems in random graph theory.  

\item	\cite{Tro11:Freedmans-Inequality}.  This paper shows that Lieb's Theorem leads to a Freedman-type inequality for matrix-valued martingales.  The associated technical report~\cite{Tro11:User-Friendly-Martingale-TR} describes additional results for matrix-valued martingales.

\item	\cite{GT11:Tail-Bounds}.  This article explains how to use the Lieb--Seiringer Theorem~\cite{LS05:Stronger-Subadditivity} to develop tail bounds for the interior eigenvalues of a sum of independent random matrices.  It contains a Chernoff-type bound for a sum of positive-semidefinite matrices, as well as several Bernstein-type bounds for sums of bounded random matrices.

\item	\cite{MJCFT12:Matrix-Concentration}.  This paper contains a strikingly different method for establishing matrix concentration inequalities.  The argument is based on work of Sourav Chatterjee~\cite{Cha07:Steins-Method} that shows how Stein's method of exchangeable pairs~\cite{Ste72:Bound-Error} leads to probability inequalities.  This technique has two main advantages.  First, it gives results for random matrices that are based on dependent random variables.  As a special case, the results apply to sums of independent random matrices.  Second, it delivers both exponential moment bounds and polynomial moment bounds for random matrices.  Indeed, the paper describes a Bernstein-type exponential inequality and also a Rosenthal-type polynomial moment bound.  Furthermore, this work contains what is arguably the simplest known proof of the noncommutative Khintchine inequality.

\item	\cite{PMT14:Efron-Stein}.  This paper improves on the work in~\cite{MJCFT12:Matrix-Concentration} by extending an argument, based on Markov chains, that was developed in Chatterjee's thesis~\cite{Cha08:Concentration-Inequalities}.  This analysis leads to satisfactory matrix analogs of scalar concentration inequalities based on logarithmic Sobolev inequalities.  In particular, it is possible to develop a matrix version of the exponential Efron--Stein inequality in this fashion.

\item	\cite{CGT12:Masked-Sample, CGT12:Masked-Sample-TR}.  The primary focus of this paper is to analyze a specific type of procedure for covariance estimation.  The appendix contains a new matrix moment inequality that is, roughly, the polynomial moment bound associated with the matrix Bernstein inequality.

\item	\cite{Kol11:Oracle-Inequalities}.  These lecture notes use matrix concentration inequalities as a tool to study some estimation problems in statistics.  They also contain some matrix Bernstein inequalities for unbounded random matrices.

\item	\cite{GN10:Note-Sampling}.  Gross and Nesme show how to extend Hoeffding's method for analyzing sampling without replacement to the matrix setting.  This result can be combined with a variety of matrix concentration inequalities.

\item	\cite{Tro11:Improved-Analysis}.  This paper combines the matrix Chernoff inequality, Theorem~\ref{thm:matrix-chernoff}, with the argument from~\cite{GN10:Note-Sampling} to obtain a matrix Chernoff bound for a sum of random positive-semidefinite matrices sampled without replacement from a fixed collection.  The result is applied to a random matrix that plays a role in numerical linear algebra.

\item	\cite{CT14:Subadditivity-Matrix}.  This paper establishes logarithmic Sobolev inequalities for random matrices, and it derives some matrix concentration inequalities as a consequence.  The methods in the paper have applications in quantum information theory, although the matrix concentration bounds are inferior to related results derived using Stein's method.
\end{itemize}

\subsection{Bounds with Intrinsic Dimension Parameters}

The following works contain matrix concentration bounds that depend on a dimension parameter that may be smaller than the ambient dimension of the matrix.

\begin{itemize}
\item	\cite{Oli10:Sums-Random}.  Oliveira shows how to develop a version of Rudelson's inequality~\cite{Rud99:Random-Vectors} using a variant of the argument of Ahlswede \& Winter from~\cite{AW02:Strong-Converse}.  Oliveira's paper is notable because the dimensional factor is controlled by the maximum rank of the random matrix, rather than the ambient dimension.

\item	\cite{MZ11:Low-Rank-Matrix-Valued}.  This work contains a matrix Chernoff bound for a sum of independent positive-semidefinite random matrices where the dimensional dependence is controlled by the maximum rank of the random matrix.  The approach is, essentially, the same as the argument in Rudelson's paper~\cite{Rud99:Random-Vectors}.  The paper applies these results to study randomized matrix multiplication algorithms.

\item	\cite{HKZ12:Tail-Inequalities}.  This paper describes a method for proving matrix concentration inequalities where the ambient dimension is replaced by the intrinsic dimension of the matrix variance.  The argument is based on an adaptation of the proof in~\cite{Tro11:Freedmans-Inequality}.  The authors give several examples in statistics and machine learning.

\item	\cite{Min11:Some-Extensions}.  This work presents a more refined technique for obtaining matrix concentration inequalities that depend on the intrinsic dimension, rather than the ambient dimension.
This paper motivated the results in Chapter~\ref{chap:intrinsic}.

\end{itemize}

\subsection{The Method of Ahlswede \& Winter}

Next, we list some papers that use the ideas from the work~\cite{AW02:Strong-Converse} of Ahslwede \& Winter to obtain matrix concentration inequalities.  In general, these results have suboptimal parameters, but they played an important role in the development of this field.

\begin{itemize}
\item	\cite{AW02:Strong-Converse}.  The original paper of Ahlswede \& Winter describes the matrix Laplace transform method, along with a number of other foundational results.  They show how to use the Golden--Thompson inequality to bound the trace of the matrix mgf, and they use this technique to prove a matrix Chernoff inequality for sums of independent and identically distributed random variables.  Their main application concerns quantum information theory.

\item	\cite{CM08:Expansion-Properties}.  Christofides and Markstr{\"o}m develop a Hoeffding-type inequality for sums of bounded random matrices using the approach of Ahlswede \& Winter.  They apply this result to study random graphs.

\item	\cite{Gro11:Recovering-Low-Rank}.	Gross presents a matrix Bernstein inequality based on the method of Ahlswede \& Winter, and he uses it to study algorithms for matrix completion.

\item	\cite{Rec11:Simpler-Approach}.  Recht describes a different version of the matrix Bernstein inequality, which also follows from the technique of Ahlswede \& Winter.  His paper also concerns algorithms for matrix completion.
\end{itemize}

\subsection{Noncommutative Moment Inequalities}

We conclude with an overview of some major works on bounds for the polynomial moments of a noncommutative martingale.  Sums of independent random matrices provide one concrete example where these results apply.  The results in this literature are as strong, or stronger, than the exponential moment inequalities that we have described in these notes.  Unfortunately, the proofs are typically quite abstract and difficult, and they do not usually lead to explicit constants.  Recently there has been some cross-fertilization between noncommutative probability and the field of matrix concentration inequalities.

Note that ``noncommutative'' is not synonymous with ``matrix'' in that there are noncommutative von Neumann algebras much stranger than the familiar algebra of finite-dimensional matrices equipped with the operator norm.

\begin{itemize}
\item	\cite{TJ74:Moduli-Smoothness}.  This classic paper gives a bound for the expected trace of an even power of a matrix Rademacher series.  These results are important, but they do not give the optimal bounds.

\item	\cite{LP86:Inegalites-Khintchine}.  This paper gives the first noncommutative Khintchine inequality, a bound for the expected trace of an even power of a matrix Rademacher series that depends on the matrix variance.

\item	\cite{LPP91:Noncommutative-Khintchine}.  This work establishes dual versions of the noncommutative Khintchine inequality.

\item	\cite{Buc01:Operator-Khintchine,Buc05:Optimal-Constants}.  These papers prove optimal noncommutative Khintchine inequalities in more general settings, and they obtain sharp constants.

\item	\cite{JX03:Noncommutative-Burkholder,JX08:Noncommutative-Burkholder-II}.  These papers establish noncommutative versions of the Burkholder--Davis--Gundy inequality for martingales.  They also give an application of these results to random matrix theory.

\item	\cite{JX05:Best-Constants}.  This paper contains an overview of noncommutative moment results, along with information about the optimal rate of growth in the constants.

\item	\cite{JZ13:Noncommutative-Bennett}.  This paper describes a fully noncommutative version of the Bennett inequality.  The proof is based on the method of Ahlswede \& Winter~\cite{AW02:Strong-Converse}.

\item	\cite{JZ12:Noncommutative-Martingale}.  This work shows how to use Oliveira's argument~\cite{Oli10:Concentration-Adjacency} to obtain some results for fully noncommutative martingales.

\item	\cite{MJCFT12:Matrix-Concentration}.  This work, described above, includes a section on matrix moment inequalities.  This paper contains what are probably the simplest available proofs of these results.

\item	\cite{CGT12:Masked-Sample}.  The appendix of this paper contains a polynomial inequality for sums of independent random matrices.
\end{itemize}

\bibliographystyle{myalpha}

\begin{thebibliography}{LPSS{\etalchar{+}}14}

\bibitem[ABH14]{ABH14:Exact-Recovery}
E.~Abb{\'e}, A.~S. Bandeira, and G.~Hall.
\newblock Exact recovery in the stochastic block model.
\newblock Available at \url{http://arXiv.org/abs/1405.3267}, June 2014.

\bibitem[AC09]{AC09:Fast-Johnson-Lindenstrauss}
N.~Ailon and B.~Chazelle.
\newblock The fast {J}ohnson--{L}indenstrauss transform and approximate nearest
  neighbors.
\newblock {\em SIAM J. Comput.}, 39(1):302--322, 2009.

\bibitem[AHK06]{AHK06:Fast-Random}
S.~Arora, E.~Hazan, and S.~Kale.
\newblock A fast random sampling algorithm for sparsifying matrices.
\newblock In {\em Approximation, Randomization, and Combinatorial Optimization.
  Algorithms and Techniques}, pages 272--279, 2006.

\bibitem[AKL13]{AKL13:Near-Optimal-Entrywise}
D.~Achlioptas, Z.~Karnin, and E.~Liberty.
\newblock Near-optimal entrywise sampling for data matrices.
\newblock In {\em Advances in Neural Information Processing Systems 26}, 2013.

\bibitem[ALMT14]{ALMT14:Living-Edge}
D.~Amelunxen, M.~Lotz, M.~B. McCoy, and J.~A. Tropp.
\newblock Living on the edge: A geometric theory of phase transitions in convex
  optimization.
\newblock {\em Inform. Inference}, 3(3):224--294, 2014.
\newblock Preprint available at \url{http://arXiv.org/abs/1303.6672}.

\bibitem[AM01]{AM01:Fast-Computation}
D.~Achlioptas and F.~McSherry.
\newblock Fast computation of low rank matrix approximations.
\newblock In {\em Proceedings of the {T}hirty-{T}hird {A}nnual {ACM}
  {S}ymposium on {T}heory of {C}omputing}, pages 611--618 (electronic). ACM,
  New York, 2001.

\bibitem[AM07]{AM07:Fast-Computation}
D.~Achlioptas and F.~McSherry.
\newblock Fast computation of low-rank matrix approximations.
\newblock {\em J. Assoc. Comput. Mach.}, 54(2):Article 10, 2007.
\newblock (electronic).

\bibitem[And79]{And79:Concavity-Certain}
T.~Ando.
\newblock Concavity of certain maps on positive definite matrices and
  applications to {H}adamard products.
\newblock {\em Linear Algebra Appl.}, 26:203--241, 1979.

\bibitem[AS66]{AS66:General-Class}
S.~M. Ali and S.~D. Silvey.
\newblock A general class of coefficients of divergence of one distribution
  from another.
\newblock {\em J. Roy. Statist. Soc. Ser. B}, 28:131--142, 1966.

\bibitem[AS00]{AS00:Probabilistic-Method}
N.~Alon and J.~H. Spencer.
\newblock {\em The probabilistic method}.
\newblock Wiley-Interscience Series in Discrete Mathematics and Optimization.
  Wiley-Interscience [John Wiley \& Sons], New York, second edition, 2000.
\newblock With an appendix on the life and work of Paul Erd{\H{o}}s.

\bibitem[AW02]{AW02:Strong-Converse}
R.~Ahlswede and A.~Winter.
\newblock Strong converse for identification via quantum channels.
\newblock {\em IEEE Trans. Inform. Theory}, 48(3):569--579, Mar. 2002.

\bibitem[Bar93]{Bar93:Universal-Approximation}
A.~R. Barron.
\newblock Universal approximation bounds for superpositions of a sigmoidal
  function.
\newblock {\em IEEE Trans. Inform. Theory}, 39(3):930--945, May 1993.

\bibitem[Bar14]{Bar14:Approximate-Version}
S.~Barman.
\newblock An approximate version of {C}arath{\'e}odory's theorem with
  applications to approximating {N}ash equilibria and dense bipartite
  subgraphs.
\newblock Available at \url{http://arXiv.org/abs/1406.2296}, June 2014.

\bibitem[BBLM05]{BBLM05:Moment-Inequalities}
S.~Boucheron, O.~Bousquet, G.~Lugosi, and P.~Massart.
\newblock Moment inequalities for functions of independent random variables.
\newblock {\em Ann. Probab.}, 33(2):514--560, 2005.

\bibitem[BDJ06]{BDJ06:Spectral-Measure}
W.~Bryc, A.~Dembo, and T.~Jiang.
\newblock Spectral measure of large random {H}ankel, {M}arkov and {T}oeplitz
  matrices.
\newblock {\em Ann. Probab.}, 34(1):1--38, 2006.

\bibitem[Bha97]{Bha97:Matrix-Analysis}
R.~Bhatia.
\newblock {\em Matrix Analysis}.
\newblock Number 169 in Graduate Texts in Mathematics. Springer, Berlin, 1997.

\bibitem[Bha07]{Bha07:Positive-Definite}
R.~Bhatia.
\newblock {\em Positive Definite Matrices}.
\newblock Princeton Univ. Press, Princeton, NJ, 2007.

\bibitem[BLM03]{BLM03:Concentration-Inequalities}
S.~Boucheron, G.~Lugosi, and P.~Massart.
\newblock Concentration inequalities using the entropy method.
\newblock {\em Ann. Probab.}, 31(3):1583--1614, 2003.

\bibitem[BLM13]{BLM13:Concentration-Inequalities}
S.~Boucheron, G.~Lugosi, and P.~Massart.
\newblock {\em Concentration Inequalities}.
\newblock Oxford University Press, Oxford, 2013.
\newblock A nonasymptotic theory of independence, With a foreword by Michel
  Ledoux.

\bibitem[Bou85]{Bou85:Lipschitz-Embedding}
J.~Bourgain.
\newblock On {L}ipschitz embedding of finite metric spaces in {H}ilbert space.
\newblock {\em Israel J. Math.}, 52(1-2):46--52, 1985.

\bibitem[Br{\`e}67]{Bre67:Relaxation-Method}
L.~M. Br{\`e}gman.
\newblock A relaxation method of finding a common point of convex sets and its
  application to the solution of problems in convex programming.
\newblock {\em \u Z. Vy\v cisl. Mat. i Mat. Fiz.}, 7:620--631, 1967.

\bibitem[BS55]{BS55:Monotone-Convex}
J.~Bendat and S.~Sherman.
\newblock Monotone and convex operator functions.
\newblock {\em Trans. Amer. Math. Soc.}, 79:58--71, 1955.

\bibitem[BS10]{BS10:Spectral-Analysis}
Z.~Bai and J.~W. Silverstein.
\newblock {\em Spectral analysis of large dimensional random matrices}.
\newblock Springer Series in Statistics. Springer, New York, second edition,
  2010.

\bibitem[BT87]{BT87:Invertibility-Large}
J.~Bourgain and L.~Tzafriri.
\newblock Invertibility of ``large'' submatrices with applications to the
  geometry of {B}anach spaces and harmonic analysis.
\newblock {\em Israel J. Math.}, 57(2):137--224, 1987.

\bibitem[BT91]{BT91:Problem-Kadison-Singer}
J.~Bourgain and L.~Tzafriri.
\newblock On a problem of {K}adison and {S}inger.
\newblock {\em J. Reine Angew. Math.}, 420:1--43, 1991.

\bibitem[BTN01]{BN01:Lectures-Modern}
A.~Ben-Tal and A.~Nemirovski.
\newblock {\em Lectures on modern convex optimization}.
\newblock MPS/SIAM Series on Optimization. Society for Industrial and Applied
  Mathematics (SIAM), Philadelphia, PA; Mathematical Programming Society (MPS),
  Philadelphia, PA, 2001.
\newblock Analysis, algorithms, and engineering applications.

\bibitem[Buc01]{Buc01:Operator-Khintchine}
A.~Buchholz.
\newblock Operator {K}hintchine inequality in non-commutative probability.
\newblock {\em Math. Ann.}, 319:1--16, 2001.

\bibitem[Buc05]{Buc05:Optimal-Constants}
A.~Buchholz.
\newblock Optimal constants in {K}hintchine-type inequalities for {F}ermions,
  {R}ademachers and $q$-{G}aussian operators.
\newblock {\em Bull. Pol. Acad. Sci. Math.}, 53(3):315--321, 2005.

\bibitem[BV14]{BV14:Sharp-Nonasymptotic}
A.~S. Bandeira and R.~Van{ }Handel.
\newblock Sharp nonasymptotic bounds on the norm of random matrices with
  independent entries.
\newblock Available at \url{http://arXiv.org/abs/1408.6185}, Aug. 2014.

\bibitem[BvdG11]{BG11:Statistics-High-Dimensional}
P.~B{\"u}hlmann and S.~van~de Geer.
\newblock {\em Statistics for high-dimensional data}.
\newblock Springer Series in Statistics. Springer, Heidelberg, 2011.
\newblock Methods, theory and applications.

\bibitem[BY93]{BY93:Limit-Smallest}
Z.~D. Bai and Y.~Q. Yin.
\newblock Limit of the smallest eigenvalue of a large-dimensional sample
  covariance matrix.
\newblock {\em Ann. Probab.}, 21(3):1275--1294, 1993.

\bibitem[Car85]{Car85:Inequalities-Bernstein}
B.~Carl.
\newblock Inequalities of {B}ernstein--{J}ackson-type and the degree of
  compactness in {B}anach spaces.
\newblock {\em Ann. Inst. Fourier (Grenoble)}, 35(3):79--118, 1985.

\bibitem[Car10]{Car10:Trace-Inequalities}
E.~Carlen.
\newblock Trace inequalities and quantum entropy: an introductory course.
\newblock In {\em Entropy and the quantum}, volume 529 of {\em Contemp. Math.},
  pages 73--140. Amer. Math. Soc., Providence, RI, 2010.

\bibitem[CD12]{CD12:Invertibility-Random}
S.~Chr{\'e}tien and S.~Darses.
\newblock Invertibility of random submatrices via tail-decoupling and a matrix
  {C}hernoff inequality.
\newblock {\em Statist. Probab. Lett.}, 82(7):1479--1487, 2012.

\bibitem[CGT12a]{CGT12:Masked-Sample}
R.~Y. Chen, A.~Gittens, and J.~A. Tropp.
\newblock The masked sample covariance estimator: {A}n analysis using matrix
  concentration inequalities.
\newblock {\em Inform. Inference}, 1(1), 2012.
\newblock \url{doi:10.1093/imaiai/ias001}.

\bibitem[CGT12b]{CGT12:Masked-Sample-TR}
R.~Y. Chen, A.~Gittens, and J.~A. Tropp.
\newblock The masked sample covariance estimator: {A}n analysis using matrix
  concentration inequalities.
\newblock ACM Report 2012-01, California Inst. Tech., Pasadena, CA, Feb. 2012.

\bibitem[Cha05]{Cha08:Concentration-Inequalities}
S.~Chatterjee.
\newblock {\em Concentration Inequalities with Exchangeable Pairs}.
\newblock ProQuest LLC, Ann Arbor, MI, 2005.
\newblock Thesis (Ph.D.)--Stanford University.

\bibitem[Cha07]{Cha07:Steins-Method}
S.~Chatterjee.
\newblock Stein's method for concentration inequalities.
\newblock {\em Probab. Theory Related Fields}, 138:305--321, 2007.

\bibitem[Che52]{Che52:Measure-Asymptotic}
H.~Chernoff.
\newblock A measure of asymptotic efficiency for tests of a hypothesis based on
  the sum of observations.
\newblock {\em Ann. Math. Statistics}, 23:493--507, 1952.

\bibitem[CL08]{CL08:Minkowski-Type}
E.~A. Carlen and E.~H. Lieb.
\newblock A {M}inkowski type trace inequality and strong subadditivity of
  quantum entropy. {II}. {C}onvexity and concavity.
\newblock {\em Lett. Math. Phys.}, 83(2):107--126, 2008.

\bibitem[CM08]{CM08:Expansion-Properties}
D.~Cristofides and K.~Markstr{\"o}m.
\newblock Expansion properties of random {C}ayley graphs and vertex transitive
  graphs via matrix martingales.
\newblock {\em Random Structures Algs.}, 32(8):88--100, 2008.

\bibitem[CRPW12]{CRPW12:Convex-Geometry}
V.~Chandrasekaran, B.~Recht, P.~A. Parrilo, and A.~S. Willsky.
\newblock The {C}onvex {G}eometry of {L}inear {I}nverse {P}roblems.
\newblock {\em Found. Comput. Math.}, 12(6):805--849, 2012.

\bibitem[CRT06]{CRT06:Robust-Uncertainty}
E.~J. Cand{\`e}s, J.~Romberg, and T.~Tao.
\newblock Robust uncertainty principles: exact signal reconstruction from
  highly incomplete frequency information.
\newblock {\em IEEE Trans. Inform. Theory}, 52(2):489--509, 2006.

\bibitem[CS75]{CS75:Entropy-Automorphisms}
A.~Connes and E.~St{\o}rmer.
\newblock Entropy for automorphisms of {$II\sb{1}$} von {N}eumann algebras.
\newblock {\em Acta Math.}, 134(3-4):289--306, 1975.

\bibitem[Csi67]{Csi67:Information-Type-Measures}
I.~Csisz{\'a}r.
\newblock Information-type measures of difference of probability distributions
  and indirect observations.
\newblock {\em Studia Sci. Math. Hungar.}, 2:299--318, 1967.

\bibitem[CT14]{CT14:Subadditivity-Matrix}
R.~Y. Chen and J.~A. Tropp.
\newblock Subadditivity of matrix $\phi$-entropy and concentration of random
  matrices.
\newblock {\em Electron. J. Probab.}, 19(27):1--30, 2014.

\bibitem[CW13]{CW13:Low-Rank-Approximation}
K.~L. Clarkson and D.~P. Woodruff.
\newblock Low rank approximation and regression in input sparsity time.
\newblock In {\em S{TOC}'13---{P}roceedings of the 2013 {ACM} {S}ymposium on
  {T}heory of {C}omputing}, pages 81--90. ACM, New York, 2013.

\bibitem[d'A11]{dAs11:Subsampling-Algorithms}
A.~d'Aspr{\'e}mont.
\newblock Subsampling algorithms for semidefinite programming.
\newblock {\em Stoch. Syst.}, 1(2):274--305, 2011.

\bibitem[DFK{\etalchar{+}}99]{DFK+99:Clustering-Large}
P.~Drineas, A.~Frieze, R.~Kannan, S.~Vempala, and V.~Vinay.
\newblock Clustering in large graphs and matrices.
\newblock In {\em Proceedings of the {T}enth {A}nnual {ACM}-{SIAM} {S}ymposium
  on {D}iscrete {A}lgorithms ({B}altimore, {MD}, 1999)}, pages 291--299. ACM,
  New York, 1999.

\bibitem[DK01]{DK01:Fast-Monte-Carlo}
P.~Drineas and R.~Kannan.
\newblock Fast {M}onte {C}arlo algorithms for approximate matrix
  multiplication.
\newblock In {\em Proc. 42nd IEEE Symp. Foundations of Computer Science
  (FOCS)}, pages 452--259, 2001.

\bibitem[DKM06]{DKM06:Fast-Monte-Carlo-I}
P.~Drineas, R.~Kannan, and M.~W. Mahoney.
\newblock Fast {M}onte {C}arlo algorithms for matrices. {I}. {A}pproximating
  matrix multiplication.
\newblock {\em SIAM J. Comput.}, 36(1):132--157, 2006.

\bibitem[DM05]{DM05:Nystrom-Method}
P.~Drineas and M.~Mahoney.
\newblock On the {N}ystr{\"o}m method for approximating a {G}ram matrix for
  improved kernel-based learning.
\newblock {\em J. Mach. Learn. Res.}, 6:2153--2175, 2005.

\bibitem[Don06]{Don06:Compressed-Sensing}
D.~L. Donoho.
\newblock Compressed sensing.
\newblock {\em IEEE Trans. Inform. Theory}, 52(4):1289--1306, Apr. 2006.

\bibitem[DS02]{DS02:Local-Operator}
K.~R. Davidson and S.~J. Szarek.
\newblock Local operator theory, random matrices, and {B}anach spaces.
\newblock In W.~B. Johnson and J.~Lindenstrauss, editors, {\em Handbook of
  Banach Space Geometry}, pages 317--366. Elsevier, Amsterdam, 2002.

\bibitem[DT07]{DT06:Matrix-Nearness}
I.~S. Dhillon and J.~A. Tropp.
\newblock Matrix nearness problems with {B}regman divergences.
\newblock {\em SIAM J. Matrix Anal. Appl.}, 29(4):1120--1146, 2007.

\bibitem[DZ11]{DZ11:Note-Elementwise}
P.~Drineas and A.~Zouzias.
\newblock A note on element-wise matrix sparsification via a matrix-valued
  {B}ernstein inequality.
\newblock {\em Inform. Process. Lett.}, 111(8):385--389, 2011.

\bibitem[Eff09]{Eff09:Matrix-Convexity}
E.~G. Effros.
\newblock A matrix convexity approach to some celebrated quantum inequalities.
\newblock {\em Proc. Natl. Acad. Sci. USA}, 106(4):1006--1008, Jan. 2009.

\bibitem[ENG11]{ENG11:Perspectives-Matrix}
A.~Ebadian, I.~Nikoufar, and M.~E. Gordji.
\newblock Perspectives of matrix convex functions.
\newblock {\em Proc. Natl. Acad. Sci. USA}, 108(18):7313--7314, 2011.

\bibitem[Eps73]{Eps73:Remarks-Two}
H.~Epstein.
\newblock Remarks on two theorems of {E.} {L}ieb.
\newblock {\em Comm. Math. Phys.}, 31:317--325, 1973.

\bibitem[ER60]{ER60:Evolution-Random}
P.~Erd{\H{o}}s and A.~R{\'e}nyi.
\newblock On the evolution of random graphs.
\newblock {\em Magyar Tud. Akad. Mat. Kutat\'o Int. K\"ozl.}, 5:17--61, 1960.

\bibitem[Fel68]{Fel68:Introduction-Probability-I}
W.~Feller.
\newblock {\em An introduction to probability theory and its applications.
  {V}ol. {I}}.
\newblock Third edition. John Wiley \& Sons, Inc., New York-London-Sydney,
  1968.

\bibitem[Fel71]{Fel71:Introduction-Probability-II}
W.~Feller.
\newblock {\em An introduction to probability theory and its applications.
  {V}ol. {II}.}
\newblock Second edition. John Wiley \& Sons, Inc., New York-London-Sydney,
  1971.

\bibitem[Fer75]{Fer75:Regularite-Trajectoires}
X.~Fernique.
\newblock Regularit\'e des trajectoires des fonctions al\'eatoires gaussiennes.
\newblock In {\em \'{E}cole d'\'{E}t\'e de {P}robabilit\'es de {S}aint-{F}lour,
  {IV}-1974}, pages 1--96. Lecture Notes in Math., Vol. 480. Springer, Berlin,
  1975.

\bibitem[FKV98]{FKV98:Fast-Monte-Carlo}
A.~Frieze, R.~Kannan, and S.~Vempala.
\newblock Fast {M}onte {C}arlo algorithms for finding low-rank approximations.
\newblock In {\em Proc. 39th Ann. IEEE Symp. Foundations of Computer Science
  (FOCS)}, pages 370--378, 1998.

\bibitem[For10]{For10:Log-Gases}
P.~J. Forrester.
\newblock {\em Log-gases and random matrices}, volume~34 of {\em London
  Mathematical Society Monographs Series}.
\newblock Princeton University Press, Princeton, NJ, 2010.

\bibitem[FR13]{RF13:Mathematical-Introduction}
S.~Foucart and H.~Rauhut.
\newblock {\em A mathematical introduction to compressive sensing}.
\newblock Applied and Numerical Harmonic Analysis. Birkh\"auser/Springer, New
  York, 2013.

\bibitem[Fre75]{Fre75:Tail-Probabilities}
D.~A. Freedman.
\newblock On tail probabilities for martingales.
\newblock {\em Ann. Probab.}, 3(1):100--118, Feb. 1975.

\bibitem[GGI{\etalchar{+}}02]{GGIMS02:Near-Optimal-Sparse}
A.~C. Gilbert, S.~Guha, P.~Indyk, S.~Muthukrishnan, and M.~Strauss.
\newblock Near-optimal sparse {F}ourier representations via sampling.
\newblock In {\em Proceedings of the {T}hirty-{F}ourth {A}nnual {ACM}
  {S}ymposium on {T}heory of {C}omputing}, pages 152--161. ACM, New York, 2002.

\bibitem[GM14]{GM13:Revisiting-Nystrom}
A.~Gittens and M.~Mahoney.
\newblock Revisiting the {N}ystr{\"o}m method for improved large-scale machine
  learning.
\newblock {\em J. Mach. Learn. Res.}, 2014.
\newblock To appear. Preprint available at
  \url{http://arXiv.org/abs/1303.1849}.

\bibitem[GN]{GN10:Note-Sampling}
D.~Gross and V.~Nesme.
\newblock Note on sampling without replacing from a finite collection of
  matrices.
\newblock Available at \url{http://arXiv.org/abs/1001.2738}.

\bibitem[Gor85]{Gor85:Some-Inequalities}
Y.~Gordon.
\newblock Some inequalities for {G}aussian processes and applications.
\newblock {\em Israel J. Math.}, 50(4):265--289, 1985.

\bibitem[GR01]{GR01:Algebraic-Graph}
C.~Godsil and G.~Royle.
\newblock {\em Algebraic Graph Theory}.
\newblock Number 207 in Graduate Texts in Mathematics. Springer, 2001.

\bibitem[Grc11]{Grc11:John-von-Neumanns}
J.~F. Grcar.
\newblock John von {N}eumann's analysis of {G}aussian elimination and the
  origins of modern numerical analysis.
\newblock {\em SIAM Rev.}, 53(4):607--682, 2011.

\bibitem[Gro11]{Gro11:Recovering-Low-Rank}
D.~Gross.
\newblock Recovering low-rank matrices from few coefficients in any basis.
\newblock {\em IEEE Trans. Inform. Theory}, 57(3):1548--1566, Mar. 2011.

\bibitem[GS01]{GS01:Probability-Random}
G.~R. Grimmett and D.~R. Stirzaker.
\newblock {\em Probability and random processes}.
\newblock Oxford University Press, New York, third edition, 2001.

\bibitem[GT09]{GT09:Error-Bounds}
A.~Gittens and J.~A. Tropp.
\newblock Error bounds for random matrix approximation schemes.
\newblock ACM Report 2014-01, California Inst. Tech., Nov. 2009.
\newblock Available at \url{http://arXiv.org/abs/0911.4108}.

\bibitem[GT14]{GT11:Tail-Bounds}
A.~Gittens and J.~A. Tropp.
\newblock {Tail bounds for all eigenvalues of a sum of random matrices}.
\newblock ACM Report 2014-02, California Inst. Tech., 2014.
\newblock Available at \url{http://arXiv.org/abs/1104.4513}.

\bibitem[GvN51]{NG51:Numerical-Inverting-II}
H.~H. Goldstine and J.~von Neumann.
\newblock Numerical inverting of matrices of high order. {II}.
\newblock {\em Proc. Amer. Math. Soc.}, 2:188--202, 1951.

\bibitem[Has09]{Has09:Superadditivity-Communication}
M.~B. Hastings.
\newblock Superadditivity of communication complexity using entangled inputs.
\newblock {\em Nature Phys.}, 5:255--257, 2009.

\bibitem[Hig08]{Hig08:Functions-Matrices}
N.~J. Higham.
\newblock {\em Functions of Matrices: Theory and Computation}.
\newblock Society for Industrial and Applied Mathematics, Philadelphia, PA,
  2008.

\bibitem[HJ94]{HJ94:Topics-Matrix}
R.~A. Horn and C.~R. Johnson.
\newblock {\em Topics in matrix analysis}.
\newblock Cambridge University Press, Cambridge, 1994.
\newblock Corrected reprint of the 1991 original.

\bibitem[HJ13]{HJ13:Matrix-Analysis}
R.~A. Horn and C.~R. Johnson.
\newblock {\em Matrix Analysis}.
\newblock Cambridge Univ. Press, 2nd edition, 2013.

\bibitem[HKZ12]{HKZ12:Tail-Inequalities}
D.~Hsu, S.~M. Kakade, and T.~Zhang.
\newblock Tail inequalities for sums of random matrices that depend on the
  intrinsic dimension.
\newblock {\em Electron. Commun. Probab.}, 17:no. 14, 13, 2012.

\bibitem[HLL83]{HLL83:Stochastic-Block}
P.~W. Holland, K.~B. Laskey, and S.~Leinhardt.
\newblock Stochastic blockmodels: First steps.
\newblock {\em Social Networks}, 5(2):109--137, 1983.

\bibitem[HMT11]{HMT11:Finding-Structure}
N.~Halko, P.-G. Martinsson, and J.~A. Tropp.
\newblock Finding structure with randomness: {S}tochastic algorithms for
  constructing approximate matrix decompositions.
\newblock {\em SIAM Rev.}, 53(2):217--288, June 2011.

\bibitem[HP82]{HP82:Jensens-Inequality}
F.~Hansen and G.~K. Pedersen.
\newblock Jensen's inequality for operators and {L}\"owner's theorem.
\newblock {\em Math. Ann.}, 258(3):229--241, 1982.

\bibitem[HP03]{HP03:Jensens-Operator}
F.~Hansen and G.~K. Pedersen.
\newblock Jensen's operator inequality.
\newblock {\em Bull. London Math. Soc.}, 35(4):553--564, 2003.

\bibitem[HP14]{HP14:Introduction-Matrix}
F.~Hiai and D.~Petz.
\newblock {\em Introduction to Matrix Analysis and Applications}.
\newblock Springer, Feb. 2014.

\bibitem[HW08]{HW08:Counterexamples-Maximal}
P.~Hayden and A.~Winter.
\newblock Counterexamples to the maximal {$p$}-norm multiplicity conjecture for
  all {$p>1$}.
\newblock {\em Comm. Math. Phys.}, 284(1):263--280, 2008.

\bibitem[HXGD14]{HXGD14:Compact-Random}
R.~Hamid, Y.~Xiao, A.~Gittens, and D.~DeCoste.
\newblock Compact random feature maps.
\newblock In {\em Proc. 31st Intl. Conf. Machine Learning}, Beijing, July 2014.

\bibitem[JL84]{JL84:Extensions-Lipschitz}
W.~B. Johnson and J.~Lindenstrauss.
\newblock Extensions of {L}ipschitz mappings into a {H}ilbert space.
\newblock In {\em Conference in modern analysis and probability ({N}ew {H}aven,
  {C}onn., 1982)}, volume~26 of {\em Contemp. Math.}, pages 189--206. Amer.
  Math. Soc., Providence, RI, 1984.

\bibitem[JX03]{JX03:Noncommutative-Burkholder}
M.~Junge and Q.~Xu.
\newblock Noncommutative {B}urkholder/{R}osenthal inequalities.
\newblock {\em Ann. Probab.}, 31(2):948--995, 2003.

\bibitem[JX05]{JX05:Best-Constants}
M.~Junge and Q.~Xu.
\newblock On the best constants in some non-commutative martingale
  inequalities.
\newblock {\em Bull. London Math. Soc.}, 37:243--253, 2005.

\bibitem[JX08]{JX08:Noncommutative-Burkholder-II}
M.~Junge and Q.~Xu.
\newblock Noncommutative {B}urkholder/{R}osenthal inequalities {II}:
  {A}pplications.
\newblock {\em Israel J. Math.}, 167:227--282, 2008.

\bibitem[JZ12]{JZ12:Noncommutative-Martingale}
M.~Junge and Q.~Zeng.
\newblock Noncommutative martingale deviation and {P}oincar{\'e} type
  inequalities with applications.
\newblock Available at \url{http://arXiv.org/abs/1211.3209}, Nov. 2012.

\bibitem[JZ13]{JZ13:Noncommutative-Bennett}
M.~Junge and Q.~Zeng.
\newblock Noncommutative {B}ennett and {R}osenthal inequalities.
\newblock {\em Ann. Probab.}, 41(6):4287--4316, 2013.

\bibitem[KA80]{KA80:Means-Positive}
F.~Kubo and T.~Ando.
\newblock Means of positive linear operators.
\newblock {\em Math. Ann.}, 246(3):205--224, 1979/80.

\bibitem[KD14]{KD14:Note-Randomized}
A.~Kundu and P.~Drineas.
\newblock A note on randomized element-wise matrix sparsification.
\newblock Available at \url{http://arXiv.org/abs/1404.0320}, Apr. 2014.

\bibitem[Kem13]{Kem13:Introduction-Random}
T.~Kemp.
\newblock Math 247a: Introduction to random matrix theory.
\newblock Available at
  \url{http://www.math.ucsd.edu/~tkemp/247A/247A.Notes.pdf}, 2013.

\bibitem[KK12]{KK12:Random-Feature}
P.~Kar and H.~Karnick.
\newblock Random feature maps for dot product kernels.
\newblock In {\em Proc. 15th Intl. Conf. Artificial Intelligence and Statistics
  (AISTATS)}, 2012.

\bibitem[KM13]{KM13:Bounding-Smallest}
V.~Koltchinskii and S.~Mendelson.
\newblock Bounding the smallest singular value of a random matrix without
  concentration.
\newblock Available at \url{http://arXiv.org/abs/1312.3580}, Dec. 2013.

\bibitem[Kol11]{Kol11:Oracle-Inequalities}
V.~Koltchinskii.
\newblock {\em Oracle inequalities in empirical risk minimization and sparse
  recovery problems}, volume 2033 of {\em Lecture Notes in Mathematics}.
\newblock Springer, Heidelberg, 2011.
\newblock Lectures from the 38th Probability Summer School held in Saint-Flour,
  2008, {\'E}cole d'{\'E}t{\'e} de Probabilit{\'e}s de Saint-Flour.
  [Saint-Flour Probability Summer School].

\bibitem[Kra36]{Kra36:Ueber-Konvexe}
F.~Kraus.
\newblock \"{U}ber konvexe {M}atrixfunktionen.
\newblock {\em Math. Z.}, 41(1):18--42, 1936.

\bibitem[KT94]{KT94:Some-Remarks}
B.~Kashin and L.~Tzafriri.
\newblock Some remarks on coordinate restriction of operators to coordinate
  subspaces.
\newblock Insitute of Mathematics Preprint~12, Hebrew University, Jerusalem,
  1993--1994.

\bibitem[Lat05]{Lat05:Some-Estimates}
R.~Lata{\l}a.
\newblock Some estimates of norms of random matrices.
\newblock {\em Proc. Amer. Math. Soc.}, 133(5):1273--1282, 2005.

\bibitem[LBW96]{LBW96:Efficient-Agnostic}
W.~S. Lee, P.~L. Bartlett, and R.~C. Williamson.
\newblock Efficient agnostic learning of neural networks with bounded fan-in.
\newblock {\em IEEE Trans. Inform. Theory}, 42(6):2118--2132, Nov. 1996.

\bibitem[Lie73]{Lie73:Convex-Trace}
E.~H. Lieb.
\newblock Convex trace functions and the {W}igner--{Y}anase--{D}yson
  conjecture.
\newblock {\em Adv. Math.}, 11:267--288, 1973.

\bibitem[Lin73]{Lin73:Entropy-Information}
G.~Lindblad.
\newblock Entropy, information and quantum measurements.
\newblock {\em Comm. Math. Phys.}, 33:305--322, 1973.

\bibitem[LLR95]{LLR95:Geometry-Graphs}
N.~Linial, E.~London, and Y.~Rabinovich.
\newblock The geometry of graphs and some of its algorithmic applications.
\newblock {\em Combinatorica}, 15(2):215--245, 1995.

\bibitem[L{\"o}w34]{Loe34:Ueber-Monotone}
K.~L{\"o}wner.
\newblock \"{U}ber monotone {M}atrixfunktionen.
\newblock {\em Math. Z.}, 38(1):177--216, 1934.

\bibitem[LP86]{LP86:Inegalites-Khintchine}
F.~Lust-Piquard.
\newblock In{\'e}galit{\'e}s de {K}hintchine dans {$C_p$ $(1 < p < \infty)$}.
\newblock {\em C. R. Math. Acad. Sci. Paris}, 303(7):289--292, 1986.

\bibitem[LPP91]{LPP91:Noncommutative-Khintchine}
F.~Lust-Piquard and G.~Pisier.
\newblock Noncommutative {K}hintchine and {P}aley inequalities.
\newblock {\em Ark. Mat.}, 29(2):241--260, 1991.

\bibitem[LPSS{\etalchar{+}}14]{LSS+14:Randomized-Nonlinear}
D.~Lopez-Paz, S.~Sra, A.~Smola, Z.~Ghahramani, and B.~Sch{\"o}lkopf.
\newblock Randomized nonlinear component analysis.
\newblock In {\em Proc. 31st Intl. Conf. Machine Learning}, Beijing, July 2014.

\bibitem[LS05]{LS05:Stronger-Subadditivity}
E.~H. Lieb and R.~Seiringer.
\newblock Stronger subadditivity of entropy.
\newblock {\em Phys. Rev. A}, 71:062329--1--9, 2005.

\bibitem[LT91]{LT91:Probability-Banach}
M.~Ledoux and M.~Talagrand.
\newblock {\em Probability in Banach Spaces: Isoperimetry and Processes}.
\newblock Springer, Berlin, 1991.

\bibitem[Lug09]{Lug09:Concentration-Measure}
G.~Lugosi.
\newblock Concentration-of-measure inequalities.
\newblock Available at \url{http://www.econ.upf.edu/~lugosi/anu.pdf}, 2009.

\bibitem[Mah11]{Mah11:Randomized-Algorithms}
M.~Mahoney.
\newblock Randomized algorithms for matrices and data.
\newblock {\em Found. Trends Mach. Learning}, 3(2):123--224, Feb. 2011.

\bibitem[Mau03]{Mau03:Bound-Deviation}
A.~Maurer.
\newblock A bound on the deviation probability for sums of non-negative random
  variables.
\newblock {\em JIPAM. J. Inequal. Pure Appl. Math.}, 4(1):Article 15, 6 pp.
  (electronic), 2003.

\bibitem[Mec07]{Mec07:Spectral-Norm}
M.~W. Meckes.
\newblock On the spectral norm of a random {T}oeplitz matrix.
\newblock {\em Electron. Comm. Probab.}, 12:315--325 (electronic), 2007.

\bibitem[Meh04]{Meh04:Random-Matrices}
M.~L. Mehta.
\newblock {\em Random matrices}, volume 142 of {\em Pure and Applied
  Mathematics (Amsterdam)}.
\newblock Elsevier/Academic Press, Amsterdam, third edition, 2004.

\bibitem[Mil71]{Mil71:New-Proof}
V.~D. Milman.
\newblock A new proof of {A}. {D}voretzky's theorem on cross-sections of convex
  bodies.
\newblock {\em Funkcional. Anal. i Prilo\v zen.}, 5(4):28--37, 1971.

\bibitem[Min11]{Min11:Some-Extensions}
S.~Minsker.
\newblock Some extensions of {B}ernstein's inequality for self-adjoint
  operators.
\newblock Available at \url{http://arXiv.org/abs/1112.5448}, Nov. 2011.

\bibitem[MJC{\etalchar{+}}14]{MJCFT12:Matrix-Concentration}
L.~Mackey, M.~I. Jordan, R.~Y. Chen, B.~Farrell, and J.~A. Tropp.
\newblock Matrix concentration inequalities via the method of exchangable
  pairs.
\newblock {\em Ann. Probab.}, 42(3):906--945, 2014.
\newblock Preprint available at \url{http://arXiv.org/abs/1201.6002}.

\bibitem[MKB79]{MKB79:Multivariate-Analysis}
K.~V. Mardia, J.~T. Kent, and J.~M. Bibby.
\newblock {\em Multivariate analysis}.
\newblock Academic Press [Harcourt Brace Jovanovich, Publishers], London-New
  York-Toronto, Ont., 1979.
\newblock Probability and Mathematical Statistics: A Series of Monographs and
  Textbooks.

\bibitem[Mon73]{Mon73:Pair-Correlation}
H.~L. Montgomery.
\newblock The pair correlation of zeros of the zeta function.
\newblock In {\em Analytic number theory ({P}roc. {S}ympos. {P}ure {M}ath.,
  {V}ol. {XXIV}, {S}t. {L}ouis {U}niv., {S}t. {L}ouis, {M}o., 1972)}, pages
  181--193. Amer. Math. Soc., Providence, R.I., 1973.

\bibitem[MP67]{MP67:Distribution-Eigenvalues}
V.~A. Mar{\v{c}}enko and L.~A. Pastur.
\newblock Distribution of eigenvalues in certain sets of random matrices.
\newblock {\em Mat. Sb. (N.S.)}, 72 (114):507--536, 1967.

\bibitem[MR95]{MR95:Randomized-Algorithms}
R.~Motwani and P.~Raghavan.
\newblock {\em Randomized Algorithms}.
\newblock Cambridge Univ. Press, Cambridge, 1995.

\bibitem[MSS14]{MSS13:Interlacing-Families-II}
A.~Marcus, D.~A. Spielman, and N.~Srivastava.
\newblock Interlacing families {II}: {M}ixed characteristic polynomials and the
  {K}adison--{S}inger problem.
\newblock {\em Ann. Math.}, June 2014.
\newblock To appear. Preprint available at
  \url{http://arXiv.org/abs/1306.3969}.

\bibitem[MT13]{MT13:Achievable-Performance}
M.~B. McCoy and J.~A. Tropp.
\newblock The achievable performance of convex demixing.
\newblock Available at \url{http://arXiv.org/abs/1309.7478}, Sep. 2013.

\bibitem[MT14]{MT12:Sharp-Recovery}
M.~McCoy and J.~A. Tropp.
\newblock Sharp recovery thresholds for convex deconvolution, with
  applications.
\newblock {\em Found. Comput. Math.}, Apr. 2014.
\newblock Preprint available at \url{http://arXiv.org/abs/1205.1580}.

\bibitem[Mui82]{Mui82:Aspects-Multivariate}
R.~J. Muirhead.
\newblock {\em Aspects of multivariate statistical theory}.
\newblock John Wiley \& Sons Inc., New York, 1982.
\newblock Wiley Series in Probability and Mathematical Statistics.

\bibitem[MZ11]{MZ11:Low-Rank-Matrix-Valued}
A.~Magen and A.~Zouzias.
\newblock Low rank matrix-valued {C}hernoff bounds and approximate matrix
  multiplication.
\newblock In {\em Proceedings of the {T}wenty-{S}econd {A}nnual {ACM}-{SIAM}
  {S}ymposium on {D}iscrete {A}lgorithms}, pages 1422--1436. SIAM,
  Philadelphia, PA, 2011.

\bibitem[Nem07]{Nem07:Sums-Random}
A.~Nemirovski.
\newblock Sums of random symmetric matrices and quadratic optimization under
  orthogonality constraints.
\newblock {\em Math. Prog. Ser. B}, 109:283--317, 2007.

\bibitem[NRV13]{NRV13:Efficient-Rounding}
A.~Naor, O.~Regev, and T.~Vidick.
\newblock Efficient rounding for the noncommutative {G}rothendieck inequality
  (extended abstract).
\newblock In {\em S{TOC}'13---{P}roceedings of the 2013 {ACM} {S}ymposium on
  {T}heory of {C}omputing}, pages 71--80. ACM, New York, 2013.

\bibitem[NS06]{NS06:Lectures-Combinatorics}
A.~Nica and R.~Speicher.
\newblock {\em Lectures on the combinatorics of free probability}, volume 335
  of {\em London Mathematical Society Lecture Note Series}.
\newblock Cambridge University Press, Cambridge, 2006.

\bibitem[NT14]{NT12:Paved-Good}
D.~Needell and J.~A. Tropp.
\newblock Paved with good intentions: analysis of a randomized block {K}aczmarz
  method.
\newblock {\em Linear Algebra Appl.}, 441:199--221, 2014.

\bibitem[Oli10a]{Oli10:Concentration-Adjacency}
R.~I. Oliveira.
\newblock Concentration of the adjacency matrix and of the {L}aplacian in
  random graphs with independent edges.
\newblock Available at \url{http://arXiv.org/abs/0911.0600}, Feb. 2010.

\bibitem[Oli10b]{Oli10:Sums-Random}
R.~I. Oliveira.
\newblock Sums of random {H}ermitian matrices and an inequality by {R}udelson.
\newblock {\em Electron. Commun. Probab.}, 15:203--212, 2010.

\bibitem[Oli11]{Oli11:Spectrum-Random}
R.~I. Oliveira.
\newblock The spectrum of random $k$-lifts of large graphs (with possibly large
  $k$).
\newblock {\em J. Combinatorics}, 1(3/4):285--306, 2011.

\bibitem[Oli13]{Oli13:Lower-Tail}
R.~I. Oliveira.
\newblock The lower tail of random quadratic forms, with applications to
  ordinary least squares and restricted eigenvalue properties.
\newblock Available at \url{http://arXiv.org/abs/1312.2903}, Dec. 2013.

\bibitem[Par98]{Par98:Symmetric-Eigenvalue}
B.~N. Parlett.
\newblock {\em The symmetric eigenvalue problem}, volume~20 of {\em Classics in
  Applied Mathematics}.
\newblock Society for Industrial and Applied Mathematics (SIAM), Philadelphia,
  PA, 1998.
\newblock Corrected reprint of the 1980 original.

\bibitem[Pet86]{Pet86:Quasi-Entropies}
D.~Petz.
\newblock Quasi-entropies for finite quantum systems.
\newblock {\em Rep. Math. Phys.}, 23(1):57--65, 1986.

\bibitem[Pet94]{Pet94:Survey-Certain}
D.~Petz.
\newblock A survey of certain trace inequalities.
\newblock In {\em Functional analysis and operator theory ({W}arsaw, 1992)},
  volume~30 of {\em Banach Center Publ.}, pages 287--298. Polish Acad. Sci.,
  Warsaw, 1994.

\bibitem[Pet10]{Pet10:From-f-Divergence}
D.~Petz.
\newblock From {$f$}-divergence to quantum quasi-entropies and their use.
\newblock {\em Entropy}, 12(3):304--325, 2010.

\bibitem[Pet11]{Pet11:Matrix-Analysis}
D.~Petz.
\newblock Matrix analysis with some applications.
\newblock Available at \url{bolyai.cs.elte.hu/~petz/matrixbme.pdf}, Feb. 2011.

\bibitem[Pin94]{Pin94:Optimum-Bounds}
I.~Pinelis.
\newblock Optimum bounds for the distributions of martingales in {B}anach
  spaces.
\newblock {\em Ann. Probab.}, 22(4):1679--1706, 1994.

\bibitem[Pis81]{Pis81:Remarques-Resultat}
G.~Pisier.
\newblock Remarques sur un r\'esultat non publi\'e de {B}. {M}aurey.
\newblock In {\em Seminar on {F}unctional {A}nalysis, 1980--1981}, pages Exp.
  No. V, 13. \'Ecole Polytech., Palaiseau, 1981.

\bibitem[Pis89]{Pis89:Volume-Convex}
G.~Pisier.
\newblock {\em The volume of convex bodies and {B}anach space geometry},
  volume~94 of {\em Cambridge Tracts in Mathematics}.
\newblock Cambridge University Press, Cambridge, 1989.

\bibitem[PMT14]{PMT14:Efron-Stein}
D.~Paulin, L.~Mackey, and J.~A. Tropp.
\newblock {E}fron--{S}tein inequalities for random matrices.
\newblock Available at \url{http://arXiv.org/abs/1408.3470}, Aug. 2014.

\bibitem[PW75]{PW75:Functional-Calculus}
W.~Pusz and S.~L. Woronowicz.
\newblock Functional calculus for sesquilinear forms and the purification map.
\newblock {\em Rep. Mathematical Phys.}, 8(2):159--170, 1975.

\bibitem[Rec11]{Rec11:Simpler-Approach}
B.~Recht.
\newblock A simpler approach to matrix completion.
\newblock {\em J. Mach. Learn. Res.}, 12:3413--3430, 2011.

\bibitem[Ros70]{Ros70:Subspaces-Lp}
H.~P. Rosenthal.
\newblock On subspaces of {$L_p$ $(p > 2)$} spanned by sequences of independent
  random variables.
\newblock {\em Israel J. Math.}, 8:273--303, 1970.

\bibitem[RR07]{RR07:Random-Features}
A.~Rahimi and B.~Recht.
\newblock Random features for large-scale kernel machines.
\newblock In {\em Advances in Neural Information Processing Systems 20}, pages
  1177--1184, Vancouver, Dec. 2007.

\bibitem[RR08]{RR08:Weighted-Sums}
A.~Rahimi and B.~Recht.
\newblock Weighted sums of random kitchen sinks: replacing minimization with
  randomization in learning.
\newblock In {\em Advances in Neural Information Processing Systems 21}, 2008.

\bibitem[RS13]{RS13:Expectation-Norm}
S.~Riemer and C.~Sch{\"u}tt.
\newblock On the expectation of the norm of random matrices with
  non-identically distributed entries.
\newblock {\em Electron. J. Probab.}, 18:no. 29, 13, 2013.

\bibitem[Rud99]{Rud99:Random-Vectors}
M.~Rudelson.
\newblock Random vectors in the isotropic position.
\newblock {\em J. Funct. Anal.}, 164:60--72, 1999.

\bibitem[Rus02]{Rus02:Inequalities-Quantum}
M.~B. Ruskai.
\newblock Inequalities for quantum entropy: {A} review with conditions for
  equality.
\newblock {\em J. Math. Phys.}, 43(9):4358--4375, Sep. 2002.

\bibitem[Rus05]{Rus05:Erratum-Inequalities}
M.~B. Ruskai.
\newblock Erratum: {I}nequalities for quantum entropy: {A} review with
  conditions for equality [\emph{{J}. {M}ath. {P}hys.} 43, 4358 (2002)].
\newblock {\em J. Math. Phys.}, 46(1):0199101, 2005.

\bibitem[RV06]{RV06:Sparse-Reconstruction}
M.~Rudelson and R.~Vershynin.
\newblock Sparse reconstruction by convex relaxation: {F}ourier and {G}aussian
  measurements.
\newblock In {\em Proc. 40th Ann. Conf. Information Sciences and Systems
  (CISS)}, Mar. 2006.

\bibitem[RV07]{RV07:Sampling-Large}
M.~Rudelson and R.~Vershynin.
\newblock Sampling from large matrices: {A}n approach through geometric
  functional analysis.
\newblock {\em J. Assoc. Comput. Mach.}, 54(4):Article 21, 19 pp., Jul. 2007.
\newblock (electronic).

\bibitem[RW11]{RW11:Information-Divergence}
M.~D. Reid and R.~C. Williamson.
\newblock Information, divergence and risk for binary experiments.
\newblock {\em J. Mach. Learn. Res.}, 12:731--817, 2011.

\bibitem[Sar06]{Sar06:Improved-Approximation}
T.~Sarl{\'o}s.
\newblock Improved approximation algorithms for large matrices via random
  projections.
\newblock In {\em Proc. 47th Ann. IEEE Symp. Foundations of Computer Science
  (FOCS)}, pages 143--152, 2006.

\bibitem[Seg00]{Seg00:Expected-Norm}
Y.~Seginer.
\newblock The expected norm of random matrices.
\newblock {\em Combin. Probab. Comput.}, 9:149--166, 2000.

\bibitem[Shi96]{Shi96:Probability}
A.~N. Shiryaev.
\newblock {\em Probability}, volume~95 of {\em Graduate Texts in Mathematics}.
\newblock Springer-Verlag, New York, second edition, 1996.
\newblock Translated from the first (1980) Russian edition by R. P. Boas.

\bibitem[So09]{So09:Moment-Inequalities}
A.~M.-C. So.
\newblock Moment inequalities for sums of random matrices and their
  applications in optimization.
\newblock {\em Math. Prog. Ser. A}, Dec. 2009.
\newblock (electronic).

\bibitem[SS98]{SS98:Learning-Kernels}
B.~Sch{\"o}lkopf and S.~Smola.
\newblock {\em Learning with Kernels}.
\newblock MIT Press, 1998.

\bibitem[SSS08]{SS08:Low-l1-Norm}
S.~Shalev-Shwartz and N.~Srebro.
\newblock Low $\ell_1$-norm and guarantees on sparsifiability.
\newblock In {\em ICML/COLT/UAI Sparse Optimization and Variable Selection
  Workshop}, July 2008.

\bibitem[SST06]{SST06:Smoothed-Analysis}
A.~Sankar, D.~A. Spielman, and S.-H. Teng.
\newblock Smoothed analysis of the condition numbers and growth factors of
  matrices.
\newblock {\em SIAM J. Matrix Anal. Appl.}, 28(2):446--476, 2006.

\bibitem[ST04]{ST04:Nearly-Linear-Time}
D.~A. Spielman and S.-H. Teng.
\newblock Nearly-linear time algorithms for graph partitioning, graph
  sparsification, and solving linear systems.
\newblock In {\em Proceedings of the 36th {A}nnual {ACM} {S}ymposium on
  {T}heory of {C}omputing}, pages 81--90 (electronic), New York, 2004. ACM.

\bibitem[Ste72]{Ste72:Bound-Error}
C.~Stein.
\newblock {A bound for the error in the normal approximation to the
  distribution of a sum of dependent random variables.}
\newblock In {\em Proc. 6th Berkeley Symp. Math. Statist. Probab.}, Berkeley,
  1972. Univ. California Press.

\bibitem[SV13]{SV13:Top-Eigenvalue}
A.~Sen and B.~Vir{\'a}g.
\newblock The top eigenvalue of the random {T}oeplitz matrix and the sine
  kernel.
\newblock {\em Ann. Probab.}, 41(6):4050--4079, 2013.

\bibitem[Tao12]{Tao12:Topics-Random}
T.~Tao.
\newblock {\em Topics in random matrix theory}, volume 132 of {\em Graduate
  Studies in Mathematics}.
\newblock American Mathematical Society, Providence, RI, 2012.

\bibitem[Thi02]{Thi02:Quantum-Mathematical}
W.~Thirring.
\newblock {\em Quantum mathematical physics}.
\newblock Springer-Verlag, Berlin, second edition, 2002.
\newblock Atoms, molecules and large systems, Translated from the 1979 and 1980
  German originals by Evans M. Harrell II.

\bibitem[TJ74]{TJ74:Moduli-Smoothness}
N.~Tomczak-Jaegermann.
\newblock The moduli of smoothness and convexity and the {R}ademacher averages
  of trace classes {$S_{p}(1\leq p<\infty )$}.
\newblock {\em Studia Math.}, 50:163--182, 1974.

\bibitem[Tro08a]{Tro08:Conditioning-Random}
J.~A. Tropp.
\newblock On the conditioning of random subdictionaries.
\newblock {\em Appl. Comput. Harmon. Anal.}, 25:1--24, 2008.

\bibitem[Tro08b]{Tro08:Norms-Random}
J.~A. Tropp.
\newblock Norms of random submatrices and sparse approximation.
\newblock {\em C. R. Math. Acad. Sci. Paris}, 346(23-24):1271--1274, 2008.

\bibitem[Tro08c]{Tro08:Random-Paving}
J.~A. Tropp.
\newblock The random paving property for uniformly bounded matrices.
\newblock {\em Studia Math.}, 185(1):67--82, 2008.

\bibitem[Tro11a]{Tro11:Freedmans-Inequality}
J.~A. Tropp.
\newblock {F}reedman's inequality for matrix martingales.
\newblock {\em Electron. Commun. Probab.}, 16:262--270, 2011.

\bibitem[Tro11b]{Tro11:User-Friendly-Martingale-TR}
J.~A. Tropp.
\newblock User-friendly tail bounds for matrix martingales.
\newblock ACM Report 2011-01, California Inst. Tech., Pasadena, CA, Jan. 2011.

\bibitem[{Tro}11c]{Tro11:User-Friendly-FOCM}
J.~A. {Tropp}.
\newblock {User-friendly tail bounds for sums of random matrices}.
\newblock {\em Found. Comput. Math.}, August 2011.

\bibitem[Tro11d]{Tro11:Improved-Analysis}
J.~A. Tropp.
\newblock Improved analysis of the subsampled randomized {H}adamard transform.
\newblock {\em Adv. Adapt. Data Anal.}, 3(1-2):115--126, 2011.

\bibitem[Tro12]{Tro12:Joint-Convexity}
J.~A. Tropp.
\newblock From joint convexity of quantum relative entropy to a concavity
  theorem of {L}ieb.
\newblock {\em Proc. Amer. Math. Soc.}, 140(5):1757--1760, 2012.

\bibitem[Tro14]{Tro14:Convex-Recovery}
J.~A. Tropp.
\newblock Convex recovery of a structured signal from independent random
  measurements.
\newblock In {\em Sampling Theory, a Renaissance}. Birkh\"auser Verlag, 2014.
\newblock To appear. Available at \url{http://arXiv.org/abs/1405.1102}.

\bibitem[TV04]{TV04:Random-Matrix}
A.~M. Tulino and S.~Verd{\'u}.
\newblock {\em Random matrix theory and wireless communications}.
\newblock Number 1(1) in Foundations and Trends in Communications and
  Information Theory. Now Publ., 2004.

\bibitem[Ver12]{Ver12:Introduction-Nonasymptotic}
R.~Vershynin.
\newblock Introduction to the non-asymptotic analysis of random matrices.
\newblock In {\em Compressed sensing}, pages 210--268. Cambridge Univ. Press,
  Cambridge, 2012.

\bibitem[vNG47]{NG47:Numerical-Inverting}
J.~von Neumann and H.~H. Goldstine.
\newblock Numerical inverting of matrices of high order.
\newblock {\em Bull. Amer. Math. Soc.}, 53:1021--1099, 1947.

\bibitem[Wig55]{Wig55:Characteristic-Vectors}
E.~P. Wigner.
\newblock Characteristic vectors of bordered matrices with infinite dimensions.
\newblock {\em Ann. of Math. (2)}, 62:548--564, 1955.

\bibitem[Wis28]{Wis28:Generalised-Product}
J.~Wishart.
\newblock The generalised product moment distribution in samples from a
  multivariate normal population.
\newblock {\em Biometrika}, 20A(1--2):32--52, 1928.

\bibitem[Woo14]{Woo14:Sketching-Tool}
D.~Woodruff.
\newblock Sketching as a tool for numerical linear algebra.
\newblock {\em Found. Trends Theor. Comput. Sci.}, 10(1--2):1--157, 2014.

\bibitem[WS01]{WS01:Nystrom-Method}
C.~K.~I. Williams and M.~Seeger.
\newblock Using the {N}ystr{\"o}m method to spped up kernel machines.
\newblock In {\em Advances in Neural Information Processing Systems 13}, pages
  682--688, Vancouver, 2001.

\bibitem[Zou13]{Zou12:Randomized-Primitives}
A.~Zouzias.
\newblock {\em Randomized primitives for linear algebra and applications}.
\newblock PhD thesis, Univ. Toronto, 2013.

\end{thebibliography}

{\small
\newcommand{\etalchar}[1]{$^{#1}$}

}

\end{document}